\documentclass[a4paper,11pt]{report}
\usepackage[utf8]{inputenc}
\usepackage[english]{babel}
\usepackage{amsthm, amsfonts,amsmath,amscd,amssymb,epsfig,stmaryrd,psfrag, mathtools}
\mathtoolsset{showonlyrefs}
\usepackage{enumitem}
\setlist{topsep = .5em,itemsep = 0.2em}

\usepackage{enumerate,bbm,array}
\usepackage{tikz} 
\usepackage{yhmath} 

\usepackage[active]{srcltx}
\usepackage[T1]{fontenc}
\usepackage{graphicx}

\usepackage{rotating}

\usepackage{xcolor}
\usepackage{framed}

\usepackage{pifont} 

\usepackage{stmaryrd}
\usepackage{nicefrac}

\usepackage{comment}
\usepackage{accents}

\usepackage{longtable}
\usepackage[margin =2.8cm]{geometry}
\usepackage{array}

\usepackage{hyperref}
\hypersetup{
	colorlinks=true, 
	linktoc=all,     
	linkcolor=blue  
}

\usepackage{centernot}

\theoremstyle{plain} \newtheorem{theorem}{Theorem}[chapter]
\theoremstyle{plain} \newtheorem{corollary}[theorem]{Corollary}
\theoremstyle{plain} \newtheorem{lemma}[theorem]{Lemma}
\theoremstyle{plain} \newtheorem{proposition}[theorem]{Proposition}
\theoremstyle{remark}\newtheorem{remark}[theorem]{Remark}
\theoremstyle{plain} \newtheorem{conjecture}[theorem]{Conjecture}
\theoremstyle{plain} \newtheorem{fact}[theorem]{Fact}
\theoremstyle{plain} 
\theoremstyle{plain} \newtheorem{Exc}{Exercise}[chapter]
\newenvironment{exo}{\vspace{.2cm}\begin{Exc}\normalfont}{\end{Exc} \vspace{.05cm}}

\theoremstyle{definition} \newtheorem{definition}[theorem]{Definition}

\usepackage{mdframed}
\mdfdefinestyle{MyFrame}{%
    linecolor=black,
    outerlinewidth=3pt,
    innertopmargin=\baselineskip,
    innerbottommargin=\baselineskip,
    innerrightmargin=2pt,
    innerleftmargin=2pt,
    backgroundcolor=gray!10!white}

\newcommand{\calC}{\mathcal{C}}
\newcommand{\calD}{\mathcal{D}}
\newcommand{\calE}{\mathcal{E}}
\newcommand{\calF}{\mathcal{F}}
\newcommand{\calG}{\mathcal{G}}
\newcommand{\calH}{\mathcal{H}}

\newcommand{\calL}{\mathcal{L}}

\newcommand{\calQ}{\mathcal{Q}}
\newcommand{\calR}{\mathcal{R}}

\newcommand{\calW}{\mathcal{W}}

\newcommand{\bfS}{\mathbf{S}}
\newcommand{\bfT}{\mathbf{T}}

\newcommand{\bbC}{\mathbb{C}}

\newcommand{\bbE}{\mathbb{E}}

\newcommand{\bbL}{\mathbb{L}}

\newcommand{\bbN}{\mathbb{N}}

\newcommand{\bbP}{\mathbb{P}}

\newcommand{\bbR}{\mathbb{R}}
\newcommand{\bbS}{\mathbb{S}}
\newcommand{\bbT}{\mathbb{T}}

\newcommand{\bbZ}{\mathbb{Z}}

\newcommand{\sfn}{{\sf n}}

\newcommand{\sfC}{{\sf C}}
\newcommand{\sfD}{{\sf D}}
\newcommand{\sfN}{{\sf N}}

\newcommand{\ep}{\varepsilon}
\newcommand{\eps}{\varepsilon}
\newcommand{\la}{\lambda}
\newcommand{\La}{\Lambda}
\newcommand{\ind}{\boldsymbol 1}
\newcommand{\Rect}{{\rm Rect}}
\newcommand{\balpha}{\boldsymbol \alpha}

\newcommand\lra{\leftrightarrow}
\newcommand\xlra{\xleftrightarrow}
\newcommand\concel[2]{\ooalign{$\hfil#1\mkern0mu/\hfil$\crcr$#1#2$}}  
\newcommand\nxlra[1]{\mathrel{\mathpalette\concel{\xleftrightarrow{#1}}}}

\renewcommand{\comment}[1]{}

\newcommand{\ol}{{\includegraphics[scale=0.2]{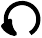}}}

\title{Exploring the phase transition of planar FK-percolation}
\author{Ioan Manolescu\addtocounter{footnote}{0}\thanks{University of Fribourg}}

\begin{document}
\maketitle
\setcounter{tocdepth}{1}
\tableofcontents

\begin{abstract}
These lecture notes accompany a mini-course given by the author at the 2023 CIME summer school ``Statistical Mechanics and Stochastic PDEs''
in Cetraro, Italy. 

	The aim of these notes is to give a quick introduction to FK-percolation (also called random-cluster model), focusing on certain recent results about the phase transition of the two dimensional model, namely its continuity or discontinuity depending on the cluster weight~$q$, and the asymptotic rotational invariance of the critical phase (when the phase transition is continuous). 
	As such, the main focus is on FK-percolation on~$\bbZ^2$ with~$q \geq 1$, but we do mention some important results valid for general dimension. 
	We purposefully avoid results specific to the case $q=2$ (i.e. the FK-Ising model), and focus on generic values of $q$.
	To favour quick access to recent results, the style is minimal, with certain proofs omitted or left as exercises. 
	
	As a {\em mise en bouche}, the first chapter treats Bernoulli percolation (also called i.i.d.\ percolation, corresponding to the case~$q = 1$), which is arguably the easiest of the FK-percolation models to define and study. We could not resist the temptation to include here the very elegant proof of the sharpness of the phase transition in general dimension due to Duminil-Copin and Tassion~\cite{DumTas16}. The role of this part is also to highlight the several levels of understanding of the phase transition of a statistical mechanics models: non-triviality, sharpness, order of the phase transition, fine properties of the critical phase. The following chapters discuss these topics for the two-dimensional FK-percolation model.  
	
	Chapter~\ref{ch:2intro_FK} is a quick introduction to FK-percolation (in general dimension), presenting its basic properties and a summary of results not contained in this work. 
	
	The following chapters are specific to two dimensions. Chapter~\ref{ch:3dichotomy} contains a celebrated result by Duminil-Copin, Sidoravicius and Tassion~\cite{DumSidTas17} which establishes a {\em dichotomy} between two types of phase transition: continuous and discontinuous. The identification of the critical point and the sharpness of the phase transition follow easily. 
	
	In Chapter~\ref{ch:4cont_v_discont}, we identify the type of phase transition depending on the cluster-weight~$q$: continuous for~$1 \leq q\leq 4$ and discontinuous for~$q >4$. The strategy employed here is to relate FK-percolation to the six-vertex model via the Baxter--Kelland--Wu (BKW) correspondence, then to compute certain quantities in the six-vertex model using its transfer-matrix representation. Eigenvalues of the transfer matrix may be estimated using the Bethe ansatz, then translated by the BKW correspondence into estimates of probabilities of events in the FK-percolation model, which ultimately allow us to determine the type of phase transition. We do not detail here how the Bethe ansatz applies to the six-vertex model and how the afore-mentioned estimates are obtained, but focus on their interpretation on the FK-percolation side. 
	
	Finally, in Chapter~\ref{ch:5rotational_inv}, we examine the case~$1 \leq q\leq 4$, when the phase transition is continuous and we may speak of a critical phase. We present two powerful results from~\cite{DumKozKra20} on the (potential) scaling limit of the critical model: namely that it is invariant under rotations by any angle~$\theta \in [0,2\pi]$ and that it is universal among certain isoradial graphs (see Section~\ref{sec:5isoradial} for precise definitions). The original paper is highly technical, and this chapter aims to present a more streamlined version of the proofs, while leaving out certain details.
\end{abstract}

\paragraph{Acknowledgements}
I thank the organisers F. Caravenna, R. Sun and N. Zygouras of the CIME summer school for organising the meeting and giving me the opportunity to present this work. 
I am grateful to Maran Mohanarangan, Dmitry Krachun and Piet Lammers for discussions and comments on these notes.

\chapter{Bernoulli percolation: the basics}\label{sec:percolation}

The simplest instance of FK-percolation is Bernoulli percolation (corresponding to cluster weight~$q = 1$), where edges are open or closed independently of each other. This chapter illustrates the concepts and phenomena valid for more general FK-percolation models in this simpler setting.

\section{Definitions}

Fix a graph~$G = (V,E)$. In this chapter,~$G$ will be the hyper-cubic lattice~$\bbZ^d$ or subgraphs of it, but for now we can consider general graphs. 

\begin{definition}
	For~$p \in (0,1)$, let~$(\omega(e))_{e \in E}$ be i.i.d.\ Bernoulli random variables of parameter~$p$. 
	Write~$\bbP_p$ for the law of~$\omega$, it is a measure on~$\Omega = \{0,1\}^E$. 
	
	We identify the {\em configuration}~$\omega$ with the subgraph of~$G$ with vertices~$V$ and edges~$\{e \in E: \omega(e) = 1\}$.
\end{definition}

We call an edge~$e$ with~$\omega(e) = 1$ {\em open} (in the configuration~$\omega$), or {\em closed } if~$\omega(e) = 0$. 
We will also identify~$\omega$ with the subset of~$E$ formed of the open edges. 
Connections in~$\omega$ will be denoted by~$\lra$, or~$\xlra{\omega}$ when~$\omega$ needs to be specified. 

\paragraph{Percolation: existence of infinite cluster.}
When studying percolation, the questions of interest revolve around the geometry of the connected components (or {\em clusters}) of~$\omega$, specifically the large ones. When~$G = \bbZ^d$, the most basic question is whether~$\omega$ contains an infinite cluster; it it does, we way that $\omega$ {\em percolates}. 

For~$G = \bbZ^d$ with~$d \geq 1$, it is immediate that, for any~$p$, the measure~$\bbP_p$ is tail-trivial.
As such
\begin{align*}
	\bbP_p[\text{there exists an infinite cluster}] \in \{0,1\} \qquad \text{ for all~$p \in (0,1)$}.
\end{align*}
Furthermore, the existence of an infinite cluster under~$\bbP_p$ is equivalent to the positivity of 
\begin{align*}
	\theta(p) := \bbP_p[0 \text{ is in an infinite cluster}] =\bbP_p[0 \lra \infty],
\end{align*}
with~$0 \lra \infty$ being an abbreviation for the fact that the cluster of~$0$ is infinite. 

\paragraph{Monotonicity in~$p$, definition of~$p_c$.}

The measures~$(\bbP_p)_{p \in (0,1)}$ may be coupled in an increasing fashion as follows. Let~$P$ be the probability measure on~$[0,1]^E$ produced by sampling i.i.d.\ uniforms~$(U_e)_{e\in E}$ on~$[0,1]$, and set 
\begin{align*}
\omega_p(e) = \begin{cases}
	0 &\text{ if~$U_e \leq 1-p$}\\
	1 &\text{ if~$U_e > 1-p$},
\end{cases}
\qquad \text{ for all~$e \in E$ and~$p \in (0,1)$.}
\end{align*}
Then~$\omega_p$ has law~$\bbP_p$ for all~$p$ and 
\begin{align}\label{eq:increasingpp}
	\omega_p(e) \leq \omega_{p'}(e)\qquad \text{ for all~$e \in E$ and~$p\leq p'$,~$P$-a.s.}
\end{align}
Due to these properties, we call~$P$ an {\em increasing coupling} of the measures~$\bbP_p$.

More generally,~\eqref{eq:increasingpp} defines a natural partial order on~$\{0,1\}^E$ 
and we will abbreviate~\eqref{eq:increasingpp} as~$\omega_p \leq \omega_{p'}$.
An event is called {\em increasing} if its indicator function is increasing for this partial order. 
In other words, increasing events are events which are stable under the addition of open edges.
Examples include~$\{\text{there exists an infinite cluster}\}$,~$\{0\lra\infty\}$ and~$\{x \lra y\}$ for any two points~$x,y \in V$.   

Due to the increasing coupling above, the probabilities of increasing events are non-decreasing functions of~$p$. 
In particular~$p \mapsto \theta(p)$ is a non-decreasing function.

\begin{definition}
	The {\em critical point} (or point of phase transition) of Bernoulli percolation~$p_c = p_c(\bbZ^d) \in [0,1]$ is defined by 
	\begin{align*}
	p_c = \sup \{ p \in (0,1) : \theta(p) = 0\}= \inf \{ p \in (0,1) : \theta(p) > 0\}.
	\end{align*}
\end{definition}

The equality of the two expressions defining~$p_c$ is due to the monotonicity of~$\theta(p)$. 
Moreover, as discussed above, we immediately conclude that 
\begin{align*}
	\bbP_p[\text{there exists an infinite cluster}] 
	= \begin{cases}
	0 \text{ if~$p < p_c$}\\
	1 \text{ if~$p > p_c$}.
	\end{cases}
\end{align*}

\paragraph{Fundamental questions.}
The questions that come to mind next are the following, in increasing order of difficulty. 
\begin{itemize}
	\item[(1)] Is the phase transition non-trivial, i.e., do we have~$0 < p_c < 1$? 
	\item[(2)] How do clusters behave away from~$p_c$? 

		In the {\em sub-critical phase}~$p<p_c$, all clusters are finite and we expect them to exhibit {\em exponential decay of radii}:
		\begin{align}\label{eq:sharp_sub_crit}
		\bbP_p[0 \lra \partial \La_n] \leq e^{-c(p) n} \qquad \text{ for all~$n$ and~$p<p_c$}, 
		\end{align}
		where~$c(p) > 0$ is a constant depending on~$p$,~$\La_n := \{-n,\dots, n\}^d$ is seen as a subgraph of~$\bbZ^d$
		and~$\partial \La_n$ is the set of vertices in~$\La_n$ with neighbours outside. 

		In the {\em super-critical phase}~$p>p_c$, there exists at least one infinite cluster; 
		we expect it to be unique and all other clusters to have an exponential decay of radii:
		\begin{align}\label{eq:sharp_super_crit}
		\bbP_p[0 \lra \partial \La_n\text{ but } 0\nxlra{} \infty]  \leq e^{-c(p) n} \qquad \text{ for all~$n$ and~$p>p_c$}, 
		\end{align}
		where~$c(p) > 0$ is a constant depending on~$p$. 
		
	\item[(3)] Does there exist an infinite cluster at~$p_c$? Or equivalently, do we have~$\theta(p_c) > 0$? 
	The same question may be rephrased as whether the phase transition is continuous ($\theta(p_c) = 0$) or discontinuous ($\theta(p_c) > 0$).
	For~$d = 3$, this remains one of the main open questions in the field. 
	
	\item[(4)] If the phase transition is continuous (and therefore all clusters are finite at~$p_c$), 
	what is the rate of decay of~$\bbP_{p_c}[0 \lra \partial \La_n]$ as~$n\to \infty$? 
	Furthermore, what is the geometry of the large clusters under the critical measure? 
\end{itemize}

In the rest of the chapter, we will answer the first two questions for Bernoulli percolation on $\bbZ^d$ for general $d \geq 1$, and
the third for Bernoulli percolation on $\bbZ^2$.  
The same questions will be considered for FK-percolation in the following chapters.

\section{Non-triviality of~$p_c$}

Throughout this section, we will work on~$\bbZ^d$ with~$d \geq 2$ (in the case of~$d = 1$, we trivially have~$p_c = 1$; see Exercise~\ref{exo:p_c=1}). 
The goal of this section is to prove the following.

\begin{theorem}\label{thm:pc_nontrivial_perco}
	For all~$d \geq 2$, we have~$0 < p_c < 1$. 
\end{theorem}

Both bounds use the celebrated Peierls argument, named after the German-British physicist Rudolf Peierls. 
This argument, most clearly illustrated in the proof of Proposition~\ref{prop:peierls1}, is a generic way of identifying trivial behaviour for models in perturbative regimes (that is, when the parameters are close to their extremes). 
It studies the competition between energy and entropy using coarse estimates.

\begin{proof}[Proof of Theorem~\ref{thm:pc_nontrivial_perco}]
	The proof follows directly from Proposition~\ref{prop:peierls1}, which shows that~$p_c \geq \frac{1}{2d-1}$
	and Proposition~\ref{prop:peierls2} which shows~$p_c\leq 2/3$.  
\end{proof}

\subsection{Lower bound on~$p_c$}\label{sec:peierls_lb}

\begin{proposition}\label{prop:peierls1}
	For all~$d \geq 2$ and~$p <  \frac{1}{2d-1}$, there exists~$c(p) > 0$ such that 
	\begin{align*}
		\bbP_p[0\lra \partial \La_n] \leq e^{-c(p) n} \qquad \text{ for all~$n\geq 1$}.
	\end{align*}
\end{proposition}

\begin{proof}
	Let~$A_n$ be the set of simple paths on~$\bbZ^d$ of length~$n$ (i.e., containing~$n$ edges) starting from~$0$. 
	Observe that, for all~$n$, 
	\begin{align*}
		\bbP_p[0\lra \partial \La_n] 
		&\leq \bbP_p[\exists \gamma \in A_n \text{ formed only of open edges}] \\
		&\leq \bbE_p[\#\{ \gamma \in A_n \text{ formed only of open edges}\}] \\
		&= \sum_{\gamma \in A_n} \bbP_p[\text{all edges of $\gamma$ are open}\}] \\
		&= |A_n| \cdot p^{n}.
	\end{align*}
	Finally, it is immediate to check that~$|A_n| \leq 2d \cdot (2d-1)^{n-1}$. 
	Inserting this estimate in the above, we obtain the desired conclusion. 
\end{proof}

\subsection{Duality of percolation}\label{sec:duality}

For the upper bound on~$p_c$, we will work with the model in two dimensions. 
The advantage of the two dimensional setting is the dual model, which we define here. 

The dual of~$\bbZ^2$ is the lattice~$(\bbZ^2 )^* = \bbZ^2 + (\frac12,\frac12)$. 
Each face of~$\bbZ^2$ contains a vertex of~$(\bbZ^2)^*$ at its centre;
each edge~$e$ of~$\bbZ^2$ has a dual edge~$e^*$ intersecting it and joining the two faces separated by~$e$. 
Duality may be defined for any planar graph; we only focus here on~$\bbZ^2$ for convenience. 

If~$\omega \in \{0,1\}^E$ denotes a percolation configuration on~$\bbZ^2$, we define its dual configuration~$\omega^*$ by 
\begin{align*}
\omega^*(e^*) =  1 - \omega(e) \qquad \text{ for all~$e \in E$}.
\end{align*}

The following observations are immediate but essential.

\begin{fact}\label{fact:duality_perco}
	If~$\omega$ is sampled according to~$\bbP_p$, then~$\omega^*$  has law~$\bbP_{1-p}$. 	
\end{fact}

Moreover, the finite clusters of~$\omega$ are surrounded by circuits of~$\omega^*$, and vice versa. 
We will generally call everything that has to do with the percolation~$\omega^*$ or the lattice~$(\bbZ^2)^*$ {\em dual}, 
while objects related to~$\omega$ and~$\bbZ^2$  are called {\em primal}. See Figure~\ref{fig:duality_perco} for an illustration. 

\begin{figure}
\begin{center}
\includegraphics[width = 0.35\textwidth]{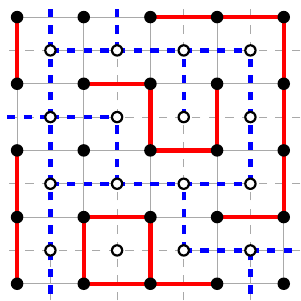}
	\caption{A piece of a square lattice (solid vertices and edges) and its dual (hollow vertices, dashed edges).
	The primal open edges of~$\omega$ are red, those of~$\omega^*$ are blue. 
	Note that the cluster of the central primal vertex is finite and surrounded by a blue circuit.}
	\label{fig:duality_perco}
\end{center}
\end{figure}

\subsection{Upper bound for~$p_c$}\label{sec:peierls2}

\begin{proposition}\label{prop:peierls2}
	For all~$d\geq 2$ and~$p > 2/3$, 
	\begin{align*}
		\bbP_p[0\lra  \infty] > 0.
	\end{align*}
\end{proposition}

\begin{proof}
	Since~$\bbZ^2$ is a subgraph of~$\bbZ^d$ for any~$d \geq 2$, we have~$p_c(\bbZ^d) \leq p_c(\bbZ^2)$ for all~$d\geq 2$. 
	Thus, we focus on~$\bbZ^2$ for the rest of the proof. 
	
	Fix some~$r \geq 0$. 
	For~$\La_r$ not to be connected to infinity, there needs to exist a dual circuit in~$\omega^*$ surrounding~$\La_r$. 
	This circuit intersects the axis~$\bbN \times \{0\}$ at some point~$(k+\tfrac12,0)$ with~$k\geq r$ 
	and needs to extend to~$L^\infty$-distance at least~$k$ from the point~$(k+\tfrac12,\tfrac12)$.
	We conclude that 
	\begin{align*}
		\bbP_p[\La_r\nxlra{}\infty] 
		&\leq \sum_{k \geq r} \bbP_p\big[\text{cluster of~$(k+\tfrac12,\tfrac12)$ in~$\omega^*$ has~$L^\infty$-radius at least~$k$}\big] \\
		&= \sum_{k \geq r} \bbP_{1-p}[0\lra \partial \La_{k}]\\
		&\leq \sum_{k \geq r} e^{-c k} = \tfrac{1}{1 - e^{-c}} \, e^{-c r},
	\end{align*}
	where the equality is due to duality and translation invariance and the last inequality is due to Proposition~\ref{prop:peierls1}. 
	Indeed, by assumption,~$1 - p < \frac13$, and Proposition~\ref{prop:peierls1} provides a constant~$c >0$, independent of~$k$ or~$r$, satisfying the above.
	Then, we may choose~$r= r(p)$ such that 
	\begin{align*}
		\bbP_p[\La_r\lra\infty] \geq \tfrac12.
	\end{align*}
	Finally, the event above is independent of the configuration inside~$\La_r$. As the probability that all edges of~$\La_r$ are open is positive, we conclude that  
	$$ \bbP_p[0\lra\infty] \geq \bbP_p[\La_r\lra\infty \text{ and all edges of~$\La_r$ open}] \geq \tfrac12\bbP_p[\text{all edges of~$\La_r$ open}] >0.$$
\end{proof}

As a byproduct of the proof, we also find that, for~$d= 2$ and~$p >  \frac23$, there exists~$c(p) > 0$ such that 
\begin{align*}
	\bbP_p[0\lra \partial \La_n \text{ but } 0\nxlra{} \infty] \leq e^{-c(p) n} \qquad \text{ for all~$n\geq 1$}.
\end{align*}

\section{Sharpness of the phase transition}

We generally consider a percolation measure to have trivial large scale behaviour if the relevant observables (namely connection probabilities) converge exponentially fast to their limits. 
As a consequence, we call the phase transition of Bernoulli percolation {\em sharp} if for all~$p \neq p_c$,~$\bbP_p[0\lra \partial \La_n]$ converges exponentially fast as~$n\to\infty$ to~$\theta(p)$. Here we will be concerned with {\em sub-critical sharpness}, which states the above exponential convergence for all~$p < p_c$.

\begin{theorem}\label{thm:sharpness_perco}
	Fix~$d\geq 2$. For all~$p < p_c$ there exists~$c(p) >0$ such that 
	\begin{align}\label{eq:exp_decay_perco}
		\bbP_p[0\lra \partial \La_n] \leq e^{-c(p) n} \qquad \text{ for all~$n\geq 1$}.
	\end{align}
\end{theorem}

The proof given below is beautiful and surprisingly simple; it is taken from~\cite{DumTas16}.

\subsection{Derivatives of increasing events}

Let~$A$ be an increasing event. We say that an edge~$e$ is {\em pivotal} for~$A$ (in a configuration~$\omega$) if
$\omega \cup \{e\} \in A$ but~$\omega\setminus\{e\} \notin A$.

\begin{proposition}[Russo's formula]\label{prop:pivs}
Suppose that~$A$ is an increasing event depending only on finitely many edges. Then~$p \mapsto \bbP_p[A]$ is a~$\calC^\infty$ function and 
\begin{align}\label{eq:pivs}
	\frac{d \bbP_p[A]}{dp} = \sum_{e \in E} \bbP_p[e \text{ is pivotal for~$A$}] =  \bbE_p[\#\text{ pivotal edges for~$A$}].
\end{align}
\end{proposition}

\begin{proof}
	Suppose that~$A$ depends only on the edges of~$\La_n$ for some~$n\geq1$. Consider the percolation restricted to~$\La_n$ and write~$|\omega|$ for the number of open edges and~$|\omega^c|$ for the number of closed edges of a configuration~$\omega$. 
	Then 
	\begin{align*}
\bbP_p[A] = \sum_\omega p^{|\omega|}(1-p)^{|\omega^c|}\ind_{\{\omega \in A\}}. 
	\end{align*}
	Differentiating this, we find, 
	\begin{align*}
		\frac{d \bbP_p[A]}{dp} 
		&= \sum_\omega \Big(|\omega| p^{|\omega|-1}(1-p)^{|\omega^c|}- |\omega^c| p^{|\omega|}(1-p)^{|\omega^c|-1}\big)\ind_{\{\omega \in A\}}\\ 
		&=  \sum_e \sum_\omega \Big(\tfrac1p \ind_{\{\omega(e) = 1\}} - \tfrac1{1-p}\ind_{\{\omega(e) = 0\}}\Big) \bbP_p[\omega] \ind_{\{\omega \in A\}}.
	\end{align*}
	For~$\omega$ such that~$\omega\cup\{e\}$ and $\omega\setminus\{e\}$ are both in $A$, the contribution of these two configuration to the above sum is~$0$. 
	The same is true when both~$\omega\cup\{e\}$ and $\omega\setminus\{e\}$ are outside of $A$. 
	When~$\omega\cup\{e\} \in A$, but~$\omega\setminus\{e\} \notin A$, the contribution of the two configurations to the above is
	$$\tfrac1p   \bbP_p[\omega \cup \{e\}] = \bbP_p[\omega \cup \{e\}] + \bbP_p[\omega \setminus \{e\}].$$ 
	This concludes the proof. 
\end{proof}

\subsection{The crucial quantity~$\varphi_p(S)$}

For a finite, connected set~$S$ of edges, at least one of which is adjacent to~$0$, write 
\begin{align*}
	\partial S = \{ u \in V: \, u \text{ is adjacent to edges inside and outside of~$S$}\}. 
\end{align*}
It is also allowed to take~$S = \emptyset$, in which case we set~$\partial S = \{0\}$. 
Let 
\begin{align*}
	\varphi_p(S) =  \bbE_p\big[\#\{ u \in \partial S: 0 \xlra{S} u\}  \big]. 
\end{align*}
Above,~$\xlra{S}$ refers to connections using only edges of~$S$. See Figure~\ref{fig:phiS} for an illustration.

The following two lemmas will imply Theorem~\ref{thm:sharpness_perco} directly. 
The proofs of the lemmas are deferred to the following two sections. 

\begin{lemma}\label{lem:phiS_exp}
	For any~$p \in (0,1)$,~$n\geq1$ any finite, connected set of edges~$S$ as above
	\begin{align}\label{eq:phiS_exp}
		\bbP_p[0\lra \partial \La_n] \le \varphi_p(S)^{\lfloor n/{\rm diam}(S)\rfloor}.
	\end{align}
\end{lemma}

\begin{lemma}\label{lem:phiS_supercrit}
	For any~$p \in (0,1)$ and~$n\geq 1$. 
	\begin{align}\label{eq:phiS_supercrit}
		\frac{d \bbP_p[0\lra \partial \La_n]}{d p}  \geq \inf_S \varphi_p(S) \cdot (1 - \bbP_p[0\lra \partial \La_n]), 
	\end{align}
	where the infimum is over all finite connected sets of edges~$S$ as above.
\end{lemma}

One should think of~\eqref{eq:phiS_supercrit} as stating that 
\begin{align*}
\frac{d \theta(p)}{d p}  \geq \inf_S \varphi_p(S) \cdot (1 - \theta(p)), 
\end{align*}
even though the derivative is not well-defined.

\begin{figure}
\begin{center}
\includegraphics[width = 0.45\textwidth, page = 4]{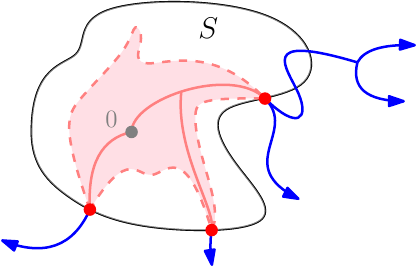}\qquad\qquad 
\includegraphics[width = 0.35\textwidth, page = 2]{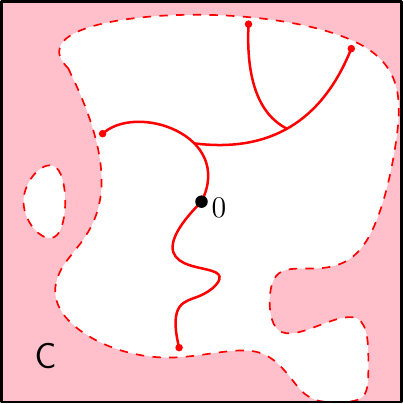}
	\caption{{\em Left:} An illustration of the argument in the proof of Lemma~\ref{lem:phiS_exp}: once the cluster~$\sfC$ of~$0$ in~$S$ has been explored, for~$0$ to be connected to distance~$n$, one of the red vertices needs to be connected to distance~$n-{\rm diam}(S)$ {\em in the complement of~$\sf C$}.
	{\em Right:} The white region containing~$0$ is the set of edges~$S$. Here~$0$ is connected to four vertices on~$\partial S$.
	When proving Lemma~\ref{lem:phiS_supercrit},~$S$ denotes the connected component of~$0$ in the complement of the cluster of~$\partial \La_n$.}
	\label{fig:phiS}
\end{center}
\end{figure}

\begin{proof}[Proof of Theorem~\ref{thm:sharpness_perco}]
	Set
	\begin{align*}
		\tilde p_c = \inf \{ p: \inf_S \varphi_p(S) \geq 1/2\}.
	\end{align*}
	
	Then, for~$p < \tilde p_c$, there exists~$S$ such that~$\varphi_p(S) < 1/2$. Fix one such~$S$. 
	Applying Lemma~\ref{lem:phiS_exp}, we conclude that, for all~$n \geq {\rm diam}(S)$, 
	\begin{align*}
		\bbP_p[0\lra \partial \La_n] \le 2^{-\lfloor n/{\rm diam}(S)\rfloor} \leq e^{-c(p) n}.
	\end{align*}
	The above may be extended to all values of~$n\geq 1$ by potentially modifying the value of the constant~$c(p) > 0$ 
	
	Conversely, for~$p > \tilde p_c$, we claim the existence of an infinite cluster. 
	First, notice that~$\inf_S \varphi_p(S)$ is increasing in~$p$, and therefore~$\inf_S \varphi_p(S) \geq 1/2$ for all~$p > \tilde p_c$. 
	Moreover, either~$\theta(p) \geq 1/2$ or, for any~$u\in (\tilde p_c, p]$,~\eqref{eq:phiS_supercrit} implies that 
	\begin{align*}
		\frac{d \bbP_u[0\lra \partial \La_n]}{d u}  \geq \inf_S \varphi_u(S) \cdot (1 -\bbP_u[0\lra \partial \La_n]) \geq \tfrac18,
	\end{align*}
	as long as~$n$ is large enough that so that~$\bbP_u[0\lra \partial \La_n] \leq \frac34~$. 
	The existence of such an~$n$ is guaranteed by the fact that~$\theta(p) <1/2$. 
	Integrating the above and taking~$n$ to infinity, we conclude that 
	\begin{align*}
		\theta(p)  \geq \tfrac18 (p-\tilde p_c) > 0.
	\end{align*}
	
	The two cases above allow us to conclude that~$p_c = \tilde p_c$ and therefore that~\eqref{eq:exp_decay_perco} holds for all~$p < p_c$. 
\end{proof}

\subsection{The sub-critical regime via~$\varphi_p(S)$: proof of  Lemma~\ref{lem:phiS_exp}}

Fix~$S$ and~$n \geq 1$. The lemma is only meaningful when~$\varphi_p(S) < 1$, so when~$S$ contains all edges incident to~$0$. We suppose this henceforth. 
Moreover, we may also take~$n \geq {\rm diam}(S)$, otherwise the statement is trivial. 

Write~$\sfC$ for the connected component of~$0$ in the configuration~$\omega \cap S$.  
Observe that, if~$0 \lra \partial \La_n$, then there exists~$u \in \partial S \cap \sfC$ such that 
$u$ is connected to~$\partial \La_n$ outside of~$\sfC$. 
In particular,~$u$ needs to be connected in~$\omega\setminus \sfC$ to a point at a distance at least~$n - {\rm diam}(S)$ of itself.
This event has probability at most~$\bbP_p[0\lra \partial \La_{n-{\rm diam}(S)}]$.
Conditioning on~$\sfC$, applying a union bound, then summing over the possible realisations~$C$ of~$\sfC$, we find 
\begin{align*}
	\bbP_p[0\lra \partial \La_{n}] 
	&\leq \sum_{C} \bbP_p[\sfC = C] \sum_{u \in \partial S \cap C} \bbP_p[u\xlra{\omega \setminus C} \partial \La_{n}\,|\, \sfC = C]\\
	&\leq \sum_{C} \bbP_p[\sfC = C] \cdot | \partial S \cap C| \cdot \bbP_p[0\lra \partial \La_{n-{\rm diam}(S)}]\\
	&=\varphi_p(S)\cdot \bbP_p[0\lra \partial \La_{n-{\rm diam}(S)}].
\end{align*}
Iterating the above proves~\eqref{eq:phiS_exp}.
\hfill $\square$

\begin{remark}
	Applying the argument above to~$S$ formed of the~$2d$ edges incident to~$0$, 
	we (almost) retrieve the Peierls argument of Proposition~\ref{prop:peierls1}.
\end{remark}

\subsection{The super-critical regime via~$\varphi_p(S)$: proof of Lemma~\ref{lem:phiS_supercrit}}

Fix~$n \geq1$. Recall from~\eqref{eq:pivs} that~$\frac{d \bbP_p[0\lra \partial \La_n]}{dp}$ is equal to the expected number of pivotals for the event~$\{0\lra \partial \La_n\}$. Pivotals may be open or closed; we will lower bound here the number of closed pivotals. 
We will work exclusively on~$\La_n$, and therefore restrict ourselves to the edges in this graph. 

Write~$\sf C$ for the set of edges connected to~$\partial \La_n$ and~$\partial_{\rm ext}\sf C$ for all edges of~$\La_n$ adjacent to, but not contained in~$\sf C$. 
If~$\sf C$ contains~$0$ (that is, contains an edge incident to~$0$) then~$0\lra \partial \La_n$ occurs, and there are no closed pivotals. 
When~$\sf C$ does not contain~$0$, let~$S$ denote the connected component of~$0$ in~$\La_n \setminus ({\sf C} \cup \partial_{\rm ext}{\sf C})$ --- see Figure~\ref{fig:phiS} (right side). 
Notice that any vertex~$u$ of~$\partial S$ is separated from~$\sf C$ by a closed edge which belongs to~$\partial_{\rm ext}{\sf C}$. Moreover, when~$u \lra 0$ (a connection which necessarily occurs in~$S$), the edge separating~$u$ from~$\sf C$ is a closed pivotal for~$\{0\lra \partial \La_n\}$. 
Thus  
\begin{align*}
\frac{d \bbP_p[0\lra \partial \La_n]}{dp} 
&\geq \sum_{C} \bbE_p[\#\{ u \in \partial S: 0 \xlra{S} u\}| \sfC = C]\cdot  \bbP_p[\sfC = C]\\
&= \sum_{C}\varphi_p(S)\cdot  \bbP_p[\sfC = C]\\
&\geq \inf_S\varphi_p(S)\cdot  \bbP_p[0 \nxlra{} \partial \La_n],
\end{align*}
where the sum is over all possible realisations~$C$ of~$\sfC$ not containing~$0$ and where~$S$ is determined by~$C$. 
In the equality, we used the fact that the conditioning on~$\sfC = C$ only gives information on the edges of~$C$ and~$\partial_{\rm ext}C$, but not on those in~$S$. The spatial independence of Bernoulli percolation is essential here. 
%
\hfill~$\square$

\section{Complement: uniqueness of the infinite cluster}

This section is not essential for the rest of the notes, but contains a beautiful argument which we could not resist presenting. 
The reader only interested in the two dimensional case may skip this section. 

We focus here on the super-critical phase, that is, when~$\theta(p) > 0$. A more basic question than the super-critical sharpness~\eqref{eq:sharp_super_crit} 
concerns the number of infinite clusters. The following result by Burton and Keane is a robust way of proving that, when an infinite cluster exists, it is a.s.\ unique. 

\begin{theorem}[Burton--Keane~\cite{BurKea89}]\label{thm:Burton-Keane}
	Fix~$d \geq 1$. Then, for all~$p \in (0,1)$ with~$\theta(p) > 0$, 
	\begin{align*}
		&\bbP_p[\text{there exists exactly one infinite cluster}] = 1.
	\end{align*}
\end{theorem}

\begin{proof}
Write~$N = N(\omega)$ for the number of infinite clusters in~$\omega$. 
Notice that~$N$ is invariant under any shift of~$\omega$, and the ergodicity of the measures~$\bbP_p$ under translations implies that~$N$ is~$
\bbP_p$-a.s.\ constant. 
Our goal is to prove that this constant is either~$0$ or~$1$. 
We will do so by first excluding the possibility that~$1 < N <\infty$, then showing that~$N \neq \infty$. 
\medskip

The former is easy. Indeed, assuming that~$p$ is such that~$\bbP_p[N = n] = 1$ for some~$2 \leq n <\infty$, we may find~$k \geq 1$ such that 
\begin{align*}
\bbP_p[\text{all infinite clusters intersect~$\La_k$}  ] \geq 1/2.
\end{align*}
Note that the event above is independent of the configuration in~$\La_k$. It follows that
\begin{align*}
\bbP_p[\text{all infinite clusters intersect~$\La_k$ and all edges of~$\La_k$ are open}  ] \geq \tfrac12 p^{|\La_k|} > 0.
\end{align*}
Finally, under the event above,~$N =1$, which contradicts the assumption that~$N = n \geq 2$ a.s.
\medskip 

Let us now exclude the possibility of having infinitely many infinite clusters. This requires a much more subtle argument --- one which may fail in certain cases (see Exercise~\ref{exo:Galton-Watson}). Assume now that~$p$ is such that~$\bbP_p[N =  \infty] = 1$. 

Call a point~$u \in \bbZ^d$ a {\em trifurcation} if it is connected to~$\infty$ by three distinct connected components of~$\omega \setminus \{u\}$. 
A local surgery argument similar to the one that allowed us to fuse all the clusters into a single one in the previous part shows that, under our assumption, 
\begin{align*}
	\bbP_p[\text{$u$ is a trifurcation}] = \bbP_p[\text{$0$ is a trifurcation}] > 0,
\end{align*}
for all~$u \in \bbZ^d$, where the equality is due to the invariance of~$\bbP_p$ under translations. 

As a consequence, if we write~$T_k$ for the number of trifurcations in~$\La_k$, we conclude that 
\begin{align*}
	\bbE_p [T_k] = |\La_k| \cdot \bbP_p[\text{$0$ is a trifurcation}],
\end{align*}
which in turn implies the existence of a constant~$c  > 0$ such that 
\begin{align}\label{eq:volume_trif}
	\bbP_p \big[T_k \geq c |\La_k|\big] \geq c,
\end{align}
for all~$k \geq 1$. 

This will come into contradiction with the following lemma, whose proof will be discussed at the end of the section. 

\begin{lemma}\label{lem:trifurcations}
	For any configuration~$\omega$ and any~$k \geq 1$, there exists at least~$T_k(\omega)$ points on~$\partial \La_k$. 
\end{lemma}

Lemma~\ref{lem:trifurcations} indeed contradicts~\eqref{eq:volume_trif}, since it states that a.s.,	~$T_k \leq |\partial \La_k| <c |\La_k|$ for~$k$ large enough. We use here the fact that~$|\partial \La_k| / |\La_k| \to 0$, which is the hallmark of an {\em amenable} graph. 
\end{proof}

\begin{proof}[Proof of Lemma~\ref{lem:trifurcations}]
	The idea of this lemma is to consider a graph obtained from~$\omega \cap \La_k$ by repeatedly removing edges. 
	First, remove edges contained in cycles of~$\omega \cap \La_k$ in arbitrary order until no cycles remain. Then repeatedly remove all edges of~$\omega \cap \La_k$ that have an endpoint of degree~$1$ strictly inside~$\La_k$.

	These procedures transform~$\omega \cap\La_k$ into a forest~$F(\omega)$ whose leaves are all contained in~$\partial \La_k$. 
	Moreover, all trifurcations of~$\omega$ contained in~$\La_k$ have degree at least~$3$ in~$F(\omega)$; indeed, the property of being a trifurcation is not affected by the passage from~$\omega$ to~$F(\omega)$. 
	
	Finally, it remains to prove by induction that the number of leaves of a forest is larger than the number of vertices of degree at least~$3$.
\end{proof}

\section{Critical planar percolation and RSW theory}\label{sec:RSW_perco}

We now focus on Bernoulli percolation on $\bbZ^2$, where the duality of the lattice allows us to identify the critical point as~$p_c =1/2$. 

Indeed, recall from Section~\ref{sec:duality} that on~$\bbZ^2$, if~$\omega \sim \bbP_{1/2}$, then~$\omega^*$ is also a percolation with parameter~$1/2$, but on the dual graph, which in the case of $\bbZ^2$ is identical to the primal graph. This is a strong indication that~$p_c=1/2$. In this section, we will confirm this prediction and study the critical phase of two-dimensional percolation. 

\begin{theorem}[Critical percolation on~$\bbZ^2$]\label{thm:poly_pc}
	The critical point of Bernoulli percolation on~$\bbZ^2$ is~$p_c = 1/2$. 
	
	Moreover, there exists~$\alpha >0$ such that 
	\begin{align}\label{eq:poly_pc}
		\tfrac1{2} n^{-1}\leq \bbP_{1/2}[0\lra \partial \La_n] \leq n^{-\alpha} \qquad \forall n\geq 1.
	\end{align}
	In particular,~$\theta(p_c) = 0$. 
\end{theorem}

\begin{remark}\label{rem:p_c=1/2_quick_proof}
	From Theorem~\ref{thm:sharpness_perco} we deduce that~$\bbP_{p}[0\lra \partial \La_n]$ decays exponentially fast for~$p< 1/2$. 
	Duality allows to also deduce the super-critical sharpness~\eqref{eq:sharp_super_crit}. 
\end{remark}

The core of Theorem~\ref{thm:poly_pc} is~\eqref{eq:poly_pc}; 
the other statements follow readily using the sharpness of the phase transition \eqref{eq:sharp_sub_crit}.
The lower bound is a relatively straightforward consequence of self-duality. 
More interesting is the upper bound, which makes use of the so-called Russo--Seymour--Welsh theory (RSW), the cornerstone of the study of the critical phase of 2D percolation. 

\begin{theorem}[RSW for Bernoulli percolation]\label{thm:RSW_perco}
	There exists a constant~$c > 0$ such that, for all~$n\geq 1$, 
	\begin{align}\label{eq:RSW_perco}
		\bbP_{1/2}\big[[-n,n] \times [0,n] \text{ contains open crossing between its left and right sides}\big] \geq c.
	\end{align}
\end{theorem}

Before proving the above, we mention a feature of Bernoulli percolation called {\em positive association} (also sometimes called FKG inequality) 
which will be discussed in more detail and proved in Section~\ref{sec:monotonicity}. 
For now, we simply state that, for any~$p \in (0,1)$, if~$A$ and~$B$ are both increasing events (or both decreasing events), 
then 
\begin{align}\label{eq:FKG_perco}
	\bbP_{p}[A\cap B] \geq\bbP_{p}[A] \bbP_{p}[B]. 
\end{align}
This property is curial to the proof of Theorem~\ref{thm:RSW_perco}, as well as for its applications. 

The proof below is not the shortest, nor the most general proof of \eqref{eq:RSW_perco}, 
but it is similar to that used for FK-percolation in Chapter~\ref{ch:3dichotomy}. 
The first proof of this type of inequality appeared in~\cite{Rus78,SeyWel78}; 
a number of variants of this type of inequality with different assumptions of independence, 
positive association and symmetries appeared in \cite{BolRio10,GriMan13,DumTas19,KSTas23} to quote only a few.

\begin{proof}[Proof of Theorem~\ref{thm:RSW_perco}]
	Fix $n\geq 1$. 
	Due to the self-duality of the model and the observation that any rectangle either contains a left-right crossing in~$\omega$, 
	or a top-bottom crossing in~$\omega^*$, we find that,
	\begin{align}\label{eq:cross_square}
		\bbP_{1/2}\big[\{0\}\times [0,n] \text{ connected to }\{n+1\}\times [0,n] \text{ inside } [0,n+1] \times [0,n]\big]
		= \tfrac12.
	\end{align}
	See also Figure~\ref{fig:RSW_perco} for an explanation.
 	One should see this as an estimate of the probability of crossing a {\em square}. The whole point of the theorem is to lengthen these crossings so as to cross rectangles of some aspect ratio strictly greater than~$1$. Indeed, once this is done, the aspect ratio may be increased by repeated applications of the positive association inequality~\eqref{eq:FKG_perco}. 
	We will use here the expression ``uniformly positive'' to mean bounded below by a positive constant independent of~$n$. 
	
	\begin{figure}
	\begin{center}
		\includegraphics[width = 0.3\textwidth]{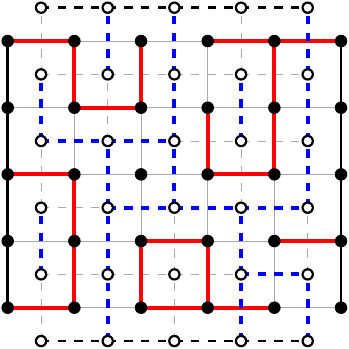} \hspace{.1\textwidth}
		\includegraphics[width = 0.3\textwidth]{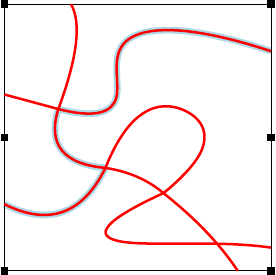}
		\caption{{\em Left: } an~$(n + 1) \times n$ rectangle is either crossed horizontally in the primal model or vertically in the dual. 
		At the self-dual point, these two events have equal probability~$1/2$.\newline
		{\em Right:} By symmetry, the probability to connect the lower half of the left side 
		${\sf BottomLeft}$  to the right side of the square is at least~$1/4$. 
		The same holds for connections between the~${\sf TopRight}$ and the left side.
		Combining these two crossings with a vertical one ensures the existence of a connection between 
		${\sf BottomLeft}$ and~${\sf TopRight}$.}
	\label{fig:RSW_perco}
	\end{center}
	\end{figure}
	
	Write~${\sf BottomLeft} = \{0\} \times [0,n/2]$ and~${\sf TopLeft} = \{0\} \times [n/2,n]$
	and~${\sf BottomRight}$ and~${\sf TopRight}$ for their translates by~$(n,0)$; see Figure~\ref{fig:RSW_perco} for an illustration.	
	We then claim that 
	\begin{align}\label{eq:BLTR}
		\bbP_{1/2}\big[{\sf BottomLeft} \xlra{[0,n]^2}{\sf TopRight}\big] \text{ is uniformly positive}.
	\end{align}
	Indeed, by symmetry with respect to vertical reflections and \eqref{eq:cross_square},
	\begin{align*}
		\bbP_{1/2}\big[{\sf BottomLeft} \xlra{[0,n]^2} \{n\} \times [0,n]\big] \ge \tfrac14\quad \text{ and } \quad
		\bbP_{1/2}\big[\{0\} \times [0,n] \xlra{[0,n]^2} {\sf TopRight}\big] \ge \tfrac14.
	\end{align*}
	When the two events above occur and~$[0,n]^2$ is crossed from top to bottom, then
	${\sf BottomLeft}$ is connected to~${\sf TopRight}$. Applying~\eqref{eq:FKG_perco},~\eqref{eq:cross_square} and the bounds above, 
	we conclude that the probability in~\eqref{eq:BLTR} is bounded below by~$\tfrac12\cdot\tfrac14\cdot\tfrac14$, and therefore is uniformly positive. 
	
	Consider now the two squares~$[-n,0]\times [0,n]$ and~$[0,n]\times [0,n]$. 
	Assume that in each of them~${\sf BottomLeft}$ is connected to~${\sf TopRight}$.
	These connections occur with uniformly positive probability. 
	When they do, write~$\Gamma_L$ for the topmost path realising the first connection 
	and~$\Gamma_R$ for the lowest path realising the second; see Figure~\ref{fig:RSW_perco2}. 
	
	Note that these ``highest'' and ``lowest'' paths are indeed well-defined, 
	and that the event~$\Gamma_L = \gamma_L$ for some potential realisation~$\gamma_L$ 
	depends on the states of the edges on and above~$\gamma_L$, but not on those below. 
	
	We now claim that, for any potential realisations~$\gamma_L$ and~$\gamma_R$ or $\Gamma_L$ and $\Gamma_R$, we have 
	\begin{align}\label{eq:GG}
		\bbP_{1/2}\big[\Gamma_L \xlra{[-n,n]\times [0,n]} \Gamma_R\,\big|\, \Gamma_L = \gamma_L \text{ and }\Gamma_R = \gamma_R\big] \geq 1/2.
	\end{align}
	Indeed, let~$\sigma$ be the horizontal reflection with respect to the axis~$\{0\}\times\bbR$ composed with the shift by~$(\tfrac12,\tfrac12)$. 
	Consider the topological rectangle~$\mathcal Q$ bounded by~$\gamma_R$,~$\sigma(\gamma_R)$,~$\gamma_L$ and~$\sigma(\gamma_L)$ (see Figure~\ref{fig:RSW_perco2} for an illustration and additional details). 
	Notice that all edges inside~$\mathcal Q$ are unaffected by the conditioning~$\Gamma_L = \gamma_L \text{ and }\Gamma_R = \gamma_R$.
	Moreover,~$\mathcal Q$ is symmetric with respect to~$\sigma$ and that it either contains a primal crossing between~$\gamma_L$ and~$\gamma_R$, or a dual crossing between~$\sigma(\gamma_R)$ and~$\sigma(\gamma_L)$. 
	Since~$\sigma$ maps one such crossing on the opposite one, we conclude that 
	\begin{align}
		\bbP_{1/2}\big[\gamma_L \xlra{\mathcal Q} \gamma_R\big] \geq 1/2.
	\end{align}
	Inserting the above into~\eqref{eq:GG}, then averaging over~$\gamma_L$ and~$\gamma_R$ that realise the connections in 
	$[-n,0]\times [0,n]$ and~$[0,n]\times [0,n]$, we obtain
	\begin{align*}
		&\bbP_{1/2}\big[[-n,n] \times [0,n] \text{ contains an open crossing between its left and right sides}\big]  \\
		&\qquad\geq \tfrac12
		\bbP_{1/2}\big[{\sf BottomLeft} \xlra{[0,n]^2} {\sf TopRight} \big]^2> 0
	\end{align*}
	uniformly in~$n$. 
\end{proof}

\begin{figure}
\begin{center}
	\includegraphics[width = 0.65\textwidth]{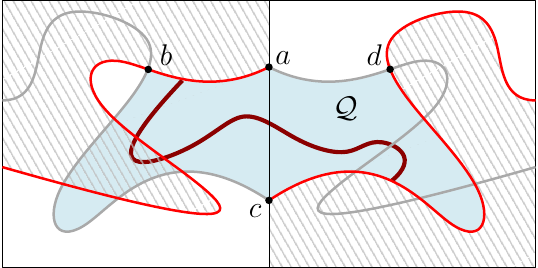}
	\caption{In each square~${\sf BottomLeft}$ is connected to~${\sf TopRight}$. 
	The left and right red paths are the topmost and bottommost, respectively, paths realising these connections in a configuration~$\omega$.
	The grey paths are their  reflections~$\sigma$ with respect to the medial line; by construction, the red and grey paths intersect. 
	When orienting the left red path from left to right, call~$b$ its last intersection point with the left grey path, and~$a$ its endpoint on the medial line. 
	Let~$d = \sigma(b)$ and~$c$ be the endpoint of the lower red path on the medial line. 
	\newline
	The blue~$\sigma$-symmetric domain~$\mathcal Q$ is bounded by the arcs~$(ab)$ and~$(cd)$ of the red paths, and their grey images through~$\sigma$.
	When sampling an independent configuration~$\tilde \omega$ in~$\mathcal Q$, the probability that~$(ab)$ and~$(cd)$ are connected is~$1/2$. 
	The explored regions (hashed) may intersect~$\mathcal Q$, however the boundary of any such intersection is an open path. 
	If one completes the explored configuration~$\omega$ by pasting~$\tilde\omega$ in the unexplored parts of~$\mathcal Q$, 
	then whenever~$\mathcal Q$ is crossed in~$\tilde \omega$, it is also crossed in~$\omega$, 
	leading ultimately to a left-right crossing of the~$2n\times n$ rectangle. }
	\label{fig:RSW_perco2}
\end{center}
\end{figure}

\begin{proof}[Proof of Theorem~\ref{thm:poly_pc}]
	We start by proving~\eqref{eq:poly_pc}. 
	The lower bound is a simple consequence of~\eqref{eq:cross_square}. Indeed, by the union bound we have
	\begin{align*}
	\tfrac12\leq\!\!\! \sum_{x \in \{0\} \times [0,n]}\!\!\!  \bbP_{1/2}\big[x \text{ connected to }\{n+1\}\times [0,n]\big]\leq (n+1) \bbP_{1/2}\big[0\lra \partial \La_{n+1}\big]. 
	\end{align*}
	Dividing by~$n+1$ provides the desired bound. 
	
	We turn to the upper bound. 
	By combining several crossings of translates and rotations by~$\pi/2$ of rectangles of the form~$[-n,n] \times [0,n]$ and using \eqref{eq:FKG_perco}, we find that 
	\begin{align*}
	\bbP_{1/2}\big[ \La_n \text{ surrounded by an open circuit contained in } \La_{2n}\big] \geq c
	\end{align*}
	for all~$n\geq 1$ and some constant~$c > 0$ independent of~$n$. 
	Moreover, the same holds for the dual model. Observe that for~$k\geq1$, if~$0 \lra \partial \La_{2^k}$, then none of the annuli 
	$\La_{2^{j+1}} \setminus \La_{2^j}$ for~$0\leq j <k$ may contain a circuit in~$\omega^*$. Finally, the configurations in these annuli are independent, and we conclude that
	\begin{align*}
		\bbP_{1/2}\big[ 0 \lra \La_{2^k} \big] 
		\leq \prod_{0\leq j<k}
		\big(1 - \bbP_{1/2}[ \La_{2^j} \text{ surrounded by a circuit in } \omega^* \cap \La_{2^{j+1}}]\big)
		\leq (1-c)^k.
	\end{align*}
	This provides the upper bound of~\eqref{eq:poly_pc} for~$n = 2^k$. The general bound may be obtained by inclusion of events and by adapting the exponent~$\alpha$. 
	
	We turn to the other statements in the theorem. 
	As a consequence of the upper bound in~\eqref{eq:poly_pc},~$\theta(1/2)= 0$, and therefore~$p_c \geq 1/2$.
	Conversely, Theorem~\ref{thm:sharpness_perco} and the lower bound in~\eqref{eq:poly_pc} imply that~$p_c \leq 1/2$. 
	Thus~$p_c = 1/2$.
\end{proof}

Let us close this section by mentioning that the polynomial bounds~\eqref{eq:poly_pc} on the connection probability form~$0$ to~$\partial \La_n$ are just one of the consequences of the RSW theory. 
Indeed, the RSW theory proved to be instrumental in the understanding of the critical phase of two-dimensional percolation models. 

It is expected that the contours of clusters of critical percolation on a rescaled lattice~$\delta \bbZ^2$ converge, as~$\delta \to 0$, 
to a certain random family of curves on~$\bbR^2$ known as CLE$_6$. This is one of the central conjectures in percolation theory; it was only proved for a variant of the model called {\em site} percolation on the triangular lattice in~\cite{Smi01,CamNew06}.
The RSW theory essentially states that the family of cluster contours of critical percolation behaves qualitatively like CLE$_6$, but is not sufficiently precise to actually identify the scaling limit, or indeed prove that such a limit exists. It does however show that sub-sequential scaling limits exist and that they are locally finite.

\section{References and further results}

\paragraph{First proof of~$p_c(\bbZ^2) = 1/2$.}
The first derivation of the critical point of Bernoulli percolation on~$\bbZ^2$ is due to Kesten~\cite{Kes80}. 
At the time, the sharpness of the phase transition was not known, so~\eqref{eq:poly_pc} only implied~$p_c \geq 1/2$. 
To prove both the sharpness of the phase transition and the fact that~$p_c \leq 1/2$, one may show ``by hand'' a sharp-threshold result. 
See Exercise~\ref{exo:p_c=1/2K} for such a proof. 

\paragraph{Sub-critical sharpness.}
The original proof of the sharpness of the phase transition for Bernoulli percolation in general dimension was obtained in~\cite{Men86,AizBar87}. 
We presented here a simpler proof due to Duminil-Copin and Tassion~\cite{DumTas16a}. 
Finally, a beautiful new proof was obtained by Vanneuville~\cite{Van22}; in addition to its elegance, this proof is remarkable as it directly implies the exponential decay of the volume of the cluster of~$0$, rather than just its radius. 
The exponential decay of the volume may also be deduced from that of the radius through renormalisation arguments. 

\paragraph{Super-critical sharpness.}
Up to now, we only discussed the sub-critical sharpness, that is, the exponential decay of connection probabilities in the sub-critical regime. 
Recall that in the super-critical regime, we also expect trivial large-scale behaviour, in that connection probabilities converge exponentially to their limits (see~\eqref{eq:sharp_super_crit} and Exercise~\ref{exo:supercrit_sharpness}).
This indeed the case for  Bernoulli percolation in any dimension. 

\begin{theorem}\label{thm:supercrit_sharpness}
	Fix~$d\geq 2$ and consider Bernoulli percolation on~$\bbZ^d$. Then, for all~$p > p_c$, there exists~$c(p) > 0$ such that 
	\begin{align*}
		\bbP_p[0\lra \partial \La_n \text{ but } 0\nxlra{} \infty] \leq e^{-c(p) n} \qquad \text{ for all~$n\geq 1$}.
	\end{align*}
\end{theorem}

In two dimensions, this theorem is easily deduced from the sub-critical sharpness of the dual model. 
For dimensions~$d \geq 3$, the key ingredient in this proof is the celebrated Grimmett--Marstrand theorem~\cite{GriMar90},
which states that the critical point of Bernoulli percolation on a slab~$\bbS_k = \bbZ^2 \times \{0,\dots, k\}^{d-2}$
tends to that of~$\bbZ^d$ when~$k \to \infty$. 

\section{Exercises: Bernoulli percolation}	

\noindent Observe that~$p_c$ may be defined on any vertex-transitive graph in the same way as on~$\bbZ^d$. 

\begin{exo}\label{exo:p_c=1}
	Show that~$p_c(\bbZ) =1$. 
\end{exo}

\begin{exo}
	Show that for the ``ladder'' graph~$\bbZ \times \{0,1\}$,~$p_c =1$. 
\end{exo}

\begin{exo}\label{exo:Galton-Watson}
	Let~$T_d$ denote the~$d + 1$-regular tree (with the root having degree~$d$ rather than~$d+1$). Prove that~$p_c (T_d)= \frac1d$ and observe that the Peierls argument works all the way up to~$p_c$.
	Deduce that the phase transition is sharp in this case. Prove that for~$p > p_c$, there exists a.s.\ infinitely many infinite clusters on~$T_d$. 

	Prove that~$\bbP_{p_c} (T_d)= 0$. 
	Show that~$\bbP_{p_c}[0 \text{ is connected to distance~$n$}] = n^{-\alpha_1 + o(1)}$ for an exponent~$\alpha_1 >0$ to be determined.\smallskip\\
	{\em Hint:} interpret the cluster of the root as a Galton-Watson process. 
\end{exo}

\begin{exo}
	Show that~$p_c(\bbZ^d)$ is decreasing in~$d$. Show that it is strictly decreasing. Show that~$p_c(\bbZ^d) \to 0$ as~$d\to \infty$. 
\end{exo}

%
%

\begin{exo}\label{exo:right_cont}
	Show that~$p \mapsto \theta(p)$ is right-continuous. \smallskip\\
	{\em Hint:}~$\theta(p)$ is the decreasing limit of the continuous and increasing functions~$p \mapsto \bbP_p[0 \lra \partial \La_n]$.
\end{exo}

\begin{exo}\label{exo:left_cont}
	Consider Bernoulli percolation on~$\bbZ^d$. The goal of this exercice is to prove continuity of~$p\mapsto\theta(p)$ for all~$p\neq p_c$. 
	Recall from Exercise~\ref{exo:right_cont} that this function is right-continuous. Thus, we only need to prove left-continuity for~$p \neq p_c$. \smallskip 
	
	Fix~$p$ and assume that~$\theta(p) > \lim_{u \nearrow p}\theta(u)$. In particular~$\theta(p) > 0$ and~$p \geq p_c$. 
	\begin{itemize}
		\item[(a)] Consider the increasing coupling~$P$ of Bernoulli percolation using uniforms~$(U_e)_{e \in E}$. Argue that our assumption implies that
		\begin{align*}
			P[0\xlra{\omega_p} \infty \text{ but } 0\nxlra{\omega_u} \infty \text{ for all~$u < p$}] > 0. 
		\end{align*}
		\item[(b)] Argue that, conditionally on~$\omega_p$,~$(U_e)_{e \in \omega_p}$ are i.i.d.\ uniforms on~$[1-p,1]$. 
		Conclude that a.s.\ for any~$n$ there exists~$u < p$ such that~$\omega_u = \omega_p$ on~$\La_n$. 
		\item[(c)] Call a vertex~$v$ {\em fragile} (for some configuration~$(U_e)_e$) if~$v\xlra{\omega_p} \infty$ but~$v\nxlra{\omega_u} \infty$ for all~$u < p$.
		Prove that if~$0$ is fragile, then a.s.\ all vertices of the infinite cluster of~$\omega_p$ are fragile. 
		\item[(d)] Using the uniqueness of the infinite cluster for $\bbP_u$, deduce that, for any~$u < p$, $$\bbP_u[\text{there exists no infinite cluster}] > 0.$$ Show that this implies~$p \leq p_c$ and conclude.  
	\end{itemize}
\end{exo}

\begin{exo}
	Prove Theorem~\ref{thm:supercrit_sharpness} for Bernoulli percolation on~$\bbZ^2$ using the sub-critical sharpness, duality and the fact that~$p_c = 1/2$. 
\end{exo}

\begin{exo}\label{exo:supercrit_sharpness}
	Consider Bernoulli percolation on~$\bbZ^d$ and~$p > p_c(\bbZ^d)$. Using the super-critical sharpness (Theorem~\ref{thm:supercrit_sharpness}) and the uniqueness of the infinite cluster (Theorem~\ref{thm:Burton-Keane}), prove that 
	\begin{align*}
	\theta(p)^2 \leq \bbP_p[x \lra y] \leq \theta(p)^2 + e^{-c(p) \|x- y\|} \qquad \text{ for all~$x,y \in V$}, 
	\end{align*}
	for some~$c(p) > 0$. 
\end{exo}

\begin{exo}(Compute~$p_c(\bbZ^2)$ without RSW)\label{exo:Zhangs}
Consider Bernoulli percolation on~$\bbZ^2$. 
The goal of this exercise is to prove that~$p_c = 1/2$ using Theorems~\ref{thm:sharpness_perco} and~\ref{thm:Burton-Keane}, but without the use of the RSW theory. 
We will also obtain here that~$\theta(p_c) = 0$. 
\begin{itemize}
\item[(a)] Use the duality observation \eqref{eq:cross_square} and Theorem~\ref{thm:sharpness_perco} to conclude that~$p_c \leq 1/2$. 
\end{itemize}
The rest of the exercise is dedicated to proving that~$p_c \geq 1/2$. We proceed by contradiction and assume that~$p_c <1/2$ and therefore that~$\theta(1/2) > 0$.
The reasoning below is sometimes referred to as {\em Zhang's argument}. 
\begin{itemize}
\item[(b)] For~$p \in (0,1)$, if~$A_1,\dots, A_k$ are increasing events of equal $\bbP_p$-probability, show that
\begin{align}\label{eq:square_root_trick}
	\bbP_p[A_1] \geq 1-\big(1- \bbP_p\big[\bigcup_{j=1}^k A_j\big]\big)^{1/k}.
\end{align}
The above is called the {\em square-root trick} and is a direct consequence of the FKG property~\eqref{eq:FKG_perco}. 
For~$k$ fixed, it may be used to argue that~$\bbP_p[A_1]$ is close to~$1$ when~$\bbP_p[\bigcup_{j=1}^k A_j]$ is close to~$1$. 
\item[(c)] From $\theta(1/2) > 0$, deduce that for any~$\eps > 0$ there exists~$N$ such that 
\begin{align}
	\bbP_{1/2}\big[\{N\} \times [-N,N] \xlra{\La_N^c} \infty \big]	\geq 1- \eps. 
\end{align}
\item[(d)] Let~$\calE_N$ be the event that the left and right sides of~$\La_N$ are connected to~$\infty$ in~$\La_N^c$ by primal open paths, 
while the top and bottom sides of~$\La_N$ are connected to~$\infty$ in~$\La_N^c$  by dual open paths. 
Use point (c) to conclude that, under the assumption~$\theta(1/2) > 0$,
\begin{align}
	\bbP_{1/2}[\calE_N]	 \to 1 \qquad \text{ as~$N\to\infty$}.
\end{align}
\item[(e)] Argue that when~$\calE_N$ occurs, either~$\omega$ or~$\omega^*$ contains at least two infinite clusters. Use Theorem~\ref{thm:Burton-Keane} to conclude that 
$\theta(1/2) =0$ and~$p_c = 1/2$. 
\end{itemize}
\end{exo}

\begin{figure}
\begin{center}
\includegraphics[width = 0.6\textwidth]{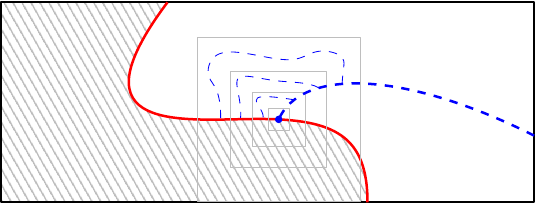}
\caption{The red path is the leftmost top-bottom crossing of the rectangle. 
A point on it connected by a dual path to the right side of the rectangle is a pivotal for the top-bottom crossing event. 
When one such point exists, one may use concentric dyadic annuli to create a logarithmic number of pivotals.}
\label{fig:Kesten_pivs}
\end{center}
\end{figure}

\begin{exo}(Compute~$p_c(\bbZ^2)$ using Kesten's tools)\label{exo:p_c=1/2K}
Consider Bernoulli percolation on~$\bbZ^2$. 
The goal of this exercise is to prove that~$p_c = 1/2$ and the sub-critical sharpness without using Theorem~\ref{thm:sharpness_perco}. 
Observe that~\eqref{eq:poly_pc} is a consequence of self-duality and the RSW theory; it does not use the fact that~$p_c = 1/2$. 
Moreover, it implies that~$p_c \geq 1/2$. 
Write~${\rm Cross}_v(n)$ for the event that the rectangle~$[-2n,2n] \times [0,n]$ contains a vertical open crossing. 
\begin{itemize}
\item[(a)] Show the existence of a constant~$c > 0$ such that if~$p$ and~$n$ satisfy
\begin{align*}
	\bbP_p[\La_n \lra \partial \La_{2n}] \leq c,
\end{align*}
then~$\bbP_p[0 \lra \partial \La_{k n}] \leq e^{-c k }$. \smallskip\\
{\em Hint:} use a ``renormalised'' Peierls argument, where a path from~$0$ to~$\partial \La_{k n}$ is shown to cross a number proportional to~$k$
of disjoint but adjacent annuli~$\La_{2n}(x) \setminus \La_n(x)$ with~$x \in n\bbZ^2$. Alternatively, see Proposition~\ref{prop:finite_size_ad}.
\item[(b)] Prove the existence of a constant~$c >0$ such that, for all~$p\leq \frac12$,
\begin{align}
	\bbE_p\big[\# \text{ pivotals for }{\rm Cross}_v(n)\,\big|\, {\rm Cross}_v(n) \big] \geq c \log n
\end{align}
{\em Hint:} When~${\rm Cross}_v(n)$ occurs, condition on the leftmost crossing, then create pivotals by constructing a dual cluster that touches this crossing coming from the right.
See Figure~\ref{fig:Kesten_pivs} for inspiration. 
\item[(c)] Conclude that for~$p< 1/2$, 
\begin{align}
\bbP_p[{\rm Cross}_v(n)] \to 0\qquad \text{ as~$n\to\infty$}. 
\end{align}
Using point (a), conclude that~$\bbP_p$ exhibits exponential decay of cluster radii. 
\item[(d)] Expressing this result in the dual model, prove that~$\theta(p) >0$ for any~$p >1/2$. 
\end{itemize}
\end{exo}


\chapter{FK-percolation: the basics}\label{ch:2intro_FK}

This chapter contains a very brief introduction to FK-percolation (also called the random-cluster model) on~$\bbZ^d$. 
Full proof and additional details may be found in~\cite{Gri06, Dum13, Dum20}.

\section{FK-percolation on finite graphs \& monotonicity}\label{sec:monotonicity}

We  start by defining the FK-percolation measure on finite graphs.
Fix some finite subgraph~$G = (V,E)$ of~$\bbZ^d$. Define its boundary by
$$\partial G = \{ v \in V(G) : \text{ incident to at least one edge outside of~$G$}\}.$$
A boundary condition~$\xi$ on~$G$ is a partition of~$\partial G$. Vertices in the same set of the partition are said to be {\em wired} together. 

Two specific boundary conditions play a special role, these are the {\em free boundary conditions}, denoted by~$\xi = 0$, where no vertices are wired together, and the {\em wired boundary conditions}, denoted by~$\xi =1$, where all vertices are wired together. 

\begin{definition}
 For~$G$ as above,~$q \geq 1$,~$p \in (0,1)$ and a boundary condition~$\xi$ on~$G$, define the FK-percolation measure~$\phi_{G,p,q}^\xi$ as the probability measure on~$\{0,1\}^E$ with 
 \begin{align*}
	\phi_{G,p,q}^\xi [\omega] = \frac{1}{Z_{G,p,q}^\xi} p^{|\omega|}(1-p)^{|E\setminus\omega|}q^{k(\omega^\xi)}, 
\end{align*}
where~$k(\omega^\xi)$ is the number of connected components of~$\omega$ where all vertices of each component of~$\xi$ are considered connected. The constant~$Z_{G,p,q}^\xi = \sum_\omega p^{|\omega|}(1-p)^{|E\setminus\omega|}q^{k(\omega^\xi)}$ is chosen so that~$\phi_{G,p,q}^\xi$ is a probability measure; it is called the {\em partition function} of the model. 
\end{definition}

Observe that Bernoulli percolation is a particular case of the above, obtained when~$q = 1$. 
For~$q \neq 1$, edges are not independent under~$\phi_{G,p,q}^\xi$; we call this a {\em dependent percolation} model. The questions of interest remain the same as in the Bernoulli case. 

The first and most basic property of FK-percolation is the Spatial Markov property.
\begin{proposition}[Spatial Markov property]
	For $G$ as above, a subgraph~$H$ of~$G$, $p\in (0,1)$, $q \geq 1$  and $\xi$ a boundary conditions on $G$, 
	\begin{align}\label{eq:SMP}
	\phi_{G,p,q}^\xi [\omega \text{ on } H \, |\,\omega \text{ on } G \setminus H  ] 
	= 	\phi_{H,p,q}^\zeta [\omega \text{ on } H]\tag{SMP}
	\end{align}
	where~$\zeta$ is the boundary condition induced by~$\omega^\xi$ on~$G \setminus H$, 
	that is, the wiring produced by~$\omega^\xi \setminus H$ between the vertices of~$\partial H$.
\end{proposition}

The proof is a direct computation which we omit. 
We turn to the question of monotonicity. Let us make a brief detour to address the general topic  of monotonicity of measures. 

\paragraph{Ordering of measures: generalities.}

There are two ways to view {\em stochastic ordering}. Consider two probability measures~$\mu,\nu$ on~$\{0,1\}^E$, where~$E$ denotes some finite set. 
We say that~$\mu \leq_{\rm st} \nu$  ($\nu$ stochastically dominates~$\mu$) if the two following equivalent conditions are satisfied:
\begin{itemize}
	\item[(a)] there exists a probability measure~$P$ producing two configurations~$\omega, \omega'$ such that~$\omega$ has law~$\mu$, 
	$\omega'$ has law~$\nu$,
	and~$\omega(e) \leq\omega'(e)$ for all~$e \in E$~$P$-a.s.\ We call~$P$ an increasing coupling of~$\mu$ and~$\nu$;
	\item[(b)]~$\mu[A] \leq \nu[A]$ for all increasing events~$A$.
\end{itemize}
The equivalence of (a) and (b) is the content of {\em Strassen's theorem}~\cite{Lin02}. 

A second, related notion is that of positive association. We say that~$\mu$ is {\em positively associated} if  
\begin{align*}
	\mu[A\cap B]\geq \mu[A]\mu[B]
\end{align*} for all increasing events~$A,B$. 
This may be understood as~$\mu[\cdot\,|\, A]\geq_{\rm st} \mu$, for any increasing event~$A$.

The following criteria are particularly convenient for proving stochastic monotonicity and positive association. 

\begin{theorem}[Holley \& FKG lattice condition] 
	 For positive measures~$\mu$,~$\nu$ on~$\{0,1\}^E$, 
	\begin{itemize}
	\item[(i)] if~$\mu(\omega \cap \omega')\nu(\omega \cup \omega') \geq \mu(\omega)\nu(\omega')$ for all~$\omega,\omega'\in \{0,1\}^E$,
	then~$\mu \leq_{\rm st} \nu$;\smallskip 
	\item[(ii)] if~$\mu(\omega \cap \omega')\mu(\omega \cup \omega') \geq \mu(\omega)\mu(\omega')$ for all~$\omega,\omega'\in \{0,1\}^E$,
	then~$\mu$ is positively associated. 
	\end{itemize}
	Moreover, it suffices to check these conditions for~$\omega$ and~$\omega'$ that differ only for at most two edges. 
\end{theorem}

The proof of the above is a beautiful use of the Glauber dynamics; we direct the reader to~\cite{Gri10} for details. 
It is worth mentioning that the condition in (ii) (sometimes called the FKG lattice condition) is stronger than positive association; 
measures satisfying it are sometimes called monotonic and have additional convenient properties.

\paragraph{Monotonicity properties of FK-percolation.}
Let us see how the properties above apply to FK-percolation. 

\begin{proposition}\label{prop:FKFKG}
	For a finite subgraph~$G$ of~$\bbZ^d$,~$q \geq 1$,~$p \in (0,1)$ and a boundary condition~$\xi$ on~$G$
	\begin{itemize}
	\item[(i)] 	$\phi_{G,p,q}^\xi$ is positively associated:
		\begin{align}\label{eq:FKGFK}
			\phi_{G,p,q}^\xi[A \cap B] \geq			\phi_{G,p,q}^\xi[A ]\cdot\phi_{G,p,q}^\xi[ B] \qquad \text{ for all~$A,B$ increasing}. \tag{FKG}
		\end{align}
	\item[(ii)] 	for~$p'\geq p$ and~$\zeta \geq \xi$ (in the sense that any vertices wired in~$\xi$ are also wired in~$\zeta$), 
\begin{align}\label{eq:Mon}
	\phi_{G,p,q}^\xi \leq_{\rm st} \phi_{G,p',q}^\zeta.\tag{Mon}
\end{align}
In other words,~$\phi_{G,p,q}^\xi$ is increasing in~$p$ and~$\xi$. 
\end{itemize}
\end{proposition}

For the above, it is crucial that~$q\geq1$. This is the main reason why the regime~$0<q <1$ is much less understood than~$q\geq 1$. 
 
An immediate consequence is  that the free and wired boundary conditions produce the minimal and maximal measures respectively: 
\begin{align*}
	\phi_{G,p,q}^0 \le_{\rm st} \phi_{G,p,q}^\xi \le_{\rm st} \phi_{G,p,q}^1 \qquad \text{ for any b.c.~$\xi$}. 
\end{align*}

Proposition~\ref{prop:FKFKG}, together with the Spatial Markov property~\eqref{eq:SMP} will be used often, sometimes in implicit ways; 
the novice reader will need some time to discover the full strength of these tools combined. 
Most often they are used as follows. If $H$ is a finite subgraph of $G$, 
\begin{align}
	\phi_{H,p,q}^0 \le_{\rm st} \phi_{G,p,q}^\xi\big|_H \le_{\rm st} \phi_{H,p,q}^1 \qquad \text{ for any b.c.~$\xi$},
	\label{eq:pushing_bc}
\end{align}
where~$|_H$ indicates the restriction to~$H$. Indeed, by~\eqref{eq:SMP}, the middle measure is a mixture of measures $\phi_{H,p,q}^\zeta$ for different boundary conditions $\xi$, 
for all of which the inequalities hold due to \eqref{eq:Mon}. Additionally,~\eqref{eq:pushing_bc} also applies when the middle measure is conditioned on events depending on the edges in $G\setminus H$.

To illustrate~\eqref{eq:Mon}, let us compute the probabilities for an edge to be open in the simplest setting: when~$G$ is formed of a single edge~$e$.
\begin{align}\label{eq:one_edge_open}
	\phi_{\{e\},p,q}^\xi[e \text{ open}]	= \begin{cases}
		p &\text{ if~$\xi = 1$} \\
		\frac{p}{p + (1-p)q}& \text{ if~$\xi = 0$}.
	\end{cases}	
\end{align}
Notice that the probabilities are indeed increasing in~$p$ and the boundary conditions. 
From the above, we deduce the following domination of and by Bernoulli percolation. 
	
\begin{corollary}\label{cor:perco_domination}
	For~$G$,~$q$,~$p$ and~$\xi$ as above, 
	\begin{align}\label{eq:perco_domination}
		\bbP_{\frac{p}{p + (1-p)q}} \leq_{\rm st} \phi_{G,p,q}^\xi \leq_{\rm st} \bbP_{p},
	\end{align}
	where~$\bbP_\cdot$ denotes the Bernoulli percolation on~$G$. 
\end{corollary}

\begin{proof}
	We do not give a full proof, but limit ourselves to mentioning that~$\phi_{G,p,q}^\xi$ may be obtained by sequentially sampling edges, using coin tosses with probabilities that depend on the edge and the previously sampled edges. Throughout the process, all coin tosses have parameters between~$\frac{p}{p + (1-p)q}$ and~${p}$. 
	
	Alternatively, one may use the Holley criterion. 
\end{proof}

Finally, let us mention that \eqref{eq:one_edge_open} also illustrates the {\em finite energy} principle, which states that, for any fixed~$p \in (0,1)$,
the probability that an edge is open is bounded away from~$0$ and~$1$, uniformly in the state of all other edge. 
This is often used repeatedly for edges in a finite set to perform local modifications, with only a limited multiplicative impact on the probabilities of the configurations.

\paragraph{Russo's formula and sharp-threshold inequalities.}

As for Bernoulli percolation, it will be useful to consider the derivatives of the probabilities of events as functions of~$p$. 

\begin{proposition}\label{prop:russo_FK}
	Let~$G$ be a finite subgraph of~$\bbZ^d$,~$q \geq 1$,~$p \in (0,1)$ and~$\xi$ a boundary condition on~$G$.
	Then, for any increasing event~$A$, 
	\begin{align}\label{eq:russo_FK}
	\tfrac{{\rm d}}{{\rm d}p}\phi_{G,p,q}^\xi[A]
	=\tfrac1{p(1-p)} \sum_{e \in E}\big( \phi_{G,p,q}^\xi[A \cap \{\omega(e) = 1\}] - \phi_{G,p,q}^\xi[A ]\phi_{G,p,q}^\xi[\omega(e)=1] \big).
	\end{align}
\end{proposition}

The summand in the right hand side of~\eqref{eq:russo_FK} is the covariance between the state of the edge~$e$ and the event~$A$ under the measure~$\phi_{G,p,q}^\xi$; write it simply~${\rm Cov}(\omega(e), A)$. 
Thus, the sum is the covariance between~$A$ and the total number of open edges. 

It may be useful to keep in mind that this covariance may also be interpreted as the {\em influence} of the edge~$e$ on~$A$.
\begin{align*}
\phi_{G,p,q}^\xi[A \cap \{\omega(e) = 1\}] - \phi_{G,p,q}^\xi[A ]\phi_{G,p,q}^\xi[\omega(e)]&\\
 = \phi_{G,p,q}^\xi[\omega(e)](1-\phi_{G,p,q}^\xi[\omega(e)])& \big(\phi_{G,p,q}^\xi[A | \omega(e) = 1] - \phi_{G,p,q}^\xi[A | \omega(e) = 0]\big),
\end{align*}
where the first two terms on the right-hand side are bounded away from~$0$ uniformly in~$p \in(0,1)$, away from~$0$ and~$1$. 
In the above, we write~$\phi_{G,p,q}^\xi[\omega(e)]$ for~$\phi_{G,p,q}^\xi[\omega(e)=1]$ since~$\phi_{G,p,q}^\xi$ is used as an expectation. 
We use the phrase ``uniformly in~$p \in(0,1)$, away from~$0$ and~$1$'' to mean uniformly in~$p \in (\eps,1-\eps)$ for any~$\eps >0$. Both conventions will be used repeatedly hereafter.

\begin{proof}
	We have 
	\begin{align*}
		\phi_{G,p,q}^\xi[A] =  \frac{1}{Z_{G,p,q}^\xi}\sum_{\omega \in A} p^{|\omega|}(1-p)^{|E\setminus\omega|}q^{k(\omega^\xi)}.
	\end{align*}
	Differentiating the above yields the desired result. 
\end{proof}

The following result, combined with~\eqref{eq:russo_FK}, 
will allow us to prove a {\em sharp-threshold} behaviour for certain connection probabilities. 
We state it here as it will be used in Section~\ref{sec:3conseq_phase_trans} to prove the sharpness of the phase transition of FK-percolation in two dimensions. 

\begin{theorem}[\cite{GraGri11}]\label{thm:BKKKL}
	Fix a finite subgraph~$G=(V,E)$ of~$\bbZ^d$,~$q \geq 1$,~$p \in (0,1)$ and a boundary condition~$\xi$ on~$G$.
	Then, for any event~$A$
	\begin{align}\label{eq:BKKKL}
 \sum_{e \in E} {\rm Cov}(A,\omega(e)) \geq c\,  \phi_{G,p,q}^\xi[A] (1 - \phi_{G,p,q}^\xi[A]) \log\frac{1}{\max_e {\rm Cov}(A,\omega(e))}
 \end{align}
 where~$c > 0$ is a constant depending only on~$p$ and~$q$, bounded away from~$0$ uniformly in~$p \in(0,1)$, away from~$0$ and~$1$. 
\end{theorem}

In~\eqref{eq:BKKKL}, the left hand side is 	$p(1-p)\tfrac{{\rm d}}{{\rm d}p}\phi_{G,p,q}^\xi[A]$ (see~\eqref{eq:russo_FK}), 
while $\phi_{G,p,q}^\xi[A] (1 - \phi_{G,p,q}^\xi[A])$ is the variance of~$\ind_A$ under the measure~$\phi_{G,p,q}^\xi$.
Thus, if no edge has a large covariance with~$A$ and if $\phi_{G,p,q}^\xi[A]$ is neither close to $0$, nor to $1$, 
then~\eqref{eq:BKKKL} states that the derivative of~$p \mapsto \phi_{G,p,q}^\xi[A]$ is large. 
It follows that~$\phi_{G,p,q}^\xi[A]$ quickly transitions from being close to~$0$ to being close to~$1$ as~$p$ increases.
We say it exhibits a sharp-threshold. 

We will not prove Theorem~\ref{thm:BKKKL}, but direct the reader to~\cite[Theorem 5.1]{GraGri11} for details. 
This type of inequality first appeared in~\cite{KahKalLin88} for product measures on~$\{0,1\}^E$ (which is to say for Bernoulli percolation) 
and was then extended to product measures on~$[0,1]^E$ in~\cite{BouKahKal92}.
Finally,~\cite{GraGri06,GraGri11} explained how to transfer such inequalities to monotonic measures such as FK-percolation. 

Other related results also allow to deduce sharp-threshold behaviour, most notably the OSSS inequality of~\cite{DumRaoTas19}.

\section{Infinite-volume measures}

The monotone {\em pushing} of boundary conditions~\eqref{eq:pushing_bc} allows to deduce that, for $H$ a subgraph of some finite graph $G$, 
\begin{align}\label{eq:pushing_bc}
	\phi_{H,p,q}^0 \leq_{\rm st} \phi_{G,p,q}^0\big|_H\qquad \text{ and }\qquad 	\phi_{H,p,q}^1 \geq_{\rm st} \phi_{G,p,q}^1\big|_H,
\end{align}
This in turn allows one to construct infinite-volume measures as monotone thermodynamical limits. 

\begin{fact}\label{fact:infinite_vol}
	For all~$p \in (0,1)$ and~$q\geq1$, the following limits exist for the weak convergence\footnote{We say that~$\mu_n \to \nu$ if, for any event~$A$ depending on a finite set of edges,~$\mu_n(A) \to \nu(A)$. }
	\begin{align}
	 	\phi_{p,q}^0= \lim_{n\to\infty}\phi_{\Lambda_n,p,q}^0 \qquad\text{ and }\qquad 
		\phi_{p,q}^1 = \lim_{n\to\infty}\phi_{\Lambda_n,p,q}^1.
	\end{align}
	The same limits are obtained for general graphs~$G_n$ that increase to~$\bbZ^d$. 
	Furthermore, both limits above are translation-invariant and ergodic probability measures on~$\{0,1\}^{E(\bbZ^d)}$.
\end{fact}

These infinite-volume measures satisfy the so-called Dobrushin--Lanford--Ruelle (DLR) condition, which essentially states that the Spatial Markov property also holds in infinite-volume. The DLR formalism allows one to define the general notion of infinite-volume measures (or Gibbs measures). In that context, Fact~\ref{fact:infinite_vol} states the existence of infinite-volume measures. We will not go further in the DLR formalism in these notes. \smallskip

It is generally not clear whether~$\phi_{p,q}^0$ and~$\phi_{p,q}^1$ are equal. In other words, when sending the boundary conditions to infinity, it is unclear whether they still manage to influence local events. 

It is a direct consequence of Proposition~\ref{prop:FKFKG} that 
\begin{align}\label{eq:phi_ordering0}
	\phi_{p,q}^0 \le_{\rm st} \phi_{p,q}^1 \qquad\text{ and } \qquad \phi_{p,q}^i \le_{\rm st} \phi_{p',q}^i \quad \text{ for all~$p<p'$ and~$i \in\{0,1\}$}.
\end{align}
Also, any limits (or infinite-volume measures) of~$\phi_{\Lambda_n,p,q}^{\xi_n}$ for sequences of boundary conditions~$\xi_n$ are always sandwiched between~$\phi_{p,q}^0$ and~$\phi_{p,q}^1$.

Observe that we have not yet managed to compare~$\phi_{p,q}^1$ and~$\phi_{p',q}^0$ for $p < p'$. The following result allows us to do this. 

\begin{proposition}\label{prop:phi_ordering}
	Fix~$q \geq 1$. There exist at most countably many values of~$p \in (0,1)$ such that~$\phi_{p,q}^0 \neq \phi_{p,q}^1$. 
	As a consequence, for all~$p < p'$,
	\begin{align}\label{eq:phi_ordering}
		\phi_{p,q}^0 \le_{\rm st} \phi_{p,q}^1 \le_{\rm st} \phi_{p',q}^0.
	\end{align}
\end{proposition}

Essentially, the proposition states that the influence of decreasing the boundary conditions can not compensate that of increasing the parameter. 

\begin{proof}
We only sketch the proof here as it is a very general approach. 
Define the {\em free energy} of FK-percolation with parameters~$p$ and~$q$ as
\begin{align}\label{eq:free_energy_FK_def}
	f(p,q) = \lim_{n} f_n^{\xi_n}(p,q) = \lim_{n} \tfrac{1}{n^d} \log Z_{\Lambda_n,p,q}^{\xi_n},
\end{align}
where the limit may be taken for any sequence of boundary conditions~$(\xi_n)$ and does not depend on this sequence\footnote{This is relatively standard. It is based on the simple observation that~$1 \leq Z_{\Lambda_n,p,q}^{\xi_n}/Z_{\Lambda_n,p,q}^{1} \le q^{|\partial \La_n|}$. The ``error'' term~$\tfrac{1}{n^d} \log q^{|\partial \La_n|}$ tends to~$0$, since~$|\partial \La_n| = o(n^d)$.}.

Write\footnote{This parametrisation is chosen to produce a convex function with no correction terms.}~$p = p(t) = \frac{e^{t}}{1 - e^{t}}$, with~$t \geq 0$, and set 
\begin{align}
	g_n^{\xi_n}(p,q) = f_n^{\xi_n}(p,q) +\tfrac{|E(\La_n)|}{n^d} \log(1-p)  \xrightarrow[n\to\infty]{} f(p,q) + d \log(1-p) =:g(p,q).
\end{align}
Explicit computations show that 
\begin{align}\label{eq:partialf_n}
		\frac{\rm d}{{\rm d} t}g_n^{\xi_n}(p,q) 
		&= \frac{1}{n^d}  \sum_{e\in E(\La_n)}  \phi_{\Lambda_n,p,q}^{\xi_n}[e \text{ open}],
\end{align}
which is increasing in~$p$. We conclude that the functions~$t \mapsto g_n^{\xi_n}(p(t),q)$ are all convex, and therefore so is~$t \mapsto g(p(t),q)$.

As a convex function,~$t \mapsto g(p(t),q)$ has left- and right-derivatives at all points, and is differentiable at all except at most countably many points. 
Furthermore, taking the limit as~$n\to\infty$ in~\eqref{eq:partialf_n}, we conclude\footnote{This step requires some care: the derivative should  be approximated by a finite increment, and the convexity of~$f$ should be used.} that, for all~$t$ for which~$t \mapsto g(p(t),q)$ is differentiable, 
\begin{align}\label{eq:f_edge_open}
	\frac{\rm d}{{\rm d} t }g(p,q) = \phi_{p,q}^0[e \text{ open}] = \phi_{p,q}^1[e \text{ open}] \qquad \text{ for any edge~$e$}.
\end{align}
The above, together with the stochastic ordering between~$\phi_{p,q}^0$ and~$\phi_{p,q}^1$ implies that~$\phi_{p,q}^0= \phi_{p,q}^1$.
Then,~\eqref{eq:phi_ordering} follows directly from the above and~\eqref{eq:phi_ordering0}.
\end{proof}

\begin{remark}\label{rem:f_FK_diff_p}
	Using the the notation of the proof above, 
	the differentiability of~$p \mapsto f(p,q)$ at a point~$p(t)$ is equivalent to that of~$t \mapsto g(p(t),q)$ at~$t$. 
	It follows that~$p \mapsto f(p,q)$ has left- and right-derivatives at all points
	and that these are equal for all $p$ except at at-most countably many points. 
	When they are equal, and $p \mapsto f(p,q)$ is differentiable, the infinite-volume measure is unique. 
	The converse implication may also be proved. 
\end{remark}

\section{Phase transition}

We are now ready to define the point of phase transition of FK-percolation, as we did for Bernoulli percolation. 

\begin{definition}
	Fix~$d \geq 2$ and~$q \geq 1$. Set 
	$$p_c = p_c(q) = \sup\{p: \, \phi^0_{p,q}[0\lra \infty] = 0\}.$$
\end{definition}

As in the case of Bernoulli percolation, $p_c$ separates two distinct regimes. Indeed, using Proposition~\ref{prop:phi_ordering}
and the ergodicity of $\phi^0_{p,q}$ and $\phi^1_{p,q}$, we conclude that 
\begin{itemize}
\item  for~$p<p_c$,~$\phi^1_{p,q}[0\lra \infty] = 0$  and~$\phi^1_{p,q}$-a.s.\ there exists no infinite cluster; 
\item for~$p>p_c$,~$\phi^0_{p,q}[0\lra \infty]  > 0$  and~$\phi^0_{p,q}$-a.s.\ there exists at least one infinite cluster. 
\end{itemize} 
In addition, the domination~\eqref{eq:perco_domination} by Bernoulli percolation allows us to deduce that 
$$0<p_c(q)<1 \qquad \text{for all~$q\geq1$ and~$d \geq 2$.}$$

We close this part with a useful application of the properties described above. 

\begin{proposition}\label{prop:unique_measure_subcrit}
	If~$p$ is such that~$\phi^1_{p,q}[0\lra \infty] = 0$, then~$\phi^0_{p,q}= \phi^1_{p,q}$. 
\end{proposition}

\begin{proof}
	This is a very instructive exercise, see Exercise~\ref{exo:influence_edge}.
\end{proof}

\section{References and further results}

We list here some important known results which will not be discussed in these notes.

\paragraph{Edwards--Sokal coupling.} 
One of the motivations for FK-percolation is its link to the~$q$-state Potts model, a spin model in which each vertex of a finite graph~$G$ is assigned a spin in~$\{1,\dots, q\}$. 
One way to view the coupling between FK-percolation and the Potts model on~$G$ is the following. Consider an integer~$q \geq 2$ and~$p \in (0,1)$. 
Sample a percolation configuration~$\omega$ on~$G$ according to~$\phi_{G,p,q}^0$, then assign independent spins in~$\{1,\dots, q\}$ to each  cluster of~$\omega$ --- that is, assign the same spin to all vertices in that cluster. The resulting spin configuration has the law of the~$q$-state Potts model with inverse temperature~$\beta = -\log(1-p)$.

This correspondence is especially fruitful for $q = 2$, when the corresponding Potts model is the Ising model. 
In this case, combining tools from FK-percolation with those coming from the Ising model (most notably its random-current representation)
often yields results only available for this specific value of $q$. We will not focus on them in these notes. 

\paragraph{Uniqueness of the infinite cluster.}
The argument of Burton and Keane that proves the uniqueness of the infinite cluster is very robust and also applies to FK-percolation.

\begin{theorem}[Uniqueness of the infinite cluster]\label{thm:Burton-Keane_FK}
	Fix~$d \geq 1$ and~$q \geq 1$. Then, for all~$p \in (0,1)$ and~$i \in \{0,1\}$, 
	\begin{align*}
		\phi_{p,q}^i[\text{exists no infinite cluster}] = 1 \quad \text{ or }\quad
		\phi_{p,q}^i[\text{exists exactly one infinite cluster}] = 1.
	\end{align*}
\end{theorem}

\paragraph{Sharpness of the phase transition.}

The sharpness result of Theorem~\ref{thm:sharpness_perco} was extended to general FK-percolation with~$q \geq 1$ in~\cite{DumRaoTas19} 
via a revolutionary use of the OSSS inequality. 

\begin{theorem}\label{thm:sharpness_FK_OSSS}
	Fix~$d\geq 2$ and~$q \ge 1$. For all~$p < p_c(q)$ there exists~$c(p) >0$ such that 
	\begin{align*}
		\phi^1_{p,q}[0\lra \partial \La_n] \leq e^{-c(p) n} \qquad \text{ for all~$n\geq 1$}.
	\end{align*}
\end{theorem}

At the time of writing, the super-critical sharpness of FK-percolation (in the sense of~\eqref{eq:sharp_super_crit}) remains an open problem  in dimensions greater than two.

\section{Exercises: basics of FK-percolation}

\begin{exo}\label{exo:continuity_FK}
	Fix~$d \geq 2$ and~$q \geq 1$. Prove that, for all~$p$, 
	\begin{align*}
		\lim_{u \nearrow p} \phi_{u,q}^0 = 		\lim_{u \nearrow p} \phi_{u,q}^1 = \phi_{p,q}^0\qquad   \text{ and }\qquad 
		\lim_{u \searrow p} \phi_{u,q}^0 = 		\lim_{u \searrow p} \phi_{u,q}^1 = \phi_{p,q}^1,
	\end{align*}
	in the sense that, for all~$A$ depending on finitely many edges~$\lim_{u \nearrow p} \phi_{u,q}^0[A]  = \phi_{p,q}^0[A]$.
\end{exo} 

%

\begin{exo}\label{exo:influence_edge}~
\begin{itemize}
	\item[(a)] 
	Fix  a finite subgraph~$G = (V,E)$ of~$\bbZ^d$, a boundary condition~$\xi$,~$p \in (0,1)$ and~$q \geq 1$;
	write~$\phi = \phi^\xi_{G,p,q}$.
	For~$A,B \subset V$ non-empty, recall that~$A \lra B$ is the event that some point of~$A$ is connected to some point of~$B$. Prove that, for any~$e \in E$, 	
	\begin{align}\label{eq:exo_influence_one_arm}
	\phi[A\lra B \,|\, \omega(e) = 1] - 	\phi[A\lra B \,|\, \omega(e) = 0] 
	\leq 	\phi[e\lra A \text{ and } e \lra  B \,|\, \omega(e) = 1].\qquad
	\end{align}
	{\em Hint:} explore the cluster of~$e$ under~$\phi[\cdot \,|\, \omega(e) = 1]$. If it does connect~$A$ to~$B$, prove that the probability of~$A\lra B$ conditionally on the realisation of the cluster is smaller than~$\phi[A\lra B\,|\, \omega(e) = 0]$.

	\item[(b)] Consider FK-percolation on~$\bbZ^d$ with~$q > 1$ and some~$p \in (0,1)$. 
	Using the same ideas as above, prove that for any~$n <N$ and any increasing event~$A$ depending only on the edges in~$\La_n$, 
	\begin{align*}
		0\le \phi_{\La_N,p,q}^1[A] -\phi_{\La_N,p,q}^0[A]  \leq \phi_{\La_N,p,q}^1[\La_n \lra\partial\La_N] ,
	\end{align*}
	
	\item[(c)]	Deduce that if~$p$ is such that~$\phi^1_{p,q}[0\lra\infty]= 0$, then~$\phi_{p,q}^0=\phi_{p,q}^1$. 
\end{itemize}
\end{exo}

\chapter{Phase transition of planar FK-percolation}\label{ch:3dichotomy}

For the rest of the notes, we focus on two dimensional FK-percolation. Unless otherwise stated, we work on~$\bbZ^2$. 
The cluster-weight~$q \geq 1$ will be fixed, and we remove it from the notation. 
Whenever the choice of~$p$ is clear, we also remove~$p$ from the notation. 

The goal of this chapter is to show that the planar FK-percolation exhibits a sharp phase transition at its self-dual point 
and to establish a dichotomy between two types of phase transition (continuous and discontinuous). 
Only in Chapter~\ref{ch:4cont_v_discont} will we establish which type of phase transition the model undergoes. 

\section{Duality for FK-percolation}

As for Bernoulli percolation, planar FK-percolation has a convenient duality property. 
For simplicity, we will state this property in the simple case of measures on boxes~$\La_n$ with either free or wired boundary conditions. 
Write~$\La_n^*$ for the subgraph of~$(\bbZ^2)^*$ formed of the edges dual to the edges of~$\La_n$.  

\begin{proposition}[Duality]\label{prop:duality}
	Fix~$n\geq 1$,~$q\geq 1$,~$p \in (0,1)$ and~$\xi \in \{0,1\}$. 
	If~$\omega$ is sampled according to~$\phi_{\La_n,p,q}^\xi$, then~$\omega^*$ has the law~$\phi_{\La_n^*,p^*,q}^{1-\xi}$,
	where~$p^* \in (0,1)$ is defined by 
	\begin{align}\label{eq:p_dual}
		\frac{p}{1-p} \cdot \frac{p^*}{1-p^*} = q.
	\end{align}
\end{proposition}

More general duality relations may be obtained, that is, for more general subgraphs and more general boundary conditions. 
Most may be deduced from the above, by fixing the configuration in parts of~$\La_n$ and applying~\eqref{eq:SMP}. 
We will not detail these generalisations, and will only use them in isolated places. 

\begin{proof}
	It suffices to consider the case~$\xi = 1$. 
	A simple induction shows that
	\begin{align}\label{eq:dual00}
		|\omega| + |\omega^*| = |E| \qquad \text{ and }\qquad k(\omega^1) - k(\omega^*) = (2n-1)^2 - |\omega|.
	\end{align}
	The first relation is obvious. The second is immediate for the empty configuration and may be extended to other configurations by induction. 	Indeed, when adding an edge~$e$ to some configuration~$\omega$, 
	either the number of connected components in the primal configuration decreases by~$1$ while that in the dual remains the same, 
	or the number of connected components of the primal configuration remains the same, but that in the dual increases by one. 
	
	The duality relation follows directly from~\eqref{eq:dual00}. 
\end{proof}

It is immediate to check that the only point for which~$p = p^*$ is the so-called {\em self-dual} point
\begin{align}\label{eq:p_sd}
	p_{\rm sd} =	p_{\rm sd}(q) = \frac{\sqrt q}{1 + \sqrt q}.
\end{align}
In light of the example of Bernoulli percolation (Section~\ref{sec:percolation}), it is natural to then conjecture that~$p_c = p_{\rm sd}$ for all~$q \geq 1$. This was first confirmed in~\cite{BefDum12} and will be deduced through different means below. 

Hereafter, we study the behaviour of the model at this self-dual parameter; that it corresponds to the point of phase transition and that the phase transition is sharp will be consequences of our investigation. We start off with an RSW-type estimate for symmetric quads.

A quad is a simply connected domain~$\calD$ with four marked points on its boundary~$a,b,c,d$ in counter-clockwise order, with the boundary of~$\calD$ formed of two arc~$(ab)$ and~$(cd)$  of the primal lattice and two arcs~$(bc)$ and~$(da)$ of the dual lattice (with a diagonal segment of length~$1/\sqrt 2$ between each arc). The {\em alternating} boundary condition on~$(\calD, a,b,c,d)$ --- denoted~${\rm alt}$ --- is produced by conditioning that all edges of the arcs~$(ab)$ and~$(cd)$ are open, while all other edges of the lattice outside of~$\calD$ are closed. 

We say that~$(\calD, a,b,c,d)$ is symmetric if there exists an isometry of~$\bbR^2$ that maps the primal lattice on the dual so that~$\calD$ is mapped to itself, with the arcs~$(ab)$ and~$(cd)$ mapped to the dual arcs~$(bc)$ and~$(da)$. 

\begin{corollary}\label{cor:symmetric_quad}
	Fix~$q\geq 1$ and a symmetric quad~$(\calD,a,b,c,d)$. 
	Then 
	\begin{align}\label{eq:symmetric_quad}
		\phi_{\calD,p_{\rm sd},q}^{\rm alt}[(ab)\xlra{\calD} (cd)] \geq \tfrac{1}{1+q}
	\end{align}
\end{corollary}

\begin{proof}
	The event~$\omega \in \{(ab)\xlra{\calD} (cd)\}$ is the complement of~$\omega^* \in \{(bc)\xlra{\calD} (da)\}$. 
	Moreover, if~$\omega \sim \phi_{\calD,p_{\rm sd},q}^{\rm alt}$ then~$\sigma(\omega^*)$ has almost the same law, where~$\sigma$ is the isometry for which~$\calD$ is symmetric. 
	Indeed,~$\sigma(\omega^*) \sim \phi_{\calD,p_{\rm sd},q}^{\rm alt'}$ 
	where the boundary conditions~${\rm alt'}$ are identical to~${\rm alt}$, except that the two wired arcs are also wired together. 
	This difference in boundary conditions affects the weight of any configuration by a factor of at most~$q$. 
	Thus, the Radon--Nikodim derivative of~$\phi_{\calD,p_{\rm sd},q}^{\rm alt'}$ with respect to~$\phi_{\calD,p_{\rm sd},q}^{\rm alt}$ 
	is between~$1/q$ and~$q$. It follows that
	\begin{align*}
	 	\phi_{\calD,p_{\rm sd},q}^{\rm alt}[\omega^* \in \{(bc)\xlra{\calD} (da)\}]
		& = \phi_{\calD,p_{\rm sd},q}^{\rm alt'}[\omega \in \{(ab)\xlra{\calD} (cd)\}]\\
	 	&\leq q\, \phi_{\calD,p_{\rm sd},q}^{\rm alt}[\omega \in \{(ab)\xlra{\calD} (cd)\}].
	\end{align*}
	The above, together with the complementarity of the two events, proves~\eqref{eq:symmetric_quad}.
\end{proof}

\begin{remark}\label{rem:symmetric_quad}
	We will mostly use the above for quads~$(\calD,a,b,c,d)$ which are not really symmetric, but are ``better'' than symmetric. 
	That is, quads so that if~$\omega \notin \{(ab) \xlra{\calD} (cd)\}$ then~$\sigma (\omega^*) \in \{(ab) \xlra{\calD}(cd)\}$.
	The proof of Corollary~\ref{cor:symmetric_quad} applies readily to this case, as well. 
\end{remark}

\section{Statement of the dichotomy theorem}

Write~$H_n$ for the event that there exists an open circuit in~$\Lambda_{2n} \setminus \Lambda_{n}$ that surrounds~$\Lambda_{n}$;
also denote by~$H_n^*$ the same event for the dual model.

\begin{theorem}[Duminil-Copin, Sidoravicius, Tassion 2017]\label{thm:dichotomy}
	Fix~$q \geq 1$ and let~$p = p_{\rm sd} = \sqrt q/(1+\sqrt q)$. Then exactly one of the two scenarios below occurs. 
	\begin{itemize}
		\item[(Con)] There exists~$c > 0$ such that, for all~$n\geq 1$ 
		\begin{align}\label{eq:Con}
			\phi_{\Lambda_{2n}\setminus \La_n}^0[H_n]  \geq c \quad\text{ and }\quad\phi_{\Lambda_{2n}\setminus\La_n}^1[H_n^*]  \geq c.
		\end{align}
		\item[(DisCon)] There exists~$c > 0$ such that, for all~$n\geq 1$  
		\begin{align}\label{eq:DisCon}
			\phi^0[0\xlra{} \partial \La_n]  \leq e^{-cn}\quad\text{ and }\quad	\phi^1[0\xlra{*} \partial \La_n]  \leq e^{-cn}.
		\end{align}
	\end{itemize}
\end{theorem}

The above implies quite easily that~$p_{\rm sd}$ is the critical point, along with the sub- and super-critical sharpness of the phase transition.

%
%

\begin{corollary}[Phase transition on~$\bbZ^2$]\label{cor:p_c(q)}
	For each~$q \geq 1$, we have 
	\begin{align*}
		p_c = p_{\rm sd}(\bbZ^2) = \tfrac{\sqrt q}{1 + \sqrt q}.
	\end{align*}
	Furthermore, for all~$p\neq p_c$,~$\phi_p^0 = \phi_p^1$ and there exists~$c = c(p) > 0$ such that 
	\begin{align}\label{eq:sharp2DFK}
		\phi^1[0\xlra{} \partial \La_n\text{ but } 0\nxlra{} \infty] \leq e^{-cn} \qquad \text{ for all~$n\geq 1$}.
	\end{align}
	Finally,~$p \mapsto \theta(p)$ is a continuous function on~$(0,1)$, except potentially at~$p_c$. 
\end{corollary}

As such, Theorem~\ref{thm:dichotomy} describes two possible behaviours at the critical parameter, 
which in turn correspond to two types of phase transitions. 
Indeed, (DisCon) corresponds to a discontinuous (or first order) phase transition, while (Con) corresponds to a continuous phase transition (or one of order two or higher).
For~$p < p_c$, define the {\em correlation length} as
\begin{align}
\xi(p)^{-1} = \lim_{n\to\infty} - \tfrac1n \log\phi_{\Lambda_n,p,q}^0[0\lra \partial \La_n];
\end{align}
see Exercise~\ref{exo:cor_len} for details on why the limit exists. 
For~$p > p_c$, set~$\xi(p) = \xi(p^*)$.
Corollary~\ref{cor:p_c(q)} shows that~$\xi(p) < \infty$ for all~$p\neq p_c$. 
In addition, the following features hold for the two types of phase transition. 

\begin{corollary}\label{cor:phase_transition}
If (DisCon) occurs at~$p_{\rm sd}$, we have 
\begin{itemize}
\item non-uniqueness of infinite-volume measure at criticality:~$\phi^0_{p_c}\neq	\phi^1_{p_c}$; 
\item~$p\mapsto \theta(p)$ is discontinuous at~$p_c$, as is the edge intensity~$p \mapsto \phi_p^0(e \text{ open})$;
\item~$p\mapsto f(p,q)$ is not differentiable at~$p_c$; 
\item the correlation length~$\xi(p)$ is bounded uniformly in~$p \in (0,1)$. 
\end{itemize}
Conversely, if (Con) occurs at~$p_{\rm sd}$, we have 
\begin{itemize}
\item uniqueness of infinite-volume measure at criticality:~$\phi^0_{p_c}=	\phi^1_{p_c}$;
\item~$p\mapsto \theta(p)$ and $p \mapsto \phi_p^0(e \text{ open})$ are continuous at~$p_c$. Moreover, there exist constants~$c, \alpha > 0$ such that 
\begin{align}\label{eq:one_arm_FK}
c\, n^{-1} \leq \phi_{p_c}[0\lra \partial \La_n] \leq n^{-\alpha} \qquad \text{ for all~$n\geq 1$};
\end{align}
\item~$p\mapsto f(p,q)$ is differentiable at~$p_c$.
\item the correlation length tends to infinity as~$p \to p_c$.
\end{itemize}
\end{corollary}

Section~\ref{sec:3conseq_phase_trans} discusses how Theorem~\ref{thm:dichotomy} implies Corollaries~\ref{cor:p_c(q)} and~\ref{cor:phase_transition}. 
We then focus on the proof of  Theorem~\ref{thm:dichotomy}, the heart of this Chapter. Section~\ref{sec:RSW_stirp} proves an instrumental RSW estimate (sometimes referred to as the {\em pushing lemma}). Then, in Section~\ref{sec:dichotomy_proof}, we prove a renormalisation inequality which implies Theorem~\ref{thm:dichotomy}.

\section{Consequences for the phase transition}\label{sec:3conseq_phase_trans}

The main goal of this section is to prove the sharpness of the phase transition and compute the value of the critical point for FK-percolation on~$\bbZ^2$ using Theorem~\ref{thm:dichotomy}. There are ways of proving these facts without referring to Theorem~\ref{thm:dichotomy} (e.g. via the general sharpness result Theorem~\ref{thm:sharpness_FK_OSSS} --- see Exercise~\ref{exo:p_c_FK_Zhang} ---  or following~\cite{BefDum12}), but as we are ultimately interested in the behaviour of the model at~$p_c$, Theorem~\ref{thm:dichotomy} will be proved anyway, and the sharpness follows fairly quickly. 
Additionally, the two regimes of Theorem~\ref{thm:dichotomy} will allow us to further describe the phase transition. 

Before proving Corollaries~\ref{cor:p_c(q)} and~\ref{cor:phase_transition}, we state a so-called ``finite size criterion'' for the exponential decay of cluster radii. 

\begin{proposition}\label{prop:finite_size_ad}
	Fix $p \in (0,1)$. 	There exists~$\delta > 0$ such that, 
	\begin{align}\label{eq:finite_size_ad}
		\sup\{ \phi_{\Lambda_{4n}}^1[H_n^*] :\, n\geq 1\}  > 1 - \delta \Leftrightarrow
		\big( \exists c > 0 \text{ s.t. }\phi_{\Lambda_n}^1[0\lra\partial \Lambda_n] \leq e^{-c n}\,\,\, \forall n\geq 1 \big).
	\end{align}
	The same holds for the dual model.
\end{proposition}

\begin{figure}
\begin{center}
\includegraphics[height = 3.2cm]{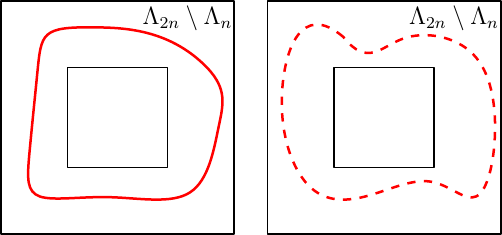}\qquad \qquad \quad
\includegraphics[height = 3.2cm]{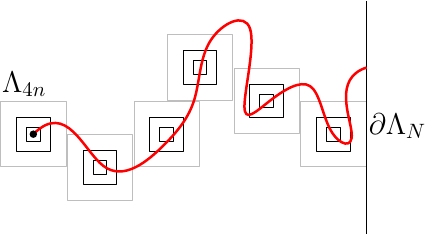}
\caption{{\em Left:} The events~$H_n$ and~$H_n^*$, respectively.
{\em Right:} If~$0$ is connected to~$\partial \La_N$, then there exists a path of at least~$\lfloor N/8n\rfloor$ disjoint but neighbouring translates of~$\La_{4n}$ by points of~$(n\bbZ)^2$ for which~$H_n^*$ fails. The black annuli are the translates of~$\La_{2n}\setminus \La_n$ for which~$H_n^*$ fails; the grey boxes are the corresponding translates of~$\La_{4n}$.}
\label{fig:Hn}
\end{center}
\end{figure}

\begin{proof}[Proof of Proposition~\ref{prop:finite_size_ad}]
	The implication from right to left is obvious since 
	\begin{align}
		1- \phi_{\Lambda_{4n}}^1[H_n^*] \leq |\partial \La_n| \phi_{\Lambda_n}^1[0\lra\partial \Lambda_n]. 
	\end{align}We focus below on the opposite implication. The idea of the proof is described in Figure~\ref{fig:Hn}.

	Fix some~$k,n \geq 1$.
	For~$N = 8kn$, the occurrence of~$0\lra\partial \Lambda_N$ implies the existence of a family of~$k$ points~$x_1,\dots, x_k \in (n\bbZ)^2$ 
	so that the translate~$\Lambda_n(x_i) \lra \Lambda_{2n}(x_i)^c$ of~$(H_n^*)^c$ occurs for each~$i = 1,\dots, k$ 
	and so that~$x_{i}$ is at a~$L^\infty$-distance~$8n$ from~$x_{i-1}$. 
	The number of such families of points may be bounded above by~$C^k$ for some fixed constant~$C$, while 
	\begin{align*}
		\phi_{\Lambda_N}^1\Big[\bigcap_{i=1}^k \{\Lambda_n(x_i) \lra \Lambda_{2n}(x_i)^c\} \Big]
		\le (1 - \phi_{\Lambda_{4n}}^1[H_n^*])^k. 
	\end{align*}
	Thus, we obtain exponential decay of~$\phi_{\Lambda_N}^1[0\lra\partial \Lambda_N]$ as soon as~$1 - \phi_{\Lambda_{4n}}^1[H_n^*] < \frac1{2C} =: \delta$. 
\end{proof}

\begin{proof}[Proof of Corollary~\ref{cor:p_c(q)}]
	The proof depends on which of the cases (DisCon) and (Con) occurs at~$p_{\rm sd}$. 
	We will prove all claims of the corollary separately in the two cases, 
	except for the continuity of~$\theta$, which will be discussed at the end of the proof. 
	\medskip 
	
	\noindent
	{\bf  (DisCon)} If (DisCon) holds for~$p_{\rm sd}$, then~\eqref{eq:phi_ordering} implies that~$\phi^1_p$ has exponential decay of cluster radii for all~$p < p_{\rm sd}$. 
	In particular,~$\phi^1_p$ contains a.s.\ no infinite cluster, which implies that~$p_c\geq p_{\rm sd}$ and~$\phi^0_p = \phi^1_p$ for all~$p < p_{\rm sd}$  (see Proposition~\ref{prop:unique_measure_subcrit}).
	
	Applying the same reasoning for the dual model proves that~$p_c \leq p_{\rm sd}$ and that for all~$p > p_{\rm sd}$ there exists a unique infinite-volume measure, for which~$\omega^*$ has exponential decay of cluster radii. The super-critical sharpness~\eqref{eq:sharp2DFK} follow directly. 	
	\medskip 
	
	\noindent
	{\bf  (Con)}
	If (Con)  holds for~$p_{\rm sd}$, a more involved argument is necessary: it requires the use of some form of sharp threshold technique (here in the form of Theorem~\ref{thm:BKKKL}) and the finite size criterion of Proposition~\ref{prop:finite_size_ad}.
	We start by proving that, for any~$p<p_{\rm sd}$,
	\begin{align}\label{eq:psd_finite_size}
		1- \phi^1_{\Lambda_{4n}, p} [H_n^*]  = \phi^1_{\Lambda_{4n}, p} [\La_n \lra \partial \La_{2n}]  \xrightarrow[n\to\infty]{} 0.
	\end{align}
	
	As explained in Exercise~\ref{exo:algebraic_decay},~\eqref{eq:Con} implies the existence of~$\alpha > 0$ such that
	\begin{align}\label{eq:one_arm_sd}
		\phi_{\La_{2n}, p_{\rm sd}}^1[0\lra \partial \La_n] \leq n^{-\alpha} \qquad \text{ for all~$n\geq 1$}.
	\end{align}
	Exercise~\ref{exo:influence_edge} allows to bound the influence of any individual edge~$e \in \La_{4n}$ on the increasing event~$\La_n \lra \partial \La_{2n}$ as follows
	\begin{align}
		{\rm Cov}(\{\La_n \lra \partial \La_{2n}\},\omega_e)
		&\leq 
		\phi^0_{\Lambda_{4n}, p} [\La_n \lra \partial \La_{2n} \,|\, \omega_e = 1] -
		\phi^0_{\Lambda_{4n}, p} [\La_n \lra \partial \La_{2n}]  \\
		&\leq \phi^0_{\Lambda_{4n}, p} [e \lra\partial \La_{2n} \text{ and }e \lra \La_{n}   \,|\, \omega_e = 1]\\
		&\leq c\, \phi^0_{\Lambda_{8n}, p} [0 \lra \partial \Lambda_{n/2}] \leq c' n^{-\alpha},\label{eq:influence_boun32}
	\end{align}
	where~$c,c' > 0$ are constants that are independent of~$n$ and uniformly positive in~$p$, away form~$0$ and~$1$;
	${\rm Cov}$ is the covariance under~$\phi^0_{\Lambda_{4n}, p}$;
	the second inequality is given by \eqref{eq:exo_influence_one_arm};
	the third is due to  \eqref{eq:Mon} and the last to~\eqref{eq:one_arm_sd} and~\eqref{eq:Mon}.
	
	Our assumption (Con) also shows that~$\sup_{n \geq 1} \phi^0_{\Lambda_{4n}, p_{\rm sd}} [\La_n \lra \partial \La_{2n}] < 1$.
	 Applying Proposition~\ref{prop:russo_FK}, Theorem~\ref{thm:BKKKL} and using~\eqref{eq:influence_boun32}, we conclude that~$p \mapsto \phi^0_{\Lambda_{4n}, p} [\La_n \lra \partial \La_{2n}]$ 
	satisfies a sharp-threshold principle below~$p_{\rm sd}$, which is to say that
	\begin{align*}
		\phi^0_{\Lambda_{4n}, p} [\La_n \lra \partial \La_{2n}] \xrightarrow[n\to \infty]{} 0 \qquad \text{ for all~$p < p_{\rm sd}$}. 
	\end{align*}
	Notice the difference in boundary conditions when compared to \eqref{eq:psd_finite_size}.
	To overcome this, we write
	\begin{align*}
		\phi^1_{\Lambda_{4n}, p} [\La_n \lra \partial \La_{2n}]  
		\leq \frac{\phi^1_{\Lambda_{4n}, p} [\La_n \lra \partial \La_{2n} \,|\, H_{2n}^*] }{\phi^1_{\Lambda_{4n}, p} [ H_{2n}^* \,|\, \La_n \lra \partial \La_{2n}]}
		\leq \frac{\phi^0_{\Lambda_{4n}, p} [\La_n \lra \partial \La_{2n}]  }{\phi^1_{\Lambda_{4n}, p} [ H_{2n}^* \,|\, \La_n \lra \partial \La_{2n}]}
	\end{align*}
	where the inequality is obtained by exploring the outermost circuit realising~$H_{2n}^*$ and ``pushing'' boundary conditions as in~\eqref{eq:pushing_bc}.
	Similarly,   we have that
	\begin{align}
		\inf_{p\leq p_{\rm sd}; n\geq 1}\phi^1_{\Lambda_{4n}, p} [ H_{2n}^* \,|\, \La_n \lra \partial \La_{2n}] \geq
	\inf_{n\geq 1}\phi^1_{\Lambda_{4n} \setminus \La_{2n}, p_{\rm sd}} [ H_{2n}^*] > 0,
	\end{align}
	 due to our assumption that (Con) occurs at~$p_{\rm sd}$. This yields~\eqref{eq:psd_finite_size}.
	
	Finally, Proposition~\ref{prop:finite_size_ad} allows us to conclude that, for all~$p < p_{\rm sd}$,~$\phi_{p}^1$ exhibits exponential decay of cluster radii \eqref{eq:sharp2DFK}, and therefore that~$\theta(p) = 0$ and~$\phi_p^0 =\phi_p^1$. 
	The same reasoning applied to the dual model proves~\eqref{eq:sharp2DFK} for~$p > p_{\rm sd}$. 
	We also conclude that~$p_c = p_{\rm sd}$.\bigskip 
	
	\noindent In closing, let us study the continuity of~$\theta$. By general monotonicity arguments, it may be proved that 
	for any~$p \in (0,1)$,
	\begin{align}
		\lim_{u \nearrow p} \theta(u) = \phi^0_p[0\lra \infty]\quad \text{ and }\quad \lim_{u \searrow p} \theta(u) = \phi^1_p[0\lra \infty].
	\end{align}
Thus, for any point~$p$ where~$\phi^0_p = \phi^1_p$,~$\theta$ is continuous. 
\end{proof}

Next, we prove the finer consequences of (Con) and (DisCon) on the phase transition, namely those listed in Corollary~\ref{cor:phase_transition}.
We do not give detailed arguments here; see the exercises for additional indications. 

\begin{proof}[Proof of Corollary~\ref{cor:phase_transition}]
Assume first that (DisCon) occurs at~$p_{\rm sd}$. Then~\eqref{eq:DisCon} directly implies that~$\phi^0_{p_c}\neq	\phi^1_{p_c}$. 

A general approach similar to that used in the proof of Proposition~\ref{prop:phi_ordering}
shows that~$p\mapsto f(p,q)$ is not differentiable at~$p_c$ and that 
$\lim_{p \nearrow p_c} \phi_p = \phi_{p_c}^0$ and~$\lim_{p \searrow p_c} \phi_p = \phi_{p_c}^1$.
This immediately implies the discontinuity at~$p_c$ of~$p\mapsto \theta(p)$ and~$p \mapsto \phi_p^0(e \text{ open})$.
It also implies that 
\begin{align}
\lim_{p \nearrow p_c} \xi(p) = \xi^0(p_c),
\end{align}
where~$\xi^0(p_c)$ is the correlation length of~$\phi_{p_c}^0$, which, by~\eqref{eq:DisCon}, is  finite. 
\medskip 

Assume now that (Con) occurs at~$p_{\rm sd}$.
Then a standard RSW construction shows that~\eqref{eq:one_arm_FK} holds for~$p = p_{\rm sd}$ (see Exercise~\ref{exo:algebraic_decay}).

The uniqueness of infinite-volume measure follows from the fact that~$\phi_{p_c}^1 [0\lra\infty] = 0$ (see Exercise~\ref{exo:influence_edge}), which in turn follows from \eqref{eq:one_arm_FK}. 
Then, by standard monotonicity arguments, we find that~$p\mapsto \theta(p)$ and~$p\mapsto  \phi_p(e \text{ open})$ are continuous and 
$p\mapsto f(p,q)$ is differentiable at~$p_c$.

Finally, Exercise~\ref{exo:cor_len_diverges} explains that~$\xi(p)$ must diverge at~$p_c$ if~$\phi^{0}_{p_c}$ does not have exponential decay of cluster radii. 
\end{proof}

\section{RSW in strips}\label{sec:RSW_stirp}

Throughout this section, we work with~$p = p_{\rm sd}$ an omit it from the notation. 
Set~${\rm Strip}_N = \bbZ \times [-N,N]$. Let~$1/0$ be the boundary conditions that are wired on the top of~${\rm Strip}_N$ and free on the bottom. 
The following RSW theorem is central to the proof of Theorem~\ref{thm:dichotomy}.

The measure~$\phi_{{\rm Strip}_N}^{1/0}$ is defined as a limit of measures on rectangles~$[-M,M]\times [-N,N]$ as~$M\to\infty$;
the boundary conditions on the lateral sides of the rectangles are irrelevant due to the finite energy property. 

\begin{theorem}[RSW in strips]\label{thm:RSWstrip}
	There exists~$c >0$ such that, for any~$\rho \geq 1$ and~$N\geq n$,
	\begin{align}\label{eq:RSWstrip}
		&\phi_{{\rm Strip}_N}^{1/0}\big[ [-\rho n,\rho n] \times [-n,n] \text{ crossed horizontally}\big]\geq c^\rho.
	\end{align}
	The same applies to crossing probabilities for the dual model.
\end{theorem}

The proof of the above is similar to the corresponding proof for Bernoulli percolation (Section~\ref{sec:RSW_perco}), with a few exceptions. 
Here, the measure is not invariant under rotations by~$\pi/2$ or vertical reflections, 
but is invariant under horizontal translations and reflections. 
It has a self-duality property which may be used to obtain a ``initial estimate'' similar to~\eqref{eq:cross_square}.

The lack of invariance under rotations by $\pi/2$ and vertical reflections will lead to the use of rectangles rather than squares, 
and to an adaptation of the notions~${\sf BottomLeft}$ and~${\sf TopRight}$.
Finally, in FK-percolation one should be aware of the effect of boundary conditions. 
The construction of the symmetric domain in the proof of~\eqref{eq:GG} is designed exactly for this purpose, and will also apply here. 

Theorem~\ref{thm:RSWstrip} is a consequence of the following lemmas. 
Write~$\Rect(m ,n) = [-m,m] \times [-n,n]$ and~$\calC_h(\Rect(m ,n))$ and~$\calC_v(\Rect(m ,n))$ for the events that~$\Rect(m ,n)$ contains a horizontal and a vertical crossing, respectively. 

\begin{lemma}(Duality)\label{lem:duality_strip}
	For any~$N \geq n$ and~$m \geq 1$, 
	\begin{align}\label{eq:duality_strip}
		\phi_{{\rm Strip}_N}^{1/0}\big[\calC_h(\Rect(m ,n))\big] + \phi_{{\rm Strip}_N}^{1/0}\big[ \calC_v(\Rect(m ,n))\big] = 1.
	\end{align}
\end{lemma}

\begin{lemma}(Non-degenerate horizontal crossings)\label{lem:cross_easy}
	There exists a constant~$0 < c \leq 1/2$ such that, for any~$N \geq n$
	\begin{align}\label{eq:cross_easy}
		\phi_{{\rm Strip}_N}^{1/0}\big[\calC_h(\Rect(c n ,n))\big] \geq c.
	\end{align}
\end{lemma}

\begin{lemma}(Lengthening crossings)\label{lem:lengthening_crossings}
	There exists a constant~$c >0$ such that, for any~$N \geq n$ and~$m \geq 1$, 
	\begin{align}\label{eq:lengthening_crossings}
		\phi_{{\rm Strip}_N}^{1/0}\big[\calC_h(\Rect(2m ,n))\big]
		\geq c\,  \phi_{{\rm Strip}_N}^{1/0}\big[ \calC_h(\Rect(m ,n))\big]^4 \phi_{{\rm Strip}_N}^{1/0}\big[ \calC_v(\Rect(m ,n))\big]^2.
	\end{align}
\end{lemma}

Of the three lemmas, the last one is the most important and hardest to prove. Before proving the lemmas, let us see how they imply Theorem~\ref{thm:RSWstrip}.

\begin{proof}[Proof of  Theorem~\ref{thm:RSWstrip}]
	Fix~$N \geq n$. Observe that the finite energy property implies that
	\begin{align}\label{eq:finite_energy_length}
		\phi_{{\rm Strip}_N}^{1/0}\big[\calC_h(\Rect(m + 1 ,n))\big] \geq  \big(\tfrac{p}{p + (1-p)q}\big)^2 \phi_{{\rm Strip}_N}^{1/0}\big[\calC_h(\Rect(m ,n))\big],
	\end{align}
	since for any~$\omega$ in~$\calC_h(\Rect(m ,n))$, opening at most two edges produces a configuration in~$\calC_h(\Rect(m + 1 ,n))$.
	
	Let~$c_0 > 0$ be the constant given by Lemma~\ref{lem:cross_easy} and let~$m$ be the largest value for which 
	\begin{align*}
		\phi_{{\rm Strip}_N}^{1/0}\big[\calC_h(\Rect(m-1 ,n))\big] \geq c_0.
	\end{align*} 
	Then, by Lemma~\ref{lem:cross_easy},~$m \geq c_0 n$. Moreover, 
	\begin{align}
		\phi_{{\rm Strip}_N}^{1/0}\big[\calC_h(\Rect(m,n))\big] \geq c_1 \quad \text{ and }\quad 
		\phi_{{\rm Strip}_N}^{1/0}\big[\calC_v(\Rect(m,n))\big] \geq 1/2,
	\end{align} 
	where~$c_1 = \big(\tfrac{p}{p + (1-p)q}\big)^2\, c_0$ is a uniformly positive constant --- for the first inequality, see~\eqref{eq:finite_energy_length}, for the second we use Lemma~\ref{lem:duality_strip} and the fact that~$c_0 \leq1/2$.
	
		\begin{figure}
	\begin{center}
	\includegraphics[width = .56\textwidth]{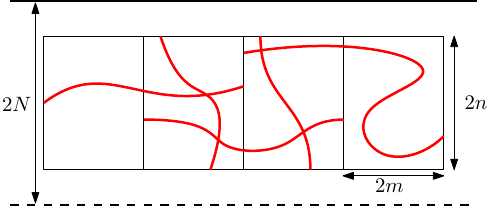}
	\caption{By combining vertical crossings of horizontal translations of~$\Rect(m ,n)$ 
	and horizontal crossings of translations of~$\Rect(2 m ,n)$, we produce horizontal crossings of~$\Rect(k m ,n)$.
	We work here in a horizontal strip with wired boundary conditions on the top and free on the bottom.}
	\label{fig:strip_lengthen}
	\end{center}
	\end{figure}
	
	Lemma~\ref{lem:lengthening_crossings} now implies that
	\begin{align*}
		\phi_{{\rm Strip}_N}^{1/0}\big[\calC_h(\Rect(2 m ,n))\big] \geq c_2
	\end{align*}
	for some uniformly positive constant~$c_2$. 
	Then, by combining horizontal and vertical crossings as in Figure~\ref{fig:strip_lengthen}, we conclude that, for any~$k\geq 2$,
	\begin{align*}
		\phi_{{\rm Strip}_N}^{1/0}\big[\calC_h(\Rect(k m ,n))\big] \geq \phi_{{\rm Strip}_N}^{1/0}\big[\calC_h(\Rect(2 m ,n))\big]^{k-1} \phi_{{\rm Strip}_N}^{1/0}\big[\calC_v(\Rect(m,n))\big]^{k-2}
		\geq (\tfrac12c_2)^k. 
	\end{align*}
	In light of the fact that~$m \geq c_0n$, this implies~\eqref{eq:RSWstrip}.

	Finally, observe that, by the duality of~$\phi_{{\rm Strip}_N}^{1/0}$, the probability in~\eqref{eq:RSWstrip} is the same for primal and dual crossings. This allows us to conclude that~\eqref{eq:RSWstrip} also holds for~$\omega^*$. 
\end{proof}

We now turn to the proof of the three lemmas. 
\begin{proof}[Proof of Lemma~\ref{lem:duality_strip}]
	Note that if~$\omega \notin \calC_h(\Rect(m ,n))$, then~$\omega^*\in \calC_v(\Rect(m ,n))$, and vice-versa. 
	Moreover, if~$\omega \sim \phi_{{\rm Strip}_N}^{1/0}$, then~$\omega^*$ has the same law as the vertical reflection of~$\omega$. 
	Thus
	\begin{align*}
		\phi_{{\rm Strip}_N}^{1/0}\big[\omega \in  \calC_v(\Rect(m ,n))\big] = \phi_{{\rm Strip}_N}^{1/0}\big[\omega^* \in  \calC_v(\Rect(m ,n))\big]
	\end{align*}
	These two observations together yield~\eqref{eq:duality_strip}.
\end{proof}

\begin{proof}[Proof of Lemma~\ref{lem:cross_easy}]
	We will prove by contradiction that 
	\begin{align*}
	\phi_{{\rm Strip}_N}^{1/0}\big[\calC_h(\Rect(n/3 ,n))\big] \geq \tfrac{1}{4(1 +  q)},
	\end{align*}
	which suffices to deduce~\eqref{eq:cross_easy}. 	
	Assume the opposite for some~$n$. Then~\eqref{eq:duality_strip} implies that 
	\begin{align*}
	\phi_{{\rm Strip}_N}^{1/0}\big[\calC_v(\Rect(n/3 ,n))\big] \geq \tfrac12.
	\end{align*}
	Consider now the translates~$\Rect_L  =\Rect(n/3 ,n) - (2n/3,0)$ and~$\Rect_R  =\Rect(n/3 ,n) + (2n/3,0)$. 
	By~\eqref{eq:FKGFK},	
	\begin{align*}
		\phi_{{\rm Strip}_N}^{1/0}\big[\calC_v(\Rect_L) \cap \calC_v(\Rect_R)\big] \geq \tfrac14.
	\end{align*}

	\begin{figure}
	\begin{center}
	\includegraphics[width = 0.45\textwidth]{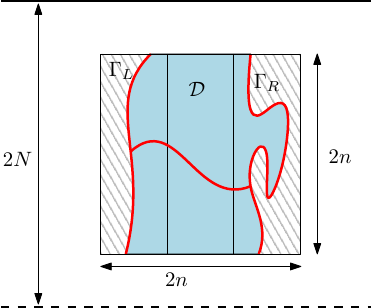}
	\caption{The vertical crossings~$\Gamma_L$ and~$\Gamma_R$ are the leftmost and rightmost crossings of their respective rectangles. 
	Identifying them only has an influence on the state of the edges on them and in the hashed areas. 
	Conditionally on~$\Gamma_L$ and~$\Gamma_R$, the part~$\calD$ of~$\Rect(n ,n)$ left between them (in light blue) is ``narrower'' than a square and has boundary conditions that are wired on its lateral sides. As such, it is crossed horizontally with probability at least~$1/(1+q)$.}
	\label{fig:3x1}
	\end{center}
	\end{figure}
	
	Define the leftmost crossing~$\Gamma_L$ on~$\Rect_L$ and the rightmost crossing~$\Gamma_R$ of~$\Rect_R$, when these two rectangles are crossed (see Figure~\ref{fig:3x1}). Let~$\calD$ be the part of~$\La_n$ between~$\Gamma_L$ and~$\Gamma_R$. Then, conditionally on~$\Gamma_L$ and~$\Gamma_R$, the boundary conditions induced on~$\calD$ by the configuration outside dominate the alternating boundary conditions (see the proof of Theorem~\ref{thm:RSW_perco} for how the conditioning on~$\Gamma_L$ and~$\Gamma_R$ affects the configuration between~$\Gamma_L$ and~$\Gamma_R$). 
	Note that~$\calD$ is not symmetric, but if~$\omega \notin \{\Gamma_L \xlra{\calD}\Gamma_R\}$ then~$\sigma (\omega^*) \in \{\Gamma_L \xlra{\calD}\Gamma_R\}$, where~$\sigma$ is the rotation of the lattice by~$\pi/2$ composed with the shift by~$(1/2,1/2)$.
	Corollary~\ref{cor:symmetric_quad} (or more precisely Remark~\ref{rem:symmetric_quad}) implies that 
	\begin{align*}
		\phi_{{\rm Strip}_N}^{1/0}\big[\Gamma_L \xlra{\calD}\Gamma_R \,\big|\, \Gamma_L,\Gamma_R \big] \geq \tfrac1{1+q}.
	\end{align*}
	Finally, notice that when the event above occurs,~$\Rect(n/3 ,n)$ is crossed horizontally. 
	Combining the last two displays we conclude that 
	\begin{align*}
		\phi_{{\rm Strip}_N}^{1/0}\big[\calC_h(\Rect(n/3 ,n))\big] \geq \tfrac{1}{4(1 +  q)},
	\end{align*}
	which contradicts our assumption. 
\end{proof}

\begin{proof}[Proof of Lemma~\ref{lem:lengthening_crossings}]
	Fix~$m$,~$n$ and~$N$ as in the statement.
	Since we work only with~$\phi_{{\rm Strip}_N}^{1/0}$, we denote this simply by~$\phi$.
	
		Write~$a,b,c,d$ for the corners of~$\Rect(m ,n)$ in counter-clockwise order, starting with the top left corner; 
		see Figure~\ref{fig:RSW_strip}, left diagram.
	Let~$v$ be the topmost point on~$(ab)$ so that 
	\begin{align}\label{eq:vmiddle}
		\phi\big[ (av) \xlra{\Rect(m ,n)} (cd)\big] \geq \tfrac12\phi\big[ \calC_h(\Rect(m ,n))\big].
	\end{align}
	Then the same lower bound holds for~$\phi[(vb) \xlra{\Rect(m ,n)} (cd)]$. 
	Write~$w$ for the horizontal reflection of~$v$, belonging to the side~$(cd)$. 
	Then, by~\eqref{eq:FKGFK}, 
	\begin{align*}
		\phi\big[ (vb) \xlra{\Rect(m ,n)} (wd)\big] 
		&\geq 
		\phi\big[ (vb) \xlra{\Rect(m ,n)} (cd)\big]
		\phi\big[ (ab) \xlra{\Rect(m ,n)} (wd)\big]
		\phi\big[ \calC_v(\Rect(m ,n))\big]\\
		&\geq
		\tfrac14\phi\big[ \calC_h(\Rect(m ,n))\big]^2 \phi\big[ \calC_v(\Rect(m ,n))\big].
	\end{align*}
	We call the event on the left-hand side above~$\calE(\Rect(m ,n))$.
	
	\begin{figure}
	\begin{center}
	\includegraphics[width = .9\textwidth]{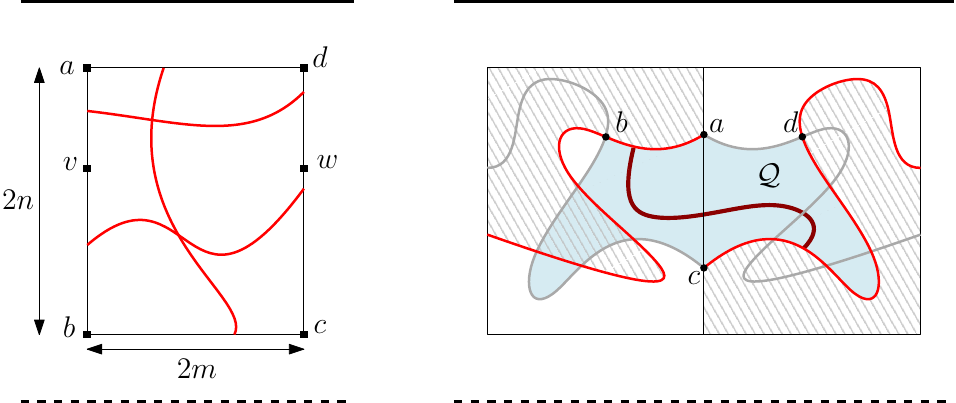}
	\caption{{\em Left:} with~$v$ chosen as in~\eqref{eq:vmiddle}, the probability of obtaining a ``diagonal'' crossing, i.e., from~$(vb)$ to~$(wd)$, is bounded away from~$0$ by a universal multiple of~$\phi\big[ \calC_h(\Rect(m ,n))\big]^2 \phi\big[ \calC_v(\Rect(m ,n))\big]$. \newline 
	{\em Right:} in two side-by-side copies of~$\Rect(m,n)$, if both the left and right rectangle contain diagonal crossings as in the left picture, then we may construct a symmetric domain~$\calQ$ with points~$a,b,c,d$ on it boundary, as depicted. If~$\calQ$ is crossed from~$(ab)$ to~$(cd)$, the left side of the left rectangle is connected to the right side of the right rectangle. 
}
	\label{fig:RSW_strip}
	\end{center}
	\end{figure}
	
	Write~$\Rect_L$ for the translate of~$\Rect(m,n)$ by~$(-m,0)$ and~$\Rect_R$ for its translate by~$(m,0)$. 
	Then, by the above and~\eqref{eq:FKGFK}, 
	\begin{align}\label{eq:rlr}
		\phi\big[ \calE(\Rect_L) \cap \calE(\Rect_R)\big] 
		\geq \tfrac1{16}\,\phi\big[ \calC_h(\Rect(m ,n))\big]^4 \phi\big[ \calC_v(\Rect(m ,n))\big]^2.
	\end{align}
	
	Write~$\Gamma_L$ for the topmost path realising~$\calE(\Rect_L)$ and~$\Gamma_R$ for the bottommost realising~$\calE(\Rect_R)$. 
	When the events are not realised, write~$\Gamma_L = \emptyset$ and~$\Gamma_R = \emptyset$, repsectively.
	
	Let~$\sigma$ be the horizontal reflection with respect to the axis~$\{0\}\times\bbR$ composed with the shift by~$(\tfrac12,\tfrac12)$. 	
	Conditionally on~$\Gamma_L$ and~$\Gamma_R$ both different from~$\emptyset$, write~$\calQ$ for the quad obtained as follows.
	Due to the definition of~$\calE$,~$\Gamma_L$ necessarily intersects~$\sigma(\Gamma_R)$; 
	let~$b$ be the last such point of intersection along~$\Gamma_L$, when going from left to right; set~$d = \sigma(b)$. 
	Let~$a$ and~$c$ be the endpoints of~$\Gamma_L$ and~$\Gamma_R$, respectively,  on~$\{0\}\times \bbR$.
	Then~$\calQ$ is the quad delimited by the arcs~$\Gamma_L$ between~$a$ and~$b$,~$\sigma(\Gamma_R)$ between~$b$ and~$c$, 
	$\Gamma_R$ between~$c$ and~$d$ and~$\sigma(\Gamma_L)$ between~$d$ and~$a$. See Figure~\ref{fig:RSW_strip}, right diagram.  
	
	Then~$\calQ$ is invariant\footnote{This is not formally true, but a finite energy argument allows us to apply Corollary~\ref{cor:symmetric_quad}.}  under~$\sigma$, with the wired arcs~$(ab)$ and~$(cd)$ being mapped by~$\sigma$ onto the free arcs~$(bc)$ and~$(da)$. As such, we would like to argue that Corollary~\ref{cor:symmetric_quad} applies, and provides a lower bound on the probability of~$(ab) \xlra{\calQ}(cd)$ under the measure conditional on~$\Gamma_L$ and~$\Gamma_R$. 
	Notice, however, that parts of~$\calQ$ may have been explored when determining~$\Gamma_L$ and~$\Gamma_R$, which complicates our analysis (see Figure~\ref{fig:RSW_strip}, right diagram, hashed regions). 
	Still, the measure in the unexplored part of~$\calQ$  induced by the conditioning on~$\Gamma_L$ and~$\Gamma_R$ 
 	is the measure with wired boundary conditions on~$(ab)$ and~$(cd)$, random boundary conditions on~$(bc)$ and~$(da)$, 
	conditioned that all explored edges are open --- indeed, the edges inside the explored region have no influence on the measure outside, only those of the boundary of the explored region do, and they are all open. 
	As such, the measure in~$\calQ$ induced by the conditioning on~$\Gamma_L$ and~$\Gamma_R$ 
	dominates the measure with alternating free and wired boundary conditions on the arcs of~$\partial \calQ$, 
	for which Corollary~\ref{cor:symmetric_quad} applies.
	
	We conclude that 
	\begin{align}\label{eq:crossQ}
		\phi\big[ \Gamma_L \xlra{\calQ} \Gamma_R \,|\, \Gamma_L, \Gamma_R\big]  \geq \tfrac{1}{1+q},
	\end{align}
	when~$\Gamma_L \neq \emptyset$ and~$\Gamma_R \neq \emptyset$.
	Finally notice that, when the event above occurs,~$\Rect(2m ,n)$ is crossed horizontally by an open path. 
	Combining~\eqref{eq:crossQ} with~\eqref{eq:rlr} proves~\eqref{eq:lengthening_crossings}.
\end{proof}

\section{Proof of the dichotomy: Theorem~\ref{thm:dichotomy}}\label{sec:dichotomy_proof}

In this section, we work with~$p = p_{\rm sd}$ and omit it from the notation. 
By Theorem~\ref{thm:RSWstrip}, the comparison of boundary conditions and \eqref{eq:FKGFK}, we conclude that 
\begin{align}\label{eq:RSW_weak}
		\phi^1[H_n]  \geq c \qquad \text{ and }\qquad		\phi^0[H_n^*]  \geq c 
\end{align}
for all~$n$ and some constant~$c > 0$. 
Notice here that the boundary conditions are {\em favourable} to the events~$H_n$ and~$H_n^*$, respectively. 

When the boundary conditions are adverse, the events~$H_n$ and~$H_n^*$ may have much smaller probabilities. 
For~$n\geq 1$, set
\begin{align*}
	u_n =  \phi_{\Lambda_{10n}}^0[H_n].
\end{align*}

\begin{proposition}\label{prop:finite_size_bc}
	There exists~$\delta > 0$ such that, if 
	$u_{n_0} < \delta$ for some~$n_0\geq 1$, then 
	\begin{align}\label{eq:finite_size_bc}
		u_{10 k n_0} \leq 2^{-k } \qquad \text{ for all~$k\geq 1$}.
	\end{align}
\end{proposition}

This proposition, together with the finite size criterion of Proposition~\ref{prop:finite_size_ad} and some clever but fairly standard manipulations, imply the dichotomy theorem.

\begin{figure}
\begin{center}
\includegraphics[width=0.23\textwidth]{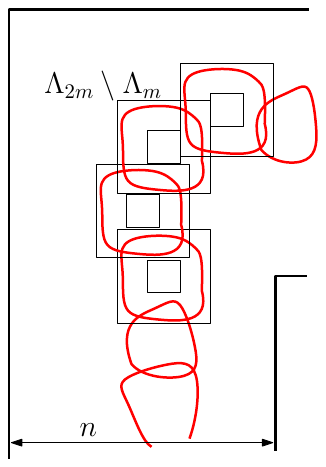}\hspace{0.15\textwidth}
\includegraphics[width=0.6\textwidth]{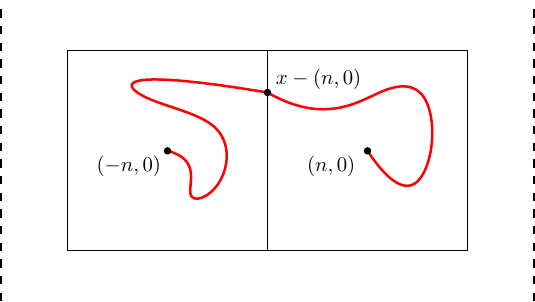}
\caption{{\em Left:} 
For~$m \leq n/20$, combining several instances of~$H_m$ produces~$H_n$. 
Comparing boundary conditions and using \eqref{eq:FKGFK}, we conclude that 
$u_n \geq \phi_{\Lambda_{2n}\setminus \La_n}^0[H_n] \geq u_n^{C n/m}$ for some universal constant~$C$. 
 \newline
{\em Right:} Consider the point~$x$ on the right side of~$\partial \La_n$ that is most likely to be connected to~$0$ in~$\La_n$.
Then, under the measure~$\phi_{\Lambda_{2N}}^0$, the point~$(-n,0)$ is connected to~$x - (n,0)$ with probability at least~$\tfrac{1}{|\partial \La_n|} \phi_{\Lambda_{N}}^0[0 \lra \partial \La_n]$. 
The same holds for connections between~$(n,0)$ and~$x - (n,0)$. 
When both connections occur,~$(-n,0)$ connects to~$(n,0)$.}
\label{fig:nNtoR}
\end{center}
\end{figure}

\begin{proof}[Proof of Theorem~\ref{thm:dichotomy}]
	Proposition~\ref{prop:finite_size_bc} implies that the sequence\footnote{Formally, this only works directly if we limit~$n$ to powers of~$10$. To access all values of~$n$, one may observe that for any~$n/200 \leq m \leq  n/20$,~$u_n \geq u_m^C$, where~$C$ is some fixed constant; see Figure~\ref{fig:nNtoR}.}~$(u_n)_{n\geq 1}$ is either uniformly bounded away from~$0$, or converges exponentially fast to~$0$.
	When the former occurs, a standard geometric construction (see Figure~\ref{fig:nNtoR}, left diagram) also implies that 
	$\phi_{\Lambda_{2n}\setminus \La_n}^0[H_n]$ is uniformly bounded away from~$0$. By self-duality, we deduce~\eqref{eq:Con}. 
	
	Assume now that the latter occurs, which is to say~$u_N \leq e^{-cN}$ for all~$N\geq 1$ and some~$c > 0$ independent of~$N$. We will prove~\eqref{eq:DisCon}.
	
	Consider~$n\leq N$, with~$N$ an integer multiple of~$n$. 
	Let~$x \in \partial \La_n$ be a maximiser of~$\phi_{\Lambda_{N}}^0[0 \xlra{\La_n}x]$; by symmetry, we may take~$x$ on the right side of~$\La_n$. 
	Using~\eqref{eq:FKGFK} (see also Figure~\ref{fig:nNtoR}, right diagram), we find that 
	\begin{align}
	\big(\tfrac{1}{|\partial \La_n|} \phi_{\Lambda_{N}}^0[0 \lra \partial \La_n]\big)^2 
	\leq \phi_{\Lambda_{N}}^0[0 \xlra{\La_n}x]^2 \leq \phi_{\Lambda_{2N}}^0[(-n,0) \lra (n,0)].
	\end{align}
	Furthermore, combining the event on the right-hand side above and its rotations~$C N/n$ times --- where~$C$ is some fixed constant --- we may construct~$H_N$. 
	Applying~\eqref{eq:pushing_bc} to push the free boundary conditions further, we find 
	\begin{align}
		\big(\phi_{\Lambda_{2N}}^0[(-n,0) \lra (n,0)]\big)^{C N / n} \leq  u_{N} \leq e^{-c N}.
	\end{align}
	The last two displays allow us to conclude that there exists~$c_0 > 0$ such that 
	\begin{align*}
	\phi_{\Lambda_{N}}^0[0 \lra \partial \La_n] \leq e^{ - c_0 n} \qquad \text{ for all~$N \geq n \geq 1$.}
	\end{align*}
	Taking~$N \to\infty$ in the above, we conclude~\eqref{eq:DisCon}.
\end{proof}

\begin{proof}[Proof of Proposition~\ref{prop:finite_size_bc}]
	The central element of this proof is the following {\em renormalisation inequality}. 
	There exists a constant~$C > 0$ such that for any~$n,k \geq 1$,
	\begin{align}\label{eq:renormalisation}
		u_{10k n} \leq (C\, u_n)^{k}. 
	\end{align}
	Indeed, Proposition~\ref{prop:finite_size_bc} follows from the above by taking~$\delta = 1/2C$. We now focus on proving~\eqref{eq:renormalisation}.

\begin{figure}
\begin{center}
\includegraphics[width = 0.47\textwidth, page = 6]{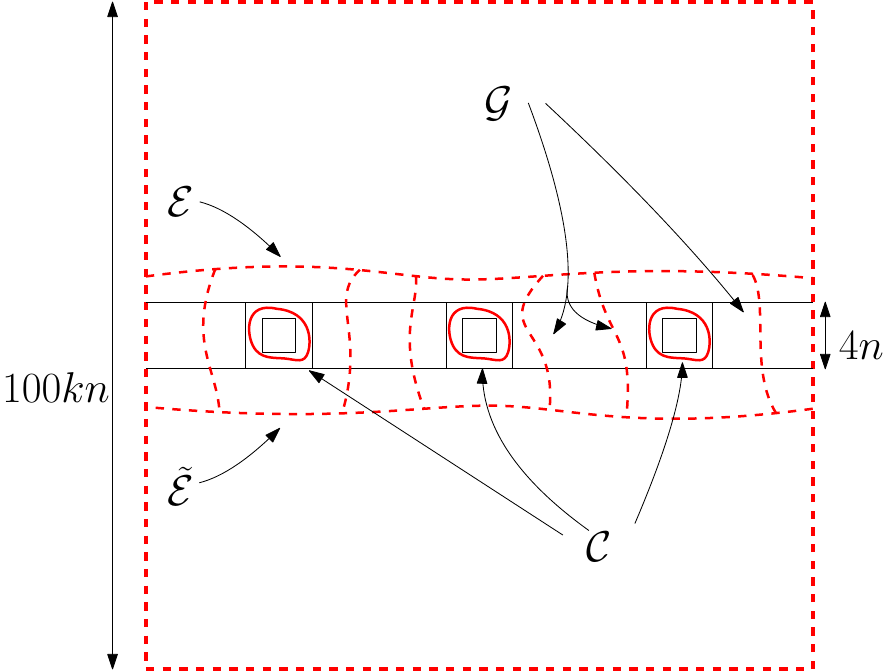}\qquad
\includegraphics[width = 0.47\textwidth, page = 2]{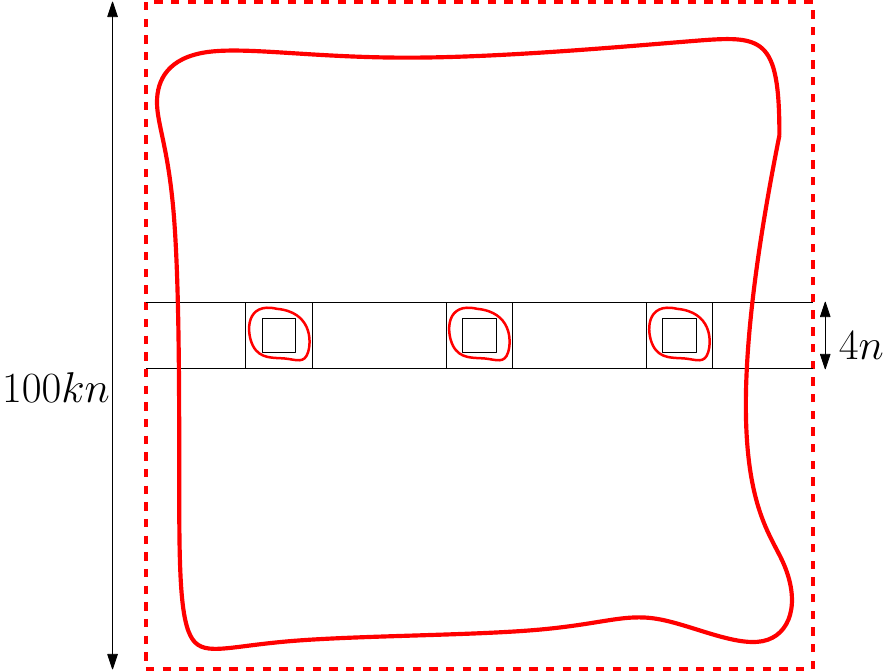}
\caption{
{\em Left:} The event~$\calC \cap \calE \cap \tilde\calE  \cap \calG$ with the elements required by each event highlighted. 
Its probability may be bounded from above by the probability of the primal circuits of~$\calC$ occurring, conditionally on the dual paths. This produces an upper bound of~$u_n^k$.  
Conversely, its probability is bounded from below in several steps: first construct~$\calC$, then, conditionally on~$\calC$, the dual paths forming~$\calE\cap \tilde\calE$ and~$\calG$ appear with probability at least~$(c_\calG\,c_\calE)^k$. \newline 
{\em Right:}~Construct~$\calC$ by first requiring that~$H_{10 k n}$ occurs. Conditionally on this event,~$\calC$ occurs with probability at least~$c_\calC^k$. As such,~$\calC$ has probability at least~$c_\calC^k \, u_{10kn}$.}
\label{fig:uu}
\end{center}
\end{figure}

	The constants~$c_*$ below are positive and independent of~$n$ and~$k$.
	Figure~\ref{fig:uu} may be useful in understanding the proof. 
	Fix~$k,n \geq 1$ and assume for simplicity that~$k$ is odd. Define the event~$\calC$
	as the intersection of the translates of~$H_n$ by the~$(20 j n ,0)$ with~$j = -(k-1)/2,\dots, (k-1)/2$. 
	By a standard exploration argument,~\eqref{eq:FKGFK} and the monotonicity in boundary conditions,
	\begin{align*}
		\phi_{\Lambda_{100 k n}}^0[\calC] 
		\geq \phi_{\Lambda_{100 k n}}^0[\calC\,|\, H_{10 k n}]\, \phi_{\Lambda_{100 k n}}^0[H_{10 k n}] 
		\geq \big(\phi^1[H_{n}]\big)^k  \,  \phi_{\Lambda_{100 k n}}^0[H_{10 k n}]
		\geq c_\calC^k \,u_{10 k n},
	\end{align*}
	where~$c_\calC = \inf_n  \phi^1[H_{n}]$, which is strictly positive  by~\eqref{eq:RSW_weak}. 
	
	Write~$\calE$ for the event that there exists a dual horizontal crossing in~$[-100 k n,100 k n] \times [2n,4n]$
	and~$\tilde \calE$ for the vertical reflection of this event. 
	Then, repeated applications of~\eqref{eq:RSWstrip} at scales~$N = 100 nk (3/4)^j$, 
	producing horizontal dual crossings of translates of~$[-100 k n,100 k n] \times	[-N/2,N/2]$, imply
	\begin{align*}
		\phi_{\Lambda_{100 k n}}^0[\calE \cap \tilde\calE \,|\,\calC] 
		\geq  \prod_{j=0}^{C \log k}\exp\big(- c\cdot 100 k \cdot \big(\tfrac34\big)^j\big) \geq c_\calE^{k},
	\end{align*}
	for universal constants~$C, c_\calE >0$. 
	
	Finally, when~$\calE \cap \tilde\calE$ occurs, write~$\calG$ for the event that there exists a dual crossings
	between the topmost dual path realising~$\calE$ and the bottommost path realising~$\tilde \calE$, in each of the squares 
	$[20jn - 10n, 20j n - 2n] \times [-4n,4n]$ and~$[20jn+ 2n, 20j n + 10n] \times [-4n,4n]$ with~$j = -(k-1)/2,\dots, (k-1)/2$. 
	By pushing of boundary conditions~\eqref{eq:pushing_bc} and~\eqref{eq:FKGFK}, we conclude that 
	\begin{align*}
		\phi_{\Lambda_{100 k n}}^0[\calG \,|\, \calE \cap \tilde\calE  \cap \calC]
		\geq   \phi_{\La_{4n}}^{\rm alt}\big[\calC_v(\La_{4n})\big]^{2k}  \geq\big(\tfrac{1}{1+q}\big)^{2k},
	\end{align*}
	where the last inequality is due to~\eqref{eq:symmetric_quad}. 
	
	Combine the three displays above to deduce that 
	\begin{align*}
		\phi_{\Lambda_{100 k n}}^0[\calC \,|\, \calE \cap \tilde\calE \cap \calG  ] 
		\geq \phi_{\Lambda_{100 k n}}^0[\calC \cap \calE \cap \tilde\calE  \cap \calG] 
		\geq c_{\rm tot}^k \, u_{10 k n},
	\end{align*}
	where~$c_{\rm tot} > 0$ is a universal constant.
	
	Now, when~$\calE \cap \tilde\calE \cap \calG$ occurs, each of the disjoint annuli~$\La_{10n}\setminus \La_{2n} + (20 j n ,0)$ contains a dual circuit. These may be explored from the outside, leading to independent measures in each of their interiors, with free boundary conditions. 
	Thus 
	\begin{align*}
		\phi_{\Lambda_{100 k n}}^0[\calC \,|\, \calE \cap \tilde\calE \cap \calG  ] \leq u_n^k.
	\end{align*}
	The last two displays combine to prove the desired inequality \eqref{eq:renormalisation}. 
\end{proof}

\section{Exercises: FK-percolation on~$\bbZ^2$, fine properties}

\begin{exo}\label{exo:p_c_FK_Zhang}
	Consider FK-percolation on~$\bbZ^2$ with~$q >1$ and~$p =  \frac{\sqrt q}{1 + \sqrt q}$. Prove that
	\begin{align*}
		\phi_{[0,n]\times[0,n+1]}^\xi\big[\text{$[0,n]\times[0,n+1]$ crossed from left to right by an open path}\big] = c,
	\end{align*}
	for all~$n$, where~$\xi$ is the boundary condition where the left and right sides are wired (and wired to each other), while the top and bottom are free. 
	Why is~$c \neq 1/2$? 
	
	Assuming the sharpness of the phase transition (Theorem~\ref{thm:sharpness_FK_OSSS}) and the uniqueness of the infinite cluster (Theorem~\ref{thm:Burton-Keane_FK}), proceed as in Exercise~\ref{exo:Zhangs} to 
	prove that~$p_c(\bbZ^2) = \frac{\sqrt q}{1 + \sqrt q}$. 
	
	Why can't we conclude that the phase transition is continuous, as in the case of Bernoulli percoaltion? \smallskip	\\
		{\em Hint:} Zhang's argument applies well when there exists a unique infinite-volume measure. Assuming~$p_c < \frac{\sqrt q}{1 + \sqrt q}$, show the existence of an open interval of values of~$p$ for which both the primal and dual model have infinite clusters. Use Proposition~\ref{prop:phi_ordering} to find such a parameter for which  the infinite-volume measure is unique. 
\end{exo}

\begin{exo}\label{exo:algebraic_decay}
	Consider FK-percolation on~$\bbZ^2$ with some~$q \geq 1$,~$p = p_{\rm sd}(q)$. 
	Assume that we are in the case {\em (Con)} of Theorem~\ref{thm:dichotomy}.
	Show that there exists~$c > 0$ such that
	\begin{align*}
		c \, n^{-1} \leq \phi_{\La_{2n}}^\xi[0\lra \partial \La_n] \leq n^{-c}\qquad \text{ for all~$n\geq1$ and all boundary conditions~$\xi$}.
	\end{align*}
	Deduce that~$\phi[|\sfC_0|] = \infty$, where~$\sfC_0$ is the cluster of~$0$, and~$\phi$ is the unique infinite-volume measure (here  used as an expectation). 
\end{exo}

\begin{exo}\label{exo:cor_len}
	Fix~$p \in (0,1)$ and~$q\geq 1$. We wish to prove that 
	\begin{align}\label{eq:u_n_cor_len}
	u_n := \tfrac1n \log\phi_{\Lambda_n,p,q}^0[0\lra \partial \La_n].
	\end{align}
	converges. If~$p < p_c$, we denote the limit by~$-1/\xi(p)$.
	\begin{itemize}
	\item[(a)] Using \eqref{eq:FKGFK} and the comparison of boundary conditions, prove that for all~$m,n\geq 0$
	\begin{align}
	\phi_{\Lambda_{n+m},p,q}^0[0\lra \partial \La_{n + m}] \geq \frac{1}{C \min\{n,m\}}	\phi_{\Lambda_{n},p,q}^0[0\lra \partial \La_{n}]\phi_{\Lambda_{m},p,q}^0[0\lra \partial \La_{m}],
	\end{align}
	where~$C$ is some universal constant. 
	\item[(b)] Use this fact to prove the existence of the limit in~\eqref{eq:u_n_cor_len}. \smallskip
	\\
	{\em Hint:} Use the subadditivity lemma that states that if a sequence~$(a_n)_{n\geq0}$ satisfies~$a_{m+n} \leq a_n + a_m$ for all~$m,n \geq 1$, then~$\frac1n a_n$ converges to~$\inf_n \frac1n a_n$. 
	\item[(c)] Show that~$\phi_{p,q}^0$ exhibits exponential decay of cluster radii if and only if~$\xi(p) < \infty$. Moreover prove that then
	\begin{align}
		\phi_{p,q}^0[0 \lra \partial \La_n] = e^{- n /\xi(p) + o(n)} \qquad \text{ as~$n\to\infty$.} 
	\end{align}
	\item[(d)] For~$p > p_c$, write~$\xi(p) = \xi(p^*)$ with~$p^*$ being the dual parameter to~$p$. Observe that,~$p^* < p_c$. 
	Show that there exist universal constants~$c,C > 0$ such that 
	\begin{align}
		e^{- C n /\xi(p) } \leq \phi_{p,q}^0[0 \lra \partial \La_n \text{ but } 0 \nxlra{} \infty] \leq e^{- c n /\xi(p) }  \qquad \text{ for all~$n$ large enough}.
	\end{align}
	\item[(e)] For~$p > p_c$, prove that 
	\begin{align}
		\phi_{p,q}^0[0 \lra \partial \La_n \text{ but } 0 \nxlra{} \infty] = e^{- 2 n /\xi(p) + o(n)} \qquad \text{ as~$n\to\infty$.} 
	\end{align}
	\end{itemize}
\end{exo}

\begin{exo}\label{exo:cor_len_diverges}
	The goal of the exercise is to study the continuity of~$\xi(p)$ for~$p\leq p_c$. 
	\begin{itemize}
	\item[(a)] Prove that~$p \mapsto \xi(p)$ is increasing for~$p < p_c$. 
	\item[(b)] Prove that~$p \mapsto \xi(p)$ is continuous for~$p < p_c$. 
	\item[(c)] Assume that~$\xi^0(p_c):=\lim_{p \nearrow p_c} \xi(p) < \infty$ as~$p \nearrow p_c$. Prove then that 
	\begin{align}
		\phi_{p_c,q}^0[0 \lra \partial \La_n] = e^{- n /\xi^0(p_c) + o(n)} \qquad \text{ as~$n\to\infty$,} 
	\end{align}
	and in particular that (DisCon) occurs at~$p_c$.
	\item[(d)] Conclude that, if (Con) occurs at~$p_c$, then~$\xi(p) \to \infty$ as~$p\nearrow p_c$.
	\end{itemize} 
\end{exo}

\begin{exo}
	Consider the torus~$\bbT_N = (\bbZ/N\bbZ)^2$ as a finite graph. 
	Fix an increasing event~$A$ which is invariant under the translations by~$(1,0)$ and~$(0,1)$. 
	Notice that, for~$q \geq 1$ and~$p \in (0,1)$, we may define~$\phi_{\bbT_N,p,q}$ as the FK-percolation measure on~$\{0,\dots, N\}^2$ with periodic boundary conditions (that is, where we wire each~$(x,0)$ to~$(x,N)$ and each~$(0,y)$ to~$(N,y)$ for $0\leq x,y\leq N$). 
	
	Use~\eqref{eq:BKKKL} to prove that, for some constant~$c > 0$ depending only on~$p$ and~$q$, 
	\begin{align}
	\frac{\rm d}{{\rm d} p} \phi_{\bbT_N,p,q}[A] \geq c \, \phi_{\bbT_N,p,q}[A](1 - \phi_{\bbT_N,p,q}[A]) \log N.
	\end{align}
\end{exo}

\begin{exo}\label{exo:limit_on_torus}
	Consider FK-percolation on~$\bbZ^2$ with some~$q \geq 1$ and~$p = p_{\rm sd}$. 
	Assume that we are in the case {\em (DisCon)} of Theorem~\ref{thm:dichotomy}.
	\begin{itemize}
	\item[(a)] Using~$\phi^0 = (\phi^1)^*$, show that, for any fixed edge~$e$,~$\phi^0[e \text{ open}] < 1/2$. 
	\item[(b)] Write~$\phi_{\bbT_{2n},p,q}$ for the FK-percolation measure on the square torus of side-length~$2n$. 
	Prove that, for any fixed edge~$e$,~$\phi_{\bbT_{2n},p,q}[e \text{ open}] = 1/2$. 
	\item[(c)] Assume that the only Gibbs measures\footnote{For the purpose of these exercises, Gibbs measures should be understood as the potential limits of finite volume measures.}
	of FK-percolation on~$\bbZ^2$ 
	which are invariant under translations and rotations by~$\pi/2$ 
	are the linear combinations of~$\phi^0$ and~$\phi^1$ (this may be proved using relatively soft tools, see Exercise~\ref{exo:Gibbs_for_2DFK}). 
	Prove that 
	\begin{align*}
		\phi_{\bbT_{2n},p,q}\longrightarrow \tfrac12 \phi^0 + \tfrac12\phi^1 \qquad \text{ as~$n\to\infty$}.
	\end{align*}
	\item[(d)] Is~$\lim_n\phi_{\bbT_{2n},p,q}$ ergodic?
	\end{itemize}
\end{exo}

\begin{exo}\label{exo:Gibbs_for_2DFK}
	Fix~$q \geq 1$ and~$p \in (0,1)$. 
	Let~$\phi$ be Gibbs measure for FK-percolation on~$\bbZ^2$
	which is invariant under translations (by two linearly independent vectors) and rotations by~$\pi/2$. 
	Furthermore, assume that~$\phi$ is ergodic with respect to translations.
	\begin{itemize}
	\item[(a)]  Follow the proof of Theorem~\ref{thm:Burton-Keane} and observe that it applies to~$\phi$. 
	Conclude that~$\phi$ contains a.s. at most one infinite cluster. 
	\item[(b)]  Use the same construction as in Exercise~\ref{exo:Zhangs} to prove that 
	$\phi$-a.s.\ there exists no infinite primal cluster or~$\phi$-a.s.\ there exists no infinite dual cluster. 
	\item[(c)]	Deduce that the only Gibbs measures	of FK-percolation on~$\bbZ^2$ 
	which are invariant under translations and rotations by~$\pi/2$ 
	are the linear combinations of~$\phi^0$ and~$\phi^1$.\smallskip \\
	{\em Hint:} use an ergodic decomposition. 
	\end{itemize} 
\end{exo}

\begin{exo}
	Fix~$q \geq 1$ and~$p \in (0,1)$. 
	Let~$\phi$ be a Gibbs measure for FK-percolation on~$\bbZ^2$
	which is invariant under translations and rotations by~$\pi/2$. 
	Assume that there exists~$c > 0$ such that  
	\begin{align*}
	\phi \big[\omega \text{ contains a horizontal crossing of }[0,2n]\times [0,n]\big] &\geq c \qquad \text{ and}\\
	\phi \big[\omega^* \text{ contains a horizontal crossing of }[0,2n]\times [0,n]\big] &\geq c \qquad \text{ for all~$n\geq 1$}.
	\end{align*}
	Prove (without using Theorem~\ref{thm:dichotomy}) that
	\begin{align*}
	\phi \big[\omega \text{ contains a horizontal crossing of }[0,kn]\times [0,n]\big] \geq c^{2k} &\qquad \text{ and}\\
	 \delta < \phi [H_n] < 1 - \delta  &\qquad \text{ for all~$n\geq 1$},
	\end{align*}
	and the same for the dual, for some~$\delta > 0$. 
	Does this imply that~$\phi [0\lra \partial \La_n]  \to 0$? 
\end{exo}

\chapter{Continuous/discontinuous phase transition}\label{ch:4cont_v_discont}

Theorem~\ref{thm:dichotomy} identified two types of phase transition for FK-percolation on~$\bbZ^2$. 
The goal of this chapter is to determine which values of~$q$ correspond to which type of phase transition.

\begin{theorem}\label{thm:q<>4}
	 The phase transition of FK-percolation on~$\bbZ^2$ is continuous if~$1 \leq q \leq 4$, and discontinuous if~$q > 4$.
\end{theorem}

The theorem above was proved in a series of papers via different methods. 
Originally, the continuity was proved in~\cite{DumSidTas17} and the discontinuity in~\cite{DumGagHar16}.
Alternative proofs of these two reults were obtained in~\cite{GlaLam23} and~\cite{RaySpi20}, repsectively.

We will present here a proof of both the continuity and discontinuity regimes inspired by the method of~\cite{DumGagHar16}. It consists in the explicit computation of the rate of decay of a certain event that allows us to distinguish the cases (Con) and (DisCon) of Theorem~\ref{thm:dichotomy}. The computation is done by relating critical FK-percolation to the six-vertex model, which we define below. The free energy of the six-vertex model is then estimated by applying the Bethe ansatz to its transfer matrix and computing its leading eigenvalues. 

This section is meant to highlight the links between FK-percolation and the six-vertex model, and the very different techniques that may be used to analyse them. 
A more specific take-home message is that the continuity/discontinuity of the phase transition of FK-percolation 
corresponds to the delocalisation/localisation of the height function of the corresponding six-vertex model, 
which in turn corresponds to the differentiability/non-differentiability of the free energy of the six-vertex model as a function of its ``slope'', at~$0$ slope.

The actual computation of the six-vertex free energies with different slope will not be detailed
as it does not employ the type of tools highlighted in these notes.

\section{Six-vertex model on the torus}

Let~$L,M \in 2\bbN$ and write~$\bbT_{L,M} = (\bbZ/L\bbZ) \times(\bbZ/M\bbZ)$ for the torus of width~$L$ and height~$M$. 
A six-vertex configuration on~$\bbT_{L,M}$ is an assignment of directions (or arrows) to each edge of~$\bbT_{L,M}$ with the restriction that 
any vertex has exactly two incoming and two  outgoing edges; we call this restriction the {\em ice rule}. 
As a result, there are only six possible configurations at each vertex, whence the name of the model. 

Each possibility carries a weight (see Figure~\ref{fig:6vertex}). Consider three positive parameters~$a,b,c > 0$. 
The weight of a configuration~$\vec\omega$ is 
\begin{align*}
	w_{\rm 6V}(\vec\omega) = a^{\# \text{vertices of type~$a$}} \cdot b^{\# \text{vertices of type~$b$}} \cdot c^{\# \text{vertices of type~$c$}} .
\end{align*}

It is standard to parametrise the model via 
\begin{align}\label{eq:Delta_6V}
\Delta := \frac{a^2 + b^2 - c^2}{2ab},
\end{align}
as models with equal~$\Delta$ are expected to behave similarly.
Henceforth, we focus exclusively on the case~$a = b = 1$, for which~$\Delta = 1 - c^2/2 \in (-\infty,1)$. 

\begin{figure}
\begin{center}
\includegraphics[width = 0.7\textwidth]{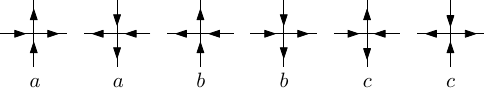}
\caption{The six possible vertex configurations obeying the ice rule. The weights are chosen to be invariant under total arrow reversal.}
\label{fig:6vertex}
\end{center}
\end{figure}

\paragraph{Preservation of arrows.}
Partition the vertical edges of the torus into horizontal rows. 
It is a direct but crucial consequence of the ice rule that, in any six-vertex configuration, the number of up-arrows is the same in each row.
We call this the {\em preservation of up-arrows}. 

Define the partition functions
\begin{align*}
	Z_{\rm 6V}^{(k)}(\bbT_{L,M}) 
	= \!\! \!\!\!\!\sum_{\substack{\vec\omega \text{ with~$L/2 + k$} \\ \text{up-arrows per row}}}  \!\!\!\!\!\!\!w_{\rm 6V}(\vec\omega) 
	\qquad \text{ and } \qquad
	Z_{\rm 6V}(\bbT_{L,M}) = \sum_{k = -L/2}^{L/2}   Z_{L,M}^{(k)}.
\end{align*}
Finally, for~$\alpha \in [-1/2,1/2]$, define the free energy of the (sloped) model as
\begin{align*}
f_{\rm 6V} = \lim_{L \to \infty} \lim_{M \to \infty} \tfrac{1}{LM} \log Z_{\rm 6V}(\bbT_{L,M}) \quad\text{ and}\quad
f_{\rm 6V}(\alpha) = \lim_{L \to \infty} \lim_{M \to \infty} \tfrac{1}{LM} \log Z_{\rm 6V}^{(\alpha L)}(\bbT_{L,M}).
\end{align*}
Simple combinatorial manipulations (see, for instance,~\cite[proof of Cor 1.4]{DumGagHar16})
show that the limits exist no matter the order in which~$L$ and~$M$ are sent to infinity. 
Moreover, they show that 
\begin{align*}
	f_{\rm 6V} =f_{\rm 6V}(0)= \lim_{L \to \infty} \lim_{M \to \infty} \tfrac{1}{LM} \log Z_{\rm 6V}^{(0)}(\bbT_{L,M}).
\end{align*}

\section{Relation to FK-percolation: BKW correspondence}\label{sec:BKW}

We next present a correspondence between critical FK-percolation on (a~$\pi/4$-rotated version of) the torus and the six-vertex model described above. It is sometimes called the Baxter--Kelland--Wu (or BKW) correspondence~\cite{BaxKelWu76}.
Figure~\ref{fig:correspondence} sums it up. 

Notice that~$\bbT_{L,M}$ is a bipartite graph; consider a bipartite colouring in black and white of its vertices. 
Let~$\bbT_{L,M}^\bullet$ be the graph containing only the black vertices of~$\bbT_{L,M}$, with edges between vertices at a distance~$\sqrt 2$ of each other. 
Write~$\Omega_{\rm FK}$ for the set of FK-percolation configurations on~$\bbT_{L,M}^\bullet$
and~$\phi_{\bbT_{L,M}}$ for the associated FK-percolation measure (the parameters~$p$ and~$q$ will be fixed and omitted from the notation). 
As on any finite graph, the FK-percolation measure may be defined on~$\bbT_{L,M}^\bullet$ with no mention of boundary conditions; 
alternatively~$\phi_{\bbT_{L,M}}$ may be viewed as a FK-percolation measure on a rectangle with periodic boundary conditions. 


Write~$\calE_k$ for the event that~$\omega \in \Omega_{\rm FK}$ contains at least~$k$ vertically crossing clusters, 
where clusters are counted on the cylinder obtained by cutting~$\bbT_{L,M}$ horizontally at height~$0$ --- see Figure~\ref{fig:cylinder_BKW}. 
The BKW correspondence will be used to prove the following. 

\begin{figure}
\begin{center}
\includegraphics[width = 0.4\textwidth]{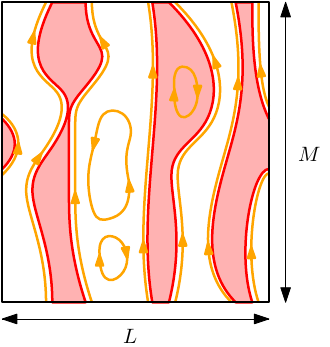}
\caption{See the torus~$\bbT_{L,M}$ as a vertical cylinder. The event~$\calE_k$ (here with~$k =3$) requires the existence of~$k$ clusters of~$\omega$ (in red) between the bottom and top of the cylinder. 
This leads to the existence of at least~$2k$ paths between the bottom and top of the torus in the loop configuration~$\omega^\circ$  (in orange). 
If all these paths are oriented upwards, the resulting six-vertex configurations will have an excess of~$2k$ up-arrows. 
Indeed, any retractible loop contributes the same number of up- and down-arrows to each row, regardless of its orientation.}
\label{fig:cylinder_BKW}
\end{center}
\end{figure}

\begin{proposition}[BKW correspondence]\label{prop:BKW}
	For~$q \geq 1$,~$p = p_c(q) = \frac{\sqrt q}{1 + \sqrt q}$ and~$c = \sqrt{2 + \sqrt{q}}$, 
	\begin{align}\label{eq:BKW1}
		f_{\rm FK} &= f_{\rm 6V} + \tfrac14\log q + \log (1+ \sqrt q) &&\text{ and } \\[.2cm]
		\phi_{\bbT_{L,M}}[\calE_{\alpha L}]^{1/LM} &= \exp\big(f_{\rm 6V}(\alpha)- f_{\rm 6V}(0) + o(1)\big) && \text{  for any~$\alpha >0$},
		\label{eq:BKW2}
	\end{align}
	where~$f_{\rm FK}$ is the finite energy of the FK-percolation defined in \eqref{eq:free_energy_FK_def} and 
	~$o(1)$ in the last line denotes a quantity converging to~$0$ as~$M \to \infty$, then~$L \to \infty$.
\end{proposition}

The rest of the section is dedicated to proving this result. 
Henceforth,~$q \geq 1$ is fixed and we set~$p =  \frac{\sqrt q}{1 + \sqrt q}$ and~$c = \sqrt{2 + \sqrt{q}}$.

\paragraph{Parameters of the correspondence.}

First, we define a parameter linking~$q$ and~$c$. Let~$\lambda$ be such that 
\begin{align}\label{eq:v}
	e^{\lambda} + e^{-\lambda} = \sqrt{q}.
\end{align} 
Notice that~$\lambda$ is real for~$q \geq 4$ and purely imaginary for~$1 \leq q < 4$. 
For $q \neq 4$, there are two possible choices for~$\lambda$ (up to sign change); we do not impose a canonical choice. 
Then,
\begin{align}\label{eq:c}
	c = e^{\frac{\lambda}2} + e^{-\frac{\lambda}2}.
\end{align}

\paragraph{Loop configurations.}
We define two more types of configurations needed in describing the BKW correspondence. 
An {\em oriented loop} on~$\bbT_{L,M}$ is an oriented cycle of~$\bbT_{L,M}$ which is edge-disjoint, non-self-intersecting and such that it turns by~$\pm \pi/2$ at every visited vertex.  
We may view oriented loops as ordered collections of edges of~$E(\bbT_{L,M})$, quotiented by cyclic permutations of the indices.  
Unoriented loops (or simply loops) are oriented loops considered up to reversal of the indices.
A (oriented) loop configuration on~$\bbT_{L,M}$ is a partition of~$E(\bbT_{L,M})$ into (oriented) loops.
Write~$\Omega_{\rm Loop}^\circ$ and~$\Omega_{\rm Loop}^{\ol}$ for the set of configurations of unoriented and oriented loops, respectively. 

Associate the following weights to unoriented and oriented loop configurations.
For an unoriented loop configuration~$\omega^{\circ}$, write~$\ell(\omega^\circ)$  for the number of loops of~$\omega^{\circ}$
and~$\ell_0(\omega^\circ)$ for the number of loops that are not retractable (on the torus) to a point.
For an oriented loop configuration~$\omega^{\ol}$, write 
$\ell_-(\omega^\ol)$ and~$\ell_+(\omega^\ol)$ for the number of retractable loops of~$\omega^\ol$ 
which are oriented clockwise and counterclockwise, respectively.
Set 
\begin{align*}
	w_{\circ} (\omega^{\circ}) &= \sqrt{q}^{\ell(\omega^\circ)} \cdot \big(\tfrac{2}{\sqrt q}\big)^{\ell_0(\omega^\circ)}  &&\text{ for all~$\omega^{\circ} \in \Omega_{\rm Loop}^{\circ}$};\\
	w_{\ol} (\omega^{\ol}) &= e^{\lambda (\ell_+(\omega^\ol) -\ell_-(\omega^\ol))} && \text{ for all~$\omega^{\ol} \in \Omega_{\rm Loop}^{\ol}$}.
\end{align*}
Additionally, for FK-percolation configurations~$\omega$ and six-vertex configurations~$\vec \omega$, recall the weights
\begin{align*}
w_{\rm FK}(\omega) = p_{\rm sd}^{|\omega|}(1-p_{\rm sd})^{|E\setminus\omega|}q^{k(\omega)}\qquad \text{ and }\qquad 
w_{\rm 6V}(\vec\omega)= c^{\#\text{type~$c$ vertices}}.
\end{align*}

\paragraph{Correspondence between configurations.}
The correspondence between configurations is best described in Figure~\ref{fig:correspondence}.

\begin{figure}
	\begin{center}
		\includegraphics[width=0.23\textwidth, page=1]{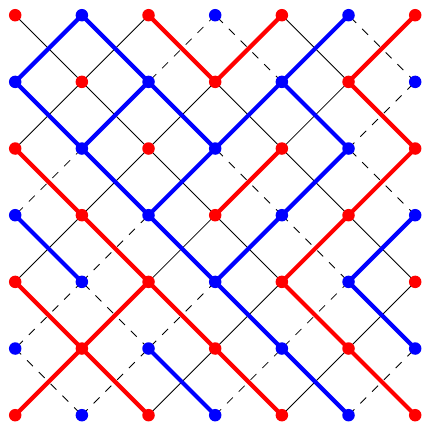}\,
		\includegraphics[width=0.23\textwidth, page=2]{maps.pdf}\,
		\includegraphics[width=0.23\textwidth, page=3]{maps.pdf}\,
		\includegraphics[width=0.23\textwidth, page=4]{maps.pdf}
	\end{center}
	\caption{The different steps in the correspondence between the FK-percolation and six-vertex models on a torus. 
	From left to right: A FK-percolation configuration and its dual, the corresponding loop configuration,
	an orientation of the loop configuration and the resulting six-vertex configuration. 
	Note that in the first picture, there exist both a primal and dual cluster winding vertically around the torus; 
	this leads to two loops that wind vertically (see second picture); 
	if these loops are oriented in the same direction (as in the third picture), 
	then the number of up arrows on every row of the six-vertex configuration is equal to~$\frac{N}2 \pm 1$. 
	}
	\label{fig:correspondence}
\end{figure}

Unoriented loop configurations are in bijection with FK-percolation configurations: associate to any FK-percolation configuration~$\omega$ the unique loop configuration whose loops do not intersect any edge of~$\omega$ or~$\omega^*$. 

An oriented loop configuration~$\omega^{\ol}$ is said to be coherent with the unoriented loop configuration containing the same loops. 
It is also said to be coherent with the six-vertex configuration whose edge-orientations are given by the orientations of the loops. 

Notice that for any unoriented loop configuration~$\omega^{\circ}$, there are $2^{\ell(\omega^\circ)}$ oriented loop configurations coherent with it.
Similarly, for any six-vertex configuration~$\vec\omega$, there are~$2^{\# \text{vertices of type~$c$}}$ oriented loop configurations 
coherent with~$\vec\omega$.

\paragraph{Correspondence for weights.}
The following lemmas relate the weights of the different configurations. 
Both results are obtained via fairly direct computations which we will only sketch; full proofs are available in~\cite{DumGagHar16}.

\begin{lemma}\label{lem:correspondence2}
	For any~$\omega \in \Omega_{\rm FK}$, 
	\begin{align}\label{eq:correspondence2}
		w_{\rm FK}(\omega) 
		= C\, \sqrt{q}^{ \ell(\omega^\circ)}q^{s(\omega)}		
		= C\, \big(\tfrac{\sqrt{q}}2 \big)^{ \ell_0(\omega^\circ)}q^{s(\omega)} \!\!\!\!\sum_{\omega^\ol\text{ coherent w. } \omega}\!\!\!\! w_{\ol} (\omega^{\ol}), 
	\end{align}
	where~$\omega^\circ$ is the loop configuration corresponding to~$\omega$ and 
	the sum is over the~$2^{\ell(\omega^\circ)}$ oriented loop configurations coherent with~$\omega^\circ$. 
	The term~$s(\omega) \in \{0,1\}$ is the indicator that~$\omega$ contains a cluster that winds around~$\bbT_{L,M}$ in both the vertical and horizontal directions.
	Finally,~$C= \frac{\sqrt q^{LM/2}}{(1+ \sqrt q)^{LM}}$.
\end{lemma}

The first equality is proved by induction on the number of open edges of~$\omega$. 
The second is obtained by observing that in the sum on the right-hand side, 
every retractable loop of~$\omega^\circ$ appears with its two possible orientations, thus with a total weight of~$e^\lambda + e^{-\lambda} = \sqrt q$; non-retractable loops appear with a weight of~$2$, compensated by the term~$\big(\tfrac{\sqrt{q}}2 \big)^{ \ell_0(\omega)}$.

\begin{lemma}\label{lem:correspondence3}
	For any six-vertex configuration~$\vec\omega$,
	\begin{align}\label{eq:correspondence3}
		w_{\rm 6V}(\vec\omega) = \sum_{\omega^\ol \text{ coherent w. } \vec\omega} w_{\ol}(\omega^\ol).
	\end{align}
\end{lemma}

The key here is to see~$w_{\ol}(\omega^\ol)$ as~$e^{\lambda {\rm wind}(\omega^\ol)/2\pi}$, where~${\rm wind}(\omega^\ol)$ is the total winding of all loops of~$\omega^{\ol}$. This winding may be computed locally: 
vertices of type~$a$ and~$b$ produce a total winding of~$0$, while vertices of type~$c$ produce a total winding of~$\pm \pi$, depending on how the loops are split at the vertex. See also Figure~\ref{fig:oriented_loop_vertices}.

\begin{figure}[t]
	\begin{center}
		\includegraphics[width=0.8\textwidth]{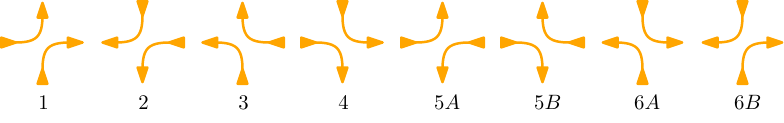}
	\end{center}
	\caption{The~$8$ different types of vertices encountered in an oriented loop configuration.}
	\label{fig:oriented_loop_vertices}
\end{figure}

\paragraph{Consequences for the partition functions: proof of Proposition~\ref{prop:BKW}.}

Summing~\eqref{eq:correspondence2} and~\eqref{eq:correspondence3} over all configurations, we easily find that 
\begin{align}\label{eq:correspondence4}
		Z_{\rm FK} \cdot \phi_{\bbT_{L,M}}\big[\big(\tfrac2{\sqrt{q}} \big)^{ \ell_0}q^{-s}\big]
		= \sum_{\omega \in \Omega_{\rm FK}} \big(\tfrac2{\sqrt{q}} \big)^{ \ell_0(\omega^\circ)}q^{-s(\omega)} \cdot w_{\rm FK}(\omega)
		= C \cdot Z_{\rm 6V},
\end{align}
with~$C$ given by Lemma~\ref{lem:correspondence2}.
This is not exactly an equality between the two partition functions, 
due to the term~$\phi_{\bbT_{L,M}}\big[\big(\tfrac2{\sqrt{q}} \big)^{ \ell_0}q^{-s}\big]$. 
Note, however, that this term is of no significance for the free energy. Indeed,~$\ell_0$ is bounded by~$L+M$, so 
\begin{align}\label{eq:free_energy}
	f_{\rm FK} 
	:= \lim_{L \to \infty} \lim_{M \to \infty} \frac{1}{LM} \log Z_{\rm FK}(\bbT_{L,M}) 
	=	f_{\rm 6V} + \tfrac14\log q - \log (1+ \sqrt q).
\end{align}
This proves~\eqref{eq:BKW1}. \medskip 

We turn to~\eqref{eq:BKW2}, and prove a more general fact. Fix~$k > 0$. 
For a six-vertex configuration~$\vec\omega$, write~$K(\vec\omega)$ for the excess of up-arrows on each row, 
that is, the number of up-arrows on any one row minus~$L/2$. Use the same notation for oriented loop configurations~$\omega^{\ol}$.

When~$\vec\omega$ and~$\omega^{\ol}$ are coherent,~$K(\vec\omega)=K(\omega^{\ol})$. 
In particular, the oriented loop configurations coherent with~$\{\vec\omega:\, K(\vec\omega) = k\}$ are exactly those with~$K(\omega^{\ol}) = k$. 

For an unoriented loop configuration~$\omega^\circ$, write~$J_k(\omega^\circ)$ for the number of ways in which its vertically-winding loops may be oriented to produce~$\omega^{\ol}$ with~$K(\omega^\ol) = k$.
For most~$\omega^\circ$, we have~$J_k(\omega^\circ) = 0$. 
Indeed, to have~$J_k(\omega^\circ) >0$,~$\omega^\circ$ needs to have at least~$2k$ loops winding vertically around the torus (where a loop winding vertically more than once is counted with multiplicity) --- see also Figure~\ref{fig:cylinder_BKW}. 
Finally, these loops are all counted in~$\ell_0(\omega^{\ol})$. 
Indeed, any retractable or horizontally-winding loop contributes the same number of up- and down-arrows to each row; 
inbalances come only from loops winding vertically around~$\bbT_{L,M}$. 

For an unoriented loop configuration~$\omega^\circ$, write~$\ell_h(\omega^\circ)$ for the number of non-retractable loops of~$\omega$ that do not wind vertically around the torus (in particular, they wind horizontally). 
Observe that re-orienting the retractable loops of~$\omega^{\ol}$, or those contributing to~$\ell_h$, does not change~$K(\omega^{\ol})$. 
Thus, applying~\eqref{eq:correspondence3}, then~\eqref{eq:correspondence2}, we find 
\begin{align*}
		Z_{\rm 6V}^{(k)} 
		=		\!\!\!\!\! \sum_{\omega^\ol : \, K(\omega^{\ol}) = k} 		\!\!\!\!\!w_{\ol}(\omega^\ol)
		= \sum_{\omega^\circ } J_k(\omega^\circ)  \frac{2^{\ell_h(\omega^\circ)}}{\sqrt q^{\ell_0(\omega^\circ)}} \sqrt q^{\ell ( \omega^\circ)}&\\
		= \frac{1}{C} \sum_{\omega} J_k(\omega^\circ)  \frac{2^{\ell_h(\omega^\circ)}}{\sqrt q^{\ell_0(\omega^\circ)}} q^{-s(\omega) } w_{\rm FK}(\omega) 
		&=\frac{Z_{\rm FK}}{C} \, \phi_{\bbT_{L,M}}\Big[J_k(\omega^\circ) \, \tfrac{ 2^{\ell_h(\omega^\circ)}}{\sqrt q^{ \ell_0(\omega^\circ)}}q^{-s(\omega) }\big].
\end{align*}

All terms in the expectation on the right-hand side, except~$J_k(\omega^\circ)$, are ultimately unimportant. 
Moreover, most of the contribution comes from configurations with~$J_k(\omega^\circ) = 1$. 
The important observation is that~$J_k(\omega^\circ) \neq 0$ only if there are at least~$k$ primal clusters crossing~$\bbT_{L,M}$ vertically (where clusters are  counted on the cylinder obtained by cutting~$\bbT_{L,M}$ horizontally at height~$0$). 

Recall that we are interested in~$f_{\rm 6V}(\alpha)$ which corresponds to the partition function over configurations with an excess~$K(\vec\omega)= \alpha L$.
Crude upper and lower bounds on the terms appearing in the probability above suffice to prove that 
\begin{align*}
	f_{\rm 6V}(\alpha)  =
	f_{\rm FK} + \lim_{L \to \infty} \lim_{M \to \infty} \tfrac{1}{LM} \log \phi_{\bbT_{L,M}}[\calE_{\alpha L}]
	 - \tfrac14\log q - \log (1+ \sqrt q)
\end{align*}
Together with~\eqref{eq:BKW1}, this proves~\eqref{eq:BKW2} (also observe that~$f_{\rm 6V} = f_{\rm 6V}(0)$).

\section{Solving six-vertex via the transfer matrix}\label{sec:Bethe_Ansatz}

The following theorem is obtained by estimating the leading eigenvalues of blocks of the transfer matrix of the six-vertex model via the Bethe ansatz. It will be used in our case to distinguish between continuous and discontinuous phase transitions of the FK-percolation model. 

\begin{theorem}\label{thm:free_energy_6V}
	As~$\alpha \searrow 0$, we have 
	\begin{align}\label{eq:BAf}
    	f_{\rm 6V}(0) - f_{\rm 6V}(\alpha)
    	=\begin{cases}
    		C \alpha + o(\alpha) &\text{ if~$c > 2$},\\
    		C \alpha^2 + o(\alpha^2) &\text{ if~$0 < c \leq 2$},
		\end{cases}
	\end{align} 
		where~$C = C(c) >0$ is a positive constant. 
\end{theorem}

Below, we give some details of the Bethe ansatz approach and the ingredients needed to prove Theorem~\ref{thm:free_energy_6V}. 
We do not provide a full proof here, as it is highly technical and does not involve tools relevant to the rest of these notes. 
See~\cite{DumKozKra22} for a full proof.

\paragraph{Transfer matrix formalism.}

We identify the possible configurations of vertical arrows on one row by~$\{\pm\}^L$. 
Consider the matrix~$V = V_L$ defined by 
\begin{align}
V(\vec x,\vec x') = \sum_{\text{possible completions}} c^{\#\text{type~$c$ vertices}}\qquad  \forall \vec x,\vec x' \in \{\pm\}^L,
\end{align}
where the completions refer to the possible assignments of horizontal arrows between~$\vec x$ and~$\vec x'$ that obey the ice rule. 
See Figure~\ref{fig:transfer_matrix} for an example.
Formally,~$V$ is an operator acting on~$(\bbC^2)^{\otimes L}$, which may be viewed as a~$2^L$-dimensional vector space with an orthonormal basis~$(\Psi_{\vec x})_{\vec x \in \{\pm\}^L}$ indexed by the possible vertical arrow configuration on one row. 

\begin{figure}
\begin{center}
\includegraphics[width = .7\textwidth, page = 1]{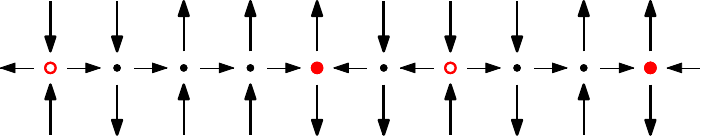}
\caption{The unique possible completion between~$\vec x$ (bottom configuration) and~$\vec x'$ (top configuration) that obeys the ice rule. The two types of $c$-vertices are marked in red; the matrix entry corresponding to the transition from $\vec x$ to~$\vec x'$ is~$c^4$.}
\label{fig:transfer_matrix}
\end{center}
\end{figure}

The transfer matrix is split into blocks corresponding to the number of up-arrows on each row (recall that this number is conserved). Indeed, if~$\vec x$ and~$\vec x'$ have a different number of up-arrows, then~$V(\vec x,\vec x') = 0$. 
Write~$V^{(k)}$ for the block with~$L/2 + k$ up-arrows. 
Then, it is immediate to check that 
\begin{align}\label{eq:Z_trace}
	Z_{\rm 6V}^{(k)}(\bbT_{L,M}) = {\rm Tr} (V^{(k)})^M.
\end{align}

Furthermore, each such block has non-negative entries and may be shown to be irreducible. 
It follows by the Perron--Frobenius theorem that each~$V^{(k)}$ has a single eigenvector of maximal eigenvalue (with all other eigenvalues being of strictly smaller modulus). We call these the Perron--Frobenius eigenvector and eigenvalue, and write~$\La^{(k)}=\La^{(k)}(L)$ for the latter. 

Then,~\eqref{eq:Z_trace} implies 
\begin{align}\label{eq:Z_trace2}
	Z_{\rm 6V}^{(k)}(\bbT_{L,M}) = (\La^{(k)}(L))^M(1 + O(e^{-\eps M})),
\end{align}
for some~$\eps = \eps(L) >0$. Since~$Z_{\rm 6V}^{(k)}(\bbT_{L,M})$ is used to determine the free energy with slope~$\alpha$, we find 
\begin{align}
	f_{\rm 6V}(\alpha) = \lim_{L\to \infty} \tfrac{1}L \log \La^{(\alpha L)}(L).
\end{align}

\paragraph{Statement of the Bethe ansatz for the six-vertex model.}

The Bethe ansatz is a general method for producing eigenvectors and eigenvalues of certain transfer matrices. 
It originates in the work of Bethe~\cite{Bet31}.  
We state it below in the context of the six-vertex model. 

Recall that~$\Delta = 1- c^2/2$ and set 
\begin{align*}
	 \mu &:= \begin{cases}
		\arccos (-\Delta)& \text{ if~$c \leq 2$}\\
		 0 & \text{ if~$c >2$}
	\end{cases}\quad 
	\text{ and } \quad 
	\mathcal{D} := (-\pi + \mu, \pi - \mu).
\end{align*}
Let~$\Theta: \mathcal{D}^2 \rightarrow \mathbb{R}$ be the unique continuous function which satisfies~$\Theta(0,0) = 0$ and
\[
	\exp(-i \Theta(x,y)) = e^{i(x-y)} \cdot \frac{e^{-ix} + e^{iy} - 2 \Delta}{e^{-iy} + e^{ix} - 2 \Delta} \, .
\]
For~$z\neq 1$, set 
\begin{align*}
	L(z):= 1 + \tfrac{c^2 z}{1-z} \, , \qquad 
	M(z):= 1 - \tfrac{c^2}{1-z} \, . 
\end{align*}

\begin{theorem}[Bethe ansatz for~$V$]\label{thm:BA} 
	Fix~$0 \leq k \leq L/2$ and set $n = L/2 - k$. 
	Let~$(p_1, p_2, \dots, p_n) \in \calD^n$ be distinct and satisfy the equations
	\begin{align}\label{eq:BA}
		N p_j = 2\pi I_j - \sum_{k=1}^n	\Theta (p_j,p_k), \qquad \forall j\in \{1,\dots,n\},
	\end{align}
	where~$I_j = j-\frac{n+1}2$ for~$j\in \{1,\dots,n\}$.
	Set  
	$\psi = \sum_{|\vec x|=n}\psi(\vec x)\, \Psi_{\vec x}$, 
	the sum is over~$\vec x \in \{\pm\}^L$ with~$n$ up arrows
	and~$\psi(\vec x)$ is given by
	$$\psi(\vec x) 
		:= \sum_{\sigma \in \mathfrak{S}_n} A_\sigma \prod_{k=1}^n \exp\left(i p_{\sigma(k)}x_k \right), \quad\text{where}\quad
		A_\sigma := \varepsilon(\sigma)\,  \prod_{1 \leq k < \ell \leq n} e^{i p_{\sigma(k)}}\  (e^{-{i} p_{\sigma(k)}}+e^{{\rm i } p_{\sigma(\ell)}}-2\Delta),
	$$
	for~$\sigma$ an element of the symmetry group~$\mathfrak{S}_n$. 
	Then~\begin{align}
		V\psi = \Lambda \psi,
	\end{align} 
	where 
	\begin{align*}
		\Lambda = 		
		\begin{cases}
			\displaystyle \prod_{j=1}^n L(e^{ip_j}) + \prod_{j=1}^n M(e^{ip_j})\quad &\text{if all~$p_1,\dots,p_n$ are non zero,} \vspace{3pt}\\
			\displaystyle
			\Big[2+ c^2 (N-1) + c^2 \sum_{j\neq \ell} \partial_x \Theta (0,p_{j}) \Big] \cdot \prod_{j \neq \ell} M(e^{ip_j})
			 &\text{if~$p_\ell = 0$ for some~$\ell$.}
		\end{cases}
	\end{align*}
\end{theorem}
\medskip 

In words, the above explains how a solution to~\eqref{eq:BA} may be used to produce an eigenvector and eigenvalue for~$V$. 
The proof of the above is an intricate, but ultimately direct computation; a full proof may be found in~\cite{DumGagHar16b}.
It is expected that there exists exactly one solution to~\eqref{eq:BA}, which produces the Perron--Frobenius eigenvector and eigenvalue of~$V^{(k)}$. 

Let us assume that one may find solutions~$p_1,\dots, p_n$ to~\eqref{eq:BA}, 
and that their empirical distribution defined as~$\rho_N = \frac1N\sum_{i=1}^n \delta_{p_i}$ converges to a measure~$\rho(x) dx$ on~$\calD$ admitting a density --- this is sometimes called {\em condensation} of the Bethe roots.
Then~\eqref{eq:BA} implies that 
\begin{align}\label{eq:BA_cont}
	2 \pi \rho(x) &= 1 + \displaystyle \int_\calD \partial_x \Theta (x,y) \rho(y) dy		\qquad \forall x\in \calD
\end{align}
and 
\begin{equation}\label{eq:ag}
\lim_{L\rightarrow\infty}\tfrac1L\log \Lambda = \int_\calD \log \left|M\left(e^{ix}\right)\right| \rho(x) dx.
\end{equation}
The condition~\eqref{eq:BA_cont} allows one to determine~$\rho$, which may then be injected in~\eqref{eq:ag} to obtain asymptotics for~$\Lambda$.

We will not prove Theorem~\ref{thm:BA} here, nor how it applies to extract Theorem~\ref{thm:free_energy_6V}. 
We limit ourselves to the following observations. 
\begin{itemize}
	\item The difficulties in applying this method are: prove that~\eqref{eq:BA} has solutions and that these solutions do condensate as~$L\to\infty$. Show that the resulting vector is not null, and therefore that the resulting~$\La$ is an eigenvalue. Prove that~$\La$ is the Perron--Frobenius eigenvalue of the corresponding block of~$V$. Finally, when this is done, one needs to compute sufficiently precise estimates on the resulting eigenvalue, using the continuous version of the Bethe equations~\eqref{eq:BA_cont} and~\eqref{eq:ag}.
	\item The Bethe ansatz changes form at~$\Delta = -1$, corresponding to~$c =2$. Outside of this value, it depends smoothly on~$\Delta$. 
	\item These difficulties were overcome in~\cite{DumGagHar16b} and~\cite{DumKozKra22}. Proving that~$\La$ is indeed the Perron--Frobenius eigenvalue was done by considering the model (for finite~$L$) at the ``trivial'' cases~$\Delta = 0$ and~$\Delta = -\infty$, where the computations are explicit, then proving that~$\Delta \mapsto (\psi,\La)$ is an analytic function, hence preserving the positivity of~$\psi$. The condensation is also proved using this type of evolution in~$\Delta$: it is proved that~$\psi(\Delta)$ can not continuously ``escape'' a region of condensation.
	\item The estimates obtained for~$\La^{(k)}/\La^{(0)}$ are not sufficiently fine to apply to~$k=1$ (or small values of~$k$). 
	For our purpose~$k = \alpha L$ suffices, where the estimates are more robust. 
	\item The ultimate computation of~$\La^{(k)}/\La^{(0)}$ proves~\eqref{eq:BAf}.
\end{itemize}

\section{Deducing the type of phase transition}

A direct consequence of Proposition~\ref{prop:BKW} and Theorem~\ref{thm:free_energy_6V} is the following estimate on~$\phi_{\bbT_{L,M}}[\calE_{\alpha L}]$.

\begin{corollary}\label{cor:long_vertical_cross}
	We have
	\begin{align}\label{eq:long_vertical_cross}
		\lim_{L\to \infty} \lim_{M\to \infty}\tfrac{1}{LM} \,\log\phi_{\bbT_{L,M}}[\calE_{\alpha L}] 
		= 
		\begin{cases}
    		- C \alpha + o(\alpha) &\text{ if~$q > 4$},\\
    		- C \alpha^2 + o(\alpha^2) &\text{ if~$1 \leq q \leq 4$},
		\end{cases}
	\end{align}
	with~$C = C(\sqrt{2 +\sqrt q})>0$ given by Theorem~\ref{thm:free_energy_6V}.
\end{corollary}

This corollary implies directly Theorem~\ref{thm:q<>4}.

\begin{proof}[Proof of Theorem~\ref{thm:q<>4}] 
Fix~$q\ge1$. Define the correlation length  at criticality as
\begin{align}
	\xi^{-1}=\xi^{-1}(p_c,q) 
	= - \lim_n \tfrac{1}n \log \phi^{0}_{\La_n}[0 \lra \partial \La_n]
	= - \lim_n \tfrac{1}n \log \phi^{0}[0 \lra \partial \La_n].
\end{align}
The equality of the two limits may be obtained by standard arguments, see Exercise~\ref{exo:cor_len}.
Also define
\begin{align}
	\pi(n) = \phi^{0}_{\La_n}[(0,-n/2) \lra (0,n/2)].
\end{align}
Then, a straightforward analysis shows that~$\frac1n\log \pi(n) \to -\xi^{-1}$.

We claim that there exists a universal constant~$c > 0$ such that, for all~$4 \leq 2k \leq L \leq M$, 
\begin{align}\label{eq:H_to_xi}
	c^{k^2 M/L} \pi(L/4k)^{4k^2 M/L } \leq \phi_{\bbT_{L,M}}[\calE_{k}] \leq  L^k e^{- (k-2) M / (\xi +o(1)) }.
\end{align}
with~$o(1)$ denoting a quantity converging to~$0$ as~$M\to\infty$. 

For the lower bound, cut up the cylinder in vertical rectangles~${\rm Rect}_j = [j\frac{L}{2k}, (j+1)\frac{L}{2k}]\times [0,M]$ for~$j = 1,\dots, 2k$.
Using RSW techniques\footnote{Here,~\eqref{eq:RSWstrip} does not apply directly, but one may easily prove the same estimate in periodic rather than~$1/0$ boundary conditions}, we may construct~$k$ dual clusters crossing the cylinder vertically in the rectangles~${\rm Rect}_j$ with~$j$ even with a probability of at least~$c^{k^2 M/L}$ for some constant $c >0$.
Conditionally on the existence of these clusters (or, more generally, on any configuration in the rectangles with~$j$ even),
the probability that each of the rectangles~${\rm Rect}_j$ with~$j$ odd contains a vertical crossing is at least~$\pi(L/4k)^{4k M/L}$. 
Indeed, combining~${4k M/L}$ translates of~$\{(0,-\frac{L}{8k}) \lra (0,\frac{L}{8k})\}$ using~\eqref{eq:FKGFK}, one obtains a crossing of~${\rm Rect}_j$.
The lower bound in~\eqref{eq:H_to_xi} follows. 

For the upper bound, assume that~$\calE_{k}$ occurs. 
Start from an arbitrary point and explore the closest primal clusters to the left and right of this point that cross the cylinder vertically. 
When conditioning on the value of these clusters, the measure in the rest of the cylinder has free boundary conditions.
Moreover, the unexplored part contains at least~$k-2$ other crossing clusters. 
We explore these clusters one by one, from left to right. 
After each cluster is explored, the probability that another cluster exists is at most~$L \phi^{0}[0 \lra \partial \La_{M}]$.
This proves the upper bound in~\eqref{eq:H_to_xi}. 

When taking~$k = \alpha L$, then taking~$M \to \infty$,~$L\to \infty$ and then~$\alpha \to 0$,~\eqref{eq:H_to_xi} translates to 
\begin{align}
	\xi^{-1} - o(1) 
	\leq \limsup_{\substack{M\to \infty\\ L \to \infty}} -\frac{1}{\alpha L M}\log \phi_{\bbT_{L,M}}[\calE_{\alpha L}] 
	\leq \xi^{-1} +o(1),
\end{align}
where $o(1)$ denotes a quantity tending to $0$ as $\alpha \to0$. 
When combining the above with~\eqref{eq:long_vertical_cross}, we conclude that~$\xi(p_c,q) = \infty$ if and only if~$1 \leq q \leq 4$.
Theorem~\ref{thm:dichotomy} allows us to conclude. 
\end{proof}

%

\chapter{Rotational invariance of the critical phase}\label{ch:5rotational_inv}

\newcommand{\onearmbound}{\alpha_{\rm arm}}

This chapter aims to give an overview of the recent result of~\cite{DumKozKra20} which states that the critical FK-percolation on~$\bbZ^2$ is asymptotically rotationally invariant. 
Henceforth, fix~$q \in [1,4]$ and omit it from the notation.

\section{Results: rotational invariance and universality}\label{sec:5intro}

\paragraph{Loop topology.}
To state our results, we first need to describe how the large scale geometry of percolation configurations should be understood. 
Recall from Section~\ref{sec:BKW} that a percolation configuration~$\omega$ on~$\bbZ^2$ may be encoded by the loop configuration~$\omega^\circ$. The latter is formed of two types of loops: those with primal edges on the inside and those with primal on the outside (we will always consider cases where~$\omega^\circ$ contains no infinite paths). In this chapter, we will consider~$\omega^\circ$ as an oriented loop configuration, with each loop being oriented so that the primal lies on its right side --- this is not the same orientation as that used in the BKW correspondence of Section~\ref{sec:BKW}. 
We will also identify~$\omega^\circ$ and~$\omega$ to avoid overburdening notation. 

For two loops~$\gamma$ and~$\gamma'$ of oriented loop configurations define the loop distance between them as 
\begin{align}\label{eq:loop_dist}
	d(\gamma, \gamma') = \inf \|\gamma - \gamma'\|_{\infty},
\end{align}
where the infimum is over all orientation-preserving parametrizations of~$\gamma$ and~$\gamma'$ by~$\bbS^1$. 

For two families of oriented loops~$\calF$ and~$\calF'$, define the Camia--Newman distance~\cite{CamNew06} by 
\begin{align}
	d_{\rm CN}(\calF,\calF') \leq \epsilon
\end{align}
if for each loop~$\gamma$ of~$\calF$ contained in~$\La_{1/\epsilon}$ of diameter greater than~$\epsilon$ 
there exists a loop~$\gamma' \in \calF'$ such that~$d(\gamma,\gamma') < \epsilon$, and vice versa. 

Finally, if~$\phi$ and~$\phi'$ are probability measures producing oriented families of loops (or equivalently percolation configurations on some rescaled lattice~$\delta \bbZ^2$), set         
\begin{align}
	{\bf d}_{\rm CN}(\phi, \phi') \leq \epsilon
\end{align}
if there exists a coupling~$P$ of~$\phi$ and~$\phi'$ producing configurations~$\calF$ and~$\calF'$
such that
\begin{align}\label{eq:CN_dist_measures}
	P[d_{\rm CN}(\calF, \calF') < \epsilon] > 1-\epsilon.
\end{align} 

With the distance defined above, we can formally mention a central conjecture in two-dimensional percolation. 

\begin{conjecture}\label{conj:CLE}
	For~$q \in [1,4]$, the critical FK-percolation measures~$\phi_{\delta \bbZ^2, p_c,q}$ on the rescaled lattice~$\delta \bbZ^2$ converge, as~$\delta \to 0$, to the nested, full plane CLE$_\kappa$ measure with 
	\begin{align*}
		\kappa(q)=4\pi/\arccos(-\tfrac{\sqrt q}2).
	\end{align*}
\end{conjecture}

We will not describe here what CLE$_\kappa$ is, but simply mention that it is a measure on families of loops, which depends on a parameter~$\kappa >0$. For details, see~\cite{Wer04,Wer09,Wer09a,She09,SheWer12}. A similar conjecture may be expressed in finite domains. 
A fundamental property of CLE$_\kappa$, and conjecturally of scaling limits of critical statistical mechanics models, is that they exhibit {\em conformal invariance}. 

From the perspective of discrete models, the convergence to CLE$_\kappa$ should be viewed as the most precise information on the critical phase. In particular, it may be used to determine different critical exponents, such as the one-arm exponent~$\alpha_1$ and the influence exponent~$\iota$, which in turn determine (almost) all critical exponents via the so-called scaling relations~\cite{DumMan20}.

\paragraph{Rotational invariance.}
The goal of this chapter is to prove that FK-percolation on~$\bbZ^2$, when rescaled, becomes asymptotically rotationally invariant. 
As a consequence, any subsequential scaling limit of its loop representation is invariant (in law) under rotations. 

\begin{theorem}\label{thm:rotation_invariance}
	Fix~$q\in[1,4]$ and~$\theta \in [0,2\pi]$. Write~$\phi_{e^{i\theta}\delta \bbZ^2}$
	for the critical FK-percolation measure on the rescaled lattice~$\delta \bbZ^2$, rotated by~$\theta$. Then, for any~$\delta >0$,
   \begin{align*}
	{\bf d}_{\rm CN}(\phi_{\delta \bbZ^2},\phi_{e^{i\theta}\delta \bbZ^2}) \leq C \delta^c,
    \end{align*}
    where~$c,C > 0$ are universal constants. 
\end{theorem}
In light of the above, we say that~$\phi_{ \bbZ^2}$ is {\em asymptotically rotationally invariant}. 
The supposed conformal invariance of the scaling limit of FK-percolation may be viewed as a result of two more basic invariances: with respect to rotations and homotheties. Indeed, conformal transformations are locally compositions of rotations and homotheties, and mixing arguments between local and global features suggest that these two types of symmetries imply overall conformal invariance. 
This points to a three part program for proving Conjecture~\ref{conj:CLE}, and more generally conformal invariance of scaling limits of 2D statistical mechanics models: invariance under homotheties, invariance under rotations and combining the two to deduce conformal invariance. 
It may be argued that the invariance of the scaling limit under rotations is the most mysterious step in this program. 
Indeed, any scaling limit is, by construction, invariant under  homotheties; the last step should be the result of mixing between scales,  something that is expected in most models and their scaling limits.

At the time of writing no general proof is available for the existence of the scaling limit, nor for the implication between scale, rotational and conformal invariance. The only exceptions are for the FK-Ising model ($q = 2$)~\cite{Smi10,CheSmi12} and for a variant of Bernoulli percolation ($q = 1$)~\cite{Smi01,CamNew06}, 
where the existence of a conformally-invariant scaling limit was proved by different means. 

Our proof of rotational invariance is intimately related to a universality result (see Theorem~\ref{thm:universalCNSS} below), which is naturally stated in the context of FK-percolation on isoradial graphs.

\paragraph{Universality on rectangular isoradial lattices.}	
We describe now an inhomogeneous FK-percolation model on some distorted embedding of the square lattice~$\bbZ^2$. 
These notions may appear strange at first, but will be shown to fit in the more general framework of isoradial graphs (see Section~\ref{sec:5isoradial}). 

For~$\alpha \in (0,\pi)$ let~$\bbL(\alpha)$ be the embedding of~$\bbZ^2$ in which horizontal edges have length~$2 \cos\alpha$ 
and vertical edges have length~$2\sin \alpha$, rotated by an angle~$\alpha/2$. 
Consider the FK-percolation model on~$\bbL(\alpha)$ with cluster-weight~$q$ and different edge-parameters~$p_{\rm hor}$ and~$p_{\rm vert}$ for the ``horizontal'' and ``vertical'' edges (that is, the edges of lengths~$2 \cos\alpha$ and~$2\sin \alpha$) respectively, given by
\begin{align}
	\frac{p_{\rm hor}}{1-p_{\rm hor}} = q \frac{1-p_{\rm vert}}{p_{\rm vert}}  = \sqrt{q} \frac{\sin( r \alpha)}{\sin( r (\pi-\alpha))}
	\qquad \text{ with } r := \tfrac{1}{\pi} \cos^{-1} \left( \tfrac{\sqrt{q}}{2} \right).
\end{align}
Write~$\phi_{\bbL(\alpha)}$ for the unique infinite-volume measure on~$\bbL(\alpha)$ with the parameters above --- the uniqueness of the measure will be discussed below.
We call~$\bbL(\alpha)$ an {\em isoradial rectangular lattice}, and~$\phi_{\bbL(\alpha)}$ its associated FK-percolation model.

Notice the duality relation~\eqref{eq:p_dual}  between~$p_{\rm hor}$ and~$p_{\rm vert}$, which implies that~$\phi_{\bbL(\alpha)}$ is self-dual. 
In the particular case of~$\alpha = \pi/2$,~$\phi_{\bbL(\pi/2)}$ is the critical FK-percolation on~$\bbZ^2$, rotated by~$\pi/4$. 
It was proved in \cite{DumLiMan18} that~$\phi_{\bbL(\alpha)}$ shares the qualitative features of critical FK-percolation on~$\bbZ^2$, most importantly an RSW property (see Theorem~\ref{thm:RSW_iso}). 

The models~$\phi_{\bbL(\alpha)}$ for different values of~$\alpha$ may be related via a series of star-triangle transformation (see Section~\ref{sec:star-triangle}); features that are stable under these transformations may then be transferred from~$\phi_{\bbL(\pi/2)}$ to all other~$\phi_{\bbL(\alpha)}$. This strategy was used in \cite{GriMan14,DumLiMan18} for RSW estimates, and will be used here to prove the following universality result. 

\begin{theorem}[Universality for isoradial rectangular lattices]\label{thm:universalCNSS}
	Fix~$q\in[1,4]$ and $\alpha \in (0,\pi)$.  There exist constants~$c,C > 0$ such that, for all~$\delta>0$,
	\begin{align}\label{eq:universalCNSS}
		{\bf d}_{\rm CN}(\phi_{\delta\bbL(\alpha)}, \phi_{\delta\bbL(\pi/2)}) < C \delta^c.
    \end{align}
\end{theorem}

In light of the above, we say that~$\phi_{\delta\bbL(\alpha)}$ and~$\phi_{\delta\bbL(\pi/2)}$ are {\em asymptotically similar}.

Observe that Theorem~\ref{thm:universalCNSS} readily implies Theorem~\ref{thm:rotation_invariance}.
Indeed, for any~$\alpha \in (0,\pi)$,~$e^{{i}\alpha/2}\mathbb R$ is an axis of symmetry for~$\bbL(\alpha)$, and therefore for~$\phi_{\bbL(\alpha)}$.
Then~\eqref{eq:universalCNSS} implies that~$\phi_{\bbL(\pi/2)}$ is  asymptotically invariant under the reflection with respect to~$e^{{i}\alpha/2}\mathbb R$. 
The composition of the reflections with respect to the horizontal axis and to~$e^{{i}\alpha/2}\mathbb R$ produce the rotation by~$\alpha$, 
and we deduce the asymptotic invariance of~$\phi_{\bbL(\pi/2)}$ by said rotation. 

As such, it would be tempting to assume that we will first prove Theorem~\ref{thm:universalCNSS}, then deduce Theorem~\ref{thm:rotation_invariance}. This will not be the case in these notes. 
We will rather start by proving the following weaker version of  Theorem~\ref{thm:universalCNSS}.

\begin{theorem}[Universality up to linear deformation]\label{thm:linear}
	For every~$q\in[1,4]$ and~$\alpha, \beta \in (0,\pi)$, 
	there exist constants~$c,C > 0$ and an invertible linear map~$M_{\beta,\alpha}:\bbR^2 \to \bbR^2$ such that 
	\begin{align}\label{eq:linear}
		{\bf d}_{\rm CN}\big[ \phi_{\delta\bbL(\beta)}, \phi_{\delta\bbL(\alpha)}\circ M_{\beta,\alpha}\big] \leq C\,\delta^c 
		\quad \text{ for all~$\delta > 0$}.
	\end{align}
\end{theorem}

The statement should be understood as~$\omega \sim \phi_{\delta\bbL(\beta)}$ and~$\omega' \sim  \phi_{\delta\bbL(\alpha)}$ may be coupled so that the loop representations of~$\omega$ and~$M_{\beta,\alpha}^{-1}(\omega')$ are close\footnote{or equivalently such that~$M_{\beta,\alpha}(\omega)$ and~$\omega'$ are close.} for~$d_{\rm CN}$ with high probability. 
Theorem~\ref{thm:universalCNSS} is equivalent to Theorem~\ref{thm:linear} with the input that~$M_{\pi/2,\alpha} = {\rm id}$.

Once Theorem~\ref{thm:linear} is proved, we will show the asymptotic rotational invariance of~$\phi_{\bbL(\pi/2)}$ (i.e. Theorem~\ref{thm:rotation_invariance}) by an argument similar to the one described above --- see Section~\ref{sec:deducing_rot_inv}.
Then, in Section~\ref{sec:drift0}, we will use  Theorem~\ref{thm:rotation_invariance} to prove that~$M_{\pi/2,\alpha} = {\rm id}$ for all~$\alpha \in (0,\pi)$, thus deducing Theorem~\ref{thm:universalCNSS}. 

One may also prove in a more direct way that~$M_{\pi/2,\alpha} = {\rm id}$, deduce Theorem~\ref{thm:universalCNSS}, then  Theorem~\ref{thm:rotation_invariance}, using different methods. 
We prefer the more convoluted strategy described above for a series of reasons. 

Let us start by mentioning that the proofs below blend two types of arguments: qualitative features of critical FK-percolation (essentially the RSW property and its consequences)
and some form of exact integrability. 
Exact integrability is commonly considered equivalent to invariance under the star-triangle transformation (a.k.a.~$Z$-invariance), 
but is mostly used in much stronger forms, such as the explicit computation of certain partition functions. 
Our aim here is to isolate as much as possible the use of exact integrability, and to limit it to the star-triangle invariance. 

Theorem~\ref{thm:linear} only uses qualitative features of critical FK-percoaltion and the star-triangle transformation.
It does not claim that the isoradial embedding (i.e. the exact ratio between the lengths of the  horizontal and vertical edges of~$\bbL(\alpha)$) 
ensures universality in the sense of~\eqref{eq:universalCNSS}. 
A separate argument is needed for the latter, which necessarily involves exact integrability. 

Indeed, the more direct route to Theorem~\ref{thm:universalCNSS} mentioned above involves proving that~$M_{\pi/2,\alpha} = {\rm id}$ 
using explicit computations of certain partition functions of the six-vertex model based on the Bethe ansatz (similar to those of Section~\ref{sec:Bethe_Ansatz}). 
As for~$\phi_{\bbL(\pi/2)}$, any FK-percolation model~$\phi_{\bbL(\alpha)}$ may be related to an inhomogeneous six-vertex model on~$\bbZ^2$, as in Section~\ref{sec:BKW}. 
Write~$\phi_{\bbT_{L,M}(\alpha)}$ for the corresponding model on the~$L\times M$ torus (that is, the torus with~$L$ columns and~$M$ rows)
and recall the definition of~$\calE_k$ from Section~\ref{sec:BKW}.
The Bethe ansatz also applies to the latter model (the eigenvectors are actually independent of~$\alpha$) and one may show that 
\begin{align}\label{eq:eigenvalues_alpha} 
	\frac{L}{M k^2 \sin(\alpha) }\log \phi_{\bbT_{L,M}(\alpha)}[\calE_{k}] = (1+o(1))\frac{L}{M k^2}\log \phi_{\bbT_{L,M}(\pi/2)}[\calE_{k}] 
\end{align}
as~$M \to \infty$, then~$L\to\infty$ and~$k/L \to 0$, in a certain regime of values of~$k$. 

It is most instructive to think of~\eqref{eq:eigenvalues_alpha} for~$k = 1$. 
It suggest that the asymptotic of the probability under~$\phi_{\bbL(\alpha)}$ for  an interface to cross vertically a tall cylinder of {\em Euclidean} dimensions~$L\times M$ is independent of the angle~$\alpha$.
As such, it is not entirely surprising that~\eqref{eq:eigenvalues_alpha} ultimately may be used to prove that~$M_{\pi/2,\alpha} = {\rm id}$.
Unfortunately, this implication is technically challenging, especially since we are unable to prove~\eqref{eq:eigenvalues_alpha} for~$k = 1$, but need to take~$k \to \infty$ as~$L \to\infty$. 
Still, it should be mentioned that the first version of \cite{DumKozKra20} did use the strategy outlined above, and a more detailed sketch may be found there. 

In the present approach, we avoid the use of~\eqref{eq:eigenvalues_alpha}, or any form of exact integrability other than the star-triangle transformation. 
That~$M_{\pi/2,\alpha} = {\rm id}$ will actually follow form a subtle interplay between Theorem~\ref{thm:linear}, an additional symmetry of~$\phi_{\bbL(\pi/2)}$,
and the way that the  star-triangle transformation acts on isoradial graphs. 
In addition to ultimately being shorter and more self-contained, 
this argument illustrates that the star-triangle transformation alone implies that the isoradial embedding is the ``correct'' embedding to ensure universality. 
The author finds the latter quite remarkable.

\paragraph{Some additional remarks}
We close this section with some additional remarks and extensions of the results listed above.

\begin{remark}\label{rem:quantum_FK}
	The constants~$c$ and~$C$ in Theorem~\ref{thm:universalCNSS} may actually be chosen uniformly in~$\alpha$.
	Thus, the universality result extends to the model obtained when taking~$\alpha \to 0$, 
	 sometimes called the quantum (or continuum) FK-percolation model 
    (see~\cite[Sec.~9.2]{Gri10} for a description and~\cite[Sec. 5]{DumLiMan18} for more details on this step). 
    
    To simplify the exposition, we will not focus on the independence of the constants~$c,C$ on~$\alpha$ in this manuscript. 
\end{remark}

\begin{remark}\label{rem:universality_iso}
The universality result of Theorem~\ref{thm:universalCNSS} extends beyond isoradial rectangular lattices. 
Indeed, it is proved in~\cite{HanMan24+} that it extends to all bi-periodic isoradial graphs.

Additionally, the universality principle suggests that any bi-periodic critical planar FK-percolation model lies in the same universality class as~$\phi_{\bbL(\pi/2)}$, and therefore~\eqref{eq:universalCNSS} should also apply to it, if properly embedded. 

The author expects that the a refinement of~\cite{HanMan24+} may incorporate all isoradial graphs (even non-periodic), but is unlikely to go beyond that. 
Indeed, Theorem~\ref{thm:universalCNSS} and its extension~\cite{HanMan24+} makes crucial use of the star-triangle transformation, 
for which isoradial graphs are a stable and natural class. 
It would be interesting to find additional transformations that allow to exit the class of isoradial graphs, 
while preserving the large scale geometry of the percolation model.
We mention here the related works of~\cite{Che20} on more general embeddings for the Ising model.
\end{remark}

\paragraph{Structure of the chapter.}
The following section contains background on isoradial graphs, the associated FK-percolation model and the star-triangle transformation. 
It introduces some of the notions used in the following sections, as well as the results of \cite{GriMan14, DumLiMan18}, which are a precursor to those presented here. 
Theorem~\ref{thm:linear} is proved in Section~\ref{sec:universality_isoradial_lin}. 
This is the most technical part of Chapter~\ref{ch:5rotational_inv}; it is self-contained and may be skipped in a first reading. 
Theorems~\ref{thm:rotation_invariance} and \ref{thm:universalCNSS} are proved in Sections~\ref{sec:deducing_rot_inv} and~\ref{sec:drift0} respectively.


\section{FK-percolation on isoradial graphs; the star-triangle transformation}\label{sec:5isoradial}

This section contains a brief introduction to isoradial graphs and the FK-percolation associated to them. 

Isoradial graphs were introduced by Duffin in~\cite{Duf68} in the context of discrete complex analysis, 
and later appeared in the physics literature in the work of Baxter~\cite{Bax78} under the name~$Z$-invariant graphs.
They have been studied extensively, in particular in the context of statistical mechanics; we refer to~\cite{CheSmi12,KenSch05, Mer01, GriMan14,DumLiMan18} for literature on the subject.

\subsection{Isoradial graphs}\label{sec:isoradial_graphs}

A rhombic tiling~$G^\diamond$  is a tiling of the plane by rhombi of edge-length~$1$. Any such graph is bipartite, and we may divide its vertices in two sets of non-adjacent vertices~$V_{\bullet}$ and~$V_{\circ}$. 
The \emph{isoradial graph}~$G$ associated to~$G^\diamond$ is the graph with vertex set~$V_{\bullet}$ and edge-set given by the diagonals of the faces of~$G^\diamond$ between vertices of~$V_{\bullet}$. If the roles of~$V_{\bullet}$ and~$V_{\circ}$ are exchanged, we obtain the dual of~$G$, which is also isoradial. 
The term isoradial was introduced by Kenyon~\cite{Ken02} and refers to the fact that each face of~$G$ is inscribed in a circle of radius~$1$.
The rhombic tiling~$G^\diamond$ is called the {\em diamond graph} of~$G$.

A {\em train-track} (or simply track) of~$G$ is a bi-infinite sequence of adjacent faces~$(r_i)_{i\in \bbZ}$ of the diamond graph~$G^\diamond$, 
with the edges shared by each~$r_i$ and~$r_{i+1}$ being parallel. The angle formed by any such edge with the horizontal axis
is called the {\em transverse angle} of the track.

Isoradial graphs considered in this paper are of a very special type, see Figure~\ref{fig:isoradial lattice unaltered}.
They will all be isoradial embeddings of the square lattice; moreover we assume that all rhombi of~$G^\diamond$ have bottom and top edges that are horizontal. 
A consequence of this assumption is that the diamond graph contains {\em horizontal tracks}~$t_i$ with transverse angles~$\alpha_i \in (0,\pi)$ and {\em vertical tracks}~$s_j$, all of which have transverse angle~$0$. Each track of one category intersects all tracks of the other category but no track of the same category. 

For a sequence of track angles~$\pmb\alpha=(\alpha_i)_{i\in \mathbb Z}\in(0,\pi)^\bbZ$, denote by~$\mathbb  L(\pmb\alpha)$ the graph whose horizontal tracks have transverse angles~$\alpha_i$, in increasing vertical order.
When~$\alpha_i=\alpha$ for every~$i$, simply write~$\bbL(\alpha) = \bbL(\balpha)$.
Note that~$\bbL(\alpha)$ has identical rectangular faces and is rotated such that~$e^{{i}\alpha/2}\mathbb R$ acts as an axis of symmetry.  
In particular,~$\bbL(\tfrac\pi2)$ is simply a rescaled and rotated (by an angle of~$\pi/4$) version of~$\mathbb Z^2$.
These are indeed the isoradial rectangular lattices described in Section~\ref{sec:5intro}.

\begin{figure}[t]
  \centering
  \includegraphics[width=0.70\textwidth]{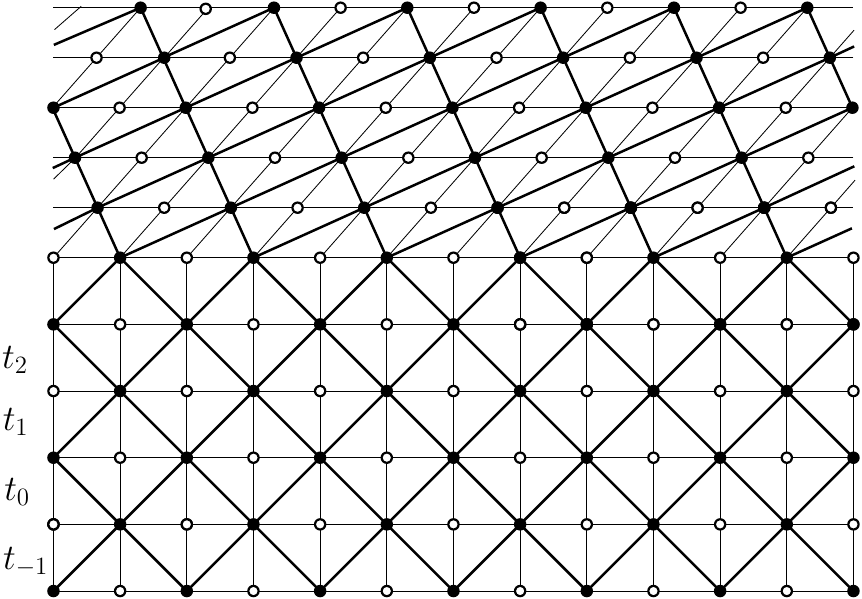}
  \caption{An example of a graph~$\bbL(\pmb\alpha)$, where~$\alpha_i$ is equal to~$\tfrac\pi2$ for~$i\le 3$, and some angle~$\alpha$ above. The diamond graph is drawn in light black lines; the solid and hollow dots are the vertices of~$V_\bullet$ and~$V_\circ$, respectively. The actual isoradial graph~$G$ is drawn in thicker black lines. 
  }
 \label{fig:isoradial lattice unaltered}
\end{figure}

When considering isoradial graphs~$G = (V,E)$, we keep the notation~$\La_n = [-n,n]^2$ and identify it with the subgraph spanned by the vertices of~$V$ contained in~$\La_n$. We write~$\partial \La_n$ for the set of vertices~$v \in V \cap \La_n$ that have a neighbour in~$V \cap \La_n^c$. 

\subsection{FK-percolation on isoradial graphs}

The isoradial embedding of a graph~$G = (V,E)$ produces different edge-weights for the edges of~$G$ as functions of their length, or more commonly the subtended angle. Indeed, if~$e$ is an edge of~$G$ and~$\theta_e$ is the angle of the rhombus of~$G^\diamond$ containing~$e$ and not bisected by~$e$ (see Figure~\ref{fig:subtended_angle}), we set
\begin{equation}\label{eq:isoraial_p_e}
	p_e:=\begin{cases}\displaystyle \frac{\sqrt{q}\sin( r (\pi - \theta_e))}{\sin( r \theta_e)+\sqrt{q}\sin( r (\pi - \theta_e))} 
    &\text{if } q<4, \\[5mm]
\qquad\ \ \displaystyle  \frac{2\pi - 2\theta_e}{2\pi-\theta_e} &  \text{if } q = 4,\\[5mm]
\displaystyle \frac{\sqrt{q}\sinh( r (\pi - \theta_e))}{\sinh( r \theta_e)+\sqrt{q}\sinh( r (\pi - \theta_e))} 
    &\text{if }q>4,\end{cases}
\end{equation}
where~$r := \tfrac{1}{\pi} \cos^{-1} \left( \tfrac{\sqrt{q}}{2} \right)$ for~$q\le 4$ and $r := \tfrac{1}{\pi} \cosh^{-1} \left( \tfrac{\sqrt{q}}{2} \right)$ for~$q>4$ --- the case~$q>4$ is not relevant for this presentation, but we give the formula to emphasise that the weights exist for all~$q \geq 1$ and change nature at~$q = 4$.
Notice that when~$\theta_e = \pi/2$ we find 
\begin{equation}
	p_e:=\tfrac{\sqrt q}{1 + \sqrt q}.
\end{equation}

\begin{figure}[t]
  \centering
  \includegraphics[width=0.3\textwidth]{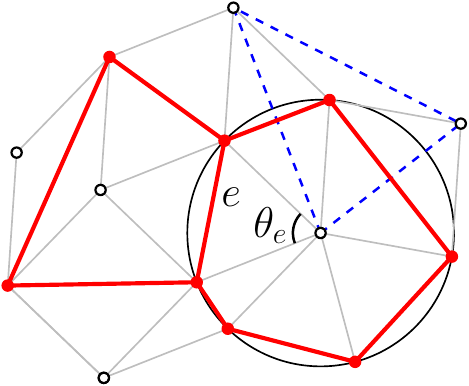}
  \caption{The edge~$e$ and its subtended angle~$\theta_e$; the red edges are those of~$G$ and the grey ones are those of the diamond graph. The dual graph (blue edges, hollow vertices) is also isoradial.}
  \label{fig:subtended_angle}
\end{figure}

For $q \geq 1$, define an infinite-volume FK-percolation measure~$\phi_G$ on~$G$ as any probability measure on~$\{0,1\}^E$ with the following property. 
For any finite set~$F \subset E$ and any configuration\footnote{The conditioning in the right-hand side of \eqref{eq:RCM_def2} is degenerate, but may be properly defined using standard manipulations; we leave this to the reader.}~$\omega_0$ on~$F^c$,
\begin{equation}\label{eq:RCM_def2}
	\phi_{G}[\omega \text{ on } F\,|\, \omega = \omega_0 \text{ on~$F^c$}]
	=\frac{1}{Z^{\omega_0}(F)} \Big(\prod_{e\in F}p_e^{\omega_e}(1-p_e)^{1-\omega_e} \Big) q^{k_F(\omega)},
\end{equation}
where~$k_F(\omega)$ is the number of connected components of~$\omega$ that intersect~$F$ and~$Z^{\omega_0}(F)$ is a normalising constant. 

A similar definition may be formulated on finite isoradial graphs with arbitrary boundary conditions. The basic properties of FK-percolation (see Section~\ref{ch:2intro_FK}) extend readily to this setting. 
Notice that the lattices~$\bbL(\alpha)$ and the associated parameters for FK-percolation given by~\eqref{eq:isoraial_p_e} are indeed those mentioned in Section~\ref{sec:5intro}.

\paragraph{Duality.}

Recall that if~$G$ is an isoradial graph, then so is its dual~$G^*$. 
Assuming that~$G$ possesses a unique infinite-volume FK-percolation measure, we have that 
\begin{align}
	\text{ if~$\omega \sim \phi_{G}$, then~$\omega^* \sim \phi_{G^*}$.}
\end{align}
In other words, the dual of an isoradial measure on some graph~$G$ is the isoradial measure on the dual graph~$G^*$. 
We stated the result here only for infinite-volume measures, but similar statements may be formulated in finite volume.

\subsection{RSW property and consequences}\label{sec:IIC}

As in the case of the square lattice, it was proved in~\cite{DumLiMan18} that, for~$q \in [1,4]$, there exists a unique infinite-volume measure on a large variety of isoradial graphs and that it has similar properties to the critical FK-percolation on~$\bbZ^2$. 
Most importantly to us, the RSW estimates of Theorem~\ref{thm:RSW_iso} also hold for the isoradial rectangular lattice which we will use below. 

\begin{theorem}[RSW on isoradial rectangular lattices]\label{thm:RSW_iso}
	For any~$1\le q\le 4$ and~$\rho,\ep>0$, there exists~$c=c(\rho,\ep)>0$ 
	such that for any~$\pmb\alpha=(\alpha_i:i\in\bbZ) \in (0,\pi)^\bbZ$ containing at most two values, 
	any~$n\geq1$ and any event~$A$ depending on the edges at distance at least~$\eps n$ from the rectangle~$R:=[0,\rho n]\times[0,n]$
    \begin{equation}\label{eq:RSW_iso}    	
    	c\le \phi_{ \bbL(\pmb\alpha)}\big[\{0\} \times [0,n] \xlra{\, R\,\,} \{\rho n\} \times [0,n]\,\big|\,A\big]\le 1-c.\tag{RSW}
    \end{equation}
\end{theorem}

In the above $\phi_{ \bbL(\pmb\alpha)}$ denotes any infinite-volume FK-percolation measure on $\bbL(\pmb\alpha)$.
The statement is actually about finite volume measure with arbitrary boundary conditions; it is stated using the infinite-volume setting to avoid introducing additional notation.  

When discussing isoradial graphs, it is convenient to think of edges as closed segments in the plane~$\bbR^2 = \bbC$.
Then, (open) paths are piecewise linear paths running along (open) edges; they are not required to end at vertices of the graph.

\begin{remark}\label{rem:RSW_uniform_angles}
Notice that the bounds in~\eqref{eq:RSW_iso} are uniform in the angles of the lattice $\bbL(\pmb\alpha)$;
in particular, we have uniform RSW estimates for $\phi_{ \bbL(\alpha)}$ as $\alpha \to 0$. 
This fact requires additional arguments than RSW estimates which only hold for angles $\alpha$ bounded away from $0$ and $\pi$.
A full proof may be found in~\cite{DumLiMan18}.

We limit Theorem~\ref{thm:RSW_iso} to lattices $\bbL(\pmb\alpha)$ with at most two different angles for the horizontal tracks as this statement follows directly from~\cite{DumLiMan18} and suffices for our purposes. An RSW estimate holding uniformly for all rectangular square lattice $\bbL(\pmb\alpha)$ with no restriction on the angle sequence may also be derived from~\cite{DumLiMan18} but requires additional work. 
\end{remark}

\begin{center}
{\bf Henceforth, fix $q \in [1,4]$. \\
For isoradial graphs $G$, $\phi_G$ will denote an infinite-volume measure on $G$ with parameter $q$ and edge-weights given by \eqref{eq:isoraial_p_e}. } 
\end{center}
For all practical purposes, one may consider the infinite-volume measure to be unique. 

\paragraph{Unique infinite-volume measure.} 
It is a direct consequence of \eqref{eq:RSW_iso} that, for all $1 \leq q \leq 4$ and $\bbL(\pmb\alpha)$ an isoradial rectangular lattice as in the statement of Theorem~\ref{thm:RSW_iso}, there exists a unique infinite-volume FK-percolation measure $\phi_{\bbL(\pmb\alpha)}$.
Furthermore, under this measure, neither the primal nor the dual configurations contain infinite clusters. 

\paragraph{Mixing.}
A second consequence of~\eqref{eq:RSW_iso} that we will use repeatedly  is the mixing property.

\begin{proposition}[Mixing property]\label{prop:mixing}
For every~$\ep>0$, there exist~$C_{\rm mix},c_{\rm mix}\in(0,\infty)$ such that for any~$\pmb\alpha=(\alpha_i:i\in\bbZ) \in (0,\pi)^\bbZ$ containing at most two values, 
every~$r\le R/2$, every event~$A$ depending on edges in~$\Lambda_r$, and every event~$B$ depending on edges outside~$\Lambda_R$, we have that 
\begin{align}
	\big|\phi_{\bbL(\pmb\alpha)}[A\cap B]-\phi_{\bbL(\pmb\alpha)}[A]\phi_{\bbL(\pmb\alpha)}[B]\big|\le C_{\rm  mix}(r/R)^{c_{\rm mix}}\phi_{\bbL(\pmb\alpha)}[A]\phi_{\bbL(\pmb\alpha)}[B].
\end{align}
\end{proposition}

\begin{proof}
	The argument is identical to the one on the square lattice, see e.g.~\cite[Proposition~2.9]{DumMan20}.
\end{proof}

\paragraph{IIC measure.}

Fix two angles~$\alpha, \beta \in (0,\pi)$. 
Write~$\bbL_{\rm mix}$ for the lattice~$\bbL(\balpha)$ with~$\balpha = (\alpha_i)_{i\in\bbZ}$ with~$\alpha_i = \alpha$ for~$i$ even and~$\alpha_i = \beta$ for~$i$ odd.
The~\eqref{eq:RSW_iso} property applies to these lattices, and that allows us to define the half-plane IIC measure. This construction is tedious, but relatively standard, so we will only state the results here, and direct the reader to~\cite{Kes86,Oul22} or Exercise~\ref{exo:IIC} for details on the proofs. 

Recall that the horizontal tracks of~$\bbL_{\rm mix}$ are denoted~$(t_k)_{k\in \bbZ}$ and the vertical ones~$(s_k)_{k\in \bbZ}$.
Consider the origin of~$\bbR^2$ to be the vertex between~$t_0$ and~$t_1$ and between~$s_0$ and~$s_1$; we will assume it to be a primal vertex.
The cell~$(i,j)$ is the set of primal and dual vertices contained between~$s_{2i-1}$ and~$s_{2i+1}$ and between~$t_{2j-1}$ and~$t_{2j+1}$. 
We associate to the cell~$(i,j)$ its lower left lattice point (due to our definition, this is a primal point) --- see Figure~\ref{fig:IIC}.

Define the unit vectors\footnote{We choose to work here in the basis~$(e_{\rm vert},e_{\rm lat})$ as the two directions are well adapted to the lattice~$\bbL_\beta$ --- this is ultimately visible in Proposition~\ref{prop:drift_RT}. The proof would however also work in any basis containing~$e_{\rm vert}$, as explained in Exercise~\ref{exo:drift_u}.}   of~$\bbR^2$~$e_{\rm vert} = (0,1)$ and~$e_{\rm lat} = (\sin\beta,-\cos \beta)$.
For a finite cluster~$\sfC$ of a configuration~$\omega$ on~$\bbL_{\rm mix}$, 
consider the leftmost cell of maximal vertical coordinate intersected by~$\sfC$.
Write~${\rm Top}(\sfC)$ for its associated lattice point. 
Write~${\rm Bottom}(\sfC)$ for the lattice point of the rightmost cell of~$\bbL_{\rm mix}$ of minimal vertical coordinate intersected by~$\sfC$.
Consider the lattice points~${\rm Left}(\sfC)$ and~${\rm Right}(\sfC)$ corresponding to the cells intersected by~$\sfC$ of minimal, respectively maximal, scalar product with~$e_{\rm lat}$. 
When multiple choices are possible, let~${\rm Left}(\sfC)$ be the bottommost such point and~${\rm Right}(\sfC)$ be the topmost. 
Finally, write 
\begin{align}\label{eq:TBLR}
    {\rm T}(\sfC) &= \langle {\rm Top}(\sfC), e_{\rm vert} \rangle;&
    {\rm B}(\sfC) &= \langle {\rm Bottom}(\sfC), e_{\rm vert} \rangle;\nonumber\\
    {\rm L}(\sfC) &= \langle {\rm Left}(\sfC), e_{\rm lat} \rangle;&
    {\rm R}(\sfC) &= \langle {\rm Right}(\sfC), e_{\rm lat} \rangle.
\end{align}
We call these the extremal coordinates of~$\sfC$ with respect to~$(e_{\rm vert},e_{\rm lat})$.

\begin{figure}
\begin{center}
\includegraphics[width=0.65\textwidth]{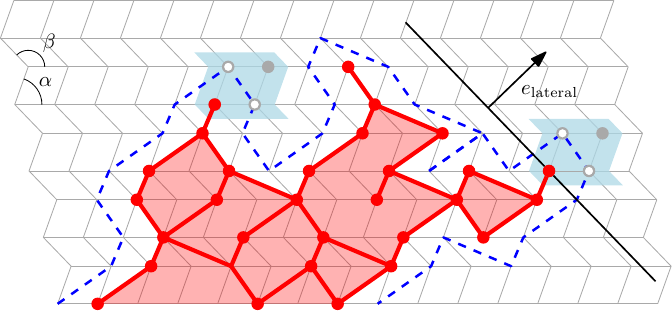}
\caption{A red cluster~$\sfC$ with the cells containing~${\rm Top}(\sfC)$ and~${\rm Right}(\sfC)$ marked in blue. 
Note that there are three topmost cells visited by~$\sfC$, the marked one is the leftmost.}
\label{fig:IIC}
\end{center}
\end{figure}

Write~$x_{n}$ for the primal vertex of~$\bbL_{\rm mix}$ between~$t_{2n}$ and~$t_{2n+1}$ and between~$s_0$ and~$s_1$. 
Similarly, write~$y_n$ for the primal vertex between~$t_{0}$ and~$t_{1}$ and between~$s_{2n}$ and~$s_{2n+1}$. 
Write~$\sfC_{x}$ for the cluster of the vertex~$x$ . 
The IIC measure is given by the following proposition. 

\begin{proposition}(IIC measure)\label{prop:IIC_def}
	The following limits exist and are called the half-plane IIC measures in the lower, upper, right and left half-planes, respectively: 
	\begin{align*}
	&\phi^{\rm IIC,T}_{\bbL_{\rm mix}}:= \lim_{n \to - \infty}	\phi_{\bbL_{\rm mix}} \big[\cdot\,\big|\, {\rm Top}(\sfC_{x_n}) = 0 \big] 
	,\qquad&&
	\phi^{\rm IIC,B}_{\bbL_{\rm mix}}:= \lim_{n \to \infty}	\phi_{\bbL_{\rm mix}} \big[\cdot\,\big|\, {\rm Bottom}(\sfC_{x_n}) = 0 \big],
	\\
	&	\phi^{\rm IIC,L}_{\bbL_{\rm mix}}:= \lim_{n \to \infty}	\phi_{\bbL_{\rm mix}} \big[\cdot\,\big|\, {\rm Left}(\sfC_{y_n}) = 0 \big] 
	\quad \text{ and } \quad&& 
	\phi^{\rm IIC,R}_{\bbL_{\rm mix}}:= \lim_{n \to -\infty}	\phi_{\bbL_{\rm mix}} \big[\cdot\,\big|\, {\rm Right}(\sfC_{y_n}) = 0 \big].
	\end{align*}
\end{proposition}

The measures above refer to the {\em local} environment around the extrema (in the sense of~\eqref{eq:TBLR}) of large clusters. 
Indeed, it may be shown that extrema of a typical large cluster are distributed according to~$\phi^{\rm IIC,*}_{\bbL_{\rm mix}}$, even when the cluster is conditioned on (reasonable) large scale features. 
This is a manifestation of a wider principle of mixing between the local and large scale features of the model and is a consequence of the~\eqref{eq:RSW_iso} property. 


\begin{lemma}\label{lem:IIC}
	There exist constants~$C,c_{\rm IIC}>0$ such that the following holds for all~$r \leq R/2$. 
	For any configuration~$\omega_0$ on~$\La_R^c$, 
	any union~$\calC$ of clusters of~$\omega_0$ with some~$x \in \calC$,
	and any event~$A$ depending on~$\Lambda_r$. , 
	\begin{align*}
        \big|\,	\phi^{\rm IIC,T}_{\bbL_{\rm mix}}[A]-\phi_{\bbL_{\rm mix}}\big[A\,\big|\, \omega = \omega_0\text{ on } \La_R^c,\,\sfC_x \cap \La_R^c = \calC,\, {\rm Top}(\sfC_x)=0\big]\,\big|\,&\le C(r/R)^{c_{\rm IIC}},
    \end{align*}
    as long as the conditioning is non-degenerate\footnote{The conditioning due on $\{\omega = \omega_0\text{ on } \La_R^c\}$ is always degenerate, but should be understood as imposing certain boundary conditions on the restriction of the measure to $\La_R$. Here, by non-degenerate we mean that there exists at least one configuration in $\La_R$ such that $\sfC_x \cap \La_R^c = \calC$ and ${\rm Top}(\sfC_x)=0$. In particular, all clusters of $\calC$ should touch $\La_R$.}. 
    Moreover, the same holds for all other directions~${\rm Bottom}(\cdot)$,~${\rm Left}(\cdot)$ and~${\rm Right}(\cdot)$, with the corresponding IIC measure. 
\end{lemma}

In the above, the constants~$C,c_{\rm IIC}>0$ only depend on the RSW property, and therefore are uniform over the choice of lattice~$\bbL_{\rm mix}$. As mentioned above, Lemma~\ref{lem:IIC} is a consequence of~\eqref{eq:RSW_iso} in the spirit of previous IIC constructions~\cite{Kes86,Jar03,BasSap17}; its proof is sketched in Exercise~\ref{exo:IIC}. Proposition~\ref{prop:IIC_def} follows from Lemma~\ref{lem:IIC}.

%

\paragraph{Arm exponents.}	
Arm exponents generally describe the speed of decay of certain connection probabilities in critical FK-percolation. 
The RSW property allows to obtain bounds for such exponents (but does not generally prove their existence), and compute some exactly. We state here the bounds necessary for our arguments. 

\begin{proposition}\label{prop:arms_bounds_iso}
	There exist constants~$C, \onearmbound > 0$ such that, 
	for any~$\pmb\alpha=(\alpha_i:i\in\bbZ) \in (0,\pi)^\bbZ$ containing at most two values
	and any~$r < R$
	\begin{align}
	\phi_{\bbL(\pmb\alpha)} [\La_r \longleftrightarrow \La_R^c] &\leq C (r/R)^{\onearmbound}, \label{eq:one_arm_bound_iso}\\
	\tfrac1C (r/R)^{2}\leq \phi_{\bbL(\pmb\alpha)} [\exists \sfC \text{ with }{\rm Top}(\sfC) \in \La_r \text{ and }\sfC \cap \La_R^c \neq \emptyset] &\leq C (r/R)^{2}. \label{eq:3hp_arm}
	\end{align}
	The second line is also valid for~${\rm Bottom}$ instead of~${\rm Top}$.
	It also holds that
	\begin{align}
	\tfrac1C (r/R)^{2}\leq \phi_{\bbL_{\rm mix}} [\exists \sfC \text{ with }{\rm Left}(\sfC) \in \La_r \text{ and }\sfC \cap \La_R^c \neq \emptyset] &\leq C (r/R)^{2} \qquad \text{and} \label{eq:3hp_arm2} \\
	\phi_{\bbL(\pmb\alpha)} [\exists \sfC \text{ with }{\rm Left}(\sfC) \in \La_r \text{ and }\sfC \cap \La_R^c \neq \emptyset] &\leq C (r/R)^{1+c}, \label{eq:3hp_arm3}
	\end{align}
	for some~$c > 0$. The same with~${\rm Right}$ instead of~${\rm Left}$. 
\end{proposition}

The events above are called the one arm- and the half-plane three-arm event, respectively, from~$\La_r$ to a distance~$R$. 
The direction of the half-plane depends on which extremum of the cluster we consider. 
	
We will not provide the proof of Proposition~\ref{prop:arms_bounds_iso}, but mention that it is a consequence of~\eqref{eq:RSW_iso} and certain symmetries of the lattice. 
The interested reader may consult Exercise~\ref{exo:algebraic_decay} for the bound~\eqref{eq:one_arm_bound_iso} on the one-arm probability and Exercises~\ref{exo:3_arms} and~\ref{exo:3_arms3} 
for the three-arm probabilities \eqref{eq:3hp_arm}, \eqref{eq:3hp_arm2} and \eqref{eq:3hp_arm3}. A full proof is available in~\cite{DumKozKra20}.

It is remarkable that the exponent of the half-plane three-arm probability in \eqref{eq:3hp_arm} and \eqref{eq:3hp_arm2} may be computed exactly; 
we say that the three-arm half-plane exponent is {\em universal}, in that it does not depend on~$q \in [1,4]$. 
Notice that~\eqref{eq:3hp_arm2} only holds for~$\bbL_{\rm mix}$, while for a generic lattices we only have the weaker bound \eqref{eq:3hp_arm3} at this stage. 
This is because the proofs of~\eqref{eq:3hp_arm} and~\eqref{eq:3hp_arm2} use the invariance of the lattice under translations orthogonal to the direction of the extremum. 
All lattices~$\bbL(\pmb\alpha)$ are invariant under horizontal translations, which allows us to deduce the exponent for~${\rm Top}$ and~${\rm Bottom}$, but not for~${\rm Left}$ or~${\rm Right}$. The lattices~$\bbL_{\rm mix}$ are invariant under two non-collinear translations, which allow to extend~\eqref{eq:3hp_arm} to the lateral extrema. 

The bound \eqref{eq:3hp_arm3} requires a more involved argument. While Theorem~\ref{thm:universalCNSS} ultimately implies that the exponent for \eqref{eq:3hp_arm3} is also equal to two, it is not possible to prove this at this stage. We mentioned this bound here, as it is necessary for the proof of Theorem~\ref{thm:linear}.

\subsection{Star-triangle and track-exchange transformations}\label{sec:star-triangle}

\paragraph{Star-triangle transformation.} 

The \emph{star-triangle transformation}, also known as the \emph{Yang--Baxter relation}
is a relation between statistical mechanics models indicative of the integrability of the system --- see~\cite{Bax89}.
We will use it here in the context of FK-percolation on isoradial graphs.

\begin{figure}
  \centering
  \includegraphics[width=0.6\textwidth]{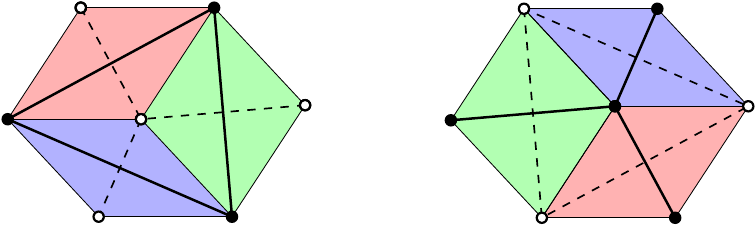}
  \caption{The three rhombi together with the drawing, on the left, of the triangle (in which case the dual graph is a star) and, on the right, of the star (in which case the dual graph is a triangle).}
  \label{fig:diamond}
\end{figure}

Fix some isoradial graph~$G$. A triangle~$ABC$ is a face of the graph bounded by three edges; 
a star~$ABCO$ is a vertex~$O$ of degree three, together with its neighbours~$A,B,C$.
In the diamond graph, the three edges forming a triangle or star correspond to three rhombi that form a hexagon --- see Figure~\ref{fig:diamond}.
For a given triangle~$ABC$ in~$G$, there exists a unique way to rearrange these three rhombi to form a star~$ABCO$; 
and vice-versa. 
This process of changing the graph as above is called the star-triangle transformation; it transforms isoradial graphs into other isoradial graphs.
Its crucial feature is that it preserves connection probabilities in the FK-percolation measures associated to the isoradial graphs.

\begin{lemma}\label{lem:stt}
	Let~$G$ be an isoradial graph and~$ABC$ be a triangle in~$G$. 
	Write~$G'$ for the graph obtained by replacing the triangle~$ABC$ with the corresponding star~$ABCO$.
	Then the connections between all vertices of~$G$ under~$\phi_G$ 
	and the corresponding vertices of~$G'$ under~$\phi_{G'}$ have the same law.
\end{lemma}

Note that the lemma above does not state anything about connections involving the vertex~$O$, 
which is natural as this vertex does not exist in~$G$.  

A convenient way to use Lemma~\ref{lem:stt} --- which also provides a proof --- is via an explicit coupling.
The idea to couple the configurations before and after applying a star-triangle transformation appeared already in~\cite{GriMan13,GriMan13a,GriMan14,DumLiMan18}.
That a coupling that preserves connections as described below exists is essentially equivalent to Lemma~\ref{lem:stt}; 
the novelty of~\cite{GriMan13,GriMan13a,GriMan14,DumLiMan18,DumKozKra20} and the present chapter 
is to follow the geometry of large clusters while applying a large number of such transformations. 

%

\begin{definition}[Star-triangle coupling] 
Consider an isoradial graph~$G$ containing a triangle~$ABC$ and let~$G'$ be the graph with the star~$ABCO$ instead. 
Introduce the {\em star-triangle coupling} between~$\omega\sim\phi_{G}$ and~$\omega'\sim\phi_{G'}$ defined as a random map~$\omega\mapsto \omega'$ as follows (see Figure~\ref{fig:simple_transformation_coupling}):
\begin{itemize}
\item For every edge~$e$ which does not belong to~$ABCO$,~$\omega'_e=\omega_e$,
\item If two or three of the  edges of~$ABC$ are open in~$\omega$, then all the edges in~$ABCO$ are open in~$\omega'$,
\item If exactly one of the edges of~$ABC$ is open in~$\omega$, say~$BC$, then the edges~$BO$ and~$OC$ are open in~$\omega'$, and the third edge of the star is closed in~$\omega'$,
\item If no edge of~$ABC$ is open in~$\omega$, then~$\omega'_{OABC}$ has 
\begin{itemize}
\item no  open edge with probability equal to~$q^2\tfrac{1-p_{OA}}{p_{OA}}\tfrac{1-p_{OB}}{p_{OB}}\tfrac{1-p_{OC}}{p_{OC}}$, 
\item the edge~$OA$ is open and the other two closed with probability~$q\tfrac{1-p_{OB}}{p_{OB}}\tfrac{1-p_{OC}}{p_{OC}}$, 
\item similarly with cyclic permutations for~$B$ and~$C$.
\end{itemize}
\end{itemize}
In the above,~$p_{e}$ is the parameter given by~\eqref{eq:isoraial_p_e} for edges~$e$ of~$G$ or~$G'$. The random choices in the last point are made independently of~$\omega$.
\end{definition}

\begin{figure}
  \centering
  \includegraphics[width=0.5\textwidth]{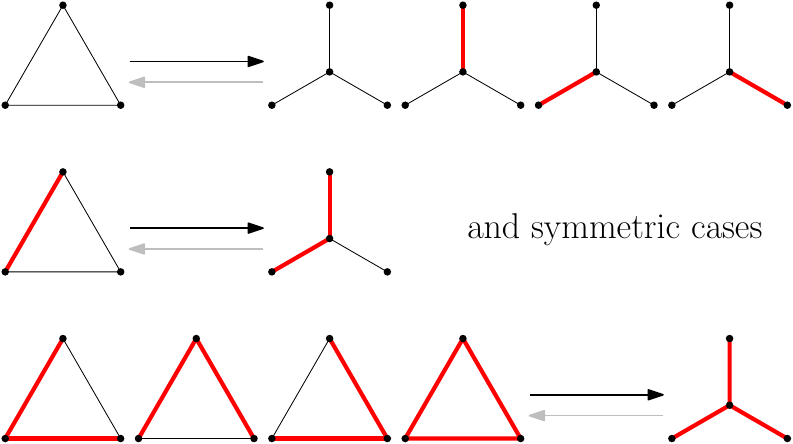}
  \caption{A picture of the transformations in the star-triangle coupling. In the first line, the outcome is chosen at random according to the probabilities listed in the definition. The reverse map (grey arrows) has a random outcome in the last line.}
  \label{fig:simple_transformation_coupling}
\end{figure}

Let us make a few observations concerning the coupling.
First, note that the transformation uses extra randomness in one case and that it is not a deterministic matching of the different configurations.
Second, the coupling preserves the connectivity between the vertices, except at the vertex~$O$. 
Third, in the coupling, given~$\omega'$, the edges of~$ABC$ in~$\omega$ are sampled as follows:
\begin{itemize}
\item If there is one or zero edge of~$ABCO$ that is open in~$\omega'$, then none of the edges in~$ABC$ is open in~$\omega$,
\item If exactly two of the edges in~$ABCO$ are open in~$\omega'$, say~$AO$ and~$BO$, then the edge~$AB$ is the only edge of~$ABC$ that is open in~$\omega$,
\item If all the edges of~$ABCO$ are open in~$\omega'$, then 
\begin{itemize}
\item all the edges of~$ABC$ are open in~$\omega$ with probability~$\tfrac1q\tfrac{p_{AB}}{1-p_{AB}}\tfrac{p_{BC}}{1-p_{BC}}\tfrac{p_{CA}}{1-p_{CA}}$, 
\item~$AB$ and~$BC$  are open and~$CA$ is closed with probability equal to~$\tfrac1q\tfrac{p_{AB}}{1-p_{AB}}\tfrac{p_{BC}}{1-p_{BC}}$, 
\item similarly with cyclic permutations.
\end{itemize}\end{itemize}

\paragraph{Track-exchange operator.} 
The previous star-triangle operator gives rise to a track-exchange operator defined as follows. 
For~$\bbL=\bbL(\pmb\alpha)$ and~$i\in\bbZ$, let~$\bbL'=\bbL(\pmb\alpha')$ be the lattice obtained by exchanging the tracks~$t_i$ and~$t_{i-1}$,
that is, exchanging~$\alpha_i$ and~$\alpha_{i-1}$ in the sequence~$\pmb\alpha$. 

\begin{fact}
	There exists a random map~$\bfT_i$ from percolation configurations on~$\bbL$ to percolation configurations on~$\bbL'$ such that 
	\begin{itemize}
	\item If~$\omega\sim \phi_{\bbL}$, then~$\bfT_i (\omega)\sim \phi_{\bbL'}$;
	\item~$\omega$ and~$\bfT_i(\omega)$ are equal for all edges outside of~$t_i$ and~$t_{i-1}$; 
	\item vertices outside of~$t_i \cap t_{i-1}$ are connected in the same way in~$\omega$ as in~$\bfT_i(\omega)$;
	\item the modification is local in the sense that there exists~$c > 0$ such that, when~$\omega \sim \phi_{\bbL}$, 
	$\bfT_i(\omega)(e)$ is determined\footnote{By that, we mean that one may produce a random sample~$\tilde\bfT_i(\omega)$ determined by 
	$\omega$ on~$\La_r(e)$ so that~$\tilde\bfT_i(\omega)(e) = \bfT_i(\omega)(e)$ with high probability. }	
	 by~$\omega$ on~$\La_r(e)$ with probability at least~$1-e^{-cr}$.
	\end{itemize}
\end{fact}

The exact nature of the coupling between~$\omega$ and~$\bfT_i(\omega)$ is not important. 
One such coupling may be obtained by compositions of the star-triangle transformation; we direct the reader to~\cite{DumKozKra20} for a full description. Other such couplings may also be constructed, see for instance~\cite{PelRom15}. 

We finish this section by a useful corollary of the track-exchange operation.

\begin{corollary}\label{cor:same_law}
	If~$\pmb\alpha$ and~$\pmb\beta$ are sequences of angles containing at most two angles each and satisfing~$\alpha_i=\beta_i$ for~$a\le i\le b$, 
	then the law of the boundary condition on the strip between~$t_a$ and~$t_b$ (including these tracks) induced by the configuration outside of this strip 
	is the same in~$\phi_{\bbL(\pmb\alpha)}$ and~$\phi_{\bbL(\pmb\beta)}$. 
\end{corollary}

A consequence of this proposition is that the laws of the configurations under~$\phi_{\bbL(\pmb\alpha)}$ and~$\phi_{\bbL(\pmb\beta)}$ restricted to the strip between~$t_a$ and~$t_b$ are the same. 

\begin{proof}
Recall that the track-exchange operator preserves the connection properties between vertices that are not between the tracks being exchanged. 
From this, one may deduce that for every event~$A$ concerning connections between vertices on the boundary of the strip, contained above and below the strip
\begin{align}\label{eq:same_law_proof}
    \phi_{\bbL(\pmb\alpha)}[A]
    =\lim_{R\rightarrow\infty}\phi_{\bbL(\pmb\alpha^{(R)})}[A]
    =\lim_{R\rightarrow\infty}\phi_{\bbL(\pmb\beta^{(R)})}[A]
    =\phi_{\bbL(\pmb\beta)}[A],
\end{align}
where 
\begin{align*}
    \pmb\alpha^{(R)}
    :=\begin{cases}
    	\alpha_i&\text{ if }|i|\le R,\\
    	\beta_{i-R+b}&\text{ if }i>R,\\
    	\beta_{i+R+a}&\text{ if }i<-R,\end{cases}\quad\text{and}\quad
    \pmb\beta^{(R)}
    :=\begin{cases}
        \beta_i&\text{ if }|i|\le R,\\
        \alpha_{i-R+b}&\text{ if }R<i\leq2R-b,\\
        \alpha_{i+R+a}&\text{ if }-2R-a\leq i<-R,\\
        \beta_i&\text{otherwise}.
    \end{cases}
\end{align*}
In the first  and last equalities of~\eqref{eq:same_law_proof}, we used the notion of measurability and the uniqueness of the infinite-volume measure. The second equality is due to the properties of the track-exchange operator,  
specifically the fact that the track exchanges involved in transforming~$\bbL(\pmb\alpha^{(R)})$ into~$\bbL(\pmb\beta^{(R)})$ 
do not affect the points on~$t_a$ and~$t_b$, and therefore do not change the connections between these points. 
\end{proof}

\section{Universality up to linear transformation}\label{sec:universality_isoradial_lin}

The goal of this section is to prove Theorem~\ref{thm:linear}. 
To start, let us sketch an approach to proving the more powerful Theorem~\ref{thm:universalCNSS}, 
so as to see exactly why  Theorem~\ref{thm:linear} is indeed more accessible than  Theorem~\ref{thm:universalCNSS}.

The idea behind proving Theorem~\ref{thm:universalCNSS} is to progressively transform~$\bbL(\beta)$ into~$\bbL(\alpha)$ using track exchanges, as in Figure~\ref{fig:transformations}. 
Throughout the transformation, we will follow the evolution of the configurations, specifically of their large clusters. 
As track exchanges do not break primal nor dual connections, large clusters survive the transformations, but their boundaries may be progressively altered. 
We would like to argue that these alterations act as i.i.d.\ modifications of the cluster boundaries with~$0$ expectation. 
As a result, the cluster boundaries move throughout the transformation, but only by~$o(N)$, and all large clusters remain essentially unchanged. 

There are three major difficulties in the above reasoning: 
encode the cluster boundary in a tractable way, prove that the modifications are i.i.d.\ (up to small errors) and that they are of~$0$ expectation. 
The latter is the most delicate point, as it essentially states that the isoradial embedding is the exact embedding that compensates for the inhomogeneity of the model. 
Ignoring this point, one expects the cluster boundaries to move according to a random walk with a (potentially non-zero) {\rm drift}. 
This leads us to Theorem~\ref{thm:linear}, and also allows us to explicitly identify the linear transformation~$M_{\beta,\alpha}$, which accounts for the drift. 
In the canonical basis of~$\bbR^2$, write 
\begin{align}\label{eq:Mbetaalpha}
M_{\beta,\alpha} 
= \begin{pmatrix}
1 &  \frac{{\rm Drift}_{\rm lat} +  {\rm Drift}_{\rm vert} \cos \beta }{\sin \beta}\cdot \frac{\sin\alpha + \sin\beta}{(\sin \alpha -{\rm Drift}_{\rm vert})\sin\beta}\\[8pt]
0 & 1  + {\rm Drift}_{\rm vert} \cdot \frac{\sin\alpha + \sin\beta}{(\sin \alpha -{\rm Drift}_{\rm vert})\sin\beta}
\end{pmatrix},
\end{align} 
where~${\rm Drift}_{\rm lat}$ and~${\rm Drift}_{\rm vert}$ are the expected increment produced by one transformation to the top- and right-extremum of a large cluster (a proper definition is given in~\eqref{eq:drift}; see also Remark~\ref{rem:M_origin} for the relation between the drift and~$M_{\beta,\alpha}$).

We mention the transformation~$M_{\beta,\alpha}$ here so as to emphasise that it does not depend on the different quantities chosen in the proof of Theorem~\ref{thm:linear}. 
Indeed, the drift is determined by the local environment around the extrema of an incipient infinite cluster (see Section~\ref{sec:IIC}). 
In particular, we will prove that it is continuous in~$\alpha$ and~$\beta$, which will be useful later on. 

\begin{lemma}\label{lem:M_cont}
	The map~$(\beta,\alpha) \mapsto M_{\beta,\alpha}$ is continuous and 
	$M_{\beta,\alpha}$ is invertible for all~$\alpha,\beta \in (0,\pi)$.
\end{lemma}

The proof of Lemma~\ref{lem:M_cont} is deferred to Section~\ref{sec:5equality_drifts}, when all the necessary notions will have been introduced. 
The rest of the section is dedicated to proving Theorem~\ref{thm:linear}. 
This is the most technical part of Chapter~\ref{ch:5rotational_inv}; the structure of the proof is described at the end of Section~\ref{sec:5sketch}, after having introduced the relevant notation.

\subsection{Setting of the proof of Theorem~\ref{thm:linear}}\label{sec:5sketch}

We will now fix some notation necessary for the proof of Theorem~\ref{thm:linear} and sketch some of its elements. 
Fix two angles~$\alpha$ and~$\beta$ and~$N \geq 1$. We will be interested in the ``large'' loops contained in the {\em observation window}~$[-N,N]\times [0,N]$ for configurations~$\omega$ and~$\omega'$ sampled according to~$\phi_{\bbL(\beta)}$ and~$\phi_{\bbL(\alpha)}$, respectively. 
More specifically, our goal is to construct a coupling between~$\omega$ and~$\omega'$ 
so that the loops of~$\omega$ of diameter larger than~$\eps N$ contained in~$[-N,N]\times [0,N]$ 
may be paired up with loops of~$\omega'$ which are close for the distance~\eqref{eq:loop_dist} and vice-versa --- here~$\eps > 0$ denotes some arbitrarily small quantity, ultimately chosen as a small negative power of~$N$. 

The exact choice of~$\eps$, as well as the proper definition of diameter of a cluster will be given below. For this section, however, we limit ourselves to a rough description. Hereafter, call clusters of diameter at least~$\eps N$ {\em macroscopic}. 

\paragraph{Lattices and track exchanges.}
Write~$\bbL_0 = \bbL(\balpha)$ for the sequence~$\balpha$ given by~$\alpha_i = \beta$ if~$i \leq K := \lceil N/\sin \beta \rceil$ and~$\alpha_i = \alpha$ otherwise. 

Define the lattices~$(\bbL_t)_{t\geq 0}$ inductively as follows. For~$t \geq 0$  set 
\begin{align}\label{eq:5S_transformation}
\bbL_{t+1} = \begin{cases}
	\bfT_1 \circ \bfT_3 \circ \bfT_5 \circ \dots (\bbL_t) \qquad &\text{ if~$t$ odd}\\
	\bfT_2 \circ \bfT_4 \circ \bfT_6 \circ \dots (\bbL_t) \qquad &\text{ if~$t$ even}.
\end{cases}
\end{align}
Write~$\bfS_t$ for the transformation applied when passing from~$\bbL_t$ to~$\bbL_{t+1}$. See Figure~\ref{fig:transformations} for an illustration. 

\begin{figure}
\begin{center}
\includegraphics[width = 0.33\textwidth, page = 1]{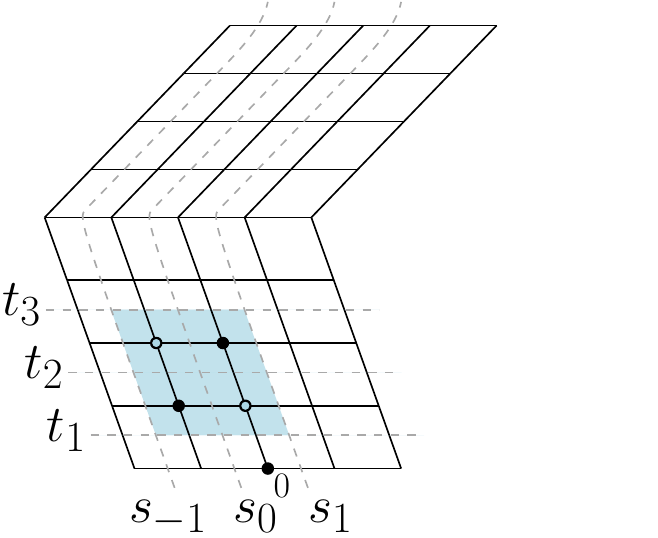}\hspace{-0.12\textwidth}
\includegraphics[width = 0.33\textwidth, page = 3]{transformations.pdf}\hspace{-0.14\textwidth}
\includegraphics[width = 0.33\textwidth, page = 6]{transformations.pdf}\hspace{-0.12\textwidth}
\includegraphics[width = 0.33\textwidth, page = 8]{transformations.pdf}
\caption{From left to right: The initial lattice~$\bbL_0$, with tracks of angle~$\beta$ at the bottom and~$\alpha$ at the top; the blue region is the cell~$(0,1)$ of the lattice. 
Applying~$\bfS_0 \circ \bfS_1$ transforms this~$\bbL_0$ into~$\bbL_2$; the mixed block is red.
After more transformations, a block of angle~$\alpha$ start appearing at the bottom (blue), which by time~$K+K'$ covers the whole window~$[-N,N]\times [0,N]$.}
\label{fig:transformations}
\end{center}
\end{figure}

It may appear as though we apply an infinite number of track exchanges when passing from~$\bbL_{t}$ to~$\bbL_{t+1}$, but this is not actually the case. 
Indeed, the transformations~$\bfT_i$ are non-trivial only then the tracks~$t_{i-1}$ and~$t_{i}$ have different transverse angles. 
When passing from~$\bbL_0$ to~$\bbL_1$, only the tracks~$t_{K+1}$ and~$t_{K}$ are exchanged. 

Generally, each lattice~$\bbL_t$ for~$t \leq K$ is formed of a{\em~$\beta$-block} of transverse angle~$\beta$; 
above which is a {\em mixed block} of alternating tracks of angles~$\alpha$ and~$\beta$, 
above which is an infinite block of tracks of angle~$\alpha$ called the{\em~$\alpha$-block}. 

For~$t > K$, the~$\beta$-block disappears and a block of tracks of angle~$\alpha$ starts forming above~$t_0$. Above this block, there is a mixed block with alternating tracks of angles~$\alpha$ and~$\beta$ (in total there are~$2K$ such tracks), above which we find again a block of tracks of angle~$\alpha$. 
A quick computation shows that the lower boundary of the mixed block is at height~$(K -t)\sin\beta$ for~$t \leq K$ and at height~$(t-K)\sin\alpha$ for~$t > K$. 

Set~$K'$ to be the smallest even integer larger than~$N/\sin \alpha$. Then, in~$\bbL_{K+K'}$, the window~$[-N,N]\times [0,N]$ is fully contained in the bottom~$\alpha$-block of the lattice. 

\paragraph{The configuration chain.}
To the lattices~$(\bbL_t)_{t\geq 0}$, we will associate a Markov chain of configurations~$(\omega_t)_{t\geq0}$ with the property that 
\begin{align*}
	\omega_t \sim \phi_{\bbL_t}\qquad  \text{ for all~$t \geq 0$}. 
\end{align*}
As such, we obtain a {\em coupling} of the laws~$(\phi_{\bbL_t})_{t\geq 0}$. In this coupling, the loops of diameter at least~$\eps N$ of~$\omega_0$ may be paired up with loops of~$\omega_{K+K'}$ that are close in the sense of~\eqref{eq:loop_dist}. 
Corollary~\ref{cor:same_law} states that~$\omega_0 \cap ([-N,N]\times [0,N])$ has the law of a configuration sampled according to~$\phi_{\bbL(\beta)}$, while~$\omega_{K+K'}\cap ([-N,N]\times [0,N])$ has the law~$\phi_{\bbL(\alpha)}$.

The actual Markov chain~$(\omega_t)_{t\geq0}$ is a little complicated to define (see Section~\ref{sec:5MarkovChain}), but the following construction offers a good understanding of it. 
Sample~$\omega_0$ according to~$\phi_{\bbL_0}$. 
Since the track exchanges appearing in~$\bfS_t$ act on disjoint tracks, they may be performed simultaneously to~$\omega_t$
and we write 
\begin{align}\label{eq:omega_t_simple}
	\omega_{t+1} = \bfS_t(\omega_t) \qquad \text{ for all~$t\geq 0$}.
\end{align}

\paragraph{Dynamics of macroscopic clusters.}
Notice that for~$t$ even,~$\bfS_t$ does not affect vertices of~$t_j\cap t_{j-1}$ with~$j$ odd. 
For~$t$ odd, it does not affect the vertices in~$t_j\cap t_{j-1}$ with~$j$ even. 
In particular, the connections between any such vertices are preserved when applying~$\bfS_t$. 
As a consequence, any primal or dual cluster of diameter at least~$2$ of~$\omega_t$ has a naturally associated cluster in~$\omega_{t+1}$. 

Clusters may still disappear throughout the process~$(\omega_t)_{t\geq 0}$, but they do so by progressively shrinking to a single point, then disappearing when that point is erased from the lattice. Conversely, new small clusters may appear and potentially grow in size. We will argue below that new clusters do not grow to size~$\eps N$ by time~$K+K'$, nor do initial macroscopic clusters disappear by this time (except on events of vanishing probability). 
This will allow us to follow the evolution of all macroscopic clusters of~$\omega_0$, and pair them to macroscopic clusters of the final configuration~$\omega_{K+K'}$. 

In addition, we would like to argue that the contours of macroscopic clusters are affected throughout the process by i.i.d.\ increments, with a potential drift. Making sense of this statement is challenging as the contour of a cluster is a complicated object. To simplify things, we will only follow the extremal coordinates (see \eqref{eq:TBLR}) of {\em mesoscopic clusters} --- by mesoscopic clusters, we mean clusters of diameter at least~$\eta N$ for some~$0 <\eta \ll \eps$. 
We will eventually prove that controlling the evolution of the extremal coordinates of mesoscopic clusters suffices to control the shape of macroscopic clusters in the sense of~\eqref{eq:loop_dist} (see Section~\ref{sec:5homotopy}). 

Finally, we would like to argue that each transformation affects the extremal coordinates of mesoscopic clusters in (almost) i.i.d.\ ways. 
As will be apparent below, the increments of the extremal coordinates of the mesoscopic clusters of~$\omega_t$ when applying~$\bfS_t$ depend on the local environments in~$\omega_t$ around these points. 
These local environments will be shown to have a local law independent of the large scale features of~$\omega_t$, in the spirit of the IIC construction of Section~\ref{sec:IIC}. This will guarantee that the increments are identical in law (up to errors vanishing as~$N\to\infty$) and that they have expectations ${\rm Drift}_{\rm vert}$ and ${\rm Drift}_{\rm lat}$ given by the IIC measures. 

That the increments are i.i.d.\ in time is more challenging. While we expect some approximate form of independence to hold in the chain~$(\omega_t)_{t\geq0}$ defined above, we are unable to prove it. To circumvent this problem, we will perform a {\em resampling} when passing from~$\omega_t$ to~$\omega_{t+1}$, which will ensure that the large scale information of~$\omega_t$ is preserved, but the local environments around the extrema of mesoscopic clusters are resampled.

\paragraph{Structure of the proof.}
In Section~\ref{sec:5MarkovChain} we define the notion of mesoscopic cluster and the actual configuration chain~$(\omega_t)_{t\geq0}$. This includes an explicit procedure for the resampling of the configuration around the extrema of mesoscopic clusters.
 
The goal of Sections~\ref{sec:control_error} and~\ref{sec:stability_meso} is to prove the stability of mesoscopic clusters --- Proposition~\ref{prop:stability_meso} ---  which states that, at first order, 
the extremal coordinates of mesoscopic clusters move during the process~$(\omega_t)_{t\geq0}$ as dictated by the drift. 

Section~\ref{sec:5homotopy} explains how the knowledge of the extremal coordinates of mesoscopic clusters allows to determine the shape of macroscopic clusters.
The key here is to consider the homotopy classes of contours of macroscopic clusters in the punctured plane obtained by removing the mesoscopic clusters. It will be a consequence of our definition of~$(\omega_t)_{t\geq0}$ and of the stability of mesoscopic clusters that these homotopy classes are preserved throughout the process. 

Finally, in Section~\ref{sec:5linear_proof}, we put together the elements of the previous sections in order to prove Theorem~\ref{thm:linear}. 

Section~\ref{sec:5equality_drifts} contains certain additional properties of the matrix~$M_{\beta,\alpha}$, namely the proof of Lemma~\ref{lem:M_cont} 
and relation between the vertical and lateral drift, which will be crucial when proving Theorem~\ref{thm:universalCNSS}.

\subsection{Actual coupling: resampling of extrema}\label{sec:5MarkovChain}
 
In this section, we define the actual Markov chain~$(\omega_t)_{t\geq 0}$. Write~$\bbP$ for the probability measure governing this chain. 
For simplicity, assume that~$K$ and~$K'$ are even. 

Fix some~$0<\eta < 1$; the choice of~$\eta$ will be explained in Section~\ref{sec:5linear_proof}. 
We will use the term ``universal constant'' to mean a positive constant independent of~$N$ and~$\eta$. 

A cluster~$\sfC$ in some configuration~$\omega$ on some lattice~$\bbL_t$ is called~{\em~$\eta$-mesoscopic} (or simply mesoscopic) 
if 
\begin{align}\label{eq:meso}
	&\eta  N \leq {\rm T}(\sfC)-{\rm B}(\sfC) \leq \eta^{1/2} N \qquad \text{and}\nonumber\\
	&{\rm T}(\sfC) \leq  N, \qquad 	{\rm L}(\sfC) \geq - N, \qquad {\rm R}(\sfC) \leq  N.
\end{align}
we call~${\rm T}(\sfC)-{\rm B}(\sfC)$ the {\em vertical diameter} of~$\sfC$. 
Write~${\rm Meso}(\omega)$ for the set of mesoscopic clusters of~$\omega$ and 
\begin{align}\label{eq:calH_def}
	\calH(\omega) = \big({\rm T}(\sfC), {\rm B}(\sfC), {\rm L}(\sfC), {\rm R}(\sfC)\big)_{\sfC \in {\rm Meso}(\omega)}
\end{align}
for the extremal coordinates of all mesoscopic clusters. 

Sample~$\omega_0 \sim \phi_{\bbL_0}$. For~$t \geq0$ even, assuming~$\omega_0,\dots, \omega_t$ already defined, 
sample a configuration~$\omega_{t+1/2}$ on~$\bbL_{t}$ as follows. 
The goal here is to have~$\omega_{t+1/2} \sim \phi_{\bbL_t}$.

Recall that~$\bbL_{\rm mix}$ denotes the lattice~$\bbL(\balpha)$ with~$\balpha = (\alpha_i)_{i\in\bbZ}$, where~$\alpha_i = \alpha$ for~$i$ even and~$\alpha_i = \beta$ for~$i$ odd.
First, sample i.i.d.\ configurations~$(\zeta_j^{{\rm Top}})_{j\in \bbN}$,~$(\zeta_j^{{\rm Bottom}})_{j\in \bbN}$, 
$(\zeta_j^{{\rm Left}})_{j\in \bbN}$ and~$(\zeta_j^{{\rm Right}})_{j\in \bbN}$
 according to the  lower, upper, right and left half-plane IIC measures~$\phi^{\rm IIC,T}_{\bbL_{\rm mix}}$, $\phi^{\rm IIC,B}_{\bbL_{\rm mix}}$, $\phi^{\rm IIC,L}_{\bbL_{\rm mix}}$ and~$\phi^{\rm IIC,R}_{\bbL_{\rm mix}}$, respectively (see Proposition~\ref{prop:IIC_def}). 

Fix~$R = \eta^C N$, where~$C > 0$ is a sufficiently large universal constant --- we will assume that~$(C -1)\onearmbound -2 \geq 1$, where~$\onearmbound > 0$ is given by~\eqref{eq:one_arm_bound_iso}. Define the extremum boxes of~$\omega_t$ as~$\La_{ R } (x)$ for 
$x \in \{{\rm Top}(\sfC),\,{\rm Bottom}(\sfC),\,{\rm Left}(\sfC),\,{\rm Right}(\sfC)\,:\, \sfC \in {\rm Meso}(\omega_t)\}$.
We say that an {\em overlap error} occurs at step~$t$ if any of the following occurs: 
\begin{itemize}
	 \item two extremum boxes intersect;
	 \item two distinct clusters of vertical diameter at least~$\eta N/2 - 2R$ intersect the same extremum box; 
	 \item there exists an extremum box that intersects two distinct blocks of the lattice (i.e.,~$\alpha$-block,~$\beta$-block or mixed block).
\end{itemize}
If an overlap error occurs, set~$\omega_{t+1/2} = \omega_t$.

If no overlap error occurs, set~$\omega_{t+1/2} = \omega_t$ everywhere except in the extremum boxes contained in the mixed block. 
We will now describe how to sample~$\omega_{t+1/2}$ in each of these boxes. 
We will do this in such a way that the mesoscopic clusters of~$\omega_t$ and~$\omega_{t+1/2}$ are identical outside of the extremum boxes
and have the same extrema in~$\omega_t$ and~$\omega_{t+1/2}$. 

Enumerate the boxes in an arbitrary way and sample them sequentially. 
Suppose that the configuration~$\omega_{t+1/2}$ in the boxes numbered~$1\dots, i-1$ has been sampled and let us describe how to sample it in the~$i^{\rm th}$ box. 
Assume that this box is centred at a point~$x = {\rm Top}(\sfC)$ for some~$\sfC \in {\rm Meso}(\omega_t)$; 
the same procedure applies if~$x$ is of the type~${\rm Bottom}(\sfC)$,~${\rm Left}(\sfC)$ or~${\rm Right}(\sfC)$.

To start, let us describe the law~$\calL$ of the configuration in~$\La_{ R } (x)$ under~$\phi_{\bbL_t}$, 
conditionally on the previously sampled edges,
on the fact that the mesoscopic clusters of~$\omega$ and~$\omega_{t}$ are identical outside of the extremum boxes
and on their extrema being the same in~$\omega$ and~$\omega_{t}$.
It is~$\phi_{\La_R(x)}^\xi$, where~$\xi$ are the  (random) boundary conditions given by the configuration outside\footnote{the boundary conditions are random, as other un-sampled boxes may influence the boundary conditions induced by the outside configuration on~$\La_{ R } (x)$.}, 
conditioned that all pieces of~$\sfC$ in~$\omega_t\setminus \La_{ R } (x)$ are connected in~$\omega$ inside~$\La_{ R } (x)$, 
but not to any other points of~$\La_{ R } (x)^c$,
and that~${\rm Top}(\sfC) = x$.
Indeed, this conditioning ensures that~$\sfC \cap \La_R(x)^c$ remains the restriction to~$\La_R(x)^c$ of a mesoscopic cluster of~$\omega$ 
with same extremal coordinates as in~$\omega_t$. 
Moreover, the resampling in~$\La_{ R } (x)$ may not produce additional mesoscopic clusters, 
since~$\sfC$ is the only cluster of vertical diameter at least~$\eta N/2 - 2R$ intersecting~$\La_{ R } (x)$.

We sample~$\omega_{t+1/2}$ in~$\La_{ R } (x)$ according to~$\calL$, as follows. 
Let~$r = \eta^C R$ for some large fixed constant~$C > 0$ --- formally, we need~$C \geq 3/c_{\rm IIC}$, where~$c_{\rm IIC} >0$ is the constant given by Lemma~\ref{lem:IIC}.
Write~$d_{\rm TV}$ for the total variation distance between the restriction of~$\calL$ to~$\La_{ r }(x)$ and the law of~$\zeta_i^{{\rm Top}}$ inside~$\La_{ r }$, translated by~$x$.

Sample a Bernoulli variable~${\rm err}_i$ of parameter~$d_{\rm TV}$
and sample~$\omega_{t+1/2}$ inside~$\La_{ R }(x)$ according to~$\calL$
so that,  
\begin{align}
{\rm err}_i = 0 \quad\Longrightarrow \quad \omega_{t+1/2} = \zeta_i^{{\rm Top}} + x \text{ on~$\La_r(x)$},
\end{align}
where~$\zeta_i^{{\rm Top}} + x$ is the translate of~$\zeta_i^{{\rm Top}}$ by~$x$.
This is indeed possible due to the definition of~$d_{\rm TV}$. 
We say that a {\em coupling error} occurs at step~$t$ if there exists~$i$ with~${\rm err}_i = 1$.

This concludes the construction of~$\omega_{t+1/2}$. Notice that~$\omega_{t+1/2}$ has the same law as~$\omega_t$, namely~$\phi_{\bbL_t}$. 
Set
\begin{align*}
	\omega_{t+1} = \bfS_t(\omega_{t+1/2}) \qquad \text{ and  } \qquad \omega_{t+2} = \bfS_{t+1}(\omega_{t+1}).
\end{align*}

We have thus constructed the Markov chain $(\omega_t)_{t\geq 0}$ and have ensured that $\omega_t \sim \phi_{\bbL_t}$ for all $t$. 
Keep in mind the fact that the resampling is done only at even time steps. 

\subsection{Controlling the errors}\label{sec:control_error}

For~$t \geq 0$ even and~$\sfC$ a mesoscopic cluster of~$\omega_t$, 
write~$\tilde\sfC$ for the corresponding cluster in~$\omega_{t+2}$; note that~$\tilde\sfC$ is a well-defined object, but may be non-mesoscopic. 
We say~${\rm Ext}(\sfC)$ for some~${\rm Ext} \in \{ {\rm Top}, {\rm Bottom}, {\rm Left}, {\rm Right}\}$ is  {\em frozen} if it is at a distance at least two from the mixed block, otherwise we call it {\em active}.
That~${\rm Top}(\sfC)$ and~${\rm Bottom}(\sfC)$ are frozen is determined by~${\rm T}(\sfC)$ and~${\rm B}(\sfC)$, respectively, and hence may be read off~$\calH(\omega_t)$. The same is not true for ${\rm Left}(\sfC)$ and ${\rm Right}(\sfC)$. 

Define the increments of the extrema of mesoscopic clusters as
\begin{align*}
	\Delta_t {A}(\sfC)& =  {A}(\tilde \sfC) - {A}(\sfC)\qquad \text{ for~$A \in \{{\rm T,B,L,R}\}$}.
\end{align*}
If~${\rm Top}(\sfC)$ is frozen, then~$\Delta_t {\rm T}(\sfC) = 0$. The same holds for~${\rm Bottom}(\sfC)$. 
For the left and right extrema, this is not the case as, even when~${\rm Left}(\sfC)$ or~${\rm Right}(\sfC)$ are frozen,
a point which is almost extremal in the direction~$e_{\rm lat}$ and which is not frozen may move and become~${\rm Left}(\tilde\sfC)$ or~${\rm Right}(\tilde\sfC)$, respectively. 
However, the increment of all extrema is zero if both~${\rm Top}(\sfC)$ and~${\rm Bottom}(\sfC)$ are frozen and on the same side of the mixed block.

%
Define the same quantities for the IIC. 
More precisely, if~$\zeta_i^{{\rm Top}}$ is the IIC configuration used in the sampling of~$\omega_{t+1/2}$ around the top of~$\sfC$, 
and if we write~$\sfC^{\rm IIC}$ for the IIC cluster of~$\zeta_i^{{\rm Top}}$ and~$\tilde\sfC^{\rm IIC}$ for the corresponding IIC cluster of 
$(\bfS_{t+1} \circ \bfS_{t})(\zeta_i^{{\rm Top}})$, set 
\begin{align}
	\Delta_t^{\rm IIC} {\rm T}(\sfC) =  \langle {\rm Top}(\tilde \sfC^{\rm IIC}) - {\rm Top}(\sfC^{\rm IIC}), e_{\rm vert}\rangle.
\end{align}
When applying~$\bfS_{t}$ to~$\bbL_{\rm mix}$, each track of angle~$\beta$ is exchanged with the track of angle~$\alpha$ above it. 
As such, the vertices at the bottom of the  tracks of angle~$\beta$ in~$\bbL_{\rm mix}$ remain unchanged, and we can make sense of the image of any cluster containing at least one such point. 
In particular, there exists a well-defined notion of image of~$\sfC^{\rm IIC}$ after application of~$\bfS_t$. 
The same holds for~$\bfS_{t+1}$. Notice that~$\bfS_t(\bbL_{\rm mix})$ is simply a translate of~$\bbL_{\rm mix}$ and that the effect of~$\bfS_{t+1}$ on~$\bfS_t(\bbL_{\rm mix})$ 
is to again exchange each track of angle~$\beta$ with the track above it. 
As such, it would be tempting to think that the effect of the two transformations on the top of the IIC is identical: that is not the case, as the cells of~$\bbL_{\rm mix}$ are distinct from those of~$\bfS_t(\bbL_{\rm mix})$, which affects the definition of~${\rm Top}(\cdot)$. 
It is, however, the case that~$\bbL_{\rm mix}$ and~$(\bfS_{t+1} \circ \bfS_t)(\bbL_{\rm mix})$ are identical (up to translation), including in their partition in cells.

If no IIC configuration was used in the sampling of~${\rm Top}(\sfC)$ (for instance, because the top was frozen or because an overlap error occurred), 
then define~$\Delta_t^{\rm IIC} {\rm T}(\sfC)$ using some arbitrary~$\zeta_i^{{\rm Top}}$, different from the ones used for other clusters. 
Define~$\Delta_t^{\rm IIC} {\rm B}(\sfC)$,~$\Delta_t^{\rm IIC} {\rm L}(\sfC)$ and~$\Delta_t^{\rm IIC} {\rm R}(\sfC)$ analogously. 

If~${\rm Top}(\sfC)$ is active, set 
\begin{align}
	\Delta_t^{\rm err} {\rm T}(\sfC) =  \Delta_t {\rm T}(\sfC)  - \Delta_t^{\rm IIC} {\rm T}(\sfC);
\end{align}
if it is frozen, set~$\Delta_t^{\rm err} {\rm T}(\sfC) =0$. Define~$\Delta_t^{\rm err} {\rm B}(\sfC)$,~$\Delta_t^{\rm err} {\rm L}(\sfC)$ and~$\Delta_t^{\rm err} {\rm R}(\sfC)$ in the same way. 

Throughout this process, the variables~$\Delta_t^{\rm IIC} {A}(\cdot)$ are i.i.d., and therefore their sums behaves like independent random walks. 
The error part may affect this dynamics, but will be proved to be small. We do so by a crude~$L^1$-bound.

\begin{proposition}\label{prop:error_exists}
	There exist universal constants~$c,C >0$ such that, for all~$\eta \geq N^{-c}$
	\begin{align}\label{eq:error_exists}
		\bbP\big[\exists \sfC \in {\rm Meso}(\omega_t) \text{ and } A \in \{{\rm T,B,L,R}\}\,: \, \Delta_t^{\rm err} {A}(\sfC) \neq 0 \big] \leq C \eta.
	\end{align}
\end{proposition}

The rest of this section is dedicated to the proof of Proposition~\ref{prop:error_exists}. 
To start, we state a bound on the number of mesoscopic clusters of any configuration~$\omega_t$.

\begin{proposition}\label{prop:mesoscopic_number}
	There exists~$C > 0$ such that, for all~$N$ and~$0<\eta < 1$ and all~$t \geq0$, 
	\begin{align}\label{eq:mesoscopic_number}
		\phi_{\bbL_t}\big[\omega \text{ has more than~$\tfrac{\la}{\eta^{2}}$~$\eta$-mesoscopic clusters}\big] 
		\leq C e^{-\la/C} \qquad \text{ for all~$\la \geq 1$}
	\end{align}
\end{proposition}

This is a direct consequence of the RSW theory and will come to no surprise to the experts.

\begin{proof}[Proof of Proposition~\ref{prop:mesoscopic_number}]
	Fix~$N$,~$\eta$ and~$t$. 
	Write~${\bf N}_{\rm meso}$ for the number of~$\eta$-mesoscopic clusters. 
	Notice that any such cluster needs to cross an annulus~$\La_{\eta N/4}(x) \setminus \La_{\eta N/8}(x)$ for some~$x \in \frac{\eta N}8\bbZ^2 \cap [-N,N]\times[0,N]~$ from the inside to the outside. For any such~$x$, write~${\bf N}_x$ for the number of such crossing clusters. 
	Then we have 
	\begin{align}\label{eq:N_macro_sum}
		{\bf N}_{\rm meso} \leq \sum_x {\bf N}_x,
	\end{align}
	where the sum is over all~$x$ as above. There are~$O(\eta^{-2})$ terms in the sum. 
	
	It is a standard consequence of~\eqref{eq:RSW_iso} that the variables~${\bf N}_x$ have exponential tails. 
	Moreover, since~\eqref{eq:RSW_iso} is uniform in boundary conditions, the variables~${\bf N}_x$ may even be bounded by variables~$\tilde {\bf N}_x$ with exponential tails and which are independent for points at mutual distance larger than~$\eta N$. 
	The conclusion follows by splitting the sum in~\eqref{eq:N_macro_sum} into a bounded number of sums, each of which is bounded by i.i.d.\ variables~$\tilde{\bf  N}_x$ with exponential tails. 
\end{proof}

We now turn to the proof of Proposition~\ref{prop:error_exists} which is also based on the RSW theory and on the IIC-mixing estimate of Lemma~\ref{lem:IIC}.
There are no major difficulties here, but the proof is tedious due to the multiple sources of error possible.

\begin{proof}[Proof of Proposition~\ref{prop:error_exists}]
	To have~$\Delta_t^{\rm err} {A}(\sfC) \neq 0$ for some~$\sfC \in {\rm Meso}(\omega_t)$ and~$A \in \{{\rm T,B,L,R}\}$, one of the following must occur: 
	\begin{itemize}
	\item[(1)] an overlap error occurs at step~$t$,
	\item[(2)] a coupling error occurs at step~$t$ or 
	\item[(3)] we have~$\Delta_t {\rm T}(\sfC) \neq \Delta_t^{\rm IIC} {\rm T}(\sfC)$ for some~$\sfC \in {\rm Meso}(\omega_t)$, even though the environment in~$\La_r({\rm Top}(\sfC))$ is given by the corresponding IIC configuration, or the same for~${\rm B}, {\rm L}$ or~${\rm R}$ instead of~${\rm T}$. 
	\end{itemize}
	We analyse each scenario separately. \smallskip 
	
	\noindent{\bf (1)}
	A relatively straightforward analysis based on~Proposition~\ref{prop:arms_bounds_iso} shows that 
	\begin{align}\label{eq:overlap_err}
		\bbP[\text{overlap error occurs at step~$t$}] \leq C \eta, 
	\end{align}
	for some universal constant~$C >0$. 
	Indeed, pave our observation window by overlapping boxes~$\{\La_{4R}(z): z \in R\bbZ^2 \cap [-N,N] \times [0,N]\}$. 
	For an overlap error to occur, at least one of the following needs to occur in $\omega_t$:
	\begin{itemize}
	\item there exists a box~$\La_{4R}(z)$ with three arms in a half-plane and at least one an additional primal arm to a distance~$\eta N/2 - 8R$, or 
	\item there exists a box~$\La_{4R}(z)$ that intersects two distinct blocks of the lattice and has three arms in a half-plane to a distance~$\eta N - 4R$.
	\end{itemize} 
	In the above, by ``three-arm in a half-plane'' we mean one of the events in \eqref{eq:3hp_arm} or \eqref{eq:3hp_arm2}.
	There are~$O(N/R)^2$ potential boxes in the first category and~$O(N/R)$ in the second. 
	By a union bound and Proposition~\ref{prop:arms_bounds_iso}, we conclude\footnote{Bounding the probability of the events comprising three arms in the upper or lower half-plane is a straightforward application of~\eqref{eq:3hp_arm}. For ``lateral'' arm events, a more careful approach is needed which  combines~\eqref{eq:3hp_arm2} and~\eqref{eq:3hp_arm3}. A detailed analysis available in~\cite{DumKozKra20}.}
	that~\eqref{eq:overlap_err} holds for appropriately chosen constants.\smallskip 
	
	\noindent{\bf (2)} Lemma~\ref{lem:IIC} implies that
\begin{align}
	\bbP[{\rm err}_i = 1] = d_{\rm TV} \leq C\big(\tfrac{r}{R}\big)^{c_{\rm IIC}} \leq C \eta^{3},
\end{align}
where the last inequality is due to the choice of~$r$. 
Combining the above with Proposition~\ref{prop:mesoscopic_number}, we conclude that 
\begin{align}\label{eq:coupling_err}
	\bbP[\text{coupling error occurs at step~$t$}] \leq C \eta
\end{align}
for some universal constant~$C > 0$.\smallskip

	\noindent{\bf (3)}
	This type of error may appear for two reasons (we describe them in the context of the top of clusters): 
	\begin{itemize}
	\item[(3.a)] the configuration in~$\La_{r/2}({\rm Top}(\sfC))$ after applying~$\bfS_{t+1}\circ \bfS_t$ is not identical to the one given by the IIC or 
	\item[(3.b)] the top of the cluster~$\sfC$ or that of the corresponding IIC cluster shifts by more than~$r/2$ to the left or to the right when applying~$\bfS_{t+1}\circ \bfS_t$, thus exiting~$\La_{r/2}({\rm Top}(\sfC))$.
	\end{itemize}
	Scenario (3.a) is very unlikely to occur. Indeed, for any given cluster~$\sfC$, scenario (3.a) has a probability bounded above by~$C e^{-c r}$ for universal constants~$C,c$ due to the way that the start-triangle transformations combine to form track exchanges. More generally, it is implied by the fact that the track exchanges are {\em local} transformations. Performing a union bound over the mesoscopic clusters of~$\omega_t$ and using Proposition~\ref{prop:mesoscopic_number}, we conclude that the probability of (3.a) is bounded above by~$C \eta^{-2}e^{-c r}$ for potentially modified universal constants~$C,c$.
	
	For scenario (3.b) to occur, an {\em almost-top} of some mesoscopic cluster~$\sfC$ needs to exist at a distance at least~$r/2$ from~${\rm Top}(\sfC)$ in~$\omega_t$, or the equivalent for the IIC clusters --- see Figure~\ref{fig:controlling_err}. 
	By a union bound on the possible positions of~${\rm Top}(\sfC)$ and that of the almost-top, using~\eqref{eq:3hp_arm}, 
	we find\footnote{
	This is not entirely true:~\eqref{eq:3hp_arm} applies for vertical extrema, but for lateral extrema a more delicate analysis which combines~\eqref{eq:3hp_arm2} and~\eqref{eq:3hp_arm3} is necessary. A detailed proof is available in~\cite{DumKozKra20}.}
	that the probability of such an event is bounded above by a multiple of
	$$ \eta^{-2}  \sum_{k\geq r}k^{-2} = \eta^{-2} r^{-1}.$$
	Recall that~$r$ was fixed to be equal to~$\eta^C N$ for some fixed constant~$C$.
	Thus, the probability above is bounded by~$\eta^{-2-C}/N$. 
	
	\begin{figure}
	\begin{center}
	\includegraphics[width = .8\textwidth]{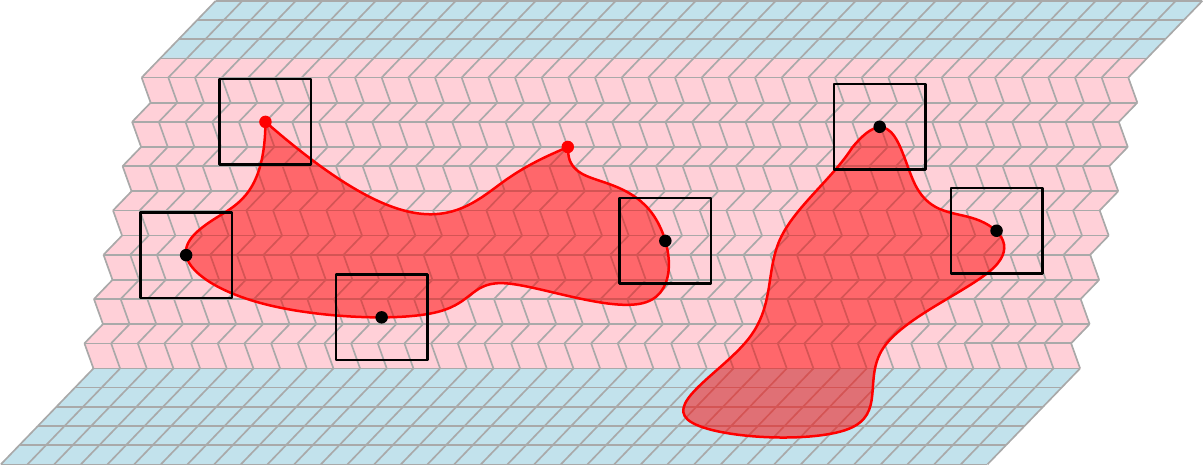}
	\caption{A step in the transformations~$\bbL_t$. The mixed block is pink, the two other ones are blue; 
	notice that the bottom block is of angle~$\alpha$, which indicates that~$t > K$.
	The two red clusters are mesoscopic; the extrema are marked, as are the extremum boxes. 
	All extrema of the left cluster are active, for the right one, only the top and right are active, the other are frozen. 
	\newline
	Notice the potential error for the move of the top of the left cluster: there exists a second point (marked in red) in the cluster
	that may move up and overtake the top.}
	\label{fig:controlling_err}
	\end{center}
	\end{figure}

	Summing the different potential sources of error we find that 
	\begin{align*}
		\bbP\big[\exists \sfC \in {\rm Meso}(\omega_t): \, \Delta_t^{\rm err} {\rm T}(\sfC) \neq 0 \text{ or }\Delta_t^{\rm err} {\rm B}(\sfC) \neq 0\big]&\\ 
		\leq C (\eta + C \eta^{-2} e^{-c r} + \eta^{-2-C}/N) 
		&\leq 3 C \eta,
	\end{align*}
	provided that~$\eta \geq N^{-c_0}$ for some fixed (small) constant~$c_0 >0$. 
\end{proof}

\subsection{Stability of mesoscopic clusters}\label{sec:stability_meso}

Write\footnote{The pre-factor~$1/2$ appears since~$\Delta_t^{\rm IIC}{\rm T}$ is an increment over two time steps of the process~$(\omega_t)_{t\geq 0}$.}
\begin{align}\label{eq:drift}
 	{\rm Drift}_{\rm vert} = \tfrac12 \bbE[\Delta_t^{\rm IIC}{\rm T}] = \tfrac12 \bbE[\Delta_t^{\rm IIC}{\rm B}] 
	\quad \text{ and }\quad
	{\rm Drift}_{\rm lat} =  \tfrac12\bbE[\Delta_t^{\rm IIC}{\rm L}] =  \tfrac12\bbE[\Delta_t^{\rm IIC}{\rm R}].
\end{align} 
The equality between these two quantities is a simple consequence of symmetry. 
Moreover, a simple argument\footnote{The non-strict inequalities are immediate. Indeed, 
applying~$\bfS_{1} \circ \bfS_{0}$ to~$\phi^{\rm IIC,T}_{\bbL_{\rm mix}}$ the top of the IIC cluster may increase or decrease by at most one cell in the vertical direction. The height of a cell of~$\bbL_{\rm mix}$ is~$\sin\alpha + \sin\beta$ and the transformation~$\bfS_{1} \circ \bfS_{0}$ shifts all cells vertically by~$\sin\alpha - \sin\beta$.
Thus~$-2\sin \beta \leq \Delta_t^{\rm IIC}{\rm T}\leq 2\sin \alpha$ a.s. The finite energy property allows to derive the strict inequalities. The same holds for ${\rm Bottom}$.}  based on the finite energy property shows that~$-\sin\beta < {\rm Drift}_{\rm vert} < \sin \alpha$.

Fix a mesoscopic cluster~$\sfC$ of~$\omega_0$.
For~$t >0$,  assuming that the cluster remains mesoscopic in each~$\omega_s$ with~$0\leq s \leq t$, write~$\sfC_t$ for the cluster in~$\omega_t$ and set
\begin{align}
	{A}_t(\sfC)  = {A}(\sfC_t) \qquad \text{ for~$A \in \{{\rm T,B,L,R}\}$}. 
\end{align}

Write~$\tau_{\rm start}^{\rm T}(\sfC)$ for the first time that~${\rm Top}(\sfC_t)$ is not frozen. 
Write~$\tau_{\rm meso}(\sfC)$ for the first time~$t$ that~$\sfC_t$ ceases to be mesoscopic. 
Finally, set~$\tau_{\rm end}^{\rm T}(\sfC)$ for the first time~$t > \tau_{\rm start}^{\rm T}(\sfC)$ such that, either~$\sfC_t$ is not mesoscopic or~${\rm Top}(\sfC_t)$ is frozen again. 
The evolution of~${\rm T}_t(\sfC)$ is then stationary up to~$\tau_{\rm start}^{\rm T}(\sfC)$ and after~$\tau_{\rm end}^{\rm T}(\sfC)$ (the latter is by convention if~$\tau_{\rm end}^{\rm T}(\sfC) = \tau_{\rm meso}(\sfC)$). 

Apply the same construction for~${\rm Bottom}(\sfC)$ to define~$\tau_{\rm start}^{\rm B}(\sfC)$ and~$\tau_{\rm end}^{\rm B}(\sfC)$. 
Note that $\tau_{\rm start}^{\rm T}(\sfC)$, $\tau_{\rm start}^{\rm B}(\sfC)$,
$\tau_{\rm end}^{\rm T}(\sfC)$, $\tau_{\rm end}^{\rm B}(\sfC)$ and~$\tau_{\rm meso}(\sfC)$ are determined by the process $(\calH(\omega_t))_t$, with the first two even determined by its initial value:
\begin{align}
\tau_{\rm start}^{\rm T}(\sfC) = K -\frac{ {\rm T}_0(\sfC)}{\sin\beta} \quad\text{ and }\quad \tau_{\rm start}^{\rm B}(\sfC) = K - \frac{{\rm B}_0(\sfC)}{\sin\beta}.
\end{align}
Here, and throughout the rest of the proof, we ignore integer parts in the definition of time-steps. 
We do not define the same stopping times for~${\rm Left}(\sfC)$ and~${\rm Right}(\sfC)$, as they would not be determined by $(\calH(\omega_t))_t$; we will work around this minor difficulty. 

Assuming that the steps of~$({\rm T}_t(\sfC), {\rm B}_t(\sfC), {\rm L}_t(\sfC), {\rm R}_t(\sfC))$ are exactly given by~${{\rm Drift}}_v$ and~${{\rm Drift}}_h$, respectively, and that~$\tau_{\rm meso}(\sfC)= \infty$, we would find that, for~$A \in \{{\rm T}, {\rm B}\}$
\begin{align}\label{eq:tau_drift}
A_t(\sfC) &= A_0(\sfC) + \big((t \wedge \tau_{\rm end}^{{\rm Drift}, A} (\sfC) - \tau_{\rm start}^A(\sfC))\vee 0\big)\cdot {\rm Drift}_{\rm vert},\\[.5em] 
 \text{ where}\qquad
	\tau_{\rm end}^{{\rm Drift}, A} (\sfC) &= {A}_0(\sfC) \frac{1 + {\rm Drift}_{\rm vert} /\sin \beta}{\sin\alpha - {\rm Drift}_{\rm vert}} + K 
\end{align} 
is the equivalent~$\tau_{\rm end}^{A}$ under the assumption that~$A_t$ moves deterministically with an increment of~${\rm Drift}_{\rm vert}$ at every step when~$A$ is active. 

For~${\rm L}_t(\sfC)$ and~${\rm R}_t(\sfC)$ the dynamics, even under the assumption that their steps are deterministic, is not that straightforward. This is because it is not clear when exactly~${\rm Left}(\sfC_t)$ and~${\rm Right}(\sfC_t)$ are frozen or active. 
Nevertheless, we can expect that
\begin{align*}
\big|{\rm L}_t(\sfC) - {\rm L}_0(\sfC) -  \big((t \wedge \tau_{\rm end}^{{\rm Drift}, {\rm T}} (\sfC) - \tau_{\rm start}^{\rm T}(\sfC))\vee 0\big)\cdot {\rm Drift}_{\rm lat} \big| \leq 
C \big({\rm T}_0(\sfC) -  {\rm B}_0(\sfC)\big),
\end{align*}
for some constant~$C$. This is because the ambiguity in whether~${\rm Left}(\sfC_t)$ is frozen or active only occurs when~${\rm Top}(\sfC_t)$ and~${\rm Bottom}(\sfC_t)$ are in different blocks of~$\bbL_t$, and this only occurs for~$O({\rm T}_0(\sfC) -  {\rm B}_0(\sfC))$ steps of the process. The same holds for~${\rm R}_t(\sfC)$. 

With this in mind, for constants~$c,C> 0$ to be fixed below, we say that a cluster~$\sfC \in {\rm Meso}(\omega_0)$ is {\em~$c,C$-stable} if 
\begin{align}\label{eq:stable_C}
	&\tau_{\rm meso}(\sfC) \geq K+K' \qquad \text{ and, for all~$0 \leq t \leq K+K'$ and~$A \in \{{\rm T,B,L,R}\}$ }\nonumber\\[.5em]
	&\big|A_t(\sfC) - A_0(\sfC) -  \big((t \wedge \tau_{\rm end}^{{\rm Drift}, {\rm T}} (\sfC) - \tau_{\rm start}^{\rm T}(\sfC))\vee 0\big)\cdot {\rm Drift}_* \big| \leq C \eta^c N ,
\end{align}
where~${\rm Drift}_*$ is~${\rm Drift}_{\rm vert}$ for~$A = {\rm T,B}$ and~${\rm Drift}_{\rm lat}$ for~$A = {\rm L,R}$. Note that we use~$\tau_{\rm end}^{{\rm Drift}, {\rm T}}(\sfC)$ for all directions, even the bottom --- this will prove irrelevant as the difference between~$\tau_{\rm end}^{{\rm Drift}, {\rm T}}(\sfC)$ and~$\tau_{\rm end}^{{\rm Drift}, {\rm B}}(\sfC)$ will be absorbed in the error term~$C \eta^c N$. 

We now state the key result which states that the dynamics only affects the extrema of mesoscopic clusters by the linear map~$M_{\beta,\alpha}$. 
This will be later shown to imply that the actual shapes of macroscopic clusters are only affected through~$M_{\beta,\alpha}$. 
We cannot expect the statement below to apply to all mesoscopic clusters of~$\omega_0$, since clusters that are barely mesoscopic in~$\omega_0$ may cease being mesoscopic during the process. For that reason, we introduce the more restrictive notion of~${\rm Meso}^+(\omega_0)$.

Fix some large constant~$C_{M}>0$ (to be precise~${C_{M}} = 2 + 2 \max\{|{\rm Drift}_{\rm vert}|, |{\rm Drift}_{\rm lat}|\}$ is sufficiently large). 
Denote by~${\rm Meso}^+(\omega_0)$ the set of clusters~$\sfC$ of~$\omega_0$ such that
\begin{align}\label{eq:Meso^+}
	C_{M}\eta N	\leq {\rm T}_0(\sfC)  - {\rm B}_0(\sfC)   \leq \tfrac{\sqrt \eta}{C_{M}} N, \quad
	{\rm T}_0(\sfC)\leq \tfrac{N}{C_{M}}, \quad &{\rm L}_0(\sfC)\geq - \tfrac{N}{C_{M}}\text{ and } {\rm R}_0(\sfC)\leq \tfrac{N}{C_{M}}.\qquad
\end{align}

\begin{proposition}\label{prop:stability_meso}
	There exist constants~$c,C >0$ such that, for~$N$ large enough and~$\eta > N^{-c}$  
	\begin{align}\label{eq:stable}
	\bbP[\text{all~$\sfC \in {\rm Meso}^+(\omega_0)$ are~$c,C$-stable}] \geq 1 - C \eta^{c}.
	\end{align}
\end{proposition}

\begin{remark}\label{rem:M_origin}
	The link between the definition \eqref{eq:stable_C} of stable cluster and the map~$M_{\beta,\alpha}$ may appear mysterious. 
	To understand this, consider~$x =(x_1,x_2) \in \bbR \times \bbR_+$ such that 
	\begin{align}
	\langle x,e_{\rm vert}\rangle  =  {\rm T}_0(\sfC) \qquad\text{and}\qquad
	\langle x,e_{\rm lat}\rangle   =  {\rm R}_0(\sfC)
	\end{align} 
	for some cluster~$\sfC \in {\rm Meso}^+(\omega_0)$. 
	Assume in addition that~$\sfC$ is~$c,C$-stable and write~$y =(y_1,y_2)$ for the unique point such that 
   \begin{align}
   	\langle y,e_{\rm vert}\rangle  =  {\rm T}_{ K+K'}(\sfC)\qquad\text{and}\qquad
	\langle y,e_{\rm lat}\rangle   =  {\rm R}_{ K+K'}(\sfC). 
   \end{align}
	Then a straightforward computation show that 
	\begin{align}
	y_1 &= x_1 +    \big({\rm Drift}_{\rm lat} \tfrac{1}{\sin\beta}  +  {\rm Drift}_{\rm vert} \tfrac{\cos\beta}{\sin\beta} \big)\cdot x_2 \tfrac{\sin\alpha + \sin\beta}{(\sin \alpha -{\rm Drift}_{\rm vert})\sin\beta} + O(\eta^c N) \qquad
	\text{ and } \\
	y_2 &= x_2 + {\rm Drift}_{\rm vert} \cdot x_2 \tfrac{\sin\alpha + \sin\beta}{(\sin \alpha -{\rm Drift}_{\rm vert})\sin\beta}+ O(\eta^c N)
	\end{align}
	which translates to~$y^{T} = M_{\beta,\alpha}x^{T} + O(\eta^c N)$. 
	
	Indeed, the term~$x_2 \frac{\sin\alpha + \sin\beta}{(\sin \alpha -{\rm Drift}_{\rm vert})\sin\beta} = \tau_{\rm end}^{{\rm Drift}, {\rm T}} (\sfC) - \tau_{\rm start}^{{\rm T}} (\sfC)$ is the approximate number of transformations that affect the cluster~$\sfC$. 
	For the vertical coordinate, observe that, on average, each transformation pushes the cluster up by~${\rm Drift}_{\rm vert}$.
	Similarly, due to the relation between the canonical and~$({\rm T}_{ t}(\sfC),{\rm R}_{ t}(\sfC))$-coordinates, 
	each transformations, on average, pushes the cluster to the right by~${\rm Drift}_{\rm lat} \tfrac{1}{\sin\beta}  +  {\rm Drift}_{\rm vert} \tfrac{\cos\beta}{\sin\beta}$. 
\end{remark}


\begin{proof}[Proof of Proposition~\ref{prop:stability_meso}]
	For any cluster~$\sfC \in {\rm Meso}(\omega_0)$ and~$A \in \{{\rm T}, {\rm B}\}$, 
	as long as~$t \leq \tau_{\rm end}^A(\sfC)$,
	\begin{align}\label{eq:Ext_t}
		{A}_t(\sfC) =  	{A}_0(\sfC)  + 
		  \sum_{\tau_{\rm start}^A(\sfC) \leq s < t}\Delta_s^{\rm IIC} {A}(\sfC) + \sum_{\tau_{\rm start}^A(\sfC) \leq s < t} \Delta_s^{\rm err} {A}(\sfC).
	\end{align}
	(Formally, the above is only valid for~$t$ even, and the sum is over~$s$ even.) Set
	\begin{align}
		{A}^{\rm IIC}_t(\sfC) =  	{A}_0(\sfC)  +  \sum_{\tau_{\rm start}^A(\sfC) \leq s < t}\Delta_s^{\rm IIC} {A}(\sfC). 
	\end{align}
	Extend the notation to~$A  \in \{{\rm L}, {\rm R}\}$ by 
	\begin{align}\label{eq:Ext_t}
		{A}^{\rm IIC}_t(\sfC) =  	{A}_0(\sfC)  +  \sum_{\tau_{\rm start}^{\rm T}(\sfC) \leq s < t}\Delta_s^{\rm IIC} {A}(\sfC).
	\end{align}
	Notice that in the latter definition, we use~$\tau_{\rm start}^{\rm T}(\sfC)$ as a starting time, even though~${\rm L}_t(\sfC)$ may be stationary beyond this time. 
	Also, for technical reasons, we continue~$	{A}^{\rm IIC}_t(\sfC)$ beyond~$\tau_{\rm end}^{.}(\sfC)$ with IIC increments which are irrelevant for the true process~${A}_t(\sfC)$.
	
	Define 
	\begin{align}
	{\rm TotErr} = \sum_{t = 0}^{K+K'} \ind_{\{\exists \sfC \in {\rm Meso}(\omega_t) \text{ and } A \in \{{\rm T,B,L,R}\}\,: \, \Delta_t^{\rm err} {A}(\sfC) \neq 0 \}}
	\end{align}
	for the total number of steps at which errors occur. 
	Recall that the variables~$\Delta_t^{\rm err} {A}$ for~$A \in \{{\rm T,B,L,R}\}$ are a.s.\ bounded. 
	Thus, for some universal constant~$C$,
	\begin{align}\label{eq:TotErr}
		|	{A}_t(\sfC) -  {A}_t^{\rm IIC}(\sfC) |\leq C \cdot {\rm TotErr} \qquad  \text{ for all~$0 \leq t \leq \tau_{\rm end}^A(\sfC)$ and~$A \in \{{\rm T},{\rm B}\}$}.
	\end{align}
	Summing~\eqref{eq:error_exists} and applying the Markov inequality, we find 
	\begin{align}\label{eq:TotErr2}
		\bbP\big[ {\rm TotErr} \leq C\, \eta^c N \big] \geq 1- \eta^c
	\end{align}
	for some universal constants~$c,C > 0$.
	
	The coordinates~${A}_t^{\rm IIC}(\sfC)$ are independent random walks with constant drift. 
	Define the event
	\begin{align*}
		{\rm GoodRW} = 
		\big\{&\forall \sfC \in {\rm Meso}^+(\omega_0), \, A \in \{ {\rm T}, {\rm B}, {\rm L}, {\rm R}\}  \text{ and } \tau_{\rm start}^{*}(\sfC)\leq t\leq K+K', \text\,\\
		&|{A}_t^{\rm IIC}(\sfC) -{A}_0^{\rm IIC}(\sfC) - (t - \tau_{\rm start}^{*}(\sfC)) \cdot {\rm Drift}_{A}| \leq N^{2/3} 
	\big\},
	\end{align*}
	where\footnote{As opposed to the notion of stability~\eqref{eq:stable_C}, we choose to work here with~$\tau_{\rm start}^{\rm T}$  for~${\rm T}^{\rm IIC}_t(\sfC)$  and~$\tau_{\rm start}^{\rm B}$ for~${\rm B}^{\rm IIC}_t(\sfC)$. This extra precision is to ensure that the vertical diameter of~$\sfC$ does not degenerate throughout the process.}~$* = {\rm T}$ if~$A \in \{ {\rm T}, {\rm L}, {\rm R}\}$ and~$* = {\rm B}$ if~$A= {\rm B}$,
	and~${\rm Drift}_{A} =  {\rm Drift}_{\rm vert}$ when~$A \in \{{\rm T,B}\}$ and~$ {\rm Drift}_{A} =  {\rm Drift}_{\rm lat}$ when~$A \in \{{\rm L,R}\}$.
		
	Standard random walk estimates 
	combined with crude bounds on~$|{\rm Meso}^+(\omega_0)|$ using~\eqref{eq:mesoscopic_number} imply that 
	\begin{align}\label{eq:GoodRW}
		\bbP\big[ {\rm GoodRW}\big] \geq 1 - C \eta.
	\end{align}
	for any~$\eta > N^{-c}$ and~$c,C > 0$ some universal constants.

	Finally, define 
	\begin{align*}
		{\rm NoDieIIC} = \big\{& \forall \sfC \in {\rm Meso}^+(\omega_0) \text{ and } 0\leq t\leq K+K',\\
		& 2 \eta  N \leq {\rm T}_t^{\rm IIC}(\sfC) - {\rm B}_t^{\rm IIC}(\sfC) \leq \tfrac12 \eta^{1/2} N, \\
		&{\rm T}_t^{\rm IIC}(\sfC) \leq \tfrac1{2} N, \, 	{\rm L}_t^{\rm IIC}(\sfC) \geq -\tfrac1{2} N \text{ and } {\rm R}_t^{\rm IIC}(\sfC) \leq \tfrac1{2} N
		  \big\} 
	\end{align*}
	In other words,~${\rm NoDieIIC}$ is the event that the IIC coordinates 
	${\rm T}_t^{\rm IIC}, {\rm B}_t^{\rm IIC}, {\rm L}_t^{\rm IIC}, {\rm R}_t^{\rm IIC}$ satisfy the conditions of~\eqref{eq:meso}, 
	with twice more restrictive constants, for all~$\sfC \in {\rm Meso}^+(\omega_0)$ and all~$0\leq t\leq K+K'$.
	Due to the conditions~\eqref{eq:Meso^+} for~$\sfC$ to be in~${\rm Meso}^+(\omega_0)$
	 --- specifically the choice of~$C_{M}$ as sufficiently large --- 
	we find that, whenever~${\rm GoodRW}$ occurs, so does~${\rm NoDieIIC}$.

	Assume henceforth that both~${\rm GoodRW}$ and~${\rm NoDieIIC}$ occur and that~${\rm TotErr} \leq C\, N^{1 - c}$ and that~$N$ is sufficiently large. 
	We will prove that the even in \eqref{eq:stable} occurs. 

	We start by proving that
	\begin{align}\label{eq:meso+survive}
		\text{for all~$\sfC \in{\rm Meso}^+(\omega_0)$,~$\tau_{\rm meso}(\sfC)> K+K'$}.
	\end{align}
	Indeed, for any~$\sfC \in{\rm Meso}^+(\omega_0)$,~${\rm T}_t(\sfC)$ and~${\rm B}_t(\sfC)$ are close to their IIC counterparts due to~\eqref{eq:TotErr} and our assumption on~${\rm TotErr}$. Factoring in the occurrence of~${\rm NoDieIIC}$, we conclude that, for all~$0\leq t\leq K+K'$,
	\begin{align} \label{eq:meso+survive2}
		 \eta  N \leq {\rm T}_t(\sfC) - {\rm B}_t(\sfC) \leq \eta^{1/2} N  \quad\text{ and }\quad {\rm T}_t(\sfC) \leq  N.
	\end{align}
	The lateral coordinates~${\rm L}_t(\sfC)$ and~${\rm R}_t(\sfC)$ require slightly more attention, due to the uncertain times when they evolve. 
	However, we know they are stationary before~$\tau_{\rm start}^{\rm T}(\sfC)$ and after~$\tau_{\rm end}^{\rm T}(\sfC)$ and 
	we know they are evolving at all times between~$\tau_{\rm start}^{\rm B}(\sfC)$ and~$\tau_{\rm end}^{\rm B}(\sfC)$.
	The upper bound on~${\rm T}_t(\sfC) - {\rm B}_t(\sfC)$ in~\eqref{eq:meso+survive2} 
	ensures that there are at most~$O(\sqrt \eta N)$ steps when~${\rm L}_t^{\rm IIC}(\sfC)$ moves, but~${\rm L}_t(\sfC)$ does not, 
	and the same for~${\rm R}_t(\sfC)$. 
	This adds an error term when comparing the lateral coordinates of the clusters~$\sfC \in {\rm Meso}^+(\omega_0)$ to their IIC equivalents, 
	but this error is sufficiently small that it allows us to deduce from~${\rm NoDieIIC}$ that
	\begin{align}
	{\rm L}_t(\sfC) \geq - N \text{ and } {\rm R}_t(\sfC) \leq  N \quad \text{ for all~$0\leq t\leq K+K'$,}
	\end{align}
	which, together with \eqref{eq:meso+survive2}, implies \eqref{eq:meso+survive}.\smallskip 
	
	We now turn to the stability of the clusters of~${\rm Meso}^+(\omega_0)$. 
	By the definition of~$\tau_{\rm end}^{\rm T} (\sfC)$, we have that 
	\begin{align}
		{\rm T}_{\tau_{\rm end}^{\rm T} (\sfC) }(\sfC) = (\tau_{\rm end}^{\rm T} (\sfC)  - K)\sin\alpha.
	\end{align}
	Comparing~${\rm T}_{\tau_{\rm end}^{\rm T} (\sfC) }(\sfC)$ to~${\rm T}^{\rm IIC}_{\tau_{\rm end}^{\rm T} (\sfC) }(\sfC)$, which in turn 
	is close to~${\rm T}_0(\sfC) +(\tau_{\rm end}^{\rm T} (\sfC)  - \tau_{\rm start}^{\rm T}(\sfC)) \cdot {\rm Drift}_{\rm vert}$, we conclude that 
	\begin{align*}
		{\rm T}_0(\sfC) + (\tau_{\rm end}^{\rm T} (\sfC) - \tau_{\rm start}^{\rm T}(\sfC)) \cdot {\rm Drift}_{\rm vert} - (\tau_{\rm end}^{\rm T} (\sfC)  - K)\sin\alpha = O(N^{1-c}).
	\end{align*}
	Recall that~$(K - \tau_{\rm start}^{\rm T}(\sfC)) \sin \beta = {\rm T}_0(\sfC)$, which implies that 
	\begin{align}\label{eq:tau_Drift_RW}
		\tau_{\rm end}^{\rm T} (\sfC) = {\rm T}_0(\sfC) \frac{1 + {\rm Drift}_{\rm vert}/ \sin \beta}{\sin\alpha - {\rm Drift}_{\rm vert}} + K + O(N^{1-c})
		= \tau_{\rm end}^{{\rm Drift, T}}(\sfC)+ O(N^{1-c}).
	\end{align}
	Finally, since~${\rm T}_t(\sfC)$ is stationary after~$\tau_{\rm end}^{\rm T}(\sfC)$, for all~$0 \leq t\leq K+K'$,
	\begin{align*}
		\big|{\rm T}_t(\sfC) - {\rm T}_0(\sfC)&   -  \big((t \wedge \tau_{\rm end}^{{\rm Drift}, {\rm T}} (\sfC) - \tau_{\rm start}^{\rm T}(\sfC))\vee 0\big){\rm Drift}_{\rm vert}\big|\\
		\leq& 
		\big|{\rm T}_{t\wedge \tau_{\rm end}^{\rm T}} (\sfC)  -	{\rm T}_{t\wedge \tau_{\rm end}^{\rm T}}^{\rm IIC}(\sfC)\big|  \\
		&+\big|{\rm T}_{t\wedge \tau_{\rm end}^{\rm T}}^{\rm IIC}(\sfC) - {\rm T}_0(\sfC) - \big((t \wedge \tau_{\rm end}^{\rm T} (\sfC) - \tau_{\rm start}^{\rm T}(\sfC))\vee 0\big){\rm Drift}_{\rm vert}\big|  \\
		&+\big|(\tau_{\rm end}^{\rm T}(\sfC) -  \tau_{\rm end}^{{\rm Drift},{\rm T}}(\sfC)){\rm Drift}_{\rm vert}\big|.
	\end{align*}
	Each of the three terms in the right-hand side are~$O(N^{1-c})$ due to~${\rm TotErr} \leq C\, N^{1 - c}$, the occurrence of~${\rm GoodRW}$ and~\eqref{eq:tau_Drift_RW}, respectively. 

	The same computation applies to~${\rm B}_t(\sfC)$,~${\rm L}_t(\sfC)$ and~${\rm R}_t(\sfC)$, with the only difference that there is an additional error due to the transformations performed
	between~$\tau_{\rm start}^{\rm T} (\sfC)$ and~$\tau_{\rm start}^{\rm B} (\sfC)$ and between~$\tau_{\rm end}^{\rm B} (\sfC)$ and~$\tau_{\rm end}^{\rm T} (\sfC)$. 
	However, the number of such transformations is bounded due to the upper bound on~${\rm T}_t (\sfC) -{\rm B}_t (\sfC)$ in \eqref{eq:meso+survive2}.
	We conclude that 
	\begin{align*}
		\big|{A}_t(\sfC) - {A}_0(\sfC)    -  \big((t \wedge \tau_{\rm end}^{{\rm Drift}, {\rm T}} (\sfC) - \tau_{\rm start}^{\rm T}(\sfC))\vee 0\big)\cdot {\rm Drift}_*\big|
		=O(N^{1-c})	+ O(\sqrt \eta N)
	\end{align*}
	for all~$\sfC \in {\rm Meso}^+(\omega_0)$ and~$A \in \{{\rm T},{\rm B},{\rm L},{\rm R}\}$, where~${\rm Drift}_*$ refers to the drift in the direction associated to~$A$ . 
	
	Thus, when~${\rm TotErr} \leq C\, N^{1 - c}$ and~${\rm GoodRW}$ occurs, all clusters~$\sfC \in{\rm Meso}^+(\omega_0)$ are stable in the sense of~\eqref{eq:stable_C}, provided that~$c,C >0$ are chosen appropriately. Finally,~\eqref{eq:TotErr2} and~\eqref{eq:GoodRW} ensure that the above events occur with probability at least~$1 - C \eta^c$. This implies the desired conclusion.  
\end{proof}

\subsection{Homotopy classes}\label{sec:5homotopy}

Fix some~$\delta > 0$ and~$N \geq 1$.  
For a configuration~$\omega$ on an isoradial rectangular lattice~$\bbL$ with constant angle, 
a nail at scales~$(\delta,N)$ for a point~$x \in \delta^{1/2}N \bbZ^2$ is any cluster~$\sfC$ of~$\omega$ such that 
\begin{align}
	|{\rm T}(\sfC)-\langle x,e_{\rm vert}\rangle  | &\leq \delta N, &&
	&|{\rm B}(\sfC)-\langle x,e_{\rm vert}\rangle  | &\leq \delta N,\\
	|{\rm L}(\sfC)-\langle x,e_{\rm lat}\rangle  | &\leq \delta N	 &\text{ and }&
	&|{\rm R}(\sfC)-\langle x,e_{\rm lat}\rangle  | &\leq \delta N.\label{eq:nails}
\end{align}

A~$(\delta,N)$-lattice of nails is the choice of a nail~$\sfN(x)$ for each point~$x \in (\delta^{1/2}N )\cdot (\bbZ^2 \cap [-\delta^{-1/4}, \delta^{-1/4}] \times [0, \delta^{-1/4}] )$. Note that the lattice of nails roughly occupies the window $[-\delta^{1/4} N ,\delta^{1/4} N] \times [0,\delta^{1/4} N]$, rather than the full observation window~$[-N,N] \times [0,N]$. This is for technical reasons that will be apparent below. 
A configuration may contain no lattice of nails or several such lattices. See Figure~\ref{fig:lattice_nails}.

\begin{figure}
\begin{center}
\includegraphics[width = 0.7\textwidth]{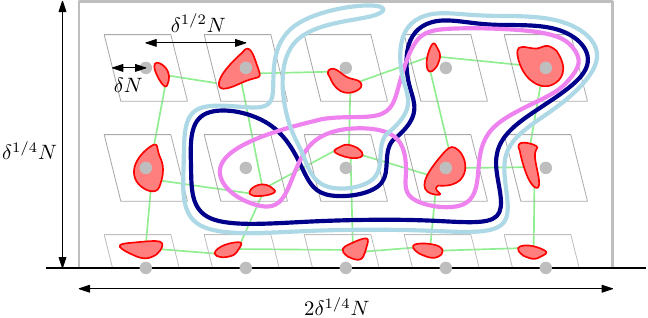}
\caption{A lattice of nails (in red) is formed of one nail for each point~$x$ in the grey lattice, within the rectangular window. The grey rhombi designate the regions in which nails should be contained, according to~\eqref{eq:nails}. The blue, light-blue and pink loops are all non-trivial. They all surround the same nails, but the homotopy class of the pink loop is different form that of the two blue loops: the homotopy class is given by the order in which the loops intersect the light green segments. 
The two blue loops are at large~$d$-distance in the sense of~\eqref{eq:loop_dist} in spite of having same homotopy class. This is because the light blue loop has a long arm wiggling between nails, without surrounding any of them; such a behaviour is unlikely due to~\eqref{eq:RSW_iso}.}
\label{fig:lattice_nails}
\end{center}
\end{figure}

For a fixed lattice of nails, a loop~$\ell$ in the loop representation of~$\omega$ is considered non-trivial if it surrounds at least two, but not all of the nails of the lattice. 
Consider the punctured plane~$\bbR^2 \setminus (\bigcup_x \sfN(x))$, where we mean that we remove the interior of the outer contour of each nail. 
Notice that any non-trivial loop is contained in~$\bbR^2 \setminus (\bigcup_x \sfN(x))$ and has a certain homotopy class. 

It will be important below to have a standard representation of homotopy classes in such a punctured plane, as we will compare homotopy classes with respect to different families of nails. We do this as follows. 
Write~$\vec E$ for the oriented edges of the  graph~$\bbZ^2 \cap [-\delta^{-1/4}, \delta^{-1/4}] \times [0, \delta^{-1/4}]~$ (each unoriented edge of the graph corresponds to two oriented edges). 
Fix a point~$\sfn_x$ in each nail~$\sfN(x)$ of the nail corresponding to~$x$. Identify~$\vec e = (xy)$ with the oriented segment between~$\sfn_x$ and~$\sfn_y$.

Let~$\calW$ be the set of finite words on the alphabet~$\vec E$ and denote the empty word by~$\emptyset$. Define the equivalence relation~$\sim$ on~$\calW$ generated by~$(u_i)_{1\le i\le p} \sim (v_j)_{1\le j\le q}$ if 
\begin{itemize}
	\item~$p=q$ and there exists~$k\in[1,p]$ such that~$u_1\ldots u_p = v_k \ldots v_p v_1\ldots v_{k-1}$ or 
	\item~$p=q+2$,~$u_1\ldots u_p = v_1 \ldots v_q$ and~$u_{q+2}$ is the same as the edge~$u_{q+1}$ but with the opposite orientation.
\end{itemize}
Define the set of {\em reduced words} as the quotient~${\calC\calW} := \calW / \sim$. 

Recall that the loops of~$\omega$ are oriented so as to have primal edges are on their right. 
Write~$w_0(\ell) \in \calW$ for the word formed of the sequence of edges~$\vec e$ crossed by~$\ell$ from left to right, 
and~$w(\ell) = w_\sfN(\ell)$ for the reduced word corresponding to~$w_0(\ell)$.

It is standard to check that this indeed encodes the homotopy class of every non-trivial loop and that~$w_\sfN(\ell)$ does not depend on the choice of points~$\sfn_x$. It may depend on the choice of the lattice of nails; however we will eventually prove that this only occurs with low probability (see Proposition~\ref{prop:homotopy_to_CN}).

\begin{definition}
We say that two percolation configurations~$\omega$ and~$\omega'$ are {\em homotopically similar} at scales~$(\delta,N)$ 
if they both contain lattices of nails~$\sfN$ and~$\sfN'$ at scales~$(\delta,N)$ such that
there exists a bijection~$\psi$ between the non-trivial loops of~$\omega$ 
with respect to~$\sfN$ and non-trivial loops of~$\omega'$ with respect to~$\sfN'$ with 
\begin{align}
	w_\sfN(\ell) = w_{\sfN'}(\psi(\ell)),
\end{align}
for all non-trivial loops~$\ell$ of~$\omega$. 
\end{definition}

The same notions above also apply to piece-wise linear deformations of percolation configurations. 

\begin{proposition}\label{prop:homotopy_to_CN}
	Let~$\bbL = \bbL(\beta)$ and~$\bbL' = \bbL(\alpha)$ be two isoradial rectangular lattices with constant angles 
	$\alpha, \beta \in (0,\pi)$ and~$M$ be an invertible linear transformation of~$\bbR^2$. 
	There exist constants~$C,c >0$ such that the following holds for all~$\delta > 0$ and~$N \geq \delta^{-C}$. 
	If~$\bbP$ is a coupling of~$\phi_{\bbL}$ and~$\phi_{\bbL'}$ such that 
	\begin{align*}
		\bbP\big[\text{$M(\omega)$ and~$\omega'$ are homotopically similar at scales~$(\delta,N)$}] \geq 1- \delta,
	\end{align*}
	then
	\begin{align}\label{eq:homotopy_to_CN}
		{\bf d}_{\rm CN}\big[ \phi_{\delta^c\bbL}\circ M^{-1}, \phi_{\delta^c\bbL'}\big] \leq C\,\delta^c.
	\end{align}
\end{proposition}

The rest of the section is dedicated to the proof of the above. The idea is explained in Figures~\ref{fig:lattice_nails} and~\ref{fig:nails2}.

We start off with a helpful consequence of the \eqref{eq:RSW_iso} property. Assume~$\delta$ and~$N$ fixed. 
For a point of~$x \in \sqrt\delta N \bbZ^2$, let~$\sfD(x)$ be the rhombus-shaped region of~$\bbR^2$ defined as 
\begin{align}
	\sfD(x) = \big\{ y \in \bbR^2:\, 
	|\langle y - x,e_{\rm vert}\rangle  |\leq \delta N \text{ and }
	|\langle y - x,e_{\rm lat}\rangle  | \leq \delta N \big\}.
\end{align}
These are the regions containing the nails at scales~$(\delta, N)$.
A rectangle~$[0,4R] \times [0,R]$ is said to contain a thin  crossing if it contains a primal or dual cluster~$\sfC$ that crosses the rectangle vertically, but surrounds none of the regions~$\sfD(x):x \in \sqrt\delta N \bbZ^2$. 
The same notion applies to crossing  of~$[0,R] \times [0,4R]$ in the horizontal direction. 

\begin{lemma}\label{lem:homotopy_to_CN}
	There exist constants~$c_0,C_0 > 0$ such that, for any isoradial square lattice~$\bbL$ with constant angle, 
	any~$\delta > 0$,~$N\geq 1$ any~$R \geq \sqrt \delta N$ 
	\begin{align}
		\phi_{\bbL}\big[ \text{$[0,4R] \times [0,R]$ contains a thin crossing}\big]\leq C_0 (\sqrt \delta N / R)^{c_0}.
	\end{align}
\end{lemma}

\begin{proof}[Proof of Lemma~\ref{lem:homotopy_to_CN}]
	We only sketch this. Explore the successive vertical crossings of~$[0,4R] \times [0,R]$ from left to right. 
	Whenever such a crossing~$\Gamma$ is revealed, the region to its right is unexplored. We can then identify~$R/\sqrt \delta N$ regions~$\sfD(x)$  
	that are within distance~$2\sqrt \delta N$ of~$\Gamma$, but lie entirely in the unexplored region; see Figure~\ref{fig:surrounding_nail}.
	Then \eqref{eq:RSW_iso} states that each such region is surrounded by the cluster of~$\Gamma$ with probability at least a uniformly positive constant. Furthermore, the events that each of these regions is surrounded may be bonded by independent Bernoulli variables of uniformly positive parameter. 
	
	\begin{figure}
	\begin{center}
	\includegraphics[width = 0.65\textwidth]{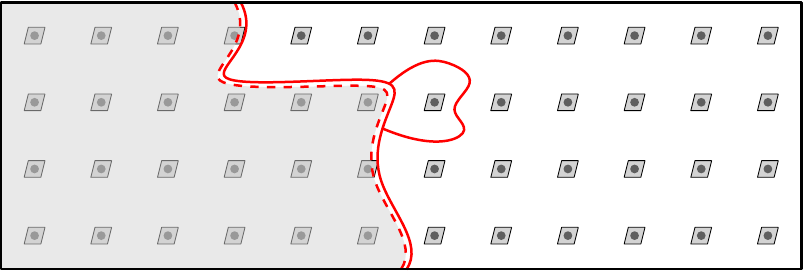}
	\caption{The left boundary~$\Gamma$ of a vertically crossing primal cluster; the conditioning only affects the grey area. 
	Under this conditioning, the cluster of~$\Gamma$ surrounds each one of the grey rhombi~$\sfD(x)$ closest to it with uniformly positive probability. }
	\label{fig:surrounding_nail}
	\end{center}
	\end{figure}
	
	We conclude that, for any such~$\Gamma$, the probability that the cluster of~$\Gamma$ forms a thin crossing is bounded above by~$(1-c)^{R/\sqrt \delta N}$ for some constant~$c>0$. 
	Combining this with a exponential tail of the number of crossings provides the desired bound. 
\end{proof}

\begin{proof}[Proof of Proposition~\ref{prop:homotopy_to_CN}]
	For simplicity, we consider first the case where~$M = {\rm id}$. 
	For a configuration~$\omega$ and a lattice of nails~$\sfN$, we may define a reference configuration\footnote{To be precise,~$\omega_{\rm ref}$ is not actually a percolation configuration, but only a family of oriented loops which might not correspond to a percolation configuration.}~$\omega_{\rm ref}$ 
	depending only on the homotopy information of~$\omega$ and which is homotopically similar to~$\omega$ at scales~$(\delta,N)$.
	It is constructed as follow; see also Figure~\ref{fig:nails2}.
	
	The nails~$\sfN_{\rm ref}$ are single primal vertices, namely those closest to the points of~$(\delta^{1/2}N )\cdot (\bbZ^2 \cap [-\delta^{-1/4}, \delta^{-1/4}] \times [0, \delta^{-1/4}] )$.
	For each non-trivial loop~$\ell$ of~$\omega$, place a loop in~$\ell_{\rm ref} \in \omega_{\rm ref}$ formed of straight line segments between the midpoints of the edges of~$(\delta^{1/2}N )\cdot (\bbZ^2 \cap[-\delta^{-1/4}, \delta^{-1/4}] \times [0, \delta^{-1/4}])$
	so that 
	\begin{align}
	w_\sfN(\ell) = w_{\sfN_{\rm ref}}(\ell_{\rm ref}).
	\end{align}
	
	We will prove that, with high probability, when~$\omega \sim \phi_{\bbL(\alpha)}$, 
	 \begin{align}\label{eq:omegaomega_ref}
		{d}_{\rm CN}\big[\tfrac{1}{\delta^{3/8} N} \omega, \tfrac{1}{\delta^{3/8} N} \omega_{\rm ref}\big] \leq \delta^{1/16}.
	\end{align}
	Applying\footnote{We actually bound the probability of~\eqref{eq:omegaomega_ref} for the configurations shifted vertically by~$-\frac12\delta^{1/4}N$, but this has no influence on our conclusion, as the measures are translationally invariant.} this to~$\omega$ and~$\omega'$ directly implies~\eqref{eq:homotopy_to_CN}, since~$\omega_{\rm ref} =\omega'_{\rm ref}$. 
	
	Assume that~\eqref{eq:omegaomega_ref} fails.
	Then, there exists a loop~$\ell \in \omega$ contained in~$[-\delta^{1/4}N,\delta^{1/4}N] \times [0,\delta^{1/4}N]~$ 
	so that 
	\begin{align}\label{eq:long_arm}
		d\big(\tfrac{1}{\delta^{3/8} N}\ell, \tfrac{1}{\delta^{3/8} N}\ell_{\rm ref}\big) \geq \delta^{1/16}.
	\end{align}
	This requires~$\ell~$ to contain a long ``arm'' of diameter~$\delta^{7/16} N$ that avoids all nails of~$\sfN$ --- 
	see Figure~\ref{fig:nails2}.
	
	In particular, if take~$R =\frac12\delta^{7/16} N$ and if we pave the rectangle~$[-\delta^{1/4}N,\delta^{1/4}N]\times [0,\delta^{1/4}N]$ with the translates of~$[0,4R] \times [0,R]$ and~$[0,R] \times [0,4R]$ by points of~$R\bbZ^2$, 
	then \eqref{eq:long_arm} implies that there exists at least one such rectangle that contains a thin crossing.
	There are~$O(\delta^{-3/8})$ such rectangles, so Lemma~\ref{lem:homotopy_to_CN} and the union bound imply that 
	\begin{align}\label{eq:homotopy_to_CN0}
		\phi_{\bbL(\alpha)} \Big[{d}_{\rm CN}\big(\tfrac{1}{\delta^{3/8} N} \omega, \tfrac{1}{\delta^{3/8} N} \omega_{\rm ref}\big) \leq \delta^{1/16}\Big]
		\leq C_1  \delta^{-3/8}	 \exp(-c_0 \delta^{-1/16})
		\leq C_2\delta^{c_2}.
	\end{align}
	For universal constants~$C_1,C_2,c_2 > 0$. 
	This implies~\eqref{eq:homotopy_to_CN}, as explained above.
	
	\begin{figure}
	\begin{center}
	\includegraphics[width = .8\textwidth, page = 2]{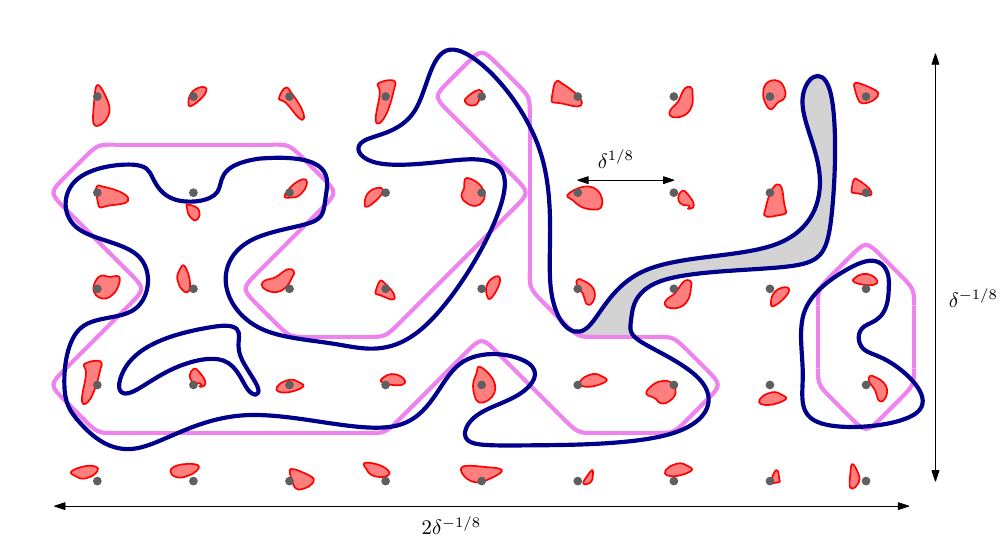}
	\caption{A configuration containing a lattice of nails~$\sfN$ in red, and several blue loops in a rescaled configuration~$\tfrac{1}{\delta^{3/8} N} \omega$. Of the blue loops, one is trivial, the other two have corresponding pink loops in~$\omega_{\rm ref}$. The large blue loop~$\ell$ is not close to~$\ell_{\rm ref}$ because of the long grey ``arm'' --- here we assume that this arm has length at least~$\delta^{1/16}$. The probability of such an arm occurring may be bounded by \eqref{eq:RSW_iso}.}
	\label{fig:nails2}
	\end{center}
	\end{figure}

	When~$M \neq {\rm id}$, we may still prove~\eqref{eq:homotopy_to_CN0} for both~$\omega$ and~$\omega'$. 
	It then suffices to observe that, since~$M$ is linear,~$M(\omega_{\rm ref}) = \omega_{\rm ref}'$.	
\end{proof}

\subsection{Proof of Theorem~\ref{thm:linear}}\label{sec:5linear_proof}

We are finally ready for the proof of Theorem~\ref{thm:linear}. 
Fix~$\alpha,\beta \in (0,\pi)$. 

In light of Proposition~\ref{prop:homotopy_to_CN}, our goal is to prove that for~$N$ large enough and~$\eta \geq N^{-c_\eta}$ for some small constant~$c_{\eta}>0$, we may couple~$\omega\sim \phi_{\bbL(\beta)}$ and~$\omega\sim \phi_{\bbL(\alpha)}$ so that
$M_{\beta,\alpha}(\omega)$ and~$\omega'$ are homotopically similar at scales~$(\eta,N)$. 
The coupling will be that of Section~\ref{sec:5MarkovChain}, as explained below. 

We call a cluster~$\sfC$ of~$\omega$ a nail$^+$ at~$x \in (\delta^{1/2}N )\cdot (\bbZ^2 \cap [-\delta^{-1/4}, \delta^{-1/4}] \times [0, \delta^{-1/4}])$
if 
\begin{align*}
	|{\rm T}(\sfC)-\langle x,e_{\rm vert}\rangle  | &\leq \tfrac{\delta}{2C_{M}} N, &&
	&|{\rm B}(\sfC)-\langle x,e_{\rm vert}\rangle  | &\leq\tfrac{\delta}{2C_{M}} N,\\
	|{\rm L}(\sfC)-\langle x,e_{\rm lat}\rangle  | &\leq \tfrac{\delta}{2C_{M}} N	 &\text{ and }&
	&|{\rm R}(\sfC)-\langle x,e_{\rm lat}\rangle  | &\leq \tfrac{\delta}{2C_{M}} N.
\end{align*}
with~$C_M$ chosen as in the definition of~${\rm Meso}^+$.
We consider here the more restricted notion of nail$^+$ rather than just nail, 
for the same reason that we considered~${\rm Meso}^+(\omega_0)$ rather~${\rm Meso}(\omega_0)$ in Proposition~\ref{prop:stability_meso}:
a buffer is needed to offset the small variations that may occur during the process~$(\omega_t)_{0 \leq t \leq K+K'}$.

\begin{lemma}\label{lem:lattice_nails_exists}
	There exist constants~$C,c >0$ such that, for all~$\delta,\eta >0$ with~$\eta <\delta^2$,~$\delta$ small enough and any~$N$ large enough
	\begin{align}\label{eq:existence_nails}
	\phi_{\bbL(\beta)}[\text{there exists a lattice of nails$^+$ of~${\rm Meso}^+$ at scales~$(\delta,N)$}] \ge 1 - C\eta^{c}  \delta^{-1/2} .
	\end{align}
\end{lemma}

\begin{proof}
	This is a direct consequence of~\eqref{eq:RSW_iso}. 
	Observe that, for given~$\delta$, there are at most~$O(\delta^{-1/2})$ points in the lattice~$(\delta^{1/2}N )\cdot (\bbZ^2 \cap [-\delta^{-1/4}, \delta^{-1/4}] \times [0, \delta^{-1/4}])$. 
	Furthermore, for~$\delta$ small enough, any such point~$x$ satisfies
	\begin{align}
		\langle x,e_{\rm vert}\rangle \leq \tfrac{N}{C_{M}} \quad \text{ and }\quad 
		- \tfrac{N}{C_{M}}\leq	\langle x,e_{\rm lat}\rangle \leq \tfrac{N}{C_{M}}, 
	\end{align}
	which is to say that it is contained in the window covered by the definition of~${\rm Meso}^+$.
	Finally, for each such point~$x$, since~$\sqrt \eta \leq \delta$,~\eqref{eq:RSW_iso} implies that 
	\begin{align}
		\phi_{\bbL(\beta)}[ \text{there exists a nail$^+$~$\sfN(x) \in {\rm Meso}^+$ at~$x$}] \geq 1 - C_0\, \eta^{-c_0}. 
	\end{align}
	for universal constants~$c_0,C_0$. Performing a union bound yields the result. 
\end{proof}

\begin{proof}[Proof of Theorem~\ref{thm:linear}]
For~$N\geq 1$, write~$\eta = N^{-c_\eta}$ and~$\delta = N^{-c_\eta^2}$, for some small~$c_\eta >0$ to be determined below. 
In particular, we assume~$c_\eta < 1/2$ so that Lemma~\ref{lem:lattice_nails_exists} applies. 

Write~${\rm Stability} = {\rm Stability}(\eta,N)$ for the event in Proposition~\ref{prop:stability_meso}, 
with the constants~$c,C$ provided by the proposition. 
By choosing~$c_\eta > 0$ small enough Lemma~\ref{lem:lattice_nails_exists} and Proposition~\ref{prop:stability_meso} ensure that 
\begin{align}
	\bbP[\text{$\omega_0$ contains a lattice of nails$^+$ of~${\rm Meso}^+(\omega_0)$ at scales~$(\delta,N)$}] &\ge 1 - N^{-c}\,\,\text{ and }
	\nonumber\\
	\bbP[\text{Stability}(\eta,N)]& \ge 1 - N^{-c},
	\label{eq:stability+}
\end{align}
for all~$N$ large enough, where~$c$ is some constant which depends on $c_\eta$, but not on $N$.

Assume henceforth that both of the events above occur and that~$N$ is large. 
Fix a lattice of nails$^+$~$\sfN_0 \subset {\rm Meso}^+(\omega_0)$. 
By~$\text{Stability}(\eta,N)$, these clusters survive throughout the process~$(\omega_t)_{0\leq t\leq K+K'}$;
denote by~$\sfN_t$ their collection in~$\omega_t$.
Furthermore, due to~$\text{Stability}(\eta,N)$ and the definition of  nail$^+$, 
$\sfN_t$ is a lattice of nails, up to a piecewise-linear transformation, for every~$t \leq K+K'$.
As such, it makes sense to encode the homotopy information of~$\omega_t$ with respect to~$\sfN_t$, in the same way as for~$\omega_0$. 

We will now argue that this homotopy information remains unchanged throughout the process. 
More precisely, any non-trivial loop~$\ell_0$ surrounding a primal or dual cluster has a corresponding loop~$\ell_t$ in~$\omega_t$ and 
$w_{\sfN_t}(\ell_t) = w_{\sfN_0}(\ell_0)$. 
Indeed, when passing from~$\omega_t$ to~$\omega_{t+1/2}$, as~$\ell_t$ has diameter larger than~$\eta N$, 
it either avoids all extremal boxes or~$\omega_t = \omega_{t+1/2}$. 
In both cases, we conclude that~$w_{\sfN_{t+1/2}}(\ell_{t+1/2}) = w_{\sfN_t}(\ell_t)$.
Furthermore, when passing from~$\omega_{t+1/2}$ to~$\omega_{t+1}$, the fact that track exchanges do not break loops ensures that the homotopy class of any non-trivial loop is preserved. As such~$w_{\sfN_{t+2}}(\ell_{t+2}) = w_{\sfN_{t+1}}(\ell_{t+1}) = w_{\sfN_{t}}(\ell_{t})$.

Finally, as explained in Remark~\ref{rem:M_origin}, the effect of the drift on the lattice~$\sfN_0$ is described by~$M_{\beta,\alpha}$.
More precisely, for all~$x\in (\delta^{1/2}N )\cdot (\bbZ^2 \cap [-\delta^{-1/4}, \delta^{-1/4}] \times [0, \delta^{-1/4}])$ 
\begin{align}
	\big|\langle y, e_{*} \rangle - \langle M_{\beta,\alpha}(x), e_{*} \rangle \big| &\leq  \tfrac{\delta}{2C_{M}} N + C N^{1-c} \qquad 
	 \forall y \in \sfN_{K+K'}(x) \text{ and } * \in \{{\rm vert},{\rm lat}\}.
\end{align}
By again assuming that~$c_\eta$ is sufficiently small and~$N$ large, we conclude that 
$M_{\beta,\alpha}^{-1}(\sfN_{K+K'})$ is a~$(\delta,N)$-lattice of nails.
Then~$\omega_0$ and~$M_{\beta,\alpha}^{-1}(\omega_{K+K'})$ are homotopically similar at scales~$(\delta,N)$. 
Applying Corollary~\ref{cor:same_law}, Proposition~\ref{prop:homotopy_to_CN} and using~\eqref{eq:stability+} we conclude that 
\begin{align}
	{\bf d}_{\rm CN}\big[ \phi_{\delta\bbL(\alpha)}, \phi_{\delta\bbL(\beta)}\circ M_{\beta,\alpha} \big] \leq C\,\delta^c \leq C N^{-c'},
\end{align}		
for all~$N$ large enough, where~$C,c,c' >0$ are universal constants. The assumption on~$N$ may be removed by altering the constants. 
\end{proof}

\subsection{An equality of drifts}\label{sec:5equality_drifts}

As an addition to Theorem~\ref{thm:linear}, we make an apparently obvious but crucial observation about the value of the drift vectors. 

\begin{proposition}\label{prop:drift_RT}
	For any~$0  < \beta < \pi$,
	\begin{align}\label{eq:drift_RT}
		{\rm Drift}_{\rm lat}(\beta,\beta/2) = {\rm Drift}_{\rm vert}(\beta,\beta/2).
	\end{align}
\end{proposition}

The above is actually a surprisingly profound fact. It ultimately shows that the isoradial embedding of the lattices~$\bbL(\alpha)$ (and more generally of bi-periodic graphs) is the right embedding of inhomogeneous FK-percolation to ensure universality in the sense of Theorem~\ref{thm:universalCNSS}. Indeed, except in this proposition, there is no use of the exact formula linking isoradial embeddings to the inhomogeneity of the FK-percolation parameters. 
Behind~\eqref{eq:drift_RT} lies the fact that the line exchanges producing the drift are composed of star-triangle transformations, which do not depend on the direction of the tracks being exchanged. 

\begin{proof} 
	Fix~$0 < \alpha  < \beta < \pi$. 
	We will prove a more general statement, which will imply~\eqref{eq:drift_RT} when taking~$\alpha = \beta /2$. 
	Several scales will appear in the construction below: consider integers~$r \ll N \ll R$, with~$r = N^{1-c}$ for some small constant~$c$.
	
	Consider the finite isoradial graph~$\bbL$ containing a block of~$R$ horizontal tracks with transverse angle~$\beta$ crossed by~$2R$ vertical tracks of transverse angle~$0$. 
	We call this the~$\beta$-block. Position it so that~$0$ is the central point on the bottom part of this block. 
	Complete the block by~$r$ horizontal tracks of transverse angle~$\alpha$ that run along its top and right side. 
	See Figure~\ref{fig:vTR}.
	
	\begin{figure}
	\begin{center}
	\includegraphics[width = .98\textwidth]{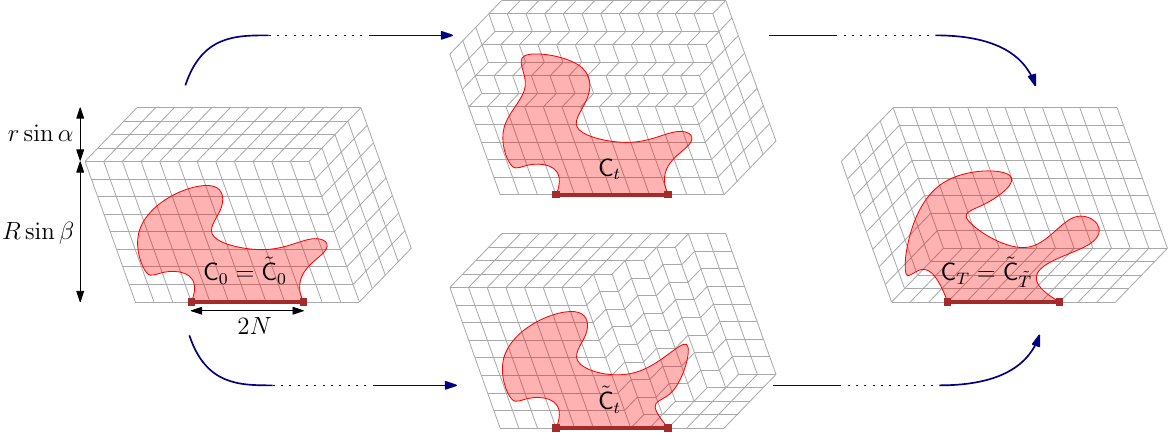}
	\caption{From~$\bbL$ on the left, there are two series of transformations leading to~$\bbL'$ on the right. 
	On the top, we proceed by exchanging horizontal tracks, bringing them down progressively. This produces a sequence of lattices~$\bbL_{t}$, and a sequence~$\sfC_t$ of transformations of the cluster of the bottom red segment. 
	On the bottom, we proceed by progressively exchanging vertical tracks. This leads to sequences~$\tilde \bbL_t$ and~$\tilde \sfC_t$.
	The result of the two processes is the same, which is to say that the laws of the final clusters~$\sfC_T$ and~$\tilde\sfC_{\tilde T}$ are the same.}
	\label{fig:vTR}
	\end{center}
	\end{figure}

	Let~$\bbL'$ be the graph obtained by switching the position of the tracks of angle~$\alpha$ with the~$\beta$-block. In~$\bbL'$, the tracks of angle~$\alpha$ run on the left and bottom side of the~$\beta$-block. 
	
	There are two ways to pass from~$\bbL$ to~$\bbL'$ using compositions of star-triangle transformations. 
	As illustrated in Figure~\ref{fig:vTR}, we may define~$\bbL = \bbL_0,\dots, \bbL_{T} = \bbL'$, by essentially ``sliding'' the tracks of angle~$\alpha$ down through the~$\beta$-block. 
	To pass from~$\bbL_{t}$ to~$\bbL_{t+1}$ we perform a series of finite track exchanges, similarly to those in the transformation~$\bfS_t$ appearing in~\eqref{eq:5S_transformation}. 
	For~$r < t \leq  T-r$, the lattice~$\bbL_t$ contains a mixed block of~$2r$ tracks which slides down as~$t$ increases. 
	To pass from~$t$ to~$t+1$, exactly~$r$ track exchanges are performed, each consisting of~$2N$ sequential start-triangle transformations. 
	Let~$(\omega_t)_{0\leq t\leq T}$ be the Markov chain of configurations on the lattices~$\bbL_t$ defined as in Section~\ref{sec:5MarkovChain}, with~$\eta = N^{1-2c} \ll r$.

	A second way to  transform~$\bbL$ into~$\bbL'$ is to slide the tracks of angle~$\alpha$ from the right to the left; 
	write~$\tilde \bbL_0,\dots, \tilde \bbL_{\tilde T}$ for this set of transformations and~$(\tilde\omega_t)_{0 \leq t\leq \tilde T}$ for the corresponding chain of configurations. 
	
	The distortion of the contours of the large clusters in each configuration may be controlled in the same way as in Theorem~\ref{thm:linear}. 
	In particular, if we write~$\sfC_t$ and~$\tilde\sfC_{t}$ to be the cluster of the interval~$[-N,N] \times \{0\}$ in~$\omega_t$ and~$\tilde\omega_t$, respectively, 
	we find that 
	\begin{align}\label{eq:drift_vh}
	 	\lim_{R \to \infty} \bbE\big[{\rm R}(\sfC_{T}) - {\rm R}(\sfC_{0})\big] 
		&= 
		\frac{2r(\sin \alpha+ \sin\beta)}{ \sin \alpha+ \sin\beta + {\rm Drift}_{\rm vert}(\beta,\alpha)}
		\cdot {\rm Drift}_{\rm lat}(\beta,\alpha) + o(r)  \text{ and } \nonumber\\
	 	\lim_{R \to \infty} \bbE\big[{\rm R}(\tilde\sfC_{\tilde T}) - {\rm R}(\tilde\sfC_{0})\big]
		& = 
		\frac{2r(\sin (\beta - \alpha)+ \sin\beta)}{ \sin (\beta-\alpha)+ \sin\beta + {\rm Drift}_{\rm vert}(\beta,\beta-\alpha)}
		\cdot {\rm Drift}_{\rm vert}(\beta,\beta-\alpha)+ o(r).
	\end{align}
	In the above, the fractions are the approximate number of transformations that affect any mesoscopic cluster. 
	In the second line, the track-exchanges effectively occur between tracks of angle~$\pi-\beta$ and~$\alpha + \beta$ and the rightmost point of~$\tilde\sfC_t$ acts as a topmost point from the point of view of the track exchanges. We used the horizontal symmetry to state that~${\rm Drift}_{\rm vert}(\pi -\beta,\pi -\beta + \alpha) =  {\rm Drift}_{\rm vert}(\beta, \pi - \beta - \alpha)$; note however that~${\rm Drift}_{\rm lat}(\pi -\beta,\pi -\beta + \alpha) =- {\rm Drift}_{\rm lat}(\beta, \pi - \beta - \alpha)$, as the direction of the vector used for reference is reversed. 
	
	Finally, notice that~${\rm R}(\sfC_{0})$ and~${\rm R}(\tilde\sfC_{0})$ have the same law, as do~${\rm R}(\sfC_{T})$ and~${\rm R}(\tilde\sfC_{\tilde T})$. 
	Combining this with~\eqref{eq:drift_vh} and taking~$N \to\infty$, we find
	\begin{align*}
		\frac{(\sin \alpha+ \sin\beta)\cdot {\rm Drift}_{\rm lat}(\beta,\alpha) }{ \sin \alpha+ \sin\beta + {\rm Drift}_{\rm vert}(\beta,\alpha)}	
	=
		\frac{(\sin (\beta - \alpha)+ \sin\beta)\cdot {\rm Drift}_{\rm vert}(\beta,\beta-\alpha)}{ \sin (\beta-\alpha)+ \sin\beta + {\rm Drift}_{\rm vert}(\beta,\beta-\alpha)}
	\end{align*}
	Taking~$\alpha = \beta/2$ in the above, we conclude~\eqref{eq:drift_RT}.
\end{proof}

\begin{remark}\label{rem:drift_RT2}
	A different proof of Proposition~\ref{prop:drift_RT} may be obtained by transforming~$\bbL_\beta$ into~$\bbL_\alpha$ using a different set of transformations, namely those in Figure~\ref{fig:drift_RT2}. This method is detailed in Exercise~\ref{exo:drift_RT}. 	
\end{remark}

\begin{proof}[Proof of Lemma~\ref{lem:M_cont}]
	We start by the proof of continuity. 
	We will focus here on the vertical drift; the proof is identical for the lateral drift. 
	Fix two angles~$\alpha, \beta$. 
	Let~$\bbL_{\rm mix}(\alpha,\beta)$ be the mixed lattice with angles~$\alpha,\beta$. 
	Also, define~$\bbL_{\rm mix}(\alpha,\beta) \cap \La_R~$ for the restriction of this lattice to~$\La_R$, 
	with the addition of the rhombi needed to perform~$\bfS_1 \circ \bfS_0$. 
	
	Fix~$\eps > 0$. 
	Due to Lemma~\ref{lem:IIC}, specifically to the fact that it is uniform in the angles~$\alpha$ and~$\beta$ 
	(outside of a neighbourhood of~$0$ and~$\pi$) and to the locality of the star-triangle transformation, 
	for~$R$ large enough
	we may couple the application of 
	$\bfS_1 \circ \bfS_0$ to~$\phi^{\rm IIC,T}_{\bbL_{\rm mix}(\alpha,\beta)}$ 
	and to~$\phi^1_{\bbL_{\rm mix}(\alpha,\beta) \cap \La_R}[\cdot \,|\,{\rm Top}(\sfC_{x_n}) = 0 ]$ so that the increments 
	$\Delta_t^{\rm IIC} {\rm T}$ and~$\tilde \Delta_t^{\rm IIC} {\rm T}$ for the top of the IIC and of~$\sfC_{x_n}$, respectively,  satisfy
	\begin{align*}
		\bbP [\Delta_t^{\rm IIC} {\rm T} \neq \tilde \Delta_t^{\rm IIC} {\rm T}] \leq \eps. 
	\end{align*}
	Moreover, the choice of~$R$ may be uniform in a vicinity of~$(\alpha,\beta)$. 
	
	Since~$\Delta_t^{\rm IIC} {\rm T}$ is bounded, we conclude that 
	$(\alpha,\beta)\mapsto \phi^{\rm IIC,T}_{\bbL_{\rm mix}(\alpha,\beta)} [\Delta_t^{\rm IIC} {\rm T}]$ is 
	the locally uniform limit of 
	$(\alpha,\beta)\mapsto \phi^1_{\bbL_{\rm mix}(\alpha,\beta) \cap \La_R}[\Delta_t^{\rm IIC} {\rm T}\,|\,{\rm Top}(\sfC_{x_n}) = 0 ]$.
	The latter is a continuous function, and therefore so is the former.

	Finally, we turn to the invertibility of~$M_{\beta,\alpha}$. By the explicit formula~\eqref{eq:Mbetaalpha}, 
	this is equivalent to 
	\begin{align}
		{\rm Drift}_{\rm vert}  \neq -\sin\beta.
	\end{align}
	As explained in Section~\ref{sec:stability_meso}, the finite energy property and a deterministic bound on the increments of the IIC yield~${\rm Drift}_{\rm vert}  > -\sin\beta$.
\end{proof}

\section{Rotational invariance and universality: proofs of Theorems~\ref{thm:rotation_invariance} and~\ref{thm:universalCNSS}}\label{sec:conclusion_rot_inv}

With Theorem~\ref{thm:linear} now proved, we turn to our two main results: 
the asymptotic rotational invariance of~$\phi_{\bbL(\pi/2)}$ (Theorem~\ref{thm:rotation_invariance}) 
and the universality among lattices~$\bbL(\alpha)$ with~$\alpha \in (0,\pi)$  (Theorem~\ref{thm:universalCNSS}).
The latter may be simply formulated as~${\rm Drift}_{\rm lat}(\pi/2,\alpha)  = {\rm Drift}_{\rm vert}(\pi/2,\alpha) = 0$ for all~$\alpha \in (0,\pi)$. 
We will start by proving the asymptotic rotational invariance of~$\phi_{\bbL(\pi/2)}$ 
without Theorem~\ref{thm:universalCNSS}, then use it to deduce that~${\rm Drift}_{\rm lat}(\pi/2,\alpha)  = {\rm Drift}_{\rm vert}(\pi/2,\alpha) = 0$ for a dense set of angles~$\alpha \in (0,\pi)$, 
which will in turn imply Theorem~\ref{thm:universalCNSS}.

\subsection{Rotational invariance: proof of Theorem~\ref{thm:rotation_invariance}}\label{sec:deducing_rot_inv}

Write~$S_{\frac\alpha2}$ for the orthogonal reflection with respect to~$e^{{i}\alpha/2}$.
We start off with a result about the invariances of~$\phi_{\delta\bbL(\beta)}$ for general angles~$\beta$. 
Recall that~$\phi_{\bbL(\alpha)}$ is invariant under~$S_{\frac\alpha2}$, and since 
$\phi_{\bbL(\beta)}$ and~$\phi_{\bbL(\alpha)}$ are related via Theorem~\ref{thm:linear}, it follows that 
$\phi_{\bbL(\beta)}$ is asymptotically invariant under~$S_{\frac\alpha2}$ conjugated with~$M_{\beta,\alpha}$.

\begin{proposition}\label{prop:Tinv}
	For any~$\alpha, \beta \in (0,\pi)$, there exists constants~$c,C > 0$ such that
	\begin{align}\label{eq:autoT}
		{\bf d}_{\rm CN}\big[\phi_{\delta\bbL(\beta)},\, \phi_{\delta\bbL(\beta)}\circ M_{\beta,\alpha}^{-1}\circ S_{\alpha/2}\circ M_{\beta,\alpha}  \big] 	 
		\leq C\, \delta^c \qquad \text{for all~$\delta > 0$}.
	\end{align}
\end{proposition}

The proposition is an immediate consequence of the argument described above. We will not give a detailed proof.

We now turn to the proof of Theorem~\ref{thm:rotation_invariance}. For~$\alpha \in (0,\pi)$ write  
\begin{align*}
	T_\alpha = M_{\frac{\pi}2,\alpha}^{-1}\circ S_{\alpha/2}\circ M_{\frac{\pi}2,\alpha}.
\end{align*}
In light of~\eqref{eq:autoT}, we will say that~$\phi_{\bbL(\frac{\pi}2)}$ is {\em asymptotically invariant} under~$T_\alpha$.
Since this is the case for all~$\alpha \in (0,\pi)$, we conclude that~$\phi_{\bbL(\frac{\pi}2)}$ is asymptotically invariant 
with respect to the group generated by~$\{T_\alpha \,:\, \alpha \in (0,\pi)\}$.
To prove the asymptotic rotation invariance of~$\phi_{\bbL(\frac{\pi}2)}$, we will show that the group generated by~$\{T_\alpha \,:\, \alpha \in (0,\pi)\}$ contains all rotations.
To start, we list some properties of~$T_{\alpha}$.

\begin{proposition}\label{prop:T}
	\begin{itemize}
		\item[(i)] For each~$\alpha \in (0,\pi)$,~$T_\alpha$ has eigenvalues~$1$ and~$-1$
		with eigenvectors~$M_{\frac{\pi}2,\alpha}^{-1}  e^{{i}\alpha/2}$ and~$M_{\frac{\pi}2,\alpha}^{-1}  e^{{i}(\alpha/2 + \pi/2)}$, respectively;
		\item[(ii)]~$\alpha \mapsto T_\alpha$ is continuous over~$(0,\pi)$;
		\item[(iii)]~$\alpha \mapsto T_\alpha$ is not constant.
	\end{itemize}
\end{proposition}

\begin{proof}
	\noindent(i) It suffices to observe that~$M_{\frac{\pi}2,\alpha}$ acts as a change of basis, and maps the vectors 
		$M_{\frac{\pi}2,\alpha}^{-1} e^{{i}\alpha/2}$ and~$M_{\frac{\pi}2,\alpha}^{-1}  e^{{i}(\alpha/2 + \pi/2)}$ onto the eigenvectors of~$S_{\alpha/2}$. \smallskip

		\noindent(ii) This is a direct consequence of the continuity of~$\alpha \mapsto M_{\frac{\pi}2,\alpha}$, which was proved in Lemma~\ref{lem:M_cont}. \smallskip

		\noindent(iii) 
		Let us proceed by contradiction and assume that~$T_\alpha = T_{\pi/2} = S_{\pi/4}$ for all~$\alpha$. 
		The gist of this point is the following. Under the assumption above, the linear maps~$M_{\frac{\pi}2,\alpha}$ should become increasingly degenerate when~$\alpha \to 0$, which contradicts the uniformity of~\eqref{eq:RSW_iso} in the angle. Other options exist for proving that~$M_{\frac{\pi}2,\alpha}$ does not degenerate, such as an explicit analysis of the drift. 
		We include a formal derivation of this point for completeness.

		Under the assumption that~$T_\alpha = T_{\pi/2} = S_{\pi/4}$ and due to point (i),~$M_{\frac{\pi}2,\alpha}$ maps~$e^{{i}\pi/4}$  onto a multiple of~$e^{{i}\alpha/2}$ 
		and~$e^{{i}3\pi/4}$  onto a multiple of~$e^{{i}(\alpha/2 + \pi/2)}$.
		Moreover, recall from~\eqref{eq:Mbetaalpha} that~$M_{\frac{\pi}2,\alpha}$ acts as the identity on the horizontal axis. 
		This completely determines~$M_{\frac{\pi}2,\alpha}$: 
		\begin{align}\label{eq:degenerate_A}
			M_{\frac{\pi}2,\alpha} = \begin{pmatrix}1 & \cos \alpha \\0 & \sin \alpha \end{pmatrix}.
		\end{align}
		
		Write~${\rm Square}_\alpha$ for the square~$\{a e^{{i}\alpha/2} + b e^{{i}(\alpha/2 + \pi/2)}: 0\leq a,b\leq 1 \}$ 
		and~$\calC_h({\rm Square}_\alpha)$ for the event that the~${\rm Square}_\alpha$ contains a ``horizontal'' crossing, 
		that is a crossing between~$\{b e^{{i}(\alpha/2 + \pi/2)}: 0\leq b\leq 1 \}$ and~$\{e^{{i}\alpha/2} + b e^{{i}(\alpha/2 + \pi/2)}: 0\leq b\leq 1 \}$.
		The same notation applies to other rectangles. 
		Theorem~\ref{thm:linear} implies that 
		\begin{align}
			\big|\phi_{\delta \bbL(\alpha)}[\calC_h({\rm Square}_\alpha)] - \phi_{\delta \bbL(\pi/2)}[\calC_h(M_{\pi/2,\alpha}^{-1}{\rm Square}_\alpha)]\big| \to 0 \text{ as~$\delta \to 0$}. 
		\end{align}
		Notice now that, due to~\eqref{eq:degenerate_A},~$M_{\pi/2,\alpha}^{-1}{\rm Square}_\alpha$ 
		is a rectangle of aspect ratio~$\tan(\alpha/2)$, which tends to~$0$ as~$\alpha \to 0$.
		This implies that its crossing probability under~$\phi_{\delta \bbL(\pi/2)}$ may be made arbitrarily close to~$1$.
		Thus, we find that
		\begin{align*}
			\lim_{\alpha \to 0}\liminf_{\delta \to 0}\phi_{\delta \bbL(\alpha)}[\calC_h({\rm Square}_\alpha)]  = 1.
		\end{align*}
		This contradicts~\eqref{eq:RSW_iso}, particularly the fact that this property is uniform over~$\alpha$. 
\end{proof}

We are now in a position to prove the asymptotic rotational invariance of~$\phi_{\bbL(\frac{\pi}2)}$.

\begin{proof}[Proof of Theorem~\ref{thm:rotation_invariance}]
	As stated in Proposition~\ref{prop:Tinv},~$\phi_{\bbL(\frac{\pi}2)}$ is asymptotically invariant under all~$T_\alpha$ with~$\alpha \in (0,\pi)$. 
	In addition, it is also invariant under the vertical reflection~$S_0$, or equivalently with respect to the rotation by~$\pi/2$ 
	--- as opposed to~$S_{\pi/4}$,~$S_0$ is not part of~$\{T_\alpha\,:\,\alpha \in (0,\pi)\}$.
	As a consequence, it is asymptotically invariant under all transformation in the group generated by~$\{T_\alpha\,:\,\alpha \in (0,\pi)\}$ and~$S_0$. 
	We will prove that this group contains all rotations. 
		
	First, let us show that for all~$\alpha \in (0,\pi)$,~$T_\alpha$ is an {orthogonal} reflection. 
	Fix~$\alpha$ and let~$u,v \in \bbC$ be the eigenvectors of~$T_\alpha$ of Euclidean norm~$1$ and eigenvalues~$1$ and~$-1$, respectively, 
	contained in the upper half-plane.
	We will proceed by contradiction. 
	Suppose that the angle between~$u$ and~$v$ is different from~$\pi/2$; without loss of generality we may assume that it is strictly below~$\pi/2$. 
	Define the ellipse
	$$
		Q = 	\{x\, u  +y \, v : x,y \in \bbR,\, x^2 + y^2 \le 1 \},
	$$
	with~$u,v,-u,-v$ on its boundary (see Figure~\ref{fig:ellipse}). 
	Denote by~$(u,v), \dots, (-v,u)$ the boundary arcs delimited by these points. 
	Notice that~$T_\alpha$ maps~$Q$ onto itself, with the points~$u$,~$v$,~$-u$ and~$-v$ mapped to~$u$,~$-v$,~$-u$ and~$v$, respectively. 

	As a consequence of the asymptotic invariance of~$\phi_{\bbL(\frac{\pi}2)}$ under~$T_\alpha$, 
	the probabilities of crossing~$Q$ form~$(u,v)$ to~$(-u,-v)$ and from~$(v,-u)$ to~$(-v,u)$ are asymptotically equal: 
	\begin{align}
				\big|\phi_{\delta\bbL(\frac{\pi}2)}\big[(u,v)\xlra{Q} (-u,-v)\big]-\phi_{\delta\bbL(\frac{\pi}2)}\big[(v,-u)\xlra{Q} (-v,u)\big]\big| 
		\xrightarrow[\delta \to 0]{} 0.
	\end{align}
	As~$\phi_{\bbL(\frac{\pi}2)}$ is also invariant under the rotation\footnote{ which may be seen as~$S_{0} \circ S_{\pi/4}$.} of angle~$\pi/2$, we conclude that 
	the probability of crossing~$Q$ from~$(u,v)$ to~$(-u,-v)$ is asymptotically close to that of its rotation by~$\pi/2$:
	\begin{align}
		\big|\phi_{\delta\bbL(\frac{\pi}2)}\big[(u,v)\xlra{Q} (-u,-v)\big] 
			- \phi_{\delta\bbL(\frac{\pi}2)}\big[(e^{{i}\frac{\pi}2}u,e^{{i}\frac{\pi}2}v)\xlra{e^{{i}\frac{\pi}2}Q} (-e^{{i}\frac{\pi}2}u,-e^{{i}\frac{\pi}2}v)\big]\big| 
			\xrightarrow[\delta \to 0]{} 0.
	\end{align}
	Combining the above, we conclude that 
	\begin{align}\label{eq:ellipse}
		\big|\phi_{\delta\bbL(\frac{\pi}2)}\big[(v,-u)\xlra{Q} (-v,u)\big]
			- \phi_{\delta\bbL(\frac{\pi}2)}\big[(e^{{i}\frac{\pi}2}u,e^{{i}\frac{\pi}2}v)\xlra{e^{{i}\frac{\pi}2}Q} (-e^{{i}\frac{\pi}2}u,-e^{{i}\frac{\pi}2}v)\big]\big| 
			\to 0,
	\end{align}
as~$\delta \to 0$.
	
	Observe now (see Figure~\ref{fig:ellipse}) that
	\begin{align}
	\big\{(e^{{i}\frac{\pi}2}u,e^{{i}\frac{\pi}2}v)\xlra{e^{{i}\frac{\pi}2}Q} (-e^{{i}\frac{\pi}2}u,-e^{{i}\frac{\pi}2}v)\big\} \subset 
		\big\{(v,-u)\xlra{Q} (-v,u) \big\}.
	\end{align}
	Furthermore,~\eqref{eq:RSW_iso} implies that 
	\begin{align}
		\liminf_{\delta \to 0}\phi_{\delta\bbL(\frac{\pi}2)}\big[(v,-u)\xlra{Q} (-v,u) \big] - \phi_{\delta\bbL(\frac{\pi}2)}\big[(e^{{i}\frac{\pi}2}u,e^{{i}\frac{\pi}2}v)\xlra{e^{{i}\frac{\pi}2}Q} (-e^{{i}\frac{\pi}2}u,-e^{{i}\frac{\pi}2}v)\big] >0
	\end{align}	
	This contradicts~\eqref{eq:ellipse}, and therefore 
	invalidates our assumption that~$T_\alpha$ is not an orthogonal reflection.\smallskip
	
	\begin{figure}
    	\begin{center}
        	\includegraphics[width = 0.44\textwidth]{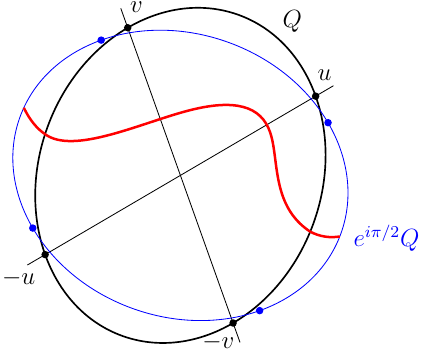}
        	\caption{If the angle between~$u$ and~$v$ is assumed strictly below~$\pi/2$, then~$Q$ is an ellipse.
        	The rotation by~$\pi/2$ of~$(u,v)\xlra{Q} (-u,-v)$ is realised by the red path; it has strictly lower probability than~$(v,-u)\xlra{Q} (-v,u)$.}
        	\label{fig:ellipse}
    	\end{center}
	\end{figure}

	We are now ready to conclude. 
	Write~$\theta(\alpha)$ for the angle between the real axis and the eigenvector~$M_{\frac{\pi}2,\alpha}^{-1} e^{{i}\alpha/2}$	of~$T_\alpha$. 
	Then the composition of~$S_0$ with~$T_\alpha$ is the rotation of angle~$2\theta(\alpha)$, 
	and we conclude that~$\phi_{\bbL(\pi/2)}$ is asymptotically invariant under this rotation. 
	
	From  Proposition~\ref{prop:T} (ii) and (iii), we conclude that~$\{\theta(\alpha): \, \alpha \in (0,\pi)\}$ 
	has a non-empty interior, whence we deduce that~$\phi_{\bbL(\pi/2)}$ is asymptotically invariant under any rotation. 
%
\end{proof}

\begin{remark}
	In the proof above, we used a special symmetry of the lattice~$\bbL(\pi/2)$, namely that with respect to the vertical reflection~$S_0$ 
	(or equivalently with respect to the rotation by~$\pi/2$). 
	The invariance of~$\bbL(\pi/2)$ with respect to~$S_{\pi/4}$, which corresponds to the generic invariance of~$\bbL(\alpha)$ with respect to~$S_{\alpha/2}$, 
	is not sufficient to conclude. 
	The best conclusion one could obtain without using the invariance under~$S_0$ is that the orbit of a given point under the group generated by~$\{T_\alpha: \, \alpha \in (0,\pi)\}$ 
	is an ellipse with axis~$e^{\pm \pi/4}\bbR$. 
	It is the vertical symmetry~(which is specific to~$\bbL(\pi/2)$ and has no correspondence for other  lattices~$\bbL(\alpha)$) 
	that allows to conclude that this ellipse is actually a circle. 
\end{remark}

\subsection{Universality: proof of Theorem~\ref{thm:universalCNSS}}\label{sec:drift0}

We turn to the proof of Theorem~\ref{thm:universalCNSS}, or equivalently to the fact that~$M_{\frac{\pi}2,\alpha} = {\rm id}$ for all~$\alpha$. 
Recall that~$\phi_{\bbL(\beta)}$ is said to be asymptotically rotationally invariant if, for any~$\theta \in [0,2\pi)$,
\begin{align*}
	{\bf d}_{\rm CN}(\phi_{\delta \bbL(\beta)},\phi_{e^{i\theta}\delta \bbL(\beta)}) \to 0 \quad\text{ as } \delta \to 0. 
\end{align*}
Similarly, say that~$\phi_{\bbL(\alpha)}$ and~$\phi_{\bbL(\beta)}$ are asymptotically similar if 
\begin{align*}
	{\bf d}_{\rm CN}(\phi_{\delta \bbL(\alpha)},\phi_{\delta \bbL(\beta)}) \to 0 \quad\text{ as } \delta \to 0. 
\end{align*}
Theorem~\ref{thm:linear} in particular states that~$\phi_{\bbL(\alpha)}$ and~$\phi_{\bbL(\beta)}$ are asymptotically similar if~$M_{\beta,\alpha} = {\rm id}$.

The key to the proof of this section is the following lemma. 

\begin{lemma}\label{lem:rotation_inv_alpha}
	Fix~$\beta \in (0,\pi)$ and assume that~$\phi_{\bbL(\beta)}$ is asymptotically rotationally invariant. 
	Then~$M_{\beta,\beta/2}  = {\rm id}$. 
\end{lemma}

\begin{proof}
	Fix such a value of~$\beta$ and set~$\alpha = \beta/2$. 
	Then, due to~\eqref{eq:drift_RT} and to the special choice of~$\alpha$, 
	we have~${\rm Drift}_{\rm vert}(\beta,\alpha) = {\rm Drift}_{\rm lat}(\beta,\alpha)$. 
	When injected in~\eqref{eq:Mbetaalpha} and after basic computation, we find that~$M_{\beta,\alpha}$ has the special form 
	\begin{align*}
	M_{\beta,\alpha} 
	= \begin{pmatrix}
	1 &  \frac{v}{\sin \alpha}\\[8pt]
	0 & 1  + \frac{v}{\cos \alpha}
	\end{pmatrix}
\qquad \text{where }v=  \frac{{\rm Drift}_{\rm lat}(\beta,\alpha)(1 +2\cos\alpha)}{2(\sin \alpha -{\rm Drift}_{\rm lat}(\beta,\alpha))}.
\end{align*}
	In particular, we notice that the vectors~$1$ and~$e^{{i}\alpha}$ are eigenvectors with eigenvalues~$1$ and~$\lambda:=1  + \frac{v}{\cos \alpha} > 0$, respectively. 
	
	Recall from Proposition~\ref{prop:Tinv} that~$\phi_{\bbL(\beta)}$ is asymptotically invariant under~$T:= M_{\beta,\alpha}^{-1} \circ S_{\alpha/2} \circ M_{\beta,\alpha}$.
	Since~$S_{\alpha/2}$ exchanges the vectors~$1$ and~$e^{{i}\alpha}$, 
	the transformation~$T$ maps the vector~$1$ to~$\lambda^{-1}e^{{i}\alpha}$ and~$e^{{i}\alpha}$ to~$\la$. 
	
	Consider the rhombic region~$R = \{x + y e^{{i}\alpha}: \,x,y\in [-1,1]\}$ as a quad with four points~$a,b,c,d$ at its corners, 
	in counter-clockwise order, starting with the top left corner. By the asymptotic invariance of~$\phi_{\delta\bbL(\beta)}$ with respect to~$T$, 
	\begin{align*}
	\lim_{\delta\to0} \phi_{\delta\bbL(\beta)}\big[(ab) \xlra{R} (cd)\big] - \phi_{\delta\bbL(\beta)}\big[(T(a)T(b)) \xlra{T(R)} (T(c)T(d))\big]  = 0. 
	\end{align*}
	Additionally, observe that~$R$ is stable under~$S_{\alpha/2}$. 
	Moreover,~$\phi_{\delta\bbL(\beta)}$ is invariant under~$S_{\beta/2}$ and is assumed asymptotically rotationally invariant,
	which implies that it is asymptotically invariant under all reflections. We conclude that
	\begin{align*}
	\lim_{\delta\to0} \phi_{\delta\bbL(\beta)}\big[(ab) \xlra{R} (cd)\big] - \phi_{\delta\bbL(\beta)}\big[(bc) \xlra{R} (da)\big]  = 0. 
	\end{align*}

	Combining the above, we find that 
	\begin{align*}
	\lim_{\delta\to0} \phi_{\delta\bbL(\beta)}\big[(bc) \xlra{R} (da)\big] - \phi_{\delta\bbL(\beta)}\big[(T(a)T(b)) \xlra{T(R)} (T(c)T(d))\big]  = 0. 
	\end{align*}
	
	\begin{figure}
	\begin{center}
	\includegraphics[width = 0.55\textwidth]{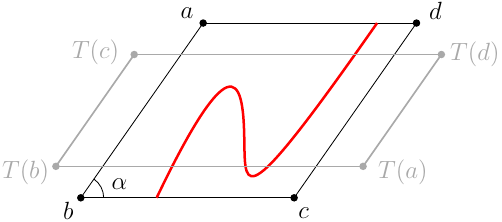}
	\caption{The rhombi~$R$ and~$T(R)$ under the assumption that~$\la > 1$. Any crossing in~$R$ between the arcs~$(bc)$  and~$(da)$ produces a crossing in~$T(R)$ between~$(T(a)T(b))$ and~$(T(c)T(d))$. However, the latter crossing is strictly easier to achieve, uniformly in the scale~$\delta$ of the lattice.}
	\label{fig:two_rhombi}
	\end{center}
	\end{figure}

	This is only possible if~$\la = 1$.
	Indeed, the quad~$T(R)$ may easily be identified as~$T(R) = \{\la y + x \la^{-1} e^{{i}\alpha}: \,x,y\in [-N,N]\}$.
	If~$\la > 1$, then~$T(R)$ is wider and shorter than~$R$ (see Figure~\ref{fig:two_rhombi}) and~\eqref{eq:RSW_iso} implies that 
\begin{align*}
	\limsup_{\delta\to0} \phi_{\delta\bbL(\beta)}\big[(bc) \xlra{R} (da)\big] - \phi_{\delta\bbL(\beta)}\big[(T(a)T(b)) \xlra{T(R)} (T(c)T(d))\big]  < 0. 
\end{align*}
	Conversely, if~$\la < 1$,~\eqref{eq:RSW_iso} shows that
\begin{align*}
	\liminf_{\delta\to0} \phi_{\delta\bbL(\beta)}\big[(bc) \xlra{R} (da)\big] - \phi_{\delta\bbL(\beta)}\big[(T(a)T(b)) \xlra{T(R)} (T(c)T(d))\big]  > 0. 
\end{align*}	
	We conclude that~$\la = 1$, which translates to~$M_{\beta,\alpha}  = {\rm id}$. 	
\end{proof}

\begin{proof}[Proof of Theorem~\ref{thm:universalCNSS}]
	Define the set 
	$$\calR = \{\alpha \in (0,\pi) : \,\phi_{\bbL(\alpha)} \text{ asymptotically similar to~$\phi_{\bbL(\pi/2)}$}\}.$$

	Theorem~\ref{thm:rotation_invariance} states that for any~$\alpha \in \calR$,~$\phi_{\bbL(\alpha)}$ is asymptotically rotationally invariant.
	In light of Lemma~\ref{lem:rotation_inv_alpha},~$\calR$ is stable by~$\alpha \mapsto \alpha/2$. 
	Moreover, by horizontal symmetry,~$\calR$ is stable by~$\alpha \mapsto \pi- \alpha$. 
	Repeatedly applying  these transformations shows that~$\calR$ is dense in~$(0,\pi)$. 
	
	Fix~$\alpha \in \calR$. That~$\phi_{\bbL(\pi/2)}$ and~$\phi_{\bbL(\alpha)}$ are asymptotically similar yields~$M_{\frac{\pi}2, \alpha} = {\rm id}$. 
	Due to the continuity of~$\alpha \mapsto M_{\frac{\pi}2,\alpha}$ (see Lemma~\ref{lem:M_cont})
	and to~$\calR$ being dense in~$(0,\pi)$, we conclude that~$M_{\frac{\pi}2, \alpha} = {\rm id}$ for all~$\alpha \in (0,\pi)$. 
	Apply  Theorem~\ref{thm:linear} to conclude. 
\end{proof}

\section{Exercises: isoradial FK-percolation}

\begin{exo}
	Show that if~$G$ is an isoradial graph, two train tracks intersect at most once and no train track intersects itself. 
\end{exo}

\begin{exo}
	Fix a simply connected union of rhombi~$G^\diamond$ that are part of a full-plane rhombic tiling. 
	Show that there exists a full-plane isoradial graph~$\calG$ such that 
	$\calG^\diamond$ contains~$G^\diamond$ as a subgraph
	and all tracks of~$\calG$ that do not intersect~$G^\diamond$ have transverse angles~$0$ or~$\pi/2$. 
	
	{\em Hint:} try to first complete~$G^\diamond$ into a convex union of rhombi.
\end{exo}

\begin{exo}
	The star-triangle transformation may be stated without referring to isoradial graphs as follows. 
	Fix the triangle and star graphs~$ABC$ and~$ABCO$, respectively, with boundary formed of the vertices~$A,B,C$. 
	Consider the FK-percolation measures~$\phi_{ABC}^\xi$ and~$\phi_{ABCO}^\xi$ on these graphs, with cluster weight~$q \geq 1$, 
	boundary conditions~$\xi$ and inhomogeneous parameters~$p_A,p_B,p_C$ and~$p_A',p_B',p_C'$
	on the edges~$BC$,~$CA$,~$AB$ and~$AO$,~$BO$ and~$CO$, respectively, as below. 
	
	\begin{center}
	\includegraphics[width = 0.5\textwidth]{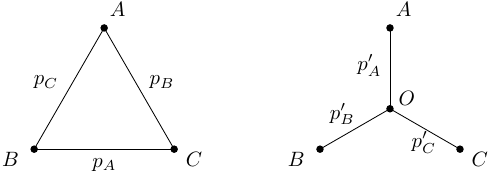}
	\end{center}
	
	Show that, for any boundary conditions~$\xi$, the connection probabilities between~$A$,~$B$ and~$C$ are the same in the two measures
	if and only if 
	\begin{align}
	y_A y_B y_C + y_A y_B + y_B y_C + y_C y_A = q \quad \text{ and } \quad y_* y_*' = q,& \\
	\text{ where }y_* = \tfrac{p_*}{1-p_*} \text{ and }y_*' = \tfrac{p_*'}{1-p_*'}\text{ for~$* \in \{A,B,C\}$}.& 
	\end{align}
	Show that there exists a unique continuous function~$\theta \mapsto y(\theta)$ on~$[0,\pi]$ satisfying 
	\begin{align}
	y(\theta_A)y(\theta_B) y(\theta_C) + y(\theta_A) y(\theta_B) + y(\theta_B) y(\theta_C) + y(\theta_C) y(\theta_A) = q  \text{ and } y(\theta)y(\pi-\theta) = q,\\
	\text{ for all~$\theta,\theta_A,\theta_B,\theta_C \in [0,\pi]$ with~$\theta_A + \theta_B + \theta_C = 2\pi$}.
	\end{align}
	Show that it is the one obtained from~\eqref{eq:isoraial_p_e}.
\end{exo}

\begin{exo}\label{exo:3_arms}
	The goal of this exercise is to prove~\eqref{eq:3hp_arm} and~\eqref{eq:3hp_arm2}. 
	We start by considering the case of~$\bbL_{\rm mix}$ for two arbitrary angles~$\alpha$,~$\beta$.  
	
	Write~$\bbT_N$ for the torus formed of~$N\times N$ cells of~$\bbL_{\rm mix}$. 
	Let~$\phi_{\bbT_N}$ be the FK-percolation measure on the torus with edge-weights given by~\eqref{eq:isoraial_p_e}.
	Fix constants~$R = \eps N$ and~$r < R$, where~$\eps > 0$ is some fixed quantity; all constants below depend on~$\eps$, 
	but should not depend on~$r, R$ or~$N$.  
	\begin{itemize}
		\item[(a)] Using~\eqref{eq:RSW_iso}, show that, uniformly in~$N$, the number of clusters in~$\bbT_N$ 
		with diameter at least~$R$ has an exponential tail. 
		\item[(b)] A point~$x \in \sfC$ with~$\sfC$ being a cluster of diameter at least~$R$ and
		\begin{align}\label{eq:almost_extremum}
		\langle x, e_{*}\rangle  \geq \max\{ \langle y, e_{*}\rangle : y \in \sfC\} - 2 
		\qquad  \text{ for some~$e_* \in \{\pm e_{\rm vert}, \pm e_{\rm lat}\}$}
		\end{align}
		 is called an {\em almost-extremum}.
		 Show that the number of almost-extrema has exponential tail, uniformly in~$R$. \smallskip\\
		{\em Hint:} Assuming that a configuration~$\omega$ has ``many'' almost-extrema, but ``few'' large clusters, 
		produce configurations~$\omega'$ by locally modifying~$\omega$ around some almost-extrema 
		so that each large cluster of~$\omega'$ has at most~$K$ almost-extrema (where~$K$ is some fixed number, say~$K = 100$).
		Observe that there are many ways to produce such configurations~$\omega'$ from~$\omega$, 
		but it is relatively ``easy'' to reconstruct~$\omega$ from~$\omega'$. \\
		We call this argument a ``multi-valued map'' principle. 
		\item[(c)] Use the periodicity of the torus to prove that the probability that~$\La_r$ contains an almost-extremum is of order~$(r/R)^2$.
		\item[(d)] Use mixing to transfer this estimate to the full-plane measure~$\phi_{\bbL_{\rm mix}}$.		
	\end{itemize}
	This proves~\eqref{eq:3hp_arm} and~\eqref{eq:3hp_arm2} for lattices~$\bbL_{\rm mix}$.
	
	Let us now deduce~\eqref{eq:3hp_arm} for a lattice~$\bbL(\balpha)$ for some sequence~$\balpha = (\alpha_j)_{j\in \bbZ}$ containing at most two values.
	For some fixed~$R$, write~$\bbL$ for the lattice with alternating angles~$\alpha_0,\alpha_1$ for tracks~$t_{-c_0R},\dots, t_{c_0R}$, angles~$\alpha_j$ for tracks~$t_{c_0+j}$ with~$j>0$ and for tracks~$t_{-c_0R+j}$ for~$j<0$, where~$c_0 = \frac2{\sin\alpha_0 + \sin\alpha_1}$.
	Write~$\bbL_{\rm mix}$ for the mixed lattice with angles~$\alpha_0$ and~$\alpha_1$. 
	Then~$\bbL$ is identical to~$\bbL_{\rm mix}$ for all tracks intersecting~$\La_R$. 
	\begin{itemize}
		\item[(e)] 
		Write \begin{align}
			{\rm Arm}_{\rm T}(r,R) &= \{\exists \sfC \text{ with }{\rm Top}(\sfC) \in \La_r \text{ and }\sfC \cap \La_R^c \neq \emptyset\}
		\qquad \text{ and }\\
		\overline{\rm Arm}_{\rm T}(r,R) &= \{\exists \sfC \text{ with }{\rm Top}(\sfC) \in \La_r \text{ and }\sfC \cap \bbR \times (-\infty,-R] \neq \emptyset\}.
		\end{align}
		Using~\eqref{eq:RSW_iso}, prove that there exists a universal constant~$c >0$ such that
		\begin{align}
			c\, \phi_{\bbL_{\rm mix}} [{\rm Arm}_{\rm T}(r,R)] 
			\leq \phi_{\bbL} \big[\overline{\rm Arm}_{\rm T}(r,2R)\big] 
			\leq \phi_{\bbL_{\rm mix}} [{\rm Arm}_{\rm T}(r,R)].
		\end{align}
		{\em Hint:} you may use the arm-separation principle, which states that with positive probability under~$\phi_{\bbL} [\cdot\,|\,{\rm Arm}_{\rm T}(r,R)]$ there exists a unique cluster~$\sfC$ realising the event~${\rm Arm}_{\rm T}(r,R)$, that~$\sfC$ is contained in~$\La_{2R}$ and surrounds the interval~$[-R/2,R/2] \times \{-R/2\}$. You may also try to prove the separation of arms. 
		\item[(f)] Using track exchanges, prove that 
		\begin{align}
			 \phi_{\bbL} \big[\overline{\rm Arm}_{\rm T}(r,2R)\big] = 	\phi_{\bbL'} \big[\overline{\rm Arm}_{\rm T}(r,2R)\big],
		\end{align}
		where~$\bbL'$ is a lattice identical to~$\bbL(\balpha)$ for all tracks intersecting~$\La_R$.
		\smallskip \\
		{\em Hint:} if~$t_{-N}$ is the lowest track included in the strip~$\bbR \times (-2R,2R]$, only perform track exchanges between tracks~$t_{-N},\dots, t_{-1}$ and~$t_2,t_3,\dots$.
		\item[(g)] Conclude that 
		\begin{align}
			c\, \phi_{\bbL_{\rm mix}} [{\rm Arm}_{\rm T}(r,R)] 
			\leq \phi_{\bbL(\balpha)} [{\rm Arm}_{\rm T}(r,R)] 
			\leq C \phi_{\bbL_{\rm mix}} [{\rm Arm}_{\rm T}(r,R)].
		\end{align}
	for universal constants~$c,C > 0$. 
	\item[(h)] Observe that, before proving that~$M_{\alpha, \beta} = {\rm id}$, there is no good reason to expect the lateral half-plane three-arm exponent to be equal to~$2$ on lattices~$\bbL(\balpha)$.
	\end{itemize}
\end{exo}

\begin{exo}\label{exo:3_arms3}
	The goal of this exercise is to prove~\eqref{eq:3hp_arm3}. Fix a lattice~$\bbL = \bbL(\balpha)$ for some sequence~$\balpha = (\alpha_j)_{j\in \bbZ}$ containing at most two values.
	
	For~$r < R$, let~${\rm Arm}_{\rm slit}(r,R)$ be the event that there exists three connections~$\gamma_1$,~$\gamma_2$ and~$\gamma_3$ in~$\bbL \setminus (\bbR_+\times \{0\})$
	between~$\La_r$ and~$\partial \La_R$, in counter-clockwise order, with~$\gamma_1,\gamma_3 \in \omega^*$ and~$\gamma_2 \in \omega$. 
	\begin{itemize}
		\item[(a)] Let~${\rm Arm}'_{\rm slit}(r,R)$ be the event that, in $\bbL \setminus (\bbR_+\times \{0\})$ there exist
		a dual connection between the top and bottom of $\La_R$ 
		and a primal connection between $\La_r$ and $\{-2R\}\times [-R,R]$ contained in $([-2R,R]\times [-R,R]) \setminus \La_r$.
		Prove that there exists some constant~$c > 0$ independent of~$r$ and~$R$ such that 
		\begin{align}
			\phi_{\bbL}[{\rm Arm}'_{\rm slit}(r,R)] \geq c \, \phi_{\bbL}[{\rm Arm}_{\rm slit}(r,R)].
		\end{align}
		{\em Hint:} use the arm-separation principle; see also Exercise~\ref{exo:3_arms} (e). 
		\item[(b)] Show that,
		\begin{align}
			\phi_{\bbL}[{\rm Arm}'_{\rm slit}(r,R)] \leq r/R.
		\end{align}
		{\em Hint:} at most one translate of the event~${\rm Arm}'_{\rm slit}(r,R)$ by~$(kr,0)$ for~$0 \leq k < R/r$ occurs. 
		\item[(c)] Use \eqref{eq:RSW_iso} to prove the existence of universal constants~$c,C > 0$ such that 
		\begin{align}
		\phi_{\bbL} [\exists \sfC \text{ with }{\rm Left}(\sfC) \in \La_r \text{ and }\sfC \cap \La_R^c \neq \emptyset] 
		\leq C (r/R)^{c}\phi_{\bbL}[{\rm Arm}_{\rm slit}(r,R)].
		\end{align}
		{\em Hint:} use the quasi-multiplicativity of arm events and show that at each ``scale'' between~$r$ and~$R$, 
		there is a positive probability to have configurations allowing the event on the right-hand side, but not that on the left-hand side. 
	\end{itemize}
	Conclude \eqref{eq:3hp_arm3}.
\end{exo}	

\begin{exo}
	Assuming that~$M_{\pi/2,\alpha} = {\rm id}$, follow the proof of 
	Theorem~\ref{thm:linear} and show that all constants may be chosen uniform in~$0 < \alpha < \pi/2$. 
	Note that there are~$O(\alpha^{-1}N)$ transformations in the process, 
	so one should show that the probability of error at each step of the process is~$O(\alpha N^{-c})$ for some constant~$c >0$.
\end{exo}

\begin{exo}\label{exo:IIC}
	The goal is to prove Lemma~\ref{lem:IIC} and Proposition~\ref{prop:IIC_def}. 
	Consider~$r < R$,~$\omega_0$ and~$\calC$ as in the statement of  Lemma~\ref{lem:IIC}. 
	Write 
	$$\phi[\cdot\,|\,\omega_0,\, \calC] := \phi_{\bbL_{\rm mix}}[\cdot\,|\, \omega = \omega_0\text{ on } \La_R^c,\,\sfC_x \cap \La_R^c = \calC,\, {\rm Top}(\sfC_x)=0].$$
	This is a measure on configurations in $\La_R$; the conditioning on $\omega_0$ is degenerate, but should be understood simply as boundary conditions. 
	For simplicity assume~$R = 2^K$ and~$r = 2^k$ for integers~$k < K$. 
	
	For a configuration~$\omega$, define the following notions. 
	Write~$\calE_j$ for the union of~$\La_{2^{K-j}}^c$ and all the primal/dual interfaces starting on~$\partial \La_{2^{K-j}}$ towards the inside of~$\La_{2^{K-j}}$, 
	explored up to the first time they exit~$\La_{2^{K-j}}$ or enter~$\La_{2^{K-j-1}}$. Write~$\calF_j$ for the connected component of~$0$ in~$\calE_j^c$. 
	The boundary of~$\calF_j$ is formed of alternating primal and dual arcs, also called petals; 
	we call~$\calF_j$ the flower domain in~$\La_{2^{K-j}} \setminus \La_{2^{K-j-1}}$ explored from the outside.

	Write~$\xi_j$ for the boundary condition on~$\calF_j$ induced by the configuration outside of~$\calF_j$; 
	it encodes how the primal petals of~$\calF_j$ are wired together outside of~$\calF_j$.
	We will abuse notation, and also encode in~$\xi_j$ which primal petals are connected to which parts of~$\calC$.	
	Say~$\xi_j$ is {\em simple} if it contains a single wired component that is connected to~$\La_R^c$ --- this component may be formed of several petals. 
	
	Set~$J= K -k-1$. When referring to the marginals in~$\La_r$ notice that
	$$\phi[\cdot\,|\,\omega_0,\, \calC,\,  \omega \text{ on }\calE_J] = \phi_{\calF_{J}}^{\xi_J}[\cdot \,|\,{\rm Top}(\sfC_{\xi_J})=0],$$
	if~$\xi_J$ is simple, where~$\sfC_{\xi_J}$ is the cluster of the single wired component of~$\xi_J$ that is connected to~$\La_R^c$.

	\begin{itemize}
		\item[(a)]  	Argue that~$(\calF_j,\xi_j)_{j = 0,\dots,J}$ is a Markov chain when~$\omega$ is sampled using~$\phi[\cdot|\omega_0,\, \calC]$; write~$P$ for the law of this chain.
	Show that if~$\xi_j$ is simple, then all following~$(\xi_\ell)_{\ell \geq j}$ are also simple. 
	
	\item[(b)] Show that the Markov chain {\em mixes} in that, for all~$j \leq J-3$,
	\begin{align}\label{eq:exo_IIC_mix1}
		c \leq \frac{P[\calF_{j+3} = G \text{ and } \xi_{j+3} = \eta\,|\, \calF_{j} = F\text{ and } \xi_{j} = \zeta]}
		{P[\calF_{j+3} = G\text{ and } \xi_{j+3} = \eta\,|\, \calF_{j} = F'\text{ and } \xi_{j} = \zeta']}
		\leq C,
	\end{align}
	for universal constants~$c,C > 0$, any potential realisations~$F, \zeta$ and~$F,\zeta'$ of~$\calF_{j}, \xi_{j}$, 
	and any realisation~$G, \eta$ of~$\calF_{j+3}, \xi_{j+3}$ which contains only two petals, with the primal petal having a ``large'' opening and lying ``deep'' in the lower half-plane. 
	Also show that 
	\begin{align}\label{eq:exo_IIC_mix2}
		\sum_{G,\eta}P[(\calF_{j+3}, \xi_{j+3}) = (G,\eta)\,|\, \calF_{j} = F\text{ and } \xi_{j} = \zeta] \geq c,
	\end{align}
	for some universal constant~$c >0$, where the sum is over~$(G,\eta)$ as described above. 
	The notions of ``large'' and ``deep'' should be defined. 
	
	\item[(c)]
	From~\eqref{eq:exo_IIC_mix1} and~\eqref{eq:exo_IIC_mix2}, deduce that, for any~$A$ depending only on~$\La_r$, 
	$$\Big| \frac{\phi[A|\omega_0,\, \calC]}{\phi[A|\omega_0',\, \calC']}-1\Big| \leq C (r/R)^c,$$
	for universal constants~$c,C>0$. 
	Use this to show Proposition~\ref{prop:IIC_def} and Lemma~\ref{lem:IIC}.\medskip 
	
	\noindent{\em Hint:} The main difficulty above is proving~\eqref{eq:exo_IIC_mix1} and~\eqref{eq:exo_IIC_mix2}. 
	Assuming~$\calF_{j} = F$ and~$\xi_{j} = \zeta$, start off by sampling the flower domain~$\calF_{\rm in}$ in~$\La_{2^{K-j-2}} \setminus \La_{2^{K-j-3}}$ explored from the outside 
	and the flower domain~$\calF_{\rm out}$ in~$\La_{2^{K-j-1}} \setminus \La_{2^{K-j-2}}$ explored from the inside, along with all edges between them, under the measure  
	$\phi_{F}^{\zeta}$.
	Use mixing to argue that this exploration is somewhat independent of~$F$ and~$\zeta$. 
	Use~\eqref{eq:RSW_iso} to argue that, with positive probability, both flower domains contain a single primal petal which is large and lies deep in the lower half-plane, 
	and that the two primal petals are connected in the lower half-plane. We say that~$(\calF_{\rm in},\calF_{\rm out})$ is {\em safe} in this case. 
	
	Notice that~$(\calF_{\rm in},\calF_{\rm out})$ sampled above are not sampled according to~$P[\cdot\, |\, \calF_{j} = F\text{ and } \xi_{j} = \zeta]$, 
	as we have not conditioned on any event inside~$\calF_j$. 
	To account for this conditioning, one should factor in the probability of~$\{\sfC_x \cap \La_R^c = \calC,\, {\rm Top}(\sfC_x)=0\}$ given~$\calF_{\rm in}$ and~$\calF_{\rm out}$. 
	Show that 
	\begin{align}
		&\phi[ \sfC_x \cap \La_R^c = \calC,\, {\rm Top}(\sfC_x)=0\,|\,\omega \text{ on }  \calE_{j},\, \calF_{\rm in},\, \calF_{\rm out} ]\\
		&\qquad \leq c \, \pi_3^{\rm T} (r,2^{K-j-3})\, 
		\phi[\text{all clusters of~$\calC$ intersect~$\La_{2^{K-j-1}}$} \,|\, \calE_{j}]
	\end{align}
	for some constant~$c > 0$ and any safe~$(\calF_{\rm in},\calF_{\rm out})$.
	Conversely, show that 
	\begin{align}
	&\phi[\sfC_x \cap \La_R^c = \calC,\, {\rm Top}(\sfC_x)=0 \,|\, \omega \text{ on }  \calE_{j}] \\
	&\qquad \leq C \pi_3^{\rm T} (r,2^{K-j-3}) \,
		\phi[\text{all clusters of~$\calC$ intersect~$\La_{2^{K-j-1}}$} \,|\, \calE_{j}]
	\end{align}
	for some~$C > 0$, where~$\pi_3^{\rm T}(r,R)$ is the probability of half-plane three-arm event between scales~$r$ and~$R$ defined in~\eqref{eq:3hp_arm}. 
	From this,  conclude~\eqref{eq:exo_IIC_mix1} and~\eqref{eq:exo_IIC_mix2}.
	\end{itemize}
\end{exo}

\begin{figure}
\begin{center}
\includegraphics[width = 1.0\textwidth]{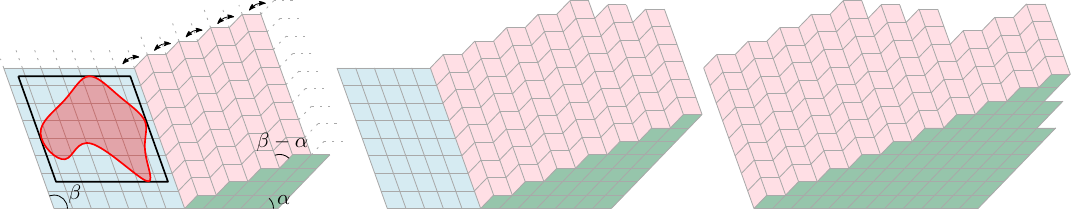}
\caption{A different set of track exchanges than those of Section~\ref{sec:5sketch} allow to transform~$\bbL_\beta$ into~$\bbL_\alpha$. 
Consider the lattice~$\bbL_0$ on the left which is formed of three blocks, which we call the~$\beta$-block (light blue), the~$\alpha$-block (green) and the mixed block (purple). The observation window~$[-N,N] \times [0,N]$ is fully contained in the~$\beta$-block. 
Applying semi-infinite track exchanges between the tracks of the mixed block makes the interfaces between the mixed- and~$\alpha$- and~$\beta$-blocks move to the left. After sufficiently many such transformations, the observation window will be fully contained in the~$\alpha$-block.
Use the same system of coordinates~$e_{\rm vert}$ and~$e_{\rm lat}$ as in the proof of Theorem~\ref{thm:linear} to track the movements of clusters;
in this system, the drifts are~$-{\rm Drift}_{\rm lat}(\beta,\beta- \alpha)$ and~${\rm Drift}_{\rm vert}(\beta,\beta- \alpha)$, respectively. }
\label{fig:drift_RT2}
\end{center}
\end{figure}

\begin{exo}(Different proof for the equality of drifts)\label{exo:drift_RT}
	Fix~$\alpha,\beta \in (0,\pi)$ with~$\alpha < \beta$. The goal of the exercise is to use a different strategy than that of Theorem~\ref{thm:linear} to transform~$\bbL_{\beta}$ into~$\bbL_\alpha$. 
	This will ultimately offer an alternative proof of Proposition~\ref{prop:drift_RT}.
	
	Consider the lattices~$(\bbL_t)_{t\geq 0}$ described in Figure~\ref{fig:drift_RT2}. We assume that we work at a scale~$N$ and follow the evolution of clusters in the observation window~$[-N,N]\times[0,N]$. We use the same vocabulary as in the proof of Theorem~\ref{thm:linear}. 
	\begin{itemize}
	\item[(a)] Argue that, for a mesoscopic cluster~$\sfC$ in the mixed block at some time~$t$,  
	\begin{align}
		&\bbE[\Delta_t{\rm T}(\sfC)] = -{\rm Drift}_{\rm lat}(\pi-\beta,\pi-\beta+ \alpha) +O(N^{1-c}) = {\rm Drift}_{\rm lat}(\beta,\beta - \alpha) +O(N^{1-c})
		\quad\text{and}\\
		&	\bbE[\Delta_t{\rm R}(\sfC)] = {\rm Drift}_{\rm vert}(\pi-\beta,\pi-\beta+ \alpha) +O(N^{1-c})= {\rm Drift}_{\rm vert}(\beta,\beta - \alpha) +O(N^{1-c}).\label{eq:drift_RT}
	\end{align}
	\item[(b)] Write the drift in cartesian coordinates corresponding to the movement in~\eqref{eq:drift_RT} as
	\begin{align}
	d = (d_1,d_2) = \big({\rm Drift}_{\rm vert}(\beta,\beta- \alpha)\tfrac{ 1}{\sin\beta} +  {\rm Drift}_{\rm lat}(\beta,\beta- \alpha)\tfrac{\cos\beta}{\sin\beta} \,,\, {\rm Drift}_{\rm lat}(\beta,\beta- \alpha)\big).
	\end{align}
	\item[(c)] Consider a point~$x(0) = (x_1(0),x_2(0))$ in the~$\beta$-block of~$\bbL_0$. Define the deterministic evolution
	\begin{align}
	x_{t+1} = \begin{cases}
			x_t &\text{ if~$x_t$ is in the~$\alpha$ or~$\beta$-blocks of~$\bbL_t$}, \\
			x_t + d \quad& \text{ if~$x_t$ is in the mixed block of~$\bbL_t$}.
			\end{cases}
	\end{align}
	Argue that, the total time spent by~$(x_t)_{t\geq 0}$ in the mixed block is~$ x_2(0)\cdot {\tau}$ where
	\begin{align}
		{\tau} =  \frac{\sin(\beta - \alpha/2)}{\sin( \alpha/2) \sin\beta + {\rm Drift}_{\rm vert}(\beta,\beta- \alpha)\sin( \alpha/2) -  {\rm Drift}_{\rm lat}(\beta,\beta- \alpha)\sin(\beta - \alpha/2)}.
	\end{align}
	Conclude that final position of~$x(t)$ is 
	\begin{align}
	x(0) + x_2(0)\Big(\frac{{\rm Drift}_{\rm vert}(\beta,\beta- \alpha) +  {\rm Drift}_{\rm lat}(\beta,\beta- \alpha) \cos\beta}{\sin\beta} \cdot {\tau}\,,\,   {\rm Drift}_{\rm lat}(\beta,\beta- \alpha)  \cdot {\tau}\Big).
	\end{align}
	Write \begin{align}
		\tilde M_{\beta,\alpha} = 
	\begin{pmatrix}
	1 & \frac{{\rm Drift}_{\rm vert}(\beta,\beta- \alpha) +  {\rm Drift}_{\rm lat}(\beta,\beta- \alpha) \cos\beta}{\sin\beta} \cdot {\tau} \\
	0 & 1 +     {\rm Drift}_{\rm lat}(\beta,\beta- \alpha)  \cdot {\tau}
	\end{pmatrix}.
	\end{align}\smallskip
	
\noindent{\em Hint:} Beware that the angle of the interface between the mixed and~$\alpha$-block is~$\alpha/2$. 
	\item[(d)] In the same way as in Theorem~\ref{thm:linear},  show that
	\begin{align}
		{\bf d}_{\rm CN}\big[ \phi_{\delta\bbL(\beta)}, \phi_{\delta\bbL(\alpha)}\circ \tilde M_{\beta,\alpha}\big] \to 0 \text{ as~$\delta \to 0$}.
	\end{align}
	\item[(e)] Using  Theorem~\ref{thm:linear} and the above, argue that~$\phi_{\bbL(\alpha)}$ is asymptotically invariant under~$M_{\beta,\alpha}^{-1}\cdot \tilde M_{\beta,\alpha}$ and conclude that~$M_{\beta,\alpha}= \tilde M_{\beta,\alpha}$.\smallskip \\
	{\em Hint:} The triangular structure of~$M_{\beta,\alpha}$ and~$\tilde M_{\beta,\alpha}$ comes into play. 
	\item[(e)] When~$\alpha = \beta/2$ notice that 
	\begin{align}
		{\tau}
=	\frac{\sin \alpha + \sin\beta}{ \sin\alpha \sin\beta + ({\rm Drift}_{\rm vert}(\beta,\alpha) -  {\rm Drift}_{\rm lat}(\beta,\alpha))\sin \alpha -  {\rm Drift}_{\rm lat}(\beta,\alpha) \sin\beta}
	\end{align}
	{\em Hint:} use the trigonometric identity~$\frac{\sin(3\alpha/2)}{\sin (\alpha/2)}\sin \alpha = \sin \alpha + \sin(2\alpha)$. 
	\item[(f)] From point (e), conclude that~${\rm Drift}_{\rm vert}(\beta,\beta/2) =  {\rm Drift}_{\rm lat}(\beta,\beta/2)$.
		\end{itemize}  
\end{exo}

\begin{exo}\label{exo:drift_u}(Drift in arbitrary direction)
	Fix a unit vector~$\vec u \in \bbR^2$ distinct from~$\pm e_{\rm vert}$. 
	\begin{itemize}
	\item[(a)] Following the steps of Section~\ref{sec:IIC}, define a drift~${\rm Drift}_{\vec u}$ in the direction~$\vec u$, in the same way that~${\rm Drift}_{\rm lat}$ is the drift in the direction~$e_{\rm lat}$. 
	\item[(b)] Check that the proof of Theorem~\ref{thm:linear} also applies using~$\vec u$ instead of~$e_{\rm lat}$, and conclude that 
	\begin{align}
		{\bf d}_{\rm CN}\big[ \phi_{\delta\bbL(\beta)}, \phi_{\delta\bbL(\alpha)}\circ \tilde M_{\beta,\alpha}\big] \to 0 \text{ as~$\delta \to 0$},
	\end{align}
	where 
	\begin{align}
\tilde M_{\beta,\alpha} 
= \begin{pmatrix}
1 &  \frac{{\rm Drift}_{\vec u} +  {\rm Drift}_{\rm vert} \cos \theta }{\sin \theta}\cdot \frac{\sin\alpha + \sin\beta}{(\sin \alpha -{\rm Drift}_{\rm vert})\sin\beta}\\[8pt]
0 & 1  + {\rm Drift}_{\rm vert} \cdot \frac{\sin\alpha + \sin\beta}{(\sin \alpha -{\rm Drift}_{\rm vert})\sin\beta}
\end{pmatrix}
	\end{align}
	and~$\theta$ is the angle between the vertical axis and~$-\vec u$.
\item[(c)] Conclude that~$\tilde M_{\beta,\alpha} =  M_{\beta,\alpha}$ and therefore that~$\frac{{\rm Drift}_{\vec u} +  {\rm Drift}_{\rm vert} \cos \theta }{\sin \theta}$ does not depend on~$\theta$. 
	\end{itemize} 
\end{exo}

\bibliographystyle{alpha}
\bibliography{biblicomplete}

\newcommand{\etalchar}[1]{$^{#1}$}
\begin{thebibliography}{DCKK{\etalchar{+}}22}

\bibitem[AB87]{AizBar87}
M.~Aizenman and D.~J. Barsky.
\newblock Sharpness of the phase transition in percolation models.
\newblock {\em Comm. Math. Phys.}, 108(3):489--526, 1987.

\bibitem[Bax78]{Bax78}
R.~J. Baxter.
\newblock Solvable eight-vertex model on an arbitrary planar lattice.
\newblock {\em Philos. Trans. Roy. Soc. London Ser. A}, 289(1359):315--346,
  1978.

\bibitem[Bax89]{Bax89}
Rodney~J. Baxter.
\newblock {\em Exactly solved models in statistical mechanics}.
\newblock Academic Press Inc. [Harcourt Brace Jovanovich Publishers], London,
  1989.
\newblock Reprint of the 1982 original.

\bibitem[BD12]{BefDum12}
V.~Beffara and H.~{Duminil-Copin}.
\newblock The self-dual point of the two-dimensional random-cluster model is
  critical for {$q\geq 1$}.
\newblock {\em Probab. Theory Related Fields}, 153(3-4):511--542, 2012.

\bibitem[Bet31]{Bet31}
Hans Bethe.
\newblock Zur theorie der metalle.
\newblock {\em Zeitschrift f{\"u}r Physik}, 71(3-4):205--226, 1931.

\bibitem[BK89]{BurKea89}
R.~M. Burton and M.~Keane.
\newblock Density and uniqueness in percolation.
\newblock {\em Comm. Math. Phys.}, 121(3):501--505, 1989.

\bibitem[BKK{\etalchar{+}}92]{BouKahKal92}
Jean Bourgain, Jeff Kahn, Gil Kalai, Yitzhak Katznelson, and Nathan Linial.
\newblock The influence of variables in product spaces.
\newblock {\em Israel J. Math.}, 77(1-2):55--64, 1992.

\bibitem[BKW76]{BaxKelWu76}
R.~J. Baxter, S.~B. Kelland, and F.~Y. Wu.
\newblock Equivalence of the potts model or whitney polynomial with an ice-type
  model.
\newblock {\em Journal of Physics A: Mathematical and General}, 9(3):397--406,
  1976.

\bibitem[BR10]{BolRio10}
B{\'{e}}la Bollob{{\'a}}s and Oliver Riordan.
\newblock Percolation on self-dual polygon configurations.
\newblock In {\em An irregular mind}, volume~21 of {\em Bolyai Soc. Math.
  Stud.}, pages 131--217. J{\'a}nos Bolyai Math. Soc., Budapest, 2010.

\bibitem[BS17]{BasSap17}
Deepan Basu and Artem Sapozhnikov.
\newblock {Kesten's incipient infinite cluster and quasi-multiplicativity of
  crossing probabilities}.
\newblock {\em Electronic Communications in Probability}, 22:1 -- 12, 2017.

\bibitem[Che24]{Che20}
Dmitry Chelkak.
\newblock Ising model and s-embeddings of planar graphs.
\newblock {\em Annales scientifiques de l'{{\'E}}cole {N}ormale
  {S}up{\'e}rieure}, 57(5):1271--1346, 2024.

\bibitem[CN06]{CamNew06}
Federico Camia and Charles~M. Newman.
\newblock Two-dimensional critical percolation: the full scaling limit.
\newblock {\em Comm. Math. Phys.}, 268(1):1--38, 2006.

\bibitem[CS12]{CheSmi12}
D.~Chelkak and S.~Smirnov.
\newblock Universality in the 2{D} {I}sing model and conformal invariance of
  fermionic observables.
\newblock {\em Invent. Math.}, 189(3):515--580, 2012.

\bibitem[DC20]{Dum20}
Hugo Duminil-Copin.
\newblock Lectures on the {I}sing and {P}otts models on the hypercubic lattice.
\newblock In {\em Random Graphs, Phase Transitions, and the Gaussian Free
  Field}, pages 35--161. Springer International Publishing, 2020.

\bibitem[DCGH{\etalchar{+}}18]{DumGagHar16b}
Hugo Duminil-Copin, Maxime Gagnebin, Matan Harel, Ioan Manolescu, and Vincent
  Tassion.
\newblock The {B}ethe ansatz for the six-vertex and {XXZ} models: {A}n
  exposition.
\newblock {\em Probability Surveys}, 15:102--130, 2018.

\bibitem[DCGH{\etalchar{+}}21]{DumGagHar16}
Hugo Duminil-Copin, Maxime Gagnebin, Matan Harel, Ioan Manolescu, and Vincent
  Tassion.
\newblock Discontinuity of the phase transition for the planar random-cluster
  and {P}otts models with $q>4$.
\newblock {\em Annales scientifiques de l'{{\'E}}cole {N}ormale
  {S}up{\'e}rieure}, 54(6), 2021.

\bibitem[DCKK{\etalchar{+}}20]{DumKozKra20}
Hugo Duminil-Copin, Karol~Kajetan Kozlowski, Dmitry Krachun, Ioan Manolescu,
  and Mendes Oulamara.
\newblock Rotational invariance in critical planar lattice models.
\newblock 2020.
\newblock preprint arXiv:2012.11672.

\bibitem[DCKK{\etalchar{+}}22]{DumKozKra22}
Hugo Duminil-Copin, Karol~Kajetan Kozlowski, Dmitry Krachun, Ioan Manolescu,
  and Tatiana Tikhonovskaia.
\newblock On the six-vertex model's free energy.
\newblock {\em Communications in Mathematical Physics}, 395(3):1383--1430,
  2022.

\bibitem[DCLM18]{DumLiMan18}
Hugo Duminil-Copin, Jhih-Huang Li, and Ioan Manolescu.
\newblock Universality for the random-cluster model on isoradial graphs.
\newblock {\em Electronic Journal of Probability}, 23, 2018.

\bibitem[DCM22]{DumMan20}
Hugo Duminil-Copin and Ioan Manolescu.
\newblock Planar random-cluster model: scaling relations.
\newblock {\em Forum of Mathematics, Pi}, 10:e23, 2022.

\bibitem[DCRT19]{DumRaoTas19}
Hugo Duminil-Copin, Aran Raoufi, and Vincent Tassion.
\newblock Sharp phase transition for the random-cluster and {P}otts models via
  decision trees.
\newblock {\em Ann. of Math. (2)}, 189(1):75--99, 2019.

\bibitem[DCST17]{DumSidTas17}
Hugo Duminil-Copin, Vladas Sidoravicius, and Vincent Tassion.
\newblock Continuity of the phase transition for planar random-cluster and
  {P}otts models with {$1 \leq q \leq 4$}.
\newblock {\em Comm. Math. Phys.}, 349(1):47--107, 2017.

\bibitem[DCT20]{DumTas19}
Hugo Duminil-Copin and Vincent Tassion.
\newblock Renormalization of crossing probabilities in the planar
  random-cluster model.
\newblock {\em Moscow Mathematical Journal}, 20(4):711--740, 2020.

\bibitem[DT16a]{DumTas16}
H.~{Duminil-Copin} and V.~Tassion.
\newblock A new proof of the sharpness of the phase transition for {B}ernoulli
  percolation and the {I}sing model.
\newblock {\em Communications in {M}athematical {P}hysics}, 343(2):725--745,
  2016.

\bibitem[DT16b]{DumTas16a}
H.~{Duminil-Copin} and V.~Tassion.
\newblock A new proof of the sharpness of the phase transition for {B}ernoulli
  percolation on $\mathbb{Z}^d$.
\newblock {\em Enseignement Math\'ematique}, 62(1-2):199--206, 2016.

\bibitem[Duf68]{Duf68}
R.~J. Duffin.
\newblock Potential theory on a rhombic lattice.
\newblock {\em J. Combinatorial Theory}, 5:258--272, 1968.

\bibitem[{Dum}13]{Dum13}
H.~{Duminil-Copin}.
\newblock {\em Parafermionic observables and their applications to planar
  statistical physics models}, volume~25 of {\em Ensaios Matematicos}.
\newblock Brazilian Mathematical Society, 2013.

\bibitem[GG06]{GraGri06}
B.~T. Graham and G.~R. Grimmett.
\newblock Influence and sharp-threshold theorems for monotonic measures.
\newblock {\em Ann. Probab.}, 34(5):1726--1745, 2006.

\bibitem[GG11]{GraGri11}
Benjamin Graham and G.~Grimmett.
\newblock Sharp thresholds for the random-cluster and {I}sing models.
\newblock {\em Ann. Appl. Probab.}, 21(1):240--265, 2011.

\bibitem[GL23]{GlaLam23}
Alexander Glazman and Piet Lammers.
\newblock Delocalisation and continuity in {2D}: loop {O(2)}, six-vertex, and
  random-cluster models.
\newblock 2023.
\newblock preprint arXiv:2306.01527.

\bibitem[GM90]{GriMar90}
G.~R. Grimmett and J.~M. Marstrand.
\newblock The supercritical phase of percolation is well behaved.
\newblock {\em Proc. Roy. Soc. London Ser. A}, 430(1879):439--457, 1990.

\bibitem[GM13a]{GriMan13}
G.~R. Grimmett and Ioan Manolescu.
\newblock Inhomogeneous bond percolation on square, triangular and hexagonal
  lattices.
\newblock {\em Ann. Probab.}, 41(4):2990--3025, 2013.

\bibitem[GM13b]{GriMan13a}
G.~R. Grimmett and Ioan Manolescu.
\newblock Universality for bond percolation in two dimensions.
\newblock {\em Ann. Probab.}, 41(5):3261--3283, 2013.

\bibitem[GM14]{GriMan14}
G.~R. Grimmett and Ioan Manolescu.
\newblock Bond percolation on isoradial graphs: criticality and universality.
\newblock {\em Probab. Theory Related Fields}, 159(1-2):273--327, 2014.

\bibitem[GPS13]{GarPetSch13}
Christophe Garban, G{{\'a}}bor Pete, and Oded Schramm.
\newblock Pivotal, cluster, and interface measures for critical planar
  percolation.
\newblock {\em J. Amer. Math. Soc.}, 26(4):939--1024, 2013.

\bibitem[Gri06]{Gri06}
G.~Grimmett.
\newblock {\em The random-cluster model}, volume 333 of {\em Grundlehren der
  Mathematischen Wissenschaften [Fundamental Principles of Mathematical
  Sciences]}.
\newblock Springer-Verlag, Berlin, 2006.

\bibitem[Gri10]{Gri10}
G.~Grimmett.
\newblock {\em Probability on graphs}, volume~1 of {\em Institute of
  Mathematical Statistics Textbooks}.
\newblock Cambridge University Press, Cambridge, 2010.
\newblock Random processes on graphs and lattices.

\bibitem[HM24]{HanMan24+}
Ulrik~Thinggaard Hansen and Ioan Manolescu.
\newblock Universality of large-scale geometry for critical random-cluster
  models.
\newblock {\em in preparation}, 2024.

\bibitem[J{\'a}r03]{Jar03}
Antal~A. J{\'a}rai.
\newblock {Incipient infinite percolation clusters in 2D}.
\newblock {\em The Annals of Probability}, 31(1):444 -- 485, 2003.

\bibitem[Ken02]{Ken02}
R.~Kenyon.
\newblock The {L}aplacian and {D}irac operators on critical planar graphs.
\newblock {\em Invent. Math.}, 150(2):409--439, 2002.

\bibitem[Kes80]{Kes80}
H.~Kesten.
\newblock The critical probability of bond percolation on the square lattice
  equals {${1\over 2}$}.
\newblock {\em Comm. Math. Phys.}, 74(1):41--59, 1980.

\bibitem[Kes86]{Kes86}
H.~Kesten.
\newblock The incipient infinite cluster in two-dimensional percolation.
\newblock {\em Probab. Theory Related Fields}, 73(3):369--394, 1986.

\bibitem[KKL88]{KahKalLin88}
Jeff~D. Kahn, Gil Kalai, and Nathan Linial.
\newblock The influence of variables on boolean functions.
\newblock {\em [Proceedings 1988] 29th Annual Symposium on Foundations of
  Computer Science}, pages 68--80, 1988.

\bibitem[KS05]{KenSch05}
Richard Kenyon and Jean-Marc Schlenker.
\newblock Rhombic embeddings of planar quad-graphs.
\newblock {\em Trans. Amer. Math. Soc.}, 357(9):3443--3458 (electronic), 2005.

\bibitem[KST23]{KSTas23}
Laurin K{\"o}hler-Schindler and Vincent Tassion.
\newblock Crossing probabilities for planar percolation.
\newblock {\em Duke Mathematical Journal}, 172(4):809--838, 2023.

\bibitem[Lin02]{Lin02}
T.~Lindvall.
\newblock {\em Lectures on the Coupling Method}.
\newblock Dover Books on Mathematics Series. Dover Publications, Incorporated,
  2002.

\bibitem[Men86]{Men86}
M.~V. Menshikov.
\newblock Coincidence of critical points in percolation problems.
\newblock {\em Dokl. Akad. Nauk SSSR}, 288(6):1308--1311, 1986.

\bibitem[Mer01]{Mer01}
Christian Mercat.
\newblock Discrete {R}iemann surfaces and the {I}sing model.
\newblock {\em Comm. Math. Phys.}, 218(1):177--216, 2001.

\bibitem[Oul22]{Oul22}
Mendes Oulamara.
\newblock {\em Random geometry and free energy of critical planar lattice
  models}.
\newblock PhD thesis, 2022.
\newblock 2022UPASM008.

\bibitem[PR15]{PelRom15}
{R.} Peled and {D.} Romik.
\newblock Bijective combinatorial proof of the commutation of transfer matrices
  in the dense {$O(1)$} loop model.
\newblock {\em S{\'e}minaire Lotharingien de Combinatoire}, 73(B73b), 2015.

\bibitem[RS20]{RaySpi20}
Gourab Ray and Yinon Spinka.
\newblock A short proof of the discontinuity of phase transition in the planar
  random-cluster model with $q>4$.
\newblock {\em Communications in Mathematical Physics}, 378(3):1977--1988,
  2020.

\bibitem[Rus78]{Rus78}
L.~Russo.
\newblock A note on percolation.
\newblock {\em Z. Wahrscheinlichkeitstheorie und Verw. Gebiete}, 43(1):39--48,
  1978.

\bibitem[She09]{She09}
Scott Sheffield.
\newblock Exploration trees and conformal loop ensembles.
\newblock {\em Duke Math. J.}, 147(1):79--129, 2009.

\bibitem[Smi01]{Smi01}
Stanislav Smirnov.
\newblock Critical percolation in the plane: conformal invariance, {C}ardy's
  formula, scaling limits.
\newblock {\em C. R. Acad. Sci. Paris S\'{e}r. I Math.}, 333(3):239--244, 2001.

\bibitem[Smi10]{Smi10}
Stanislav Smirnov.
\newblock Conformal invariance in random cluster models. {I}. {H}olomorphic
  fermions in the {I}sing model.
\newblock {\em Ann. of Math. (2)}, 172(2):1435--1467, 2010.

\bibitem[SW78]{SeyWel78}
P.~D. Seymour and D.~J.~A. Welsh.
\newblock Percolation probabilities on the square lattice.
\newblock {\em Ann. Discrete Math.}, 3:227--245, 1978.
\newblock Advances in graph theory (Cambridge Combinatorial Conf., Trinity
  College, Cambridge, 1977).

\bibitem[SW12]{SheWer12}
Scott Sheffield and Wendelin Werner.
\newblock Conformal loop ensembles: the {M}arkovian characterization and the
  loop-soup construction.
\newblock {\em Ann. of Math. (2)}, 176(3):1827--1917, 2012.

\bibitem[Van22]{Van22}
Hugo Vanneuville.
\newblock Sharpness of {B}ernoulli percolation via couplings.
\newblock 2022.
\newblock preprint arXiv:2201.08223.

\bibitem[Wer04]{Wer04}
Wendelin Werner.
\newblock Random planar curves and {S}chramm-{L}oewner evolutions.
\newblock In {\em Lectures on probability theory and statistics}, volume 1840
  of {\em Lecture Notes in Math.}, pages 107--195. Springer, Berlin, 2004.

\bibitem[Wer09a]{Wer09}
Wendelin Werner.
\newblock Lectures on two-dimensional critical percolation.
\newblock In {\em Statistical mechanics}, volume~16 of {\em IAS/Park City Math.
  Ser.}, pages 297--360. Amer. Math. Soc., Providence, RI, 2009.

\bibitem[Wer09b]{Wer09a}
Wendelin Werner.
\newblock {\em Percolation et mod{\`e}le d'{I}sing}, volume~16 of {\em Cours
  Sp\'{e}cialis\'{e}s [Specialized Courses]}.
\newblock Soci\'{e}t\'{e} Math\'{e}matique de France, Paris, 2009.

\end{thebibliography}
\end{document}